\documentclass[reqno,11pt]{amsart}

\usepackage{amscd}
\usepackage{amsmath}
\usepackage{amssymb}
\usepackage[american]{babel}
\usepackage{bbm}
\usepackage{bookmark}
\usepackage{cmap} 
\usepackage{dsfont}
\usepackage{enumerate}
\usepackage{epigraph}
\usepackage[mathscr]{euscript}
\usepackage[myheadings]{fullpage}
\usepackage{graphicx}
\usepackage[geometry]{ifsym}
\usepackage{mathabx}
\usepackage{accents}
\usepackage{mathrsfs}
\usepackage{mathtools}
\usepackage{refcount}
\usepackage{setspace}
\usepackage{stmaryrd}
\usepackage{thinsp}
\usepackage{verbatim}
\usepackage[all,2cell]{xy} \UseTwocells
\usepackage{xr}
\usepackage[multiple]{footmisc}
\usepackage{hyperref}

\numberwithin{equation}{subsection}

\newtheorem{thm}{Theorem}[subsection]
\newtheorem*{thm*}{Theorem}

\newtheorem{cor}[thm]{Corollary}
\newtheorem*{cor*}{Corollary}
\newtheorem{lem}[thm]{Lemma}

\newtheorem{prop}[thm]{Proposition}
\newtheorem{prop-const}[thm]{Proposition-Construction}

\newtheorem*{conjecture*}{Conjecture}

\newtheorem*{princ*}{Principle}

\theoremstyle{remark}
\newtheorem{rem}[thm]{Remark}
\newtheorem{example}[thm]{Example}

\newtheorem{defin}[thm]{Definition}

\newtheorem{notation}[thm]{Notation}
\newtheorem{warning}[thm]{Warning}

\newtheorem{convention}[thm]{Convention}

\newcounter{steps}[thm]
\newtheorem{step}[steps]{Step}

\newcommand{\preprime}{\prescript{\prime}{}}
\newcommand{\presup}[1]{\prescript{#1}{}\!}

\newcommand{\xar}[1]{\xrightarrow{#1}}
\newcommand{\rar}[1]{\xar{#1}}
\newcommand{\isom}{\rar{\simeq}}

\newcommand{\rightrightrightarrows}{%
        \mathrel{\vcenter{\mathsurround0pt
                \ialign{##\crcr
                        \noalign{\nointerlineskip}$\rightarrow$\crcr
                        \noalign{\nointerlineskip}$\rightarrow$\crcr
                        \noalign{\nointerlineskip}$\rightarrow$\crcr
                }%
        }}%
}

\newcommand{\into}{\hookrightarrow}
\newcommand{\onto}{\twoheadrightarrow}

\newcommand{\la}{\lambda}
\newcommand{\La}{\Lambda}

\newcommand{\bDelta}{\mathbf{\Delta}}

\newcommand{\bA}{{\mathbb A}}
\newcommand{\bB}{{\mathbb B}}

\newcommand{\bD}{{\mathbb D}}

\newcommand{\bG}{{\mathbb G}}

\newcommand{\bP}{{\mathbb P}}
\newcommand{\bQ}{{\mathbb Q}}

\newcommand{\bV}{{\mathbb V}}
\newcommand{\bW}{{\mathbb W}}

\newcommand{\bZ}{{\mathbb Z}}

\newcommand{\cD}{{\mathcal D}}

\newcommand{\cI}{{\mathcal I}}

\newcommand{\cK}{{\mathcal K}}

\newcommand{\sA}{{\EuScript A}}
\newcommand{\sB}{{\EuScript B}}
\newcommand{\sC}{{\EuScript C}}
\newcommand{\sD}{{\EuScript D}}

\newcommand{\sF}{{\EuScript F}}
\newcommand{\sG}{{\EuScript G}}
\newcommand{\sH}{{\EuScript H}}

\newcommand{\sK}{{\EuScript K}}
\newcommand{\sL}{{\EuScript L}}

\newcommand{\sO}{{\EuScript O}}
\newcommand{\sP}{{\EuScript P}}

\newcommand{\sW}{{\EuScript W}}

\newcommand{\fb}{{\mathfrak b}}

\newcommand{\fg}{{\mathfrak g}}
\newcommand{\fh}{{\mathfrak h}}

\newcommand{\fk}{{\mathfrak k}}
\newcommand{\fl}{{\mathfrak l}}

\newcommand{\fn}{{\mathfrak n}}

\newcommand{\fs}{{\mathfrak s}}
\newcommand{\ft}{{\mathfrak t}}

\newcommand{\fz}{{\mathfrak z}}

\newcommand{\on}{\operatorname}
\newcommand{\ol}[1]{\overline{#1}{}}
\newcommand{\ul}{\underline}

\newcommand{\mathendash}{\text{\textendash}}

\newcommand{\ldotsplus}{\mathinner{\ldotp\ldotp\ldotp\ldotp}}

\newcommand{\Coker}{\on{Coker}}
\newcommand{\End}{\on{End}}
\newcommand{\Hom}{\on{Hom}}
\newcommand{\Ext}{\on{Ext}}

\newcommand{\Spec}{\on{Spec}}
\newcommand{\Spf}{\on{Spf}}

\newcommand{\id}{\on{id}}

\newcommand{\Ad}{\on{Ad}}

\newcommand{\ind}{\on{ind}}

\newcommand{\Rep}{\mathsf{Rep}}
\newcommand{\gr}{\on{gr}}

\newcommand{\act}{\on{act}}

\newcommand{\coact}{\on{coact}}

\renewcommand{\dot}{\bullet}

\newcommand{\Fun}{\on{Fun}}

\newcommand{\vph}{\varphi}

\newcommand{\Vect}{\mathsf{Vect}}
\newcommand{\Fil}{\mathsf{Fil}\,	}
\newcommand{\BiFil}{\mathsf{BiFil}\,}

\newcommand{\Gr}{\on{Gr}}
\newcommand{\Spr}{\on{Spr}}
\newcommand{\Whit}{\mathsf{Whit}}

\newcommand{\Sym}{\on{Sym}}

\newcommand{\LocSys}{\on{LocSys}}

\newcommand{\Op}{\on{Op}}

\renewcommand{\mod}{\mathendash\mathsf{mod}}
\newcommand{\comod}{\mathendash\mathsf{comod}}

\newcommand{\Fl}{\on{Fl}}
\newcommand{\sFl}{\Fl^{\sinf}}

\newcommand{\sinf}{\!\frac{\infty}{2}} 

\newcommand{\colim}{\on{colim}}

\newcommand{\DGCat}{\mathsf{DGCat}}
\newcommand{\ShvCat}{\mathsf{ShvCat}}

\renewcommand{\lim}{\on{lim}}
\newcommand{\Ind}{\mathsf{Ind}}
\newcommand{\Pro}{\mathsf{Pro}}

\newcommand{\Tot}{\on{Tot}}
\newcommand{\heart}{\heartsuit}

\newcommand{\Oblv}{\on{Oblv}}
\newcommand{\Av}{\on{Av}}

\newcommand{\ren}{\mathendash ren}

\newcommand{\ld}{\check}

\newcommand{\IndCoh}{\mathsf{IndCoh}}
\newcommand{\QCoh}{\mathsf{QCoh}}

\newcommand{\Alg}{\mathsf{Alg}}

\newcommand{\LieAlg}{\mathsf{LieAlg}}
\newcommand{\Lie}{\on{Lie}}

\renewcommand{\o}[1]{\accentset{\circ}{#1}{}}

\renewcommand{\subset}{\subseteq}

\renewcommand{\sl}{\fs\fl}

\makeatletter
\newcommand{\biggg}{\bBigg@{4}}
\newcommand{\Biggg}{\bBigg@{5}}
\makeatother

\date{\today}

\begin{document}

\frenchspacing

\setlength{\epigraphwidth}{0.4\textwidth}
\renewcommand{\epigraphsize}{\footnotesize}

\title{$\sW$-algebras and Whittaker categories}

\author{Sam Raskin}

\address{Massachusetts Institute of Technology, 
77 Massachusetts Avenue, Cambridge, MA 02139.} 
\email{sraskin@mit.edu}

\begin{abstract}

Affine $\sW$-algebras are a somewhat complicated family of (topological)
associative algebras associated with a semisimple Lie algebra, quantizing
functions on the algebraic loop space of Kostant's slice. They have
attracted a great deal of attention in geometric
representation theory because of Feigin-Frenkel's duality
theorem for them, which identifies $\sW$-algebras for a Lie algebra and
its Langlands dual through a subtle construction.

The main result of this paper is an affine version of Skryabin's theorem,
describing the category of modules over the $\sW$-algebra in simpler
categorical terms. But unlike the classical story, 
it is essential to work with derived categories in the affine setting.

One novel feature is the use of geometric techniques to
study $\sW$-algebras: the theory of $D$-modules on the loop group
and the geometry of the affine Grassmannian are indispensable tools. 
These are used to give an infinite family
of affine analogues of the Bezrukavnikov-Braverman-Mirkovic theorem, providing a geometric version of Rodier's
compact approximation to the Whittaker model from the arithmetic setting.
We also use these methods to generalize
Beraldo's theorem identifying Whittaker invariants and Whittaker coinvariants, extending his result
from $GL_n$ to a general reductive group. 
At integral level, these methods seem to have deep 
intrinsic meaning in the local geometric Langlands program. 

The theory developed here provides systematic proofs
of many classical results in the subject.
In particular, we clarify the
exactness properties of the quantum 
Drinfeld-Sokolov functor. 

\end{abstract}

\dedicatory{For Arthur}

\maketitle

\setcounter{tocdepth}{1}
\tableofcontents

\section{Introduction}

\subsection{} This paper is about generalizing Skryabin's theorem, a simple
result about \emph{finite} $\sW$-algebras,\footnote{One finds competing explanations for the name in the
literature. 

The implicit connection to Whittaker models from harmonic analysis
has been suggested repeatedly, e.g. in \cite{kac-desole}.
The explanation makes a lot of sense, since the subject was started in \cite{kostant-whittaker}; here
the word ``Whittaker" in the title is explicitly meant to evoke the above meaning.

Arakawa \cite{arakawa-report} suggests the 
name comes because affine $\sW$-algebras generalize
the Virasoro algebra, and \emph{W} succeeds \emph{V} in the alphabet. 

De Sole and Kac also suggest in \cite{kac-desole} \S 0.2 that because $\sW$-algebras
quantize functions on the space of invariant polynomials of the group,
which can be thought of as invariant polynomials for the Weyl group
considered as a Chevalley group, that the name derives from \emph{Weyl}. 

I tried to hunt the answer down in the literature, with only partial success.
For $\sl_3$, the affine $\sW$-algebra has two 1-parameter families of
generators; one family has to do with Virasoro, so is denoted 
$L_n$ by standard tradition. In Zamolodchikov's first paper \cite{zamolodchikov}
on the subject, which introduced the affine $\sW$-algebra for $\sl_3$,
he denotes the other family by $V_n$. Howeover,
in his \emph{second} paper \cite{zamolodchikov-fateev}
on the subject, joint with Fateev, the second family is denoted by $W_n$.

As far as I could tell, the name originates from this choice of notation
in the second paper. I do not know what this choice was made. The connection
to Virasoro was transparent at that time, but I am not sure
about the connection to Whittaker models and Kostant's work. I'm not even
sure that the connection to $\sl_3$ would have been clear yet.}
to the more subtle setting of 
\emph{affine} $\sW$-algebras. 

The main new construction of the paper is of a more general nature.
It provides a \emph{compact approximation} to the 
\emph{Whittaker model}, which
corresponds under the conjectural
local geometric Langlands correspondence to the
stratification of the moduli space of de Rham local systems
by slope. This method, which we call the
\emph{adolescent Whittaker} construction, is a new
one, though closely related to \cite{rodier}. 
It appears to be fundamental in geometric Langlands, 
and may be of interest to specialists
in the Langlands program without an interest in $\sW$-algebras.

But in what follows, we emphasize the role of $\sW$-algebras, which
are ultimately the main players in this paper. We provide a survey
of the subject below. The reader with no interest in this part of the paper
can safely skip that material.

\subsection{Some notation}

We work over a ground field $k$ of characteristic $0$ throughout the paper.

Let $\fg$ be a reductive Lie algebra. Let $\fn$ be the radical of a Borel $\fb$.
We fix Chevalley generators $e_i \in \fn$. Let $\psi: \fn \to k$ be 
the \emph{non-degenerate} character with $\psi(e_i) = 1$ for all $i$.

Throughout the paper, notation is \emph{always} assumed derived
(see \S \ref{ss:minus-infty} for an explanation why).
Our derived categories are considered as DG categories.\footnote{This
means that they are enriched over chain complexes of vector
spaces (and satisfy a few additional hypotheses).
However, this notion should be considered in the homotopic sense,
i.e., as $\infty$-categories in the sense of Lurie, c.f. \cite{dgcat}.}
They all admit arbitrary (small) colimits (equivalently, arbitrary direct sums), i.e., they are \emph{cocomplete}.
They are considered as objects of $\DGCat_{cont}$, the $\infty$-category of
cocomplete DG categories and continuous functors (i.e., functors preserving all colimits) . 
 
E.g., $\Vect$ means the DG category of chain complexes of ($k$-)vector spaces. Similarly, 
$\fg\mod$ denotes the DG category of chain complexes of $\fg$-modules.
We use the notation $\Vect^{\heart}$, $\fg\mod^{\heart}$, etc.
to refer to the usual abelian categories (since these
are the hearts of standard $t$-structures on the DG categories).

We use $\ul{\Hom}(-,-)$ to denote the chain complex of maps in
a DG category.

\subsection{Finite $\sW$-algebras}

We begin by describing what $\sW$-algebras are in the finite-dimensional
setting, and what about them we wish to generalize.

\subsection{}

We have the \emph{finite Drinfeld-Sokolov} functor:

\[
\Psi^{fin}: \fg\mod \to \Vect
\]

\noindent defined by: 

\[
M \mapsto C^{\dot}(\fn,M \otimes -\psi)
\]

\noindent where $C^{\dot}$ indicates the cohomological Chevalley complex (i.e., Lie algebra cohomology),
and $-\psi$ is abusive notation for
the 1-dimensional $\fn$-module defined by the character $-\psi$ (the reason for the sign will be apparent later).
 
The non-derived version of this functor was introduced in 
\cite{kostant-whittaker}, where its basic properties were
established.\footnote{Almost established, in any case.
One often finds this source cited for results which are not
proved there, but whose proofs can readily be extracted from it.}

Define the DG algebra 
$\sW^{fin}$ as the \emph{endomorphisms} of this
functor.\footnote{Explicitly, this means
we take $\ul{\End}_{\fg\mod}(\Psi^{fin}(U(\fg)))$,
where $\Psi^{fin}(U(\fg))$ is regarded as a $\fg$-module
through the bimodule structure on $U(\fg)$. 

More explicitly still, note that 
$C_{\dot}(\fn,-) = C^{\dot}(\fn,-)[\dim \fn] \otimes \det(\fn)$
(with $C_{\dot}$ being Lie algebra homology),
so 
$\Psi^{fin}(U(\fg)) = \ind_{\fn}^{\fg}(-\psi)[-\dim \fn] \otimes \det(\fn)^{\vee}$
(the sign occurs in switching between right and left actions).
So we compute that
$\ul{\End}_{\fg\mod}(\Psi^{fin}(U(\fg)))$ is
$C^{\dot}\big(\fn,C_{\dot}(\fn, U(\fg) \otimes -\psi) \otimes \psi\big)$.}
One has:

\begin{thm}[$\approx$ Kostant \cite{kostant-whittaker}]\label{t:w-fin-facts}

\begin{enumerate}

\item $\sW^{fin}$ is concentrated in cohomological degree $0$,
i.e., $\sW^{fin} = H^0(\sW^{fin})$.

\item $\sW^{fin}$ carries a canonical filtration whose
associated graded is \emph{slightly non-canonically}
isomorphic to the algebra of functions on
the Kostant slice $f+\fb^e \simeq f+\fb/N \simeq \fg//G$.
Here $f$ is a principal nilpotent element related to
$\psi$, $e$ fits into a principal $\sl_2$ with $f$ with
$[e,f] \in \fb$. We use the quotient symbol $/$ to indicate
the stack quotient (which happens to be an
affine scheme in this case), and $//$ to indicate the GIT quotient.

This isomorphism is completely determined by
a choice of $\Ad$-invariant isomorphism $\fg \simeq \fg^{\vee}$.
Then for $\pi:\fg \to \fg^{\vee} \to \fn^{\vee}$, $f$ should
be the unique nilpotent element in $\pi^{-1}(\psi)$.
 
\end{enumerate}

\end{thm}

The proof is quick using \emph{Kazhdan-Kostant} filtrations, 
c.f. \cite{kostant-whittaker} \S 1-2, \cite{gan-ginzburg} \S 4,
or \S \ref{ss:kk-start} from the appendix of the present paper.

\begin{rem}

In fact, it is straightforward to show using these methods that
the canonical map $Z(\fg) \to \sW^{fin}$ is an isomorphism;
in particular, $\sW^{fin}$ is commutative. 
(This is also all but proved in \cite{kostant-whittaker}.)

We encourage the
reader to forget this fact as much as possible. The
affine $\sW$-algebras are not (usually) commutative. 
There is a more subtle point as well: this identification is not
true \emph{derivedly}, i.e., there are non-vanishing 
higher Hochschild cohomology groups for $U(\fg)$.
(From this perspective, the algebra $\sW^{fin}$ can be thought of
as a \emph{construction} of the \emph{usual} (non-derived) 
center of $U(\fg)$ adapted to derived settings.)

\end{rem}

\subsection{Skryabin's theorem}\label{ss:fin-skry}

One has the following description of the category of modules
over $\sW^{fin}$.

Let $\fg\mod^{N,\psi} \subset \fg\mod$ denote the full subcategory
consisting of twisted Harish-Chandra modules, i.e,. the
full subcategory consisting of complexes on whose cohomologies
the operators $x-\psi(x)$ act locally nilpotently for every $x \in \fn$.

\begin{thm}[Skryabin's theorem]\label{t:fin-skryabin}

There is a canonical $t$-exact equivalence of DG categories
$\fg\mod^{N,\psi} \isom \sW^{fin}\mod$ fitting into a commutative 
diagram:

\[
\xymatrix{
\fg\mod^{N,\psi} \ar[rr]^{\simeq} \ar[dr]^{\Psi^{fin}} && 
\sW^{fin}\mod \ar[dl]_{\Oblv} \\
& \Vect.
}
\]

\noindent Here $\Oblv$ denotes the forgetful functor.

\end{thm}

The proof is easy: the induced module 
$\ind_{\fn}^{\fg}(\psi) \in \fg\mod^{N,\psi}$
is a compact generator of this DG category
(it generates because $N$ is unipotent), so by general
DG category principles, we have an equivalence:

\[
\ul{\Hom}_{\fg\mod^{N,\psi}}(\ind_{\fn}^{\fg}(\psi),-):
\fg\mod^{N,\psi} \simeq \End_{\fg\mod^{N,\psi}}(\ind_{\fn}^{\fg}(\psi))\mod.
\]

\noindent By definition, this functor is $\Psi^{fin}$,
and $\sW^{fin}$ was defined as these endomorphisms.

\begin{rem}

In particular, we see that the \emph{only} role played by the 
non-degeneracy of the character
$\psi$ here is to make $\sW^{fin}$ a classical (i.e., non-DG)
associative algebra. The result remains true in general, as long
as we systematically work in the DG setting.
This will not be the case anymore once we pass to the affine setting.

\end{rem}

\subsection{Some references}

There are a many good places to learn more about finite $\sW$-algebras,
and we name just a few here for the reader's convenience. 
However, we note that many authors are interested in a more subtle
generalization taking an arbitrary (i.e., possibly non-principal) 
nilpotent element in $\fg$ as input.

The original source is \cite{kostant-whittaker}, as noted above,
and it contains most of the ideas indicated above.
The general definition may be found in \cite{premet},
whose appendix by Skryabin contains the original proof
of Theorem \ref{t:fin-skryabin}.

There are many convenient surveys, e.g. 
\cite{losev-survey} and \cite{wang}. See also 
\cite{arakawa-survey} and \cite{kac-desole}
for treatments that emphasize the affine point of view as well.

\subsection{Affine $\sW$-algebras}

Here $\fg$ is replaced by the loop algebra
$\fg((t)) \coloneqq \fg \otimes_k k((t))$.

More generally, recall that an $\Ad$-invariant bilinear form
$\kappa$ defines the \emph{affine Kac-Moody} algebra
$\widehat{\fg}_{\kappa}$,
which is a central extension of $\fg((t))$ by the abelian 1-dimensional
Lie algebra $k$. The form $\kappa$ is called the \emph{level}.
We recall that the Kac-Moody cocycle vanishes on $\fg[[t]]$ 
and $\fn((t))$.
When we speak of modules for the loop/Kac-Moody algebras,
we agree that they are \emph{discrete} (or \emph{smooth}), 
i.e., every vector is
annihilated by $t^N\fg[[t]]$ for some $N$ (depending on the vector),
and that the generator of the central $k$ in the Kac-Moody algebra
acts by the identity.

\begin{rem}

The reader may ignore the level $\kappa$, reading
$\widehat{\fg}_{\kappa}$ as $\fg((t))$ everywhere, and not miss out
on much fun in this paper. 
The major downside is that the level plays a key role in 
Feigin-Frenkel duality, which is a major source of motivation here.

\end{rem}

In the affine theory, the role of 
$\fn$ is then replaced by $\fn((t))$, and the character
$\psi:\fn \to k$ is replaced by the (conductor $0$) \emph{Whittaker}
character:

\[
\fn((t)) \to \fn/[\fn,\fn] ((t)) = \oplus_{i \in \cI_G} k((t)) \xar{\sum} k((t)) 
\xar{\on{Res}} k
\]

\noindent where $\on{Res}$ is the residue with respect to 
$dt$\footnote{It
more general settings relevant for global geometric Langlands, 
it can be important to remove the choice of $1$-form by incorporating
twists. C.f. \cite{cpsi} or \cite{fgkv}.} 
and $\cI_G$ is the set of simple roots (alias: vertices of the Dynkin diagram of $G$).
Abusing notation, we let $\psi:\fn((t)) \to k$ denote the corresponding character.

\begin{rem}

Affine $\sW$-algebras were introduced in mathematical
physics by Zamolodchikov \cite{zamolodchikov}
in the case $\fg = \sl_3$.  (The above perspective via
Lie algebras and quantum Hamiltonian reduction 
is found in Feigin-Frenkel \cite{ff-ds}
and its antecedents, which include
\cite{drinfeld-sokolov} and \cite{bershadsky-ooguri} and others.)
We refer to \cite{w-symmetry} for further discussion of the
role of $\sW$-algebras in physics.

\end{rem}

\begin{rem}

The relevance in local geometric Langlands 
is that for $K = k((t))$, the loop group\footnote{The author
prefers to denote loop Lie algebras as $\fg((t))$ and loop Lie groups
as $G(K)$.} $G(K)$ plays the role of the $p$-adic reductive group
in usual harmonic analysis. Then $\fg((t))$ is its Lie algebra, and it
is not surprising that one accesses the group through its Lie algebra.
In fact, in the Frenkel-Gaitsgory philosophy (\cite{fg2}, \cite{frenkel-book}), the DG category of modules
$\fg((t))\mod$ (or in truth, the Kac-Moody representations
of \emph{critical} level)
is regarded as an important canonical \emph{categorified} representation
of $G(K)$. The goal of the affine Skryabin theorem below is to explain
that its Whittaker model is closely tied to affine $\sW$-algebras.

\end{rem}

\subsection{}

Below, we use $\fg((t))\mod$ and $\widehat{\fg}_{\kappa}\mod$
to denote the DG categories of modules over the loop and
Kac-Moody algebras. These categories are defined in
\cite{dmod-aff-flag}, and there are subtle points that we discuss
in \S \ref{ss:minus-infty}.
 
However, for now the reader may ignore these points, and we will discuss
them in what follows. It is enough to know that they have $t$-structures
with hearts the relevant abelian categories of (discrete) modules,
and that they are almost-but-not-quite 
the derived categories of these abelian categories.

\subsection{}

We have not actually defined the affine $\sW$-algebra above,
and we will not quite do this below.
It is not so easy as in the finite situation because of the use of topological
algebras, which are well-adapted to abelian categories 
but not so well to DG categories.
Because (semi-infinite) Lie algebra cohomology plays such a key role,
(even in the finite case), one has to be careful balancing abelian and
derived categories. (This is 
one of the major technical issues
settled by the affine Skryabin theorem.)

But still, let us review the major features in what follows.

\subsection{Semi-infinite cohomology}

One has the \emph{semi-infinite} cohomology functor:

\[
H^{\sinf}(\fn((t)),\fn[[t]],-):\fn((t))\mod \to \Vect.
\]

\noindent 
We refer to \S \ref{ss:sinf-intro} in the
appendix for a simple construction via Lie algebra cohomology. 
We simply remark that it can be thought of as mixing
Lie algebra cohomology for $\fn[[t]]$ with Lie algebra homology
for $\fn((t))/\fn[[t]]$, except that the latter does not make sense.
(Properly, one takes a direct limit over cohomologies for an increasing 
sequence of compact open subalgebras.)

We obtain the (quantum) \emph{Drinfeld-Sokolov} functor:

\[
\Psi = H^{\sinf}(\fn((t)),\fn[[t]],(-) \otimes -\psi):
\widehat{\fg}_{\kappa}\mod \to \Vect.
\]

\noindent (Here we recall that there is a forgetful
functor $\widehat{\fg}_{\kappa}\mod \to \fn((t))\mod$
because the Kac-Moody cocycle vanishes
on $\fn((t)) \subset \fg((t))$.)

\subsection{}

The $\sW$-algebra at level $\kappa$ comes in two
guises: as a vertex algebra $\sW_{\kappa}$
and as a topological associative algebra
(in the sense of \cite{chiral}) $\sW_{\kappa}^{as}$.
We remind that
whatever a vertex algebra structure is, it allows one to construct
a topological associative algebra, and this is the relationship between
the two constructions above.
The algebra $\sW_{\kappa}^{as}$ is analogous to the completion of the 
enveloping algebra of $\fg((t))$ defined by the
left ideals generated by $t^N\fg[[t]]$ over all $N \geq 0$;
note that discrete modules for this algebra are the same as discrete
modules for $\fg((t))$ as defined before, and there is a similar
relationship here.

The vertex algebra $\sW_{\kappa}$ is easier to define: 
$\sW_{\kappa} = \Psi(\bV_{\kappa}) (= H^0 \Psi(\bV_{\kappa}))$,
where 
$\bV_{\kappa} = \ind_{\fg[[t]]}^{\widehat{\fg}_{\kappa}}(k) \in
\widehat{\fg}_{\kappa}\mod^{\heart}$ is the \emph{vacuum}
representation, which we remind is a vertex algebra with the
same abelian category of modules as the Kac-Moody algebra itself. 

Either perspective defines the same abelian category
of modules, and we denote it by $\sW_{\kappa}\mod^{\heart}$.
The relationship to the above is that the cohomologies
of $\Psi(M)$ for any $M \in \widehat{\fg}_{\kappa}\mod$ 
are acted on by the $\sW$-algebra (at level $\kappa$).

\begin{rem}

We refer to \cite{fbz} \S 15 and \cite{chiral} \S 3.8 for more complete 
discussion on the constructions. See also \S \ref{ss:w-recollection} below.

\end{rem}

\subsection{}

We now review the major properties of affine $\sW$-algebras
for the reader's convenience. 
Note the parallel with Theorem \ref{t:w-fin-facts}.

\begin{thm}\label{t:w-aff-facts}

\begin{enumerate}

\item $\sW_{\kappa}$ is concentrated in cohomological degree $0$,
i.e., $\sW_{\kappa} = H^0\Psi(\bV_{\kappa})$.

\item $\sW_{\kappa}$ and $\sW_{\kappa}^{as}$ carry
canonical filtrations whose
associated graded are \emph{slightly non-canonically}
isomorphic to the algebra of functions on
the affine Kostant slices 
$f+\fb^e[[t]] \simeq f+\fb[[t]]/N(O) \simeq (\fg//G)(O)$ and
 $f+\fb^e((t)) \simeq f+\fb((t))/N(K) \simeq (\fg//G)(K)$
respectively.

These isomorphisms are completely determined by
a choice of $\Ad$-invariant isomorphism $\fg \simeq \fg^{\vee}$
and the non-vanishing 1-form $dt$ on the formal disc.

\item (Feigin-Frenkel duality, \cite{feigin-frenkel-duality}.)
There is a canonical vertex algebra isomorphism:

\[
\sW_{\fg,\kappa} \simeq \sW_{\ld{\fg},\ld{\kappa}}
\]

\noindent where the notation indicates the affine $\sW$-algebras
for $\fg$ and for the Langlands dual Lie algebra. 
The construction of the 
dual level $\ld{\kappa}$ is reviewed in \S \ref{ss:ff-duality-intro}. 
(Note Warning \ref{w:ff-ratl}.)
 
\end{enumerate}

\end{thm}

\begin{rem}

The proofs of the first two results 
are similar to the finite-dimensional case up to
a subtlety about the convergence of a spectral sequence.
The result is due originally to \cite{dbt}. See also
\cite{fbz} Chapter 15, \cite{arakawa-survey}. 
(Below, \S \ref{s:ds} provides a slightly non-standard approach to
dealing with the convergence of the spectral sequence.)

\end{rem}

\begin{rem}

The vertex algebra $\sW_{\kappa}$ is filtered \emph{as a vertex algebra},
and the associated graded is commutative here, so it makes
sense to speak of an algebra structure on it.

\end{rem}

\begin{rem}

For $\kappa = \kappa_{crit}$ \emph{critical}, i.e., $-\frac{1}{2}$ times the 
Killing form for $\fg$, $\sW_{crit} \coloneqq \sW_{\kappa_{crit}}$
(resp. $\sW_{crit}^{as}$) is commutative and coincides 
with $\bV_{crit}^{G(O)}$ (resp. the center of the topological enveloping
algebra of $\widehat{\fg}_{crit}$). Here the Feigin-Frenkel duality
relates the center with opers for the Langlands dual group.

At other levels, the affine
$\sW$-algebra is non-commutative (e.g., it contains a copy of
the Virasoro algebra).

\end{rem}

\begin{rem}

The Feigin-Frenkel duality is highly suggestive in local geometric
Langlands. We refer to \cite{fg2} and \cite{frenkel-book} for a discussion
of its role in the traditional subject, and to 
\cite{quantum-langlands-summary} for a formulation 
of \emph{quantum} local geometric Langlands,
which is a general conjectural framework that helps to explain the
meaning of the Feigin-Frenkel result.

\end{rem}

\subsection{Affine Skryabin, group actions, and Whittaker}

Frenkel-Gaitsgory proposed a framework for local geometric Langlands
based on the idea of \emph{group actions on categories} for
group indschemes like $G(K)$. At a certain point, it became clear that
the natural class of categories to be acted on are cocomplete DG categories.

For derived categories of abelian categories, an appropriate substitute
for the general theory was developed in the appendices to \cite{fg2}.
But this theory is inadequate for the general local Langlands framework, 
and Gaitsgory and his collaborators have spent a number of years 
and expended a great deal of energy developing the 
necessary language and methods here.

His ideas about actions of group indschemes (such as $G(K)$ or $N(K)$)
on DG categories were developed in detail in \cite{dario-*/!}, and we
refer the reader there to learn this material.
In particular, for $\sC \in \DGCat_{cont}$ acted on (strongly) by $G(K)$
(or the twisted version of this notion 
incorporating the level $\kappa$),
one can form the \emph{Whittaker category} $\Whit(\sC) \in \DGCat_{cont}$
as \emph{invariants} (or coinvariants: see \S \ref{ss:whit?})
for $N(K)$ with respect to the character
$\psi:N(K) \to \bG_a$.

Although the definitions in this theory are simple,
the idea to take this very infinite-dimensional (and potentially
quite pathological) construction seriously was a quite nontrivial one.
This breakthrough was made by Gaitsgory in
his notes \cite{lesjours}, where
he showed several quite nontrivial results, including a comparison
with an ad hoc definition with good properties used in
\cite{fgv}. 
These ideas were further advanced by Beraldo in \cite{dario-*/!}, who 
extended Gaitsgory's results and thereby gave more evidence
that the Whittaker construction is a good one. 
(The most psychologically difficult points about this formalism
are highlighted below in \S \ref{ss:minus-infty}.)

We comment more on the Whittaker construction in \S \ref{ss:whit?}. 
For now, we merely state: 

\begin{itemize}

\item Basic examples of $\sC$ acted on by $G(K)$ are
$D$-modules on a suitably nice space acted on by
$G(K)$, or more relevantly for us, $\widehat{\fg}_{\kappa}\mod$.

\item One can fairly write
$\Whit(\sC) = \sC^{N(K),\psi}$
in such a way that the general formalism in the finite-dimensional
setting would produce 
$\fg\mod^{N,\psi}$ from \S \ref{ss:fin-skry}.

\item There are canonical functors $\Whit(\sC) \to \sC$ and
$\sC \to \Whit(\sC)$.

However, they do \emph{not} satisfy any adjunction.
This reflects of the infinite-dimensional nature of $N(K)$,
specifically that it is an indscheme but not a scheme.
The functor $\Whit(\sC) \to \sC$ should be regarded
as the forgetful functor. The functor $\sC \to \Whit(\sC)$
can be informally thought of as ``$\Av_*^{\psi}$ plus
an infinite cohomological shift," although
neither $\Av_*^{\psi}$ nor the infinite cohomological shift
make sense.

\end{itemize}

\subsection{}

We can now formulate the main theorem of this paper:

\begin{thm*}[Affine Skryabin theorem, Thm. \ref{t:aff-skry}]

There is a canonical equivalence 
$\Whit(\widehat{\fg}_{\kappa}\mod) \simeq \sW_{\kappa}\mod$
such that:

\[
\widehat{\fg}_{\kappa}\mod \to \Whit(\widehat{\fg}_{\kappa}\mod) \to
\sW_{\kappa}\mod \to \Vect
\]

\noindent is computed by the Drinfeld-Sokolov functor $\Psi$.

\end{thm*}

As a consequence\footnote{At this point, we should note that
the category $\sW_{\kappa}\mod$ has not been properly defined:
the issues are discussed in \S \ref{ss:minus-infty} and \S \ref{ss:generators}.
For this reason, Theorem \ref{t:ff} requires somewhat more input
than just Theorem \ref{t:aff-skry}: this is the content of \S \ref{s:free-field}.}
of affine Skryabin and Feigin-Frenkel duality, we obtain:

\begin{thm*}[Categorical Feigin-Frenkel duality, Thm. \ref{t:ff}]

There is a canonical equivalence of categories:

\[
\Whit(\widehat{\fg}_{\kappa}\mod) \simeq 
\Whit(\widehat{\ld{\fg}}_{\ld{\kappa}}\mod).
\]

\end{thm*}

At critical level, this identifies 
$\Whit(\widehat{\fg}_{crit}\mod)$ with $\QCoh(\Op_{\ld{G}}(\o{D}))$,
i.e., the DG category of quasi-coherent sheaves on the
indscheme of opers for the Langlands dual group.

\begin{rem}

Categorical Feigin-Frenkel duality, especially at the critical
level, was the initial motivation for the present paper.
Although it was anticipated for a long time, 
it does not seem that there is a more
straightforward approach than the one given in this paper.

Although it may seem quite formal given usual Feigin-Frenkel duality, 
it allows one to easily convert 
results about the Whittaker model from local geometric Langlands  
to statements about Kac-Moody algebras. This allows for
simpler arguments (and many extensions) of the Frenkel-Gaitsgory
work at critical level.

\end{rem}

\begin{rem}

Each of the above theorems will come as no surprise to experts
in the area. It has long been known that something like this
must be true. For example, immediately after the relevant definitions
were given in \cite{dmod-aff-flag} and \cite{lesjours},
Gaitsgory explicitly postulated both results in 
\cite{quantum-langlands-summary}.\footnote{
However, we note that
away from the commutative cases, Theorem \ref{t:aff-skry} could 
not have been formulated as an honest conjecture because it was not
known how to define the full derived category $\sW_{\kappa}\mod$,
c.f. \S \ref{ss:generators}.}\footnote{
Perhaps each of the results above should be attributed to Gaitsgory
as conjectures.
I learned this circle of ideas as his graduate student, and am 
not completely sure 
where the boundary between folklore and his ideas is
on this particular point. In any case, the necessary 
language to formulate such a result was developed by him 
and his collaborators explicitly so that such theorems could be formulated
and proved.} 

\end{rem}

\subsection{Why wasn't this result proved sooner?}

There are several reasons, discussed in more detail in what follows:

\begin{itemize}

\item The $\Whit$ notation is a priori ambiguous. (See \S \ref{ss:whit?}.)

\item Representation theorists have not encountered objects 
of $\Whit(\widehat{\fg}_{\kappa}\mod)$ before.
(See \S \ref{ss:minus-infty}.)

\item In the proof of the classical Skryabin theorem, we compared
both sides by matching $\Ext$s (as $A_{\infty}$-algebras) 
between compact generators. Who are the compact
generators of $\Whit(\widehat{\fg}_{\kappa}\mod)$?
Who are the compact generators of $\sW_{\kappa}\mod$?\footnote{Here
the fact that $\sW_{\kappa}^{as}$ is a topological algebra and
not a discrete one rears its head. For a usual algebra,
$A$ is a compact generator of $A\mod$. For a topological algebra,
$A$ does not make sense as an object of $A\mod$ (which is not
even a clearly defined DG category in all cases, as we discuss below).}
(See \S \ref{ss:generators}.)

\end{itemize}

\begin{rem}

We wish to highlight: the finite Skryabin theorem
may lead us to believe that the affine Skryabin theorem is something
formal, which can be established one way or another by 
a brute force argument. Experience shows that this
is not the case, and that the above issues
(especially the last one) are substantial ones.

However, there is a miracle to highlight early on: all of these issues
are resolved almost at once by the \emph{adolescent Whittaker} 
construction, described briefly in \S \ref{ss:intro-adolescent}
below, and in more detail in \S \ref{s:adolescence}. 

\end{rem}

\subsection{What is the Whittaker category?}\label{ss:whit?}

The discussion here follows \cite{lesjours} and \cite{dario-*/!}.
We refer to the latter for details on the definitions and constructions.

There are two a priori candidates for $\Whit(\sC)$: invariants
and coinvariants, denoted by $\sC^{N(K),\psi}$ and 
$\sC_{N(K),\psi}$ respectively. The invariants are 
equipped with a tautological fully-faithful functor to $\sC$, and the 
coinvariants receive a canonical functor from $\sC$. Each
satisfy natural universal properties in the language of group actions
on categories.

In the formalism of group actions on categories, invariants
and coinvariants coincide when taken with respect to a
group \emph{scheme}, such as $N(O)$ (or $G(O)$). But
since $N(K)$ is a group indscheme, it is easy to
see that there is no such equivalence in general.

So there is a serious question: do we mean Whittaker
invariants or coinvariants? 

\subsection{}

There is a canonical functor $\sC_{N(K),\psi} \to \sC^{N(K),\psi}$
constructed by Gaitsgory in \cite{lesjours} (it is denoted
by $\Theta$ in \cite{dario-*/!}). 
Gaitsgory conjectured that for $\sC$ acted on by $G(K)$
(as opposed to merely by $N(K)$), this functor is an equivalence,
meaning that $\Whit(\sC)$ could be understood unambiguously.
This conjecture was shown by Gaitsgory for $G = GL_2$ and
by Beraldo for $G = GL_n$.

The first major result of this paper settles the problem 
for general reductive $G$:

\begin{thm*}[Thm. \ref{t:inv-coinv}]

For any reductive $G$ and any $\sC$ acted on by $G(K)$
(possibly with level $\kappa$),
Gaitsgory's functor
yields an equivalence $\sC_{N(K),\psi} \simeq \sC^{N(K),\psi}$.

\end{thm*}

So we unambiguously have $\Whit(\sC)$, and we have
the canonical functors $\Whit(\sC) \simeq \sC^{N(K),\psi} \to \sC$ and
$\sC \to \sC_{N(K),\psi} \simeq \Whit(\sC)$.

\begin{rem}

This result is obviously of a technical nature, and may not especially
excite the reader more interested in $\sW$-algebras than in
Whittaker categories.
But in fact, the proof, which uses the adolescent Whittaker construction,
is more interesting, and has significant implications for
affine $\sW$-algebras. This is discussed further in what follows (and
in \S \ref{s:ds}).

\end{rem}

\subsection{Where are these Whittaker modules?}\label{ss:minus-infty}

Recall that the heart of the $t$-structure on 
$\fg\mod^{N,\psi}$ consisted of $\fg$-modules $M \in \Vect^{\heart}$ 
such that $x-\psi(x)$ act locally nilpotently for every $x \in \fn$.
Moreover, $\fg\mod^{N,\psi}$ is the derived category of this
abelian category.

We might expect to interpret $\Whit(\widehat{\fg}_{\kappa}\mod)$
similarly. However, \emph{there are no discrete modules
of $\widehat{\fg}_{\kappa}$ for which
$x-\psi(x)$ acts locally nilpotently for every $x \in \fn((t))$}.

\subsection{}\label{ss:km-reps-intro} Why isn't this a contradiction of the affine Skryabin 
theorem? 

The reason lies in subtleties of the
DG category $\widehat{\fg}_{\kappa}\mod$. We use the
the (``renormalized") version of this DG category defined
in \cite{dmod-aff-flag}. We refer to the notes \cite{km-indcoh} 
for a more detailed treatment. The point is that
$\widehat{\fg}_{\kappa}\mod$ has a $t$-structure
such that the bounded below derived category
$\widehat{\fg}_{\kappa}\mod^+$ is the bounded below
derived (DG) category of the abelian category
$\widehat{\fg}_{\kappa}\mod^{\heart}$. However, there
are ``many" objects ``in cohomological degree $-\infty$,"\footnote{Formally,
this means the object lies in $\widehat{\fg}_{\kappa}\mod^{\leq -N}$
for all $N$.}
and in fact, all Whittaker objects (for $\fg$ nonabelian) have this
property.

So, somewhat remarkably, the affine Skryabin theorem finds 
all of the representations of the $\sW$-algebra in the 
``invisible" part of the DG category.

\begin{rem}

That Whittaker $D$-modules lie in cohomological degree
$-\infty$ should not come as a surprise, as this happens
geometrically as well. Indeed, e.g., the Whittaker
$D$-module on $\Gr_N = N(K)/N(O)$ is the
dualizing $D$-module twisted by the exponential character.
Since $\Gr_N$ is isomorphic to (ind-)infinite-dimensional affine
space, its dualizing $D$-module unsurprisingly lies in cohomological
degree $-\infty$. 

\end{rem}

\subsection{}

One word on the definition of $\widehat{\fg}_{\kappa}\mod$,
in case it helps.
We are used to thinking of $A$-modules as abelian groups
with an action of $A$. How should we think about 
objects of $\widehat{\fg}_{\kappa}\mod$?

The answer is that for $M \in \widehat{\fg}_{\kappa}\mod$,
and for every $N \geq 0$, we have the ability to form:

\[
C^{\dot}(t^N \fg[[t]],M) \in \Vect
\]

\noindent where this notation indicates the Lie algebra cohomology
of $t^N \fg[[t]]$ with coefficients in $M$. Understanding
this properly: that these functors should be continuous in
the variable $M$ and should satisfy certain functoriality properties,
one obtains the definition from \cite{dmod-aff-flag}.

Note that one recovers the vector space underlying $M$
as:

\[
\underset{N}{\colim} \, C^{\dot}(t^N\fg[[t]],M).
\]  

\noindent So the mechanism for non-zero objects to 
exist in cohomological degree $-\infty$ is that they 
have non-vanishing Lie algebra cohomologies with
respect to some $t^N\fg[[t]]$, but in the direct limit this cohomology
dies.

\begin{rem}\label{r:km-defin}

It will be helpful in what follows to more formally
review a definition of $\widehat{\fg}_{\kappa}\mod$.
One takes $D^+(\widehat{\fg}_{\kappa}\mod)$ the
(DG) bounded below derived category, takes the full subcategory
generated under cones by the induced modules 
$\ind_{t^N\fg[[t]]}^{\widehat{\fg}_{\kappa}}(k)$, and then forms
the ind-category of this. According to \cite{dmod-aff-flag} \S 23,
this has a $t$-structure with the anticipated bounded below part.

We mention this definition to highlight the key role played by
the family of modules $t^N\fg[[t]]$ in the definition.

\end{rem}

\subsection{Who are the compact generators?}\label{ss:generators}

A more serious problem is to identify compact generators on both
sides. Because $N(K)$ is a group indscheme (not a group scheme),
it is not at all clear that $\Whit(\widehat{\fg}_{\kappa}\mod)$ is
compactly generated.

As far as I know, even the compact generators in 
$\sW_{\kappa}\mod$ were not previously constructed, even though
this basic problem about affine $\sW$-algebras can be understood
without any categorical formalism. 
The natural expectation is that there are modules
$\sW_{\kappa}^n \in \sW_{\kappa}\mod^{\heart}$ similar
to the modules 
$\ind_{t^n\fg[[t]]}^{\widehat{\fg}_{\kappa}} k \in 
\widehat{\fg}_{\kappa}\mod^{\heart}$. (So for $n = 0$, we should
have the ``vacuum" representation $\sW_{\kappa}$, and for
the Virasoro algebra, the construction is clear;
at critical level, there is a theory of \emph{opers with singularities}
due to \cite{hitchin} that settles the issue;
for higher rank $\fg$ and $n >0$, the construction is not so clear.)

We refer the interested reader 
to the beginning of \S \ref{s:ds} where the naive
expectations are formulated in detail; this material does not require
having read the preceding parts of the paper except, at a certain
point, needing some elementary Lie theoretic notation from 
\S \ref{s:adolescence}.

\begin{rem}

As in Remark \ref{r:km-defin}, the modules $\sW_{\kappa}^n$
actually play an essential role in defining the
DG category $\sW_{\kappa}\mod$. That is,
$\sW_{\kappa}\mod$ has a $t$-structure for
which $\sW_{\kappa}\mod^+$ is the bounded below derived
category of $\sW_{\kappa}\mod^{\heart}$, but the actual
definition of the full derived category takes the modules
$\sW_{\kappa}^n$ as input.

\end{rem}

\subsection{The adolescent Whittaker construction}\label{ss:intro-adolescent}

Each of the problems above are naturally solved using the
\emph{adolescent Whittaker construction}, which shows
that the Whittaker category is more highly structured than was previously
known. This construction gives a \emph{stratification} of the Whittaker
category by simpler pieces.
We refer to \S \ref{s:adolescence}, most notably Theorem \ref{t:!-avg}, 
where the results are constructions are formulated in detail.
Here we give a more tactile description. 

\begin{example}[C.f. Cor. \ref{c:ff-crit}]\label{e:adolescent-crit}

Recall from categorical Feigin-Frenkel duality
that $\Whit(\widehat{\fg}_{crit}\mod)$ is equivalent
to $\QCoh(\Op_{\ld{G}}(\o{\cD}))$. 
For $n>0$, define 
$\Whit^{\leq n}(\widehat{\fg}_{crit}\mod) \subset \Whit(\widehat{\fg}_{crit}\mod)$
to be the full subcategory corresponding to the subcategory
of quasi-coherent sheaves set-theoretically supported
on $\Op_{\ld{G}}^{\leq n} \subset \Op_{\ld{G}}(\o{\cD})$, 
the subscheme of opers \emph{with singularity $\leq n$}
(in the sense of \cite{hitchin} \S 3.8, see also \cite{fg2}).

\end{example}

The surprising fact is that the above construction makes
sense for every $\sC$ acted on by $G(K)$ (possibly with level $\kappa$).
Namely, we define categories $\Whit^{\leq n}(\sC)$ for all $n \geq 0$.
These are connected by adjoint functors:

\[
\Whit^{\leq n}(\sC) \rightleftarrows \Whit^{\leq n+1}(\sC) \in \DGCat_{cont}.
\]

\noindent The direct limit over $n$ (formed in $\DGCat_{cont}$)
is $\Whit(\sC)$. For $n > 0$, the induced functor
$\Whit^{\leq n}(\sC) \to \Whit(\sC)$ is fully-faithful.

For low values of $n$, this construction was previously known:

\begin{itemize}

\item For $n = 0$, $\Whit^{\leq 0}(\sC) = \sC^{G(O)}$.

\item For $n = 1$, $\Whit^{\leq 1}(\sC)$ is the \emph{baby Whittaker}
category of \cite{arkhipov-bezrukavnikov}. (This is the
reason for the terminology \emph{adolescent}.) 

\end{itemize}

Also, for $G = T$ a torus, we have:

\begin{itemize}

\item $\Whit(\sC) = \sC$ (tautologically) and 
$\Whit^{\leq n}(\sC) = \sC^{\cK_n}$ for $\cK_n \subset T(O)$
the $n$th congruence subgroup.

\end{itemize}

\begin{rem}

For $\kappa$ integral, local geometric Langlands predicts that
$\Whit(\sC)$ is a \emph{category over $\LocSys_{\ld{G}}(\o{\cD})$}
(see \cite{locsys} for a discussion of this notion).
Motivated by the case of critical level Kac-Moody representations
and the Frenkel-Gaitsgory philosophy,
we expect $\Whit^{\leq n}(\sC)$
to be the base-change of $\Whit(\sC)$ along
$\LocSys_{\ld{G}}^{\leq n-1} \to \LocSys_{\ld{G}}(\o{\cD})$,
where $\LocSys_{\ld{G}}^{\leq n-1}$ is the locus of 
local systems with slope $\leq n-1$; by definition,
the map $\LocSys_{\ld{G}}^{\leq n-1} \to \LocSys_{\ld{G}}(\o{\cD})$
is formally \'etale for $n \neq 0$, and for $n = 0$
we agree $\LocSys_{\ld{G}}^{\leq -1} = \bB \ld{G}$, i.e., the
trivial local system.

This gives a conceptual explanation for the use of
the baby Whittaker category in \cite{arkhipov-bezrukavnikov}:
they are working with regular singular (alias: slope $0$) 
local systems, so the baby Whittaker category $\Whit^{\leq 1}$
is enough.

\end{rem}

\begin{rem}

In the classical framework of $p$-adic groups,
a similar construction was given by Rodier in
\cite{rodier} (although his normalizations meaningfully
differ for $\on{rank}(G)>1$).

\end{rem}

\subsection{}

In general, $\Whit^{\leq n}(\sC)$ is defined as invariants
with respect to a compact open subgroup of $G(K)$.
Since the functors $\Whit^{\leq n}(\sC) \to \Whit(\sC)$ admit
continuous right adjoints, this means in practice
that \emph{$\Whit(\sC)$ is as good as invariants
with respect to a group scheme}. 
E.g., these functors preserve compact objects,
so compact generation with respect to
congruence subgroups implies it for the Whittaker category.
This is true for $\sC = D(X)$ for $G(K)$ acting on $X$ a
reasonable indscheme, or for $\sC = \widehat{\fg}_{\kappa}\mod$.

Since the functor 
$\Psi:\Whit(\widehat{\fg}_{\kappa}\mod) \to \sW_{\kappa}\mod$
is supposed to preserve compacts (being an equivalence),
we obtain candidates for the modules $\sW_{\kappa}^n$.
This idea turns out to be fruitful, and is pursued in \S \ref{s:ds}.

As an another instance of the above idea, 
the comparison Theorem \ref{t:inv-coinv} between invariants
and coinvariants falls out immediately from the adolescent
formalism, as does Gaitsgory's functor between them. 

\begin{example}\label{e:biwhit-cpt}

For instance, replacing $G$ by $G \times G$,
we find that the DG category of $D$-modules on $G(K)$
with both left and right Whittaker equivariance is compactly
generated. Local geometric Langlands predicts that
at integral level, this category is equivalent to 
$\QCoh(\LocSys_{\ld{G}}(\o{\cD})$. The corresponding
compact generation on the spectral side was
obtained in \cite{locsys}.

\end{example}

\subsection{Key construction: $t$-structures on Whittaker categories}

In the proof of Theorem \ref{t:aff-skry}, we use a general construction
of $t$-structures on Whittaker categories, which we highlight here because
it may be of interest outside of the theory of $\sW$-algebras. 

The point is that (up to shift) the canonical functors
$\Whit^{\leq n}(\sC) \to \Whit^{\leq n+1}(\sC)$ are
$t$-exact essentially whenever $\sC$ has a $t$-structure
compatible with the action of $G(K)$. 
Roughly, this is because by Theorem \ref{t:!-avg},
this functor may be realized either as the top
possible cohomology of a $*$-averaging functor, or the
bottom possible cohomology of a $!$-averaging functor,
and therefore is exact.\footnote{Naively,
the argument here is standard: see
e.g. \cite{bbm}.
We refer to Appendix \ref{a:av} for a discussion of technical
points.}

Ultimately, this is what allows us to prove Theorem \ref{t:aff-skry}:
being topological algebras, $\sW$-algebras are of abelian categorical
nature (as remarked above), and we analyze
$\Whit(\widehat{\fg}_{\kappa}\mod)$ with abelian categories via
this $t$-structure.

\subsection{}

In particular, if $D(X)$ is a reasonable indscheme with
$X$ acted on by $G(K)$, then $\Whit(D(X))$ has a canonical $t$-structure.

For instance, in the setting of
Example \ref{e:biwhit-cpt}, one finds
that the category of bi-Whittaker equivariant $D$-modules on $G(K)$ 
has a canonical $t$-structure.\footnote{The proof of the affine
Skryabin theorem can be modified to identify the corresponding
abelian category with (classical) chiral modules over the chiral algebra
$(\Psi \boxtimes \Psi)(CDO_{\kappa})$, where
$CDO_{\kappa}$ is the level $\kappa$
chiral differential operators, as constructed e.g. in \cite{dmod-loopgroup}.} 

Therefore, we should anticipate there
being a canonical $t$-structure on $\QCoh(\LocSys_{\ld{G}}(\o{\cD}))$.
Can it be constructed independently of
local geometric Langlands? This would generalize 
the perverse coherent $t$-structure on the nilpotent cone 
constructed in \cite{perverse-coherent} by Bezrukavnikov,
but to a much more complicated setting.

\begin{rem}

More generally, $\QCoh(\LocSys_{\ld{G}}(\o{\cD}))$ can be regarded
as the $\kappa \to \infty$ limit of the categories
$\widehat{\fg}_{\kappa}\mod^{G(K),w}$ of 
\emph{Harish-Chandra bimodules} for the Kac-Moody algebra.
The (conjectural) perverse coherent
$t$-structure on $\QCoh(\LocSys_{\ld{G}}(\o{\cD}))$ 
should deform in this setting.
Indeed, local geometric Langlands (in the quantum setting) continues
to predict that Harish-Chandra bimodules should be equivalent to 
bi-Whittaker $D$-modules at an appropriate level.

\end{rem}

\subsection{Summary}

The above story may be summarized as follows:
the adolescent Whittaker method from \S \ref{s:adolescence}, which
arises from the geometry of loop groups and
``pure"\footnote{That is, that part that explicitly ties to
the traditional arithmetic Langlands program, 
so e.g. does not mention Kac-Moody algebras.} 
local geometric Langlands, can be imported
to the setting of affine $\sW$-algebras to illuminate
the subject and solve some basic problems. 

But we may have reversed the logic in this presentation. The existence
of the modules $\sW_{\kappa}^n \in \sW_{\kappa}\mod^{\heart}$ 
was readily anticipated
by anyone who considered the issue. Moreover,
the affine Skryabin theorem predicts the existence
of the $t$-structure on $\Whit(\widehat{\fg}_{\kappa}\mod)$.
The theory of opers with singularities (c.f. Example \ref{e:adolescent-crit})
was long known.
So perhaps our knowledge of affine $\sW$-algebras
should rather have anticipated the adolescent Whittaker theory.

\subsection{Structure of the paper}\label{ss:str-paper}

We now outline the contents of the paper, trying
to highlight what parts may be of interest to different types of readers.

The adolescent Whittaker theory is developed in 
\S \ref{s:adolescence}. Some supporting results about loop
groups are given in \S \ref{s:rodier}. The latter material includes
some calculations about affine Springer fibers and centralizers
of elements of the ``integral Kostant slice" $f+\Lie(I)$ for
$I$ being the Iwahori subgroup (and $\Lie(I)$ being its Lie algebra).
This material, plus the analogue for $D$-module categories 
of the $t$-structure construction from \S \ref{s:skryabin},
is the new material on local geometric Langlands 
from this paper that does not mention $\sW$-algebras.

In \S \ref{s:ds}, we solve the problem from \S \ref{ss:generators}
of finding generators for the affine $\sW$-algebra.
The construction uses ideas from \S \ref{s:adolescence},
but another construction using more classical ideas
is given in \S \ref{s:free-field}. These two sections
do not use any categorical machinery; rather, they formulate and
solve problems purely about affine $\sW$-algebras.

The affine Skryabin theorem is proved in \S \ref{s:skryabin}.
This combines the categorical methods with Kac-Moody algebras.

Finally, some applications, such as the categorical
Feigin-Frenkel theorem, are given in \S \ref{s:applications}.
We discuss exactness properties of the Drinfeld-Sokolov
functor here. 

There are two appendices. Appendix \ref{a:hc} gives
background material on homological algebra for
Tate Lie algebras. It includes supporting material
on semi-infinite cohomology and its calculation. All the material
there is standard, but perhaps it warranted an updated exposition.
Appendix \ref{a:av} is about $t$-exactness of $!$-averaging
functors, which as indicated above, plays a key role. 

\subsection{Notation and conventions}

We assume the reader is familiar with the
theory of $D$-modules in infinite type and group actions
on (cocomplete DG) categories. We refer to 
\cite{dario-*/!} and \cite{dmod} for these subjects.
For a group $H$ acting on $\sC$, we let $\sC^H$ denote
the strong invariants, $\sC^{H,w}$ denote the weak 
invariants, and $\sC^{H,\psi}$ denote the invariants
twisted by the exponential $D$-module along
a character $\psi:H \to \bG_a$ (c.f. \cite{dario-*/!} \S 2.5.4).

We use homotopical algebra methods freely, 
although this is inessential at some points.
Our default categorical
language is the $\infty$-categorical language of Lurie. 
So \emph{category} means $(\infty,1)$-category, and the rest
of the categorical notions (notably limits
and colimits) are understood correspondingly.

We use $\otimes$ to denote the natural symmetric
monoidal operation on $\DGCat_{cont}$. 

\subsection{Kac-Moody groups}\label{ss:km-conv}

We let: 

\[
1 \to Z_{KM} \to \widehat{G(K)} \to G(K) \to 1
\]

\noindent denote the Kac-Moody extension. Some explanation
is in order.

Here $Z_{KM}$ is a certain (finite-dimensional) torus,
which is $\bG_m$ for $G$ a simple group.
Precisely, it suffices to define its character lattice;
it will be a certain sublattice of the $k$-vector space
of \emph{levels}, i.e., $\Ad$-invariant symmetric bilinear
forms on $\fg$. 

Recall that $\fg = \fz_{\fg} \oplus [\fg,\fg]$ where
$\fz_{\fg}$ is the center. Moreover, recall that
$[\fg,\fg]$ is canonically a direct sum of simple Lie algebras $\fg_i$
(its minimal normal subalgebras). For any level $\kappa$,
the spaces $\fz_{\fg}$ and $\fg_i$ are pairwise orthogonal.
We then take the lattice of levels such that
$\kappa|_{\fz_{\fg}}$ is \emph{even}, i.e., the pairing of
any two coweights of $G/[G,G]$ is an even integer,
and $\kappa|_{\fg_i}$ is an integral multiple of the
Killing form. Clearly such levels $k$-span the space
of all levels.

For any level of the above type, we obtain a central
extension of $G(K)$ by $\bG_m$ whose Lie algebra
is the Kac-Moody extension $\widehat{\fg}_{\kappa}$ of 
$\fg((t))$. Indeed, the adjoint group $G^{ad}$ is canonically
a product of groups $G_i$, so we obtain
$G \to G/[G,G] \times \prod_i G_i$; our extension of $G(K)$
is the Baer sum of the extensions
of $G/[G,G](K)$ defined by the Contou-Carr\'ere symbol
and the extensions of $G_i(K)$ defined by the usual
determinant line method. This is additive in the level,
so it is equivalent to say that 
there is an extension $\widehat{G(K)}$ of $G(K)$ by
$Z_{KM}$ whose pushout by a character
$Z_{KM} \to \bG_m$ (equivalently, a level of the specified type) 
is the one we just explained.

\subsection{}

For \emph{any} level $\kappa$, there is a canonical
multiplicative $D$-module on $Z_{KM}$ whose
underlying $\sO$-module is multiplicatively trivialized.
E.g., if $Z_{KM} = \bG_m$, this is the $D$-module ``$z^{\lambda}$"
for $\lambda$ the given scalar. 
We let $D_{\kappa}(G(K))$ denote the corresponding
category of twisted $D$-modules: by definition,
this is the category of $D$-modules on $\widehat{G(K)}$
$Z_{KM}$-equivariant against our multiplicative character.
This DG category is equipped with a convolution monoidal structure.

As in \cite{dario-*/!}, this allows us to speak about DG categories
acted on at level $\kappa$ (for a general level $\kappa$).

\subsection{Acknowledgements}

I have benefited from discussions with a number of people
in working on this material. I'm happy to thank
Dima Arinkin, Sasha Beilinson, David Ben-Zvi,
Dario Beraldo, Roman Bezrukavnikov, Sasha Braverman,
Vladimir Drinfeld, Giorgia Fortuna, Edward Frenkel,
Reimundo Heluani, Masoud Kamgarpour, 
Ivan Mirkovic, and Xinwen Zhu
for their support, encouragement, and for their generosity
in sharing their ideas.

Special thanks are due to Dennis Gaitsgory. 
This spirit of representation theory was invented
by him, and his singular influence extends throughout this work.

\section{Compact approximation to the Whittaker model}\label{s:adolescence} 

\subsection{} 

In this section, we will prove the following theorem:

\begin{thm}\label{t:inv-coinv}

For every $\sC$ acted on by $G(K)$ at level $\kappa$,
there is an equivalence:

\[
\sC_{N(K),\psi} \isom \sC^{N(K),\psi}
\]

\noindent functorial in $\sC$.

\end{thm}

\begin{rem}

The functor\footnote{This functor admits a very simple description
in the language \cite{dmod} of $D$-modules on infinite type (ind)schemes.
First, note that convolution by any $(N(K),\psi)$-biequivariant $D$-module
on $G(K)$ induces a functor $\sC_{N(K),\psi} \to \sC^{N(K),\psi}$.
Then we should take the convolution with the
\emph{renormalized} pushforward from $N(K)$ to $G(K)$
of the character sheaf.
We remind that the word \emph{renormalized}
indicates something specific to the infinite type setup, and in particular 
that it indicates
that we have chosen trivializations of dimension torsors.

In particular, we see that this functor makes sense for any
category acted on by $N(K)$, i.e., the $G(K)$ action is not
necessary. However, we remind that a peculiarity of the
infinite type framework is that if we took $\sC = \Vect\cdot \psi$
(i.e., the $N(K)$ action on $\Vect$ corresponding to the
character $\psi$), then this functor is zero, even though
$\sC_{N(K),\psi} = \sC^{N(K),\psi} = \Vect$ (the identification
being realized here by a different functor).

In any case, the explicit description of the functor will
be immediate from the proof given below, and we do not
particularly emphasize it.
}
 realizing this equivalence was constructed by
Gaitsgory in \cite{lesjours}, and conjectured
to be an equivalence. This theorem was proved in 
\emph{loc. cit}. for rank 1 groups, and in \cite{dario-*/!}
for $G = GL_n$.

\end{rem}

\begin{rem}

From the onset, we 
draw the reader's attention to Theorem \ref{t:!-avg}, which is the
essential tool in proving Theorem \ref{t:inv-coinv}, and which plays
a key role throughout this paper.

\end{rem}

\subsection{The adolescent Whittaker constructions}

The main tool for proving Theorem \ref{t:inv-coinv}
is the following sequence of subgroups of $G(K)$.

\begin{defin}

For $n \geq 0$, we let $\o{I}_n$ denote the subgroup
$\Ad_{-n\check{\rho}(t)}(G(O) \times_{G(O/t^n)} N(O/t^n))$.

\end{defin}

In other words, we take elements of $G(O)$ that lie in $N$ modulo
$t^n$, and conjugate this subgroup in a way that enlarges
congruence subgroups
of $N(K)$, fixes $T(K)$, and shrinks congruence subgroups of
$N^-(K)$.

\begin{example}\label{e:01infty}

For $n = 0$, $\o{I}_0 = G(O)$.\footnote{This is the only
case where $\o{I}_n$ is not prounipotent. So for many problems,
a claim about all $\o{I}_n$ is proved by treating the $n = 0$ case 
separately, where the claim may be degenerate anyway.
Despite this clumsiness, it seems to be most natural to
include the $n = 0$ case on equal footing wherever possible.}  
For $n = 1$, $\o{I}_1$ is conjugated by $\check{\rho}(t)$
from the radical of Iwahori.
While $\o{I}_n$ neither contains nor is contained in 
$\o{I}_{n+1}$, these groups limit to $N(K)$. 

\end{example}

\begin{notation}

Motivated by the limiting behavior above, we add
$N(K)$ to this family of subgroups 
by setting $\o{I}_n = N(K)$ for $n = \infty$.

\end{notation}

\begin{example}

For $G = GL_2$ and $0<n<\infty$, $\o{I}_n$ is the subgroup of matrices:

\[
\begin{pmatrix}
1+t^n a & t^{-n} b \\
t^{2n} c & 1+t^n d
\end{pmatrix} 
\]

\noindent for $a,b,c,d \in O$. Similarly, for $G = GL_3$, we obtain
the subgroup of matrices:

\[
\begin{pmatrix}
1+t^n a & t^{-n} b & t^{-2n} c \\
t^{2n} d & 1+t^n e & t^{-n} f \\
t^{3n} g & t^{2n} h & 1+t^n i
\end{pmatrix}.
\]

\end{example}

\begin{example}

For $G = T$ a torus, $\o{I}_n$ is the $n$th congruence subgroup.

\end{example}

\begin{rem}[Triangular decomposition]

We will repeatedly use the following fact without mention.
For all $n>0$, note that $\o{I}_n$ admits a 
\emph{triangular decomposition}:

\[
\o{I}_n = (\o{I}_n \cap N^-(K)) 
\times (\o{I}_n \cap T(K)) \times
(\o{I}_n \cap N(K))
\]

\noindent with the isomorphism induced by the multiplication map.
The same is true if we reverse the order of the factors. 

\end{rem}

\begin{rem}[Splitting the Kac-Moody extension]\label{r:whit-km}

Note that $\widehat{G(K)} \to G(K)$ is canonically split over
$G(O)$. By transport of structure, it is canonically split over 
$\Ad_{-n\check{\rho}(t)}(G(O))$ as well, and in particular,
over $\o{I}_n$. 

Varying $n$, these splittings coincide on all intersections
$\o{I}_n \cap \o{I}_m$. Indeed, we may safely assume one of
$n$ and $m$ is non-zero, in which case this intersection is
prounipotent. But the splittings differ by a homomorphism
$\o{I}_n \cap \o{I}_m \to \bG_m$, which must be trivial
by prounipotence.

Therefore, $D_{\kappa}(\o{I}_n) \simeq D(\o{I}_n)$ as
monoidal categories, 
and as $n$ varies these equivalences are compatible with restriction
to intersections between the subgroups $\o{I}_n$. 

\end{rem}

\subsection{}

One has the following straightforward construction
of characters of the $\o{I}_n$.

\begin{lem}

For every $n$, there is a unique homomorphism
$\psi_{\o{I}_n}:\o{I}_n \to \bG_a$ annihilating 
$B^-(O)\cap \o{I}_n$ and, with
$\psi_{\o{I}_n}|_{N(K) \cap \o{I}_n} = \psi|_{\o{I}_n\cap N(K)}$. 
For a pair of integers $n,m$, the
corresponding characters coincide on
the intersection $\o{I}_n \cap \o{I}_m$.

\end{lem}

\begin{notation}

To encourage the reader to think of the characters $\psi_{\o{I}_n}$
and $\psi$ as all being ``the same," we denote them all by
$\psi$ when there is no risk of confusion.

\end{notation}

\begin{example}

Since the $n = 0$ case can be the most confusing: the character
is trivial in this case.

\end{example}

\subsection{}

In the remainder of this section, $\sC \in \DGCat_{cont}$ is
equipped with a $G(K)$ action of level $\kappa$.

\begin{defin}

For $0 \leq n \leq \infty$, 
we define the $n$th \emph{adolescent Whittaker category}
as:\footnote{Note that this category makes sense because of
Remark \ref{r:whit-km}.}

\[
\Whit^{\leq n}(\sC) \coloneqq \sC^{\o{I}_n,\psi}.
\]

\end{defin}

\begin{rem}

For $n>0$, $\Whit^{\leq n}(\sC)$ is a subcategory of $\sC$
since $\o{I}_n$ is prounipotent if $n<\infty$ and ind-prounipotent
if $n =\infty$.

\end{rem}

\begin{rem}

In line with Example \ref{e:01infty}, $\Whit^{\leq 0}(\sC) = \sC^{G(O)}$
is the \emph{spherical} category, $\Whit^{\leq 1}(\sC)$ is\footnote{At
least if the center of $G$ is connected, so that
$\o{I}_1$ is actually conjugate to the radical of Iwahori
by an element of $G(K)$ (and not merely $G^{ad}(K)$), but the reader is advised
to ignore this point.} the \emph{baby} Whittaker category 
of \cite{arkhipov-bezrukavnikov}, and in the limit as $n \to \infty$,
we have $\Whit^{\leq \infty}(\sC) = \sC^{N(K),\psi}$ the
(grown-up) Whittaker category.

\end{rem}

\subsection{How are the categories $\Whit^{\leq n}$ related as
we vary $n$?}

Since the $\o{I}_n$ groups are not contained one in another, 
we can only relate these categories via averaging. This is to
say, they are not related so directly, though not so indirectly either.
But in Theorem \ref{t:!-avg}, we will show that this question
is interesting than it appears.

\subsection{}

For $n \leq m < \infty$, we have a functor:

\[
\iota_{n,m,*}: \Whit^{\leq n}(\sC) \to \Whit^{\leq m}(\sC)
\]

\noindent given as the composition:

\[
\sC^{\o{I}_n,\psi} \xar{\Oblv} \sC^{\o{I}_n \cap \o{I}_m,\psi}
\xar{\Av_*^{\psi}} \sC^{\o{I}_m,\psi} 
\]

\noindent and for all $n \leq m$ we have a functor:

\[
\iota_{n,m}^!: \Whit^{\leq m}(\sC) \to \Whit^{\leq n}(\sC)
\]

\noindent given as the composition:

\[
\sC^{\o{I}_m,\psi} \xar{\Oblv} \sC^{\o{I}_n \cap \o{I}_m,\psi}
\xar{\Av_*^{\psi}} \sC^{\o{I}_n,\psi}.
\]

\begin{notation}

We use the following convention systematically:
for $m = \infty$, we suppress $m$ from the notation.
So we use the notation $\iota_n^!$ instead of 
$\iota_{n,\infty}^!$.

\end{notation}

\begin{rem}

We remind that $*$-averaging only makes sense for group schemes,
not group indschemes, so $\iota_{n,*} = \iota_{n,\infty,*}$ 
does not make sense. (Here \emph{makes sense} means that while
there is a non-continuous right adjoint for formal reasons, this
functor is pathological.)

\end{rem}

\begin{rem}

These functors compose well: e.g., for 
$\ell \leq n \leq m<\infty$, we have
$\iota_{n,m,*} \circ \iota_{\ell,n,*}= \iota_{\ell,m,*}$.

\end{rem}

\begin{rem}

There is no a priori adjunction between the functors
$\iota_{n,m,*}$ and $\iota_{n,m}^!$, since
$\Oblv$ is a left adjoint while $\Av_*^{\psi}$ is a right adjoint.
Rather, their relationship, in the terminology of \cite{dgcat}, 
is that these functors
are \emph{dual} to one other (when $\sC$ is
dualizable).

\end{rem}

\begin{warning}\label{w:adolescent-forget}

Neither of the constructions $\iota_{n,m,*}$ or $\iota_{n,m}^!$
is compatible with forgetful functors to $\sC$.
We advise to mostly forget about the forgetful
functors to $\sC$ and to remember these functors instead.

\end{warning}

\subsection{Formulation of the main result}\label{ss:!-avg}

Let $\Delta \coloneqq 2(\check{\rho},\rho) \in \bZ^{\geq 0}$.
Note that:\footnote{
Just for fun, we remark that $\Delta$ can also be 
calculated as
$\sum_i \begin{pmatrix} d_i \\ 2 \end{pmatrix}$, where
the $d_i$ are the exponents of the semisimple Lie algebra $[\fg,\fg]$. 
}

\[
n\Delta = \dim(\Ad_{-n\check{\rho}(t)} N(O)/N(O)).
\]

The main result of this section is the following.

\begin{thm}\label{t:!-avg}

\begin{enumerate}

\item\label{i:!-avg-exists}

For all $n \leq m \leq \infty$, the functor $\iota_{n,m}^!$ admits a 
left adjoint $\iota_{n,m,!}$. 

\item\label{i:!-avg-ff}

For all $0<n \leq m \leq \infty$, $\iota_{n,m,!}$ is fully-faithful.

\item\label{i:!-avg=*-avg}

If $m \neq \infty$, there is a canonical
 isomorphism:
 
\[
\iota_{n,m,!} \simeq 
\iota_{n,m,*}[2(m-n)\Delta].
\]

\noindent These isomorphisms are compatible with compositions,
e.g. the induced isomorphisms between 
$\iota_{n,m,*} \circ \iota_{\ell,n,*} = \iota_{\ell,m,*}$
and $\iota_{n,m,!} \circ \iota_{\ell,n,!} = \iota_{\ell,m,!}$ canonically
coincide.

\end{enumerate}

\end{thm}

\begin{rem}

The method below is also used in a finite-dimensional
situation in \cite{bbm}, so we consider
this an affine analogue of the first part of their Theorem 1.5 (1). 

\end{rem}

\begin{rem}

The functor $\iota_{n,m,!}$ is tautologically computed by
forgetting down to $\sC^{\o{I}_n \cap \o{I}_m,\psi}$ and then
$!$-averaging to $\sC^{\o{I}_m,\psi}$: the claim in this theorem
is that the $!$-averaging is actually defined.

\end{rem}

\begin{rem}

In the $n = 0$ and $m = \infty$ case
this result says that we can $!$-average spherical objects to obtain 
Whittaker equivariant objects. This is an old observation that has
been known for as long as the words have made sense. (Actually,
the $\o{I}_n$ groups were found by reverse engineering
while trying to generalize that
argument along the lines of the
proof of Theorem \ref{t:!-avg} \eqref{i:!-avg-exists} given below
in the $m = \infty$ case.)

\end{rem}

\begin{rem}

A quite similar pattern appeared long ago
in the $p$-adic setting in \cite{rodier}, though with a (mildly)
different series of subgroups in place of the $\o{I}_n$.

\end{rem}

\begin{rem}[Relationship to the work of Beraldo-Gaitsgory]

In the case $G = GL_r$ (resp. $G = GL_2$), 
these results all follow from \cite{dario-*/!}
(resp. \cite{lesjours}). Indeed, in \emph{loc. cit}., the
authors construct closed subgroups\footnote{For example, for
$GL_2$, one has:

\[
H_n = \Big\{ 
\begin{pmatrix} 1+ t^na & t^{-n}b \\ 0 & 1 \end{pmatrix}
\mid a,b \in O\Big\}.
\]
\noindent In the general case, $H_n$ is the intersection of
$\o{I}_n$ with the mirabolic subgroup.
}
$H_n \subset \o{I}_n$
(specific to $GL_r$) such that the maps:

\[
H_n/H_n \cap H_m \to \o{I}_n/\o{I}_n \cap \o{I}_m
\] 

\noindent are isomorphisms for all $m \geq n$. Moreover, using
Fourier techniques reminiscent of the mirabolic theory, 
Beraldo shows that $*$-averaging 
$\sC^{H_{n+1},\psi} \to \sC^{H_n,\psi}$ is an equivalence
for all $\sC$ as above, with inverse given by the appropriately
shifted $*$-averaging functor. These facts are easily seen
to imply Theorem \ref{t:!-avg} in this case.

The methods for a general reductive group 
are (by necessity) quite different.

(We also remark that for the application to
$\sW$-algebras, it is essential to work with
compact \emph{open} subgroups of $G(K)$.) 

\end{rem}

\subsection{}

The remainder of this section is structured as follows.
First, in \S \ref{ss:inv-coinv-pf}, we will deduce Theorem \ref{t:inv-coinv}
from Theorem \ref{t:!-avg}. Then,
modulo a lemma whose proof will be delayed to \S \ref{s:rodier}, 
the remainder of this
section will be dedicated to the proof of Theorem \ref{t:!-avg}.

\subsection{Application to co/invariants}\label{ss:inv-coinv-pf}

We begin by making precise the sense in which the
adolescent Whittaker constructions limit to the usual Whittaker construction.

Let $\sC$ be as always. For all $0<n\leq m<\infty$, observe that the diagrams:

\[
\vcenter{
\xymatrix{
\sC^{N(K) \cap \o{I}_m,\psi} \ar[r]^{\Oblv} 
\ar[d]^{\Av_*^{\psi}} &
\sC^{N(K) \cap \o{I}_n,\psi}\ar[d]^{\Av_*^{\psi}} \\
\sC^{\o{I}_m,\psi} = \Whit^{\leq m}(\sC)  \ar[r]^{\iota_{n,m}^!} & 
\Whit^{\leq n}(\sC)  = \sC^{\o{I}_n,\psi}  
}}
\]

\noindent and:

\begin{equation}\label{eq:n->infty-str/coinv}
\vcenter{
\xymatrix{
\sC^{\o{I}_n,\psi} = \Whit^{\leq n}(\sC) \ar[r]^{\iota_{n,m,*}} \ar[d]^{\Oblv} &
\Whit^{\leq m}(\sC)  = \sC^{\o{I}_m,\psi} \ar[d]^{\Oblv}
\\
\sC^{N(K) \cap \o{I}_n,\psi} \ar[r]^{\Av_*^{\psi}} & 
\sC^{N(K) \cap \o{I}_m,\psi}
}}
\end{equation}

\noindent commute.

\begin{lem}\label{l:n->infty}

The induced functors in $\DGCat_{cont}$:

\[
\begin{gathered}
\sC^{N(K),\psi} \coloneqq
\underset{n,\Oblv}{\lim} \, 
\sC^{N(K) \cap \o{I}_n,\psi}  \to 
\underset{n,\iota_{n,m}^!}{\lim} \, \Whit^{\leq n}(\sC) \\
\underset{n,\iota_{n,m,*}}{\colim} \, \Whit^{\leq n}(\sC) \to 
\underset{n,\Av_*^{\psi}}{\colim} \, \sC^{N(K) \cap \o{I}_n,\psi} \eqqcolon
\sC_{N(K),\psi} 
\end{gathered}
\]

\noindent are equivalences. 

\end{lem}

\begin{proof}

First, we treat the coinvariants statement.

Suppose $\Gamma_i \subset G(K)$ is a decreasing sequence
and let $\Gamma_{\infty}  \coloneqq \cap_i \Gamma_i$ such that:

\[
G(K) / \Gamma_{\infty} = \underset{i}{\lim} \, G(K)/\Gamma_i.
\]

Then we claim that:

\[
\underset{i}{\colim} \,  \sC^{\Gamma_i} = \sC^{\Gamma}
\]

\noindent where the structure maps in the
colimit are forgetful functors. Indeed, this follows
because the $D$-module $\delta_{\Gamma} \in D(G(K))$ 
is the colimit of the $\delta_{\Gamma_i}$, and these
convolution with these $\delta$ $D$-modules gives the appropriate
(colocalizing) averaging functors.

Now for $\infty>m \geq n>0$, define $\o{I}_{n,m}$ as
$(\o{I}_m \cap B(K)) \cdot (N(K)\cap \o{I}_n)$.
Note that $\o{I}_{n,m} \supseteq \o{I}_{n,m+1}$,
and $\cap_m \o{I}_n = N(K) \cap \o{I}_n$.

Moreover, for a pair $(m^{\prime},n^{\prime}) \in \bZ^{>0} \times \bZ^{>0}$
with $m^{\prime} \geq n^{\prime},m$, the functor:

\[
\Av_*^{\psi}: 
\sC^{N(K) \cap \o{I}_n,\psi} \to \sC^{N(K) \cap \o{I}_{n^{\prime}},\psi}
\]

\noindent takes the subcategory $\sC^{\o{I}_{n,m},\psi}$ into
$\sC^{\o{I}_{n^{\prime},m^{\prime}},\psi}$.
Therefore, we have:

\[
\underset{n,\Av_*^{\psi}}{\colim} \, \sC^{N(K) \cap \o{I}_n,\psi} =
\underset{m \geq n}{\colim} \, \sC^{\o{I}_{n,m},\psi} =
\underset{n}{\colim} \, \sC^{\o{I}_n,\psi}
\]

\noindent as desired.

The version for invariants follows formally from the coinvariants
version: compute the DG category of $G(K)$-equivariant
functors $D(G(K))_{N(K),\psi} = 
\underset{n}{\colim} \, D(G(K))^{\o{I}_n,\psi}$ to $\sC$
in two different ways.

\end{proof}

\begin{proof}[Proof that Theorem \ref{t:!-avg}
implies Theorem \ref{t:inv-coinv}]

By Lemma \ref{l:n->infty}, 
$\sC_{N(K),\psi}$ is the colimit in $\DGCat_{cont}$
of the $\Whit^{\leq n}(\sC)$ with the functors
$\alpha_{n,m,*}$ as structural functors.
Intertwining via the autoequivalences:

\[
\begin{gathered}
\Whit^{\leq n}(\sC) \isom \Whit^{\leq n}(\sC) \\
\sF \mapsto \sF [2 n \Delta]
\end{gathered}
\]

\noindent we see that $\sC_{N(K),\psi}$ is also equivalent
to the colimit where instead we use the structural
functors:

\[
\iota_{n,m,*}[2(m-n)\Delta]
\overset{Thm. \ref{t:!-avg} \eqref{i:!-avg=*-avg}}
\simeq \iota_{n,m,!}.
\]

Now this is a colimit in $\DGCat_{cont}$ under left adjoints,
so it is equivalent to the limit under the right adjoints  
$\iota_{n,m}^!$. Applying Lemma \ref{l:n->infty} again
gives this limit as $\sC^{N(K),\psi}$, as desired.

\end{proof}

We record the following consequence of the argument.

\begin{cor}

The equivalent categories $\sC^{N(K),\psi}$ and $\sC_{N(K),\psi}$
are obtained from the categories
$\Whit^{\leq n}(\sC)$ by either taking the colimit under the functors
$\iota_{n,m,!}$, or the limit under their right adjoints $\iota_{n,m}^!$.

\end{cor}

\begin{rem}

For the sake of clarity: note that we did not use
Theorem \ref{t:!-avg} \eqref{i:!-avg-ff} in this
deduction. But this result \emph{will} be used
in proving the other parts of the theorem.

\end{rem}

\subsection{Notation} 

In the remainder of this section, we prove Theorem \ref{t:!-avg}.

To keep the notation from becoming too overburdened:
for a subgroup $H \subset G(K)$, we use the notation
$H^n$ for $\Ad_{-n\check{\rho}(t)}(H)$. 
For example, $N(O)^n = N(K) \cap \o{I}_n$.

We also use the notation $\Whit(\sC)$ to indicate
Whittaker \emph{invariants} 
$\sC^{N(K),\psi} = \Whit^{\leq \infty}(\sC)$.

\subsection{Proof of Theorem \ref{t:!-avg} \eqref{i:!-avg-exists} for $m = \infty$}

We begin by showing that objects of $\Whit^{\leq n}(\sC)$
can be {$!$-averaged} to $\Whit(\sC)$.

\subsection{}

We begin by giving the argument in the geometric setting, 
where it is easier to understand, and
later will explain how to adapt to the categorical setup.  
So suppose $\sC = D(X)$
for $X$ a suitably nice (e.g., ind-finite type) 
indscheme acted on by $G(K)$, ignoring the 
irrelevant level for the time being.\footnote{If we were not
forgetting the level, we would equip $X$ with a
$Z_{KM}$-torsor $\sP$ with  
an action of $\widehat{G(K)}$ extending the given action
of the central $Z_{KM}$.} 

For $\sF \in \Whit^{\leq n}(D(X))$, we want to show that 
$\act_!(\psi \widetilde{\boxtimes} \sF)$ is defined,
where:\footnote{Here $N(K) \overset{N(O)^n} {\times} X$
is the standard notation for 
the quotient of $N(K) \times X$ by the diagonal action of $N(O)^n$
acting on the right on $N(K)$ and ``on the left" on $X$.}

\[
\psi \widetilde{\boxtimes} \sF \in
D(N(K) \overset{N(O)^n} {\times} X)
\]

\noindent is the descent of $\psi \boxtimes \sF$ using 
equivariance of $\sF$ and $\act$ denotes the action map
$N(K) \overset{N(O)^n} {\times} X \to X$.

So first we should compactify the map $\act$.
Note that:

\[
N(K) \overset{N(O)^n} {\times} X =
N(K) \o{I}_n \overset{\o{I}_n} {\times} X =
N(K)G(O)^n
\overset{G(O)^n} {\times} X
\]

\noindent since 
$G(O)^n\cap N(K) = N(K) \cap \o{I}_n = N(O)^n$.

Then define
$\ol{N(K)G(O)^n}$ as the pullback
to $G(K)$ of the closure of the $N(K)$ orbit through $1$
in $G(K)/G(O)^n$. Since the latter is
isomorphic to the affine Grassmannian, so is ind-proper,
$\ol{N(K)G(O)^n}/G(O)^n$
is ind-proper as well. Therefore, the map:

\[
\ol{\act}:
\ol{N(K)G(O)^n}
\overset{G(O)^n} {\times} X \to X
\]

\noindent is ind-proper as well.\footnote{The notation is
potentially confusing: $\ol{\act}$ is just induced by the
usual action map $G(K) \times X \to X$.}

Let $j$ denote the open embedding:

\[
N(K) \overset{N(O)^n} {\times} X =
N(K)G(O)^n
\overset{G(O)^n} {\times} X \into 
\ol{N(K)G(O)^n}
\overset{G(O)^n} {\times} X.
\]

\noindent Then it remains to verify the \emph{cleanness} result: 

\[
j_!(\psi \widetilde{\boxtimes} \sF ) \isom j_{*,dR}(\psi \widetilde{\boxtimes} \sF ).
\]

\noindent In particular, the left hand side is defined.

To this end, we begin with a lemma, where we reintroduce the
level $\kappa$ for clarity in the generalization.

\begin{lem}\label{l:cleanness}

Let $\cK_n \subset G(O)$ denote the $n$th congruence subgroup, 
so $\cK_n^n = \Ad_{-n\check{\rho}(t)}(\cK_n)$ by our convention.

Then $\Whit(D_{\kappa}(\ol{N(K)G(O)^n}/\cK_n^n)) = 
\Whit(D_{\kappa}(N(K)G(O)^n/\cK_n^n))$, i.e., every Whittaker equivariant
$\kappa$-twisted 
$D$-module on $N(K)G(O)^n/\cK_n^n$ cleanly extends
to the closure.

\end{lem}

\begin{proof}

Recall that $\ol{N(K)G(O)^n}$ is stratified by
the locally closed $N(K) \check{\la}(t) G(O)^n$ for
$\check{\la} \in - \check{\La}^{pos}$. So it suffices to show
that $\Whit(D_{\kappa}(N(K) \check{\la}(t) G(O)^n/\cK_n^n)) = 0$ for
all $0 \neq \check{\la} \in -\check{\La}^{pos}$.

Using:

\[
N(K) \check{\la}(t) G(O)^n/\cK_n^n = 
\check{\la}(t) N(K)G(O)^n/\cK_n = 
\check{\la}(t) N(K) \overset{N(O)^n}{\times} G(O)^n/\cK_n 
\]

\noindent we obtain:

\[
\Whit(D_{\kappa}(N(K) \check{\la}(t) G(O)^n/\cK_n^n)) =
D(G(O)^n/\cK_n^n)^{N(O)^n,\psi^{\check{\la}}}
\]

\noindent where 
$\psi^{\check{\la}} \coloneqq \psi \circ \Ad_{\check{\la}(t)}$.
(We forget $\kappa$ because we are dealing with a conjugate
of $G(O)$ now.)

As $\cK_n^n \subset G(O)$ is normal, it
acts trivially on $G(O)^n/\cK_n^n$. 
Because $N(O) \subset \cK_n^n$, we see that
the $N(O)$ action on $G(O)^n/\cK_n^n$ is trivial.
Therefore, it suffices to see that $\psi^{\check{\la}}|_{N(O)}$
is non-trivial for $0\neq \check{\la} \in -\check{\La}^{pos}$.

This is standard:
$\psi^{\check{\la}}(\exp(t^m e_i)) = 
\psi(\exp(t^{m+(\check{\la},\alpha_i)} e_i)$, which is
$1$ if $m+(\check{\la},\alpha_i) = -1$.
So it suffices to show that $(\check{\la},\alpha_i)<0$ for some $i \in \cI_G$
(then take $m = -1 - (\check{\la},\alpha_i) \geq 0$).
Since $(\check{\la},\rho)<0$ by assumption on $\check{\la}$, this
is clear.

\end{proof}

\begin{rem}

More generally, this argument shows that 
any $D$-module on $\overline{N(K)G(O)^n} \times X$
satisfying Whittaker equivariance for the left action on the
first factor and $\cK_n^n$-equivariance for the right action
on the first factor is cleanly extended from $N(K)G(O)^n \times X$.
Note that the action of $G(K)$ on $X$ is not used here.
(In the categorical setting, one should instead note that
for any $\sC \in \DGCat_{cont}$ considered with a
trivial $G(K)$-action,
$\Whit(D_{\kappa}(\ol{N(K)G(O)^n}/\cK_n^n) \otimes \sC) = 
\Whit(D_{\kappa}(N(K)G(O)^n/\cK_n^n) \otimes \sC)$.)

\end{rem}

Therefore, it suffices to see that
the pullback of $\psi \widetilde{\boxtimes} \sF$ to
$N(K)G(O)^n \times X$ is Whittaker equivariant
for the action of $N(K)$ on the first factor, and
$\cK_n^n$-equivariant for the right action on the first factor.
Note that the map 
$N(K)G(O)^n \times X \to N(K) \overset{N(O)^n}{\times} X$
is given by the formula $(g_1\cdot g_2,x) \mapsto (g_1,g_2\cdot x)$,
so the first claim is obvious.

For the second: the $n = 0$ case is easy to deal with directly,
so we assume $n>0$. Then by prounipotence of
$\cK_n^n$, it suffices
to check equivariance after we pullback further
to $N(K) \times G(O)^n \times X$. Our map
to $N(K) \overset{N(O)^n}{\times} X$ then lifts
to $N(K) \times X$. Now observe that $\psi \boxtimes \sF$
is $\cK_n^n$-equivariant for the action of $\cK_n^n$ on
the $X$-factor, since $\cK_n^n \subset \o{I}_n$
and $\psi_{\o{I}_n}|_{\cK_n^n}$ is trivial.
Moreover, the map:

\[
N(K) \times G(O)^n \times X \xar{(g_1,g_2,x) \mapsto (g_1,g_2\cdot x)}
N(K) \times X \to N(K) \times X/\cK_n^n
\]

\noindent descends to a map from 
$N(K) \times G(O)^n/\cK_n^n \times X$, by normality
of $\cK_n^n \subset G(O)^n$. This gives the claim, completing the
argument.

\subsection{}\label{ss:clean-cat}

We now indicate what changes should be made in the general
categorical setting. So let $\sC$ be acted on by $G(K)$ at level
$\kappa$.

We have a coaction functor:

\[
\sC \to D(N(K)) \otimes \sC
\]

\noindent (obtained from the coaction functor
$\sC \to D_{\kappa}(G(K)) \otimes \sC$ encoding the action
of $G(K)$ by $!$-restriction along $N(K) \into G(K)$). This induces
a functor:

\[
\act^!:
\Whit^{\leq n}(\sC) = \sC^{\o{I}_n,\psi} \to 
D(N(K)) \overset{N(O)^n}{\otimes} \sC 
\]

\noindent where the funny notation indicates we take the
$N(O)^n$-invariants for the diagonal action mixing
the action on $\sC$ with the right action on $D(N(K))$. 

Similarly, we have 
$\act_{*,dR}: D(N(K)) \overset{N(O)^n}{\otimes} \sC \to
\sC$.

As before, we need to show that the left adjoint $\act_!$ to
$\act^!$ is
defined on $\psi \widetilde{\boxtimes} \sF$ for every
$\sF \in \Whit^{\leq n}(\sC)$. Again, it suffices to show the corresponding
clean extension property for:\footnote{Because the
Kac-Moody cocycle is non-trivial on 
$\fn((t)) \times \Ad_{-n\check{\rho}(t)}\fg[[t]]$,
there is risk of thinking that we should be including $\kappa$ in
the middle term here. But in fact, the $Z_{KM}$-torsor
$\widehat{G(K)} \to G(K)$ is canonically
$N(K) \times G(O)^n$-equivariantly
trivial over this locus, essentially because the determinant line bundle is
canonically trivial over $\Gr_N \subset \Gr_G$. So there is
no risk of making a mistake here.}

\[
\psi \widetilde{\boxtimes} \sF \in 
D(N(K)) \overset{N(O)^n}{\otimes} \sC =
D(N(K)G(O)^n)
\overset{G(O)^n} {\otimes} \sC \overset{j_{*,dR}}{\subset}
D_{\kappa}(\ol{N(K)G(O)^n})
\overset{G(O)^n} {\otimes} \sC.
\]

From here, the argument proceeds as explained in the geometric
setting.

\subsection{Proof of Theorem \ref{t:!-avg} \eqref{i:!-avg-ff} for $m = \infty$}

Next, we discuss the fully-faithfulness of $\iota_{n,!}$.
Throughout this section, $n>0$ (as this was an obviously necessary
hypothesis).

We will reduce the claim to the following geometric result:

\begin{lem}\label{l:rodier}

For $g \in N(K)$, we have:

\[
\psi_{\o{I}_n}|_{\o{I}_n \cap \Ad_g \o{I}_n} =
\psi_{\o{I}_n} \circ \Ad_{g^{-1}}|_{\o{I}_n \cap \Ad_g \o{I}_n}
\]

\noindent if and only if $g \in N(K) \cap \o{I}_n$.

In other words, if $g \in N(K)$ but $ \not\in \o{I}_n$,
there exists $h \in \o{I}_n \cap \Ad_g\o{I}_n$ with
$\psi_{\o{I}_n}(h) \neq \psi_{\o{I}_n}(g^{-1}hg)$.

\end{lem}

The proof of this lemma will be the subject of \S \ref{s:rodier}.

\begin{proof}[Proof that Lemma \ref{l:rodier} implies 
Theorem \ref{t:!-avg} \eqref{i:!-avg-ff} for $m = \infty$]

\step

First, note that $\iota_{n,!}$ must be given 
by convolution with a kernel 
$\sK_n \in D_{\kappa}(G(K))^{(N(K),\psi),(\o{I}_n,-\psi)}$,
where the notation indicates that 
the kernel is $(N(K),\psi)$-equivariant for the left action.
and $(\o{I}_n,-\psi)$-equivariant for the right action.

Indeed, consider the $D_{\kappa}(G(K))$ as a
$D_{-\kappa}(G(K))$-module\footnote{Since we are exclusively
working with $D$-modules in this section and not
quasi-coherent sheaves, there is no need to incorporate
any critical twist here.} 
category via the right action, so that this action
commutes with the $D_{\kappa}(G(K))$-module structure. With
invariants understood with respect to the \emph{left} action,
we have the functor:

\[
\iota_{n,!}:
D_{\kappa}(G(K))^{\o{I}_n,\psi} \to 
D_{\kappa}(G(K))^{N(K),\psi}
\] 

\noindent 
which is a morphism of $D_{-\kappa}(G(K))$-module categories
for formal reasons\footnote{Namely, the proof
of Theorem \ref{t:!-avg} \eqref{i:!-avg-exists} ($m = \infty$)
shows that $\iota_{n,!}$ upgrades to a natural
transformation between the functors
$\Whit^{\leq n},\Whit:D_{\kappa}(G(K))\mod \to \DGCat_{cont}$
considered as morphisms of $\DGCat_{cont}$-enriched categories.
Then use the fact that $D_{-\kappa}(G(K))$ acts on
$D_{\kappa}(G(K)) \in D_{\kappa}(G(K))\mod$.}
Since the $(\o{I}_n,\psi)$-invariants
coincide with coinvariants, universal properties give
an object $\sK_n \in D_{\kappa}(G(K))^{(N(K),\psi),(\o{I}_n,-\psi)}$.

It is tautological that for $\sC = D_{\kappa}(G(K))$,
$\iota_{n,!}$ is given by convolution by
$\sK_n$. For general $\sC$, this follows formally by
functoriality, since:

\[
\iota_{n,!}: \sC^{\o{I}_n,\psi} =
D_{\kappa}(G(K))^{\o{I}_n,\psi} 
\underset{D_{\kappa}(G(K))}{\otimes} \sC \to 
D_{\kappa}(G(K))^{N(K),\psi} 
\underset{D_{\kappa}(G(K))}{\otimes} \sC \to 
\sC^{N(K),\psi}.
\]

\step

We are trying to show that the unit map
$\id \to \iota_n^! \iota_{n,!}$ is an equivalence.
Let us rewrite this goal in terms of kernels. 

Note that $\iota_n^!$ is given as the composition:
$\Whit(\sC) = \sC^{N(K),\psi} \xar{\Oblv} \sC \xar{\Av_*^{\psi}} 
\sC^{\o{I}_n,\psi} = \Whit^{\leq n}(\sC)$.\footnote{We are
using the hypothesis that $n>0$ here, so that $\o{I}_n$
is prounipotent.} 
In other words,
$\iota_n^!$ is given by convolution with
$\delta_{\o{I}_n}^{\psi} \in D_{\kappa}(G(K))$, where
this notation indicates the
de Rham pushforward of the character sheaf on 
$\o{I}_n$. 

Therefore, it suffices to show that
$\delta_{\o{I}_n}^{\psi} \isom 
\delta_{\o{I}_n}^{\psi} \star \sK_n =
\iota_n^! \iota_{n,!}(\delta_{\o{I}_n}^{\psi})$
is an equivalence.

\step 

Next, we observe that $\sK_n$ can be readily
calculated:

Namely, it suffices to calculate $\iota_{n,!}$
applied to $\delta_{\o{I}_n}^{\psi}$ (this is the
de Rham pushforward of the character sheaf on 
$\o{I}_n$). 
This object is obtained by $!$-averaging, so is
obtained by $!$-extending $\delta_{N(K)\o{I}_n}^{\psi}$,
the (pullback of the)
exponential $D$-module from $N(K)\o{I}_n$.\footnote{Because
the given splittings of the Kac-Moody extension for $N(K)$ and
$\o{I}_n$ coincide on their intersection, and because 
the characters $\psi$ coincide here as well, this twisted $D$-module 
makes sense.
}
Note that by Lemma \ref{l:cleanness}, this extension is clean,
i.e., the $!$-extension coincides with the $*$-extension.

\step 

By the cleanness noted above, the convolution we are trying to
compute is 
the renormalized $D$-module pushforward
of $\delta_{\o{I}_n}^{\psi} \boxtimes \delta_{N(K)\o{I}_n}^{\psi}$
along the multiplication map:

\[
\o{I}_n \times N(K)\o{I}_n \to G(K).
\]

We claim that this $D$-module is supported on $\o{I}_n$, i.e.,
is obtained by de Rham pushforward of some $D$-module
on this subscheme. In this step, we will show this
for $\kappa = 0$, and in the next step, we will show it
for general $\kappa$.

Indeed, since this $D$-module is equivariant for compact open
subgroups, this is really a problem about $D$-modules on
ind-finite type schemes. So it suffices to show that the
$!$-fibers of this convolution at all geometric points
vanish outside of $\o{I}_n$. 

For $\gamma \in G(K)$ a geometric
point, the $!$-fiber obviously vanishes unless 
$\gamma$ can be written as  $h_1 g h_2$ with $h_i \in \o{I}_n$
and $g \in N(K)$. It suffices to show that the fiber
vanishes unless $g \in N(K) \cap \o{I}_n$.

We will do this using the following paradigm:
if $H$ is a prounipotent group acting on $X$, 
$\psi:H \to \bG_a$ is a character, and $x \in X$ is a geometric point
with stabilizer $H_x \subset H$, then if 
$\psi|_{H_x}$ is non-trivial, any $(H,\psi)$-equivariant
$D$-module on $X$ vanishes along the $H$-orbit through $x$.

So consider $G(K)$ as acted on by 
$\o{I}_n \times \o{I}_n$. Note that our convolution
is equivariant with respect to the character
$(\psi,-\psi)$ (the minus sign occurring due to the
sign appearing for the right action). Note that the stabilizer
of $g$ for this action is the subgroup: 

\[
\o{I}_n \cap \Ad_g(\o{I}_n) 
\overset{h \mapsto (h,\Ad_{g^{-1}} h)}{\into}
\o{I}_n \times \o{I}_n.
\]

\noindent By Lemma \ref{l:rodier}, if 
$g \not \in N(K) \cap \o{I}_n$, then the character
$(\psi,-\psi)$ restricted to this stabilizer subgroup is
non-trivial, so our $!$-fiber vanishes, as desired.

\step

Next, we explain how the above calculation works 
for $\kappa$-twisted
$D$-modules.

Recall that $\kappa$-twisted $D$-modules are $D$-modules
on $\widehat{G(K)}$ satisfying some equivariance
with respect to the Kac-Moody center, which will actually
not be relevant here. We want to show that
any $(\o{I}_n,\psi)$-biequivariant\footnote{Here ``biequivariant"
should certainly be understood with the sign change on the
character on the right.} has vanishing $!$-fiber 
at any point $\widetilde{g} \in \widehat{G(K)}$ mapping to
$g \in N(K) \subset G(K)$ with
$g \not \in \o{I}_n$. (We remind that this
equivariance makes sense in the first place because
the Kac-Moody extension is split over $\o{I}_n$.)
 
So we need a version of Lemma \ref{l:rodier} for
$\widetilde{g}$ instead of $g$. 
Let $\sigma:\o{I}_n \to \widehat{G(K)}$
denote the splitting. Then we want the conclusion
of Lemma \ref{l:rodier} but for some $h \in \o{I}_n$ with

\[
\sigma(h) \in \sigma(\o{I}_n) \cap \Ad_g\sigma(\o{I}_n)
\]

\noindent instead. Note here that 
$\Ad_g$ makes sense on the left hand side, and coincides with
$\Ad_{\widetilde{g}}$, because $Z_{KM}$ is central
in $\widehat{G(K)}$. We will actually show:

\begin{equation}\label{eq:sigma-stab}
\sigma(\o{I}_n) \cap \Ad_g\sigma(\o{I}_n) = \sigma(\o{I}_n \cap \Ad_g\o{I}_n)
\end{equation}

\noindent which immediately gives the claim by Lemma \ref{l:rodier}.

Let $\pi$ denote the projection $\widehat{G(K)} \to G(K)$.
This map is $G(K)$-equivariant for the adjoint actions,
which immediately implies that the left hand side of
\eqref{eq:sigma-stab} is contained in the right hand side.

Now note that $\o{I}_n \cap \Ad_g \o{I}_n$ acts
on $\pi^{-1}(g)$ via the action map:

\[
\begin{gathered}
\o{I}_n \cap \Ad_g \o{I}_n \times \pi^{-1}(g) \to \pi^{-1}(g) \\
(h,\widetilde{g}) \mapsto \sigma(h)\widetilde{g}\sigma(g^{-1}h^{-1}g). 
\end{gathered}
\]

\noindent The right hand side obviously lies in $\pi^{-1}(g)$,
and this honestly defines an action map because
$\sigma$ is a homomorphism, and
$\sigma \circ \Ad_{g^{-1}}:\Ad_g(\o{I}_n) \to \widehat{G(K)}$
is too. 

Moreover, this action commutes with the $Z_{KM}$-action
on the fiber, so is induced by a homomorphism 
$\o{I}_n \to Z_{KM}$. This homomorphism must be trivial because
$\o{I}_n$ is prounipotent while $Z_{KM}$ is a torus.
Therefore, we obtain:

\[
\sigma(h)\widetilde{g}\sigma(g^{-1}h^{-1}g) = \widetilde{g}
\]

\noindent i.e.:

\[
\Ad_{g^{-1}}(\sigma(h)) = \sigma(\Ad_{g^{-1}}(h)), \hspace{.5cm}
h \in \o{I}_n \cap \Ad_g(\o{I}_n).
\]

This implies that the right hand side
of \eqref{eq:sigma-stab} is contained in the left hand side,
since for $h$ as above, 
$\sigma(h) = \Ad_g(\sigma(\Ad_{g^{-1}} h)) \in \Ad_g(\sigma(\o{I}_n))$.

\step 

At this point, we have seen that 
$\delta_{\o{I}_n}^{\psi} \star \sK_n$ is
supported on $\o{I}_n$. Obviously, it is $(\o{I}_n,\psi)$-equivariant,
and its $!$-fiber at the identity is $k$ by prounipotence of
$\o{I}_n$. Therefore, this convolution is isomorphic
to $\delta_{\o{I}_n}^{\psi}$ and the unit map is an isomorphism,
as desired.

\end{proof}

\subsection{Proof of Theorem \ref{t:!-avg} \eqref{i:!-avg-exists} and \eqref{i:!-avg-ff} 
for $m$ general}

Next, we claim that the work we have done so far implies 
the corresponding results on $\iota_{n,m,!}$ for $m$ general.

First, we need:

\begin{lem}\label{l:ladj-composition}

Let $G_2: \sC_2 \to \sC_1$ and $G_1: \sC_1 \to \sC_0$ be
functors such that $G_2$ admits a fully-faithful left adjoint $F_2$
and $G_1 \circ G_2$ admits a left adjoint $\Xi$. 
Then $G_1$ admits a left adjoint, which is computed as
$G_2 \circ \Xi$.

\end{lem}

\begin{proof}

It suffices to show that $\Psi$ maps $\sC_0$ into the subcategory
$F_2(\sC_1) \subset \sC_2$, which is equivalent to saying that:

\[
F_2G_2\Xi \to \Xi
\]

\noindent is an isomorphism. But note that we have a map:

\[
\Xi \to F_2G_2\Xi
\]

\noindent induced by adjunction from the unit map:

\[
\id_{\sC_0} \to G_1G_2F_2G_2\Xi = G_1G_2\Xi
\]

\noindent where we are using $\id_{\sC_1} = G_2F_2$. It
is straightforward to show that this map is inverse to
the given one.\footnote{Here is a more conceptual argument
that does not require checking anything, but which
feels too abstract for such a simple claim. For notational
reasons, suppose all finite colimits exist and are preserved by
every functor in sight (otherwise, use opposite Yoneda categories
instead of $\Pro$-categories), and assume all categories
are accessible (otherwise, play with universes).
Then $\Pro(G_1):\Pro(\sC_1) \to \Pro(\sC_0)$ admits a left
adjoint $\widetilde{F}_1$, and we have
$\widetilde{F}_1 = \Pro(G_2) \Pro(F_2) \widetilde{F_1} =
\Pro(G_2) \Pro(\Xi)$. The right hand side obviously maps
$\sC_0$ into $\sC_1$, so we obtain $\widetilde{F}_1 = \Pro(G_2 \Xi)$
as desired.}

\end{proof}

Regarding the theorem, we may suppose $m>n$, so in particular,
$m \neq 0$ and $\iota_{m,!}$ is fully-faithful. Then
the lemma implies $\iota_{n,m,!} = \iota_m^! \circ \iota_{n,!}$.
Moreover, if $n>0$, then because 
$\iota_{m,!} \circ \iota_{n,m,!} = \iota_{n,!}$ and
both $\iota_{m,!}$ and $\iota_{n,!}$ are fully-faithful,
clearly $\iota_{n,m,!}$ is as well.

\subsection{Proof of Theorem \ref{t:!-avg} \eqref{i:!-avg=*-avg}}\label{ss:!=*-pf}

It remains to show that for $n\leq m<\infty$, $\iota_{n,m,!}$
and $\iota_{n,m,*}$ differ by a cohomological shift. We may safely
assume $n<m$, so that $m \neq 0$.

Let $\delta_{\o{I}_m\o{I}_n}^{\psi} \in D(\o{I}_m\o{I}_n)$
denote the pullback of the exponential $D$-module
along the (well-defined) map $\psi: \o{I}_m\o{I}_n \to \bG_a$.
Note that the Kac-Moody extension canonically splits over
$\o{I}_m\o{I}_n$, since the splittings over each of these
subgroups coincides on their intersection; in particular,
$\delta_{\o{I}_m\o{I}_n}^{\psi} $ can be considered
as $\kappa$-twisted.

The main geometric result is:

\begin{lem}\label{l:mn-clean}

The $D$-module $\delta_{\o{I}_m\o{I}_n}^{\psi}$ cleanly
extends to $D_{\kappa}(G(K))$.

\end{lem}

\begin{proof}

Let $\ol{\o{I}_m\o{I}_n}$ denote the closure in $G(K)$,
and let $j:\o{I}_m\o{I}_n \into \ol{\o{I}_m\o{I}_n}$ 
denote the open embedding. We want to show:

\begin{equation}\label{eq:mn-clean}
j_!(\delta_{\o{I}_m\o{I}_n}^{\psi} ) \isom 
j_{*,dR}(\delta_{\o{I}_m\o{I}_n}^{\psi}).
\end{equation}

\noindent Note that these $\kappa$-twisted 
$D$-modules are $(\o{I}_m,\psi)$
equivariant with respect to the left action of $G(K)$ on itself.
Therefore, by Theorem \ref{t:!-avg} \eqref{i:!-avg-ff},
it suffices to show that this map is an isomorphism
after applying $\iota_{m,!}$, i.e., $!$-averaging to Whittaker
with respect to the left action.

Note that $\o{I}_m\o{I}_n \subset N(K)\o{I}_n$, so the same
is true of their closures. Therefore, the $!$-averages
of both sides of \eqref{eq:mn-clean} are supported
on $\ol{N(K)\o{I}_n}$. 

We claim that 
$\ol{N(K)\o{I}_n} \cap N(K) \Ad_{-n\check{\rho}(t)}G(O) = N(K)\o{I}_n$.
Indeed, it suffices to show that 
$N(K)\o{I}_n \subset N(K)\Ad_{-n\check{\rho}(t)}G(O)$ is closed,
and for this it suffices to show this mod $N(K)$, i.e.,
that: 

\[
\o{I}_n/\Ad_{-n\check{\rho}(t)} N(O) \to 
\Ad_{-n\check{\rho}(t)}G(O)/\Ad_{-n\check{\rho}(t)} N(O)
\]

\noindent is a closed embedding. The $n = 0$ case is obvious, and
for $n>0$, this is the embedding of an orbit for a prounipotent
group on a quasi-affine scheme, so is a closed embedding.

Therefore, by Lemma \ref{l:cleanness},
our $!$-averages are cleanly extended from $N(K)\o{I}_n$. Moreover,
they are $(N(K),\psi)$ and $(\o{I}_n,-\psi)$-equivariant, so it
suffices to show that our map induces an isomorphism
between fibers at $1 \in G(K)$, which is clear.

\end{proof}

\begin{rem}

At the end of the day, we have shown that an $(\o{I}_m,\psi)$
and $(\o{I}_n,\psi)$-equivariant $D$-module on
$\ol{\o{I}_m\o{I}_n}$ is cleanly extended from the $\o{I}_m\o{I}_n$
by a somewhat indirect argument, using
Theorem \ref{t:!-avg} \eqref{i:!-avg-ff} to reduce to
the case $m = \infty$ and then using our explicit knowledge
of double coset representatives for $N(K)$ and $G(O)$.
Presumably there is a more direct argument 
analyzing $\ol{\o{I}_m\o{I}_n}$ as is. Such an argument would
mean that Theorem \ref{t:!-avg} \eqref{i:!-avg-ff}
(hence \S \ref{s:rodier}) would play an inessential
(if still philosophically important) role in proving Theorem \ref{t:!-avg}.

\end{rem}

We now conclude the proof of Theorem \ref{t:!-avg}.
As in the proof of Theorem \ref{t:!-avg} \eqref{i:!-avg-ff} that
$\iota_{n,m,!}$ and $\iota_{n,m,*}$ are defined by kernels:
in the notation above, the former is given by
$j_!(\delta_{\o{I}_m\o{I}_n}^{\psi})$ and the latter
by 
$j_{*,dR}(\delta_{\o{I}_m\o{I}_n}^{\psi} )
[-2(m-n)\Delta]$,
where the shift appears since in the finite-dimensional
setting, $*$-averaging is given by convolution with the
constant sheaf, not the dualizing sheaf.

\section{Some supporting calculations}\label{s:rodier}

\subsection{}

In this section, we prove Lemma \ref{l:rodier}.\footnote{This
result is essentially \cite{rodier} Lemma 4, except that
he works with a slightly different series of subgroups
(but with similar enough properties that the same arguments
should work uniformly for both). Unfortunately, I couldn't understand
Rodier's argument: his result
relies on Lemma 13 from \emph{loc. cit}., which in particular says that
every element of $f+t^N\fg[[t]]$ can be conjugated into a Borel;
this is not true since we can approximate $f$ by elliptic
elements in the Kostant slice (e.g. $G = GL_2$ and
take $\begin{pmatrix} 0 & t^{2N+1} \\ 1 & 0 \end{pmatrix}$).
Unfortunately, the arguments we give below use loop rotation in an
essential way, so cannot be adapted to Rodier's
$p$-adic setting.} 

The argument is presented according to its internal
logical development, so the reader who cannot see beyond 
Lemma \ref{l:rodier} should read this section backward, beginning
in \S \ref{ss:rodier-pf}.

\subsection{} Throughout this section, we fix
$\kappa$ an $\Ad$-invariant identification $\fg \simeq \fg^{\vee}$.
For $i \in \cI_G$, 
we let $f_i \in \fg$ be a non-zero vector of weight
$-\alpha_i$ normalized so that $\kappa(e_i,f_i) = 1$. 
We let $f = \sum f_i$, so that $f$ maps to our fixed character
of $\fn$.

We obtain an identification $\fg((t))^{\vee} \simeq \fg((t))dt$ via the residue
pairing. Note that the image of $f dt$ in $\Lie\o{I}_n^{\vee}$
is $\psi_{\o{I}_n}$ for all $0 \leq n \leq \infty$.

\subsection{}

It is convenient to use the notation $I \subset G(O)$ for the
Iwahori subgroup (i.e., $G(O) \times_G B$). The reader is cautioned
that this is not a typo: the groups $\o{I}_n$ will not appear until later in this 
section.

\subsection{Some affine Springer theory}

The main geometric input be the following result on
affine Springer fibers.

Let $\xi \in \fg((t))$. Recall that $\Spr^{\xi} \subset \Gr_G = G(K)/G(O)$
is:

\[
\{g \in G(K) \mid \Ad_{g^{-1}}(\xi) \in \fg[[t]]\}/G(O).
\]

\noindent Note that $\Spr^{\xi}$ is closed in $\Gr_G$.

\begin{thm}\label{t:spr-fib}

For $\xi \in f + \Lie(I) \subset \fg[[t]]$, 
$\Spr^{\xi} \cap \ol{\Gr}_{N} = \{1\} \in \Gr_G$. Here
$\ol{\Gr}_N$ is the usual closure of the semi-infinite orbit
$\Gr_N \subset \Gr_G$.

\end{thm}

\begin{rem}

We only
need this result at the level of reduced indschemes and so treat
it as such, but one can show that this intersection is actually transverse
by an infinitesimal version of what follows: 
see Remark \ref{r:inf-spr}.

\end{rem}

\begin{rem}

Note that elements of $f+\Lie(I) \subset \fg[[t]] \subset \fg((t))$ 
are regular (but certainly not necessarily semisimple),
since their reductions mod $t$ are by usual Kostant theory.
It is helpful to think of $f+\Lie(I)$ as a fattened version of the
scheme $f+\fb[[t]]$, which is the jet space of the scheme $f+\fb$ 
appearing in and well-understood by usual Kostant theory.

\end{rem}

\begin{rem}\label{r:spr-fdim-1}

The above can be regarded as an affine version of the statement
that (usual) Springer fibers 
of points of $f+\fb$ lie in the open Bruhat cell $B w_0 B \subset G/B$,
which can be proved by similar methods.

\end{rem}

\begin{proof}[Proof of Theorem \ref{t:spr-fib}]

The method is the following: we explicitly compute in the case
$\xi = f$, and then use dynamical methods to reduce to it.

\step\label{st:f-spr}

First, we treat the case $\xi = f$.

Suppose $g \in \Spr^f$. By Iwasawa, we can write
$g = \gamma \check{\la}(t) G(O)$ for $\gamma \in N^-(K)$ and 
$\check{\la} \in \check{\La}$. This point lies in our Springer fiber
if and only if:

\[
\Ad_{-\check{\la}(t)} \Ad_{\gamma^{-1}}(f) \in \fg[[t]].
\]

\noindent Note that $\Ad_{\gamma^{-1}}(f) \in f + [\fn^-,\fn^-]((t))$.
Since $\Ad_{-\check{\la}(t)}$ preserves $[\fn^-,\fn^-]((t))$, 
we obtain:

\[
\Ad_{-\check{\la}(t)} \Ad_{\gamma^{-1}}(f) \in 
\sum_{i \in \cI_G} t^{(\check{\la},\alpha_i)} f_i + [\fn^-,\fn^-]((t)).
\]

\noindent If this point lies in $\fg[[t]]$, then each
$(\check{\la},\alpha_i) \geq 0$, i.e., $\check{\la}$ is dominant.

On the other hand, recall that $\ol{\Gr}_N$ is stratified
by 
$\Gr_B^{\check{\mu}} = N(K) \check{\mu}(t) G(O)/G(O)$ for $\check{\mu} \in -\check{\La}^{pos}$.
Also, recall that:

\[
N^-(K) \check{\la}(t) G(O) \cap
N(K) \check{\mu}(t) G(O) \neq \emptyset \Leftrightarrow
\check{\mu} - \check{\la} \in \check{\La}^{pos}
\]

\noindent (this is easiest to see from the perspective of Zastava spaces,
e.g., in the formalism of \cite{mirkovic-notes} or \cite{cpsi} \S 2).

Therefore, our point $g$ above can only lie in $\ol{\Gr}_N$
if $\check{\mu} - \check{\la} \in \check{\La}^{pos}$. 
On the one hand, this implies $(\rho,\check{\mu}-\check{\la}) \geq 0$.
On the other hand, $\check{\mu} \in -\check{\La}^{pos}$
means $(\rho,\check{\mu}) \leq 0$, and $\check{\la}$ being
dominant (and $\rho \in \frac{1}{2} \La^{pos}$) means
$(\rho,-\check{\la}) \leq 0$. But since 
$\check{\mu} - \check{\la}, -\check{\mu} \in \check{\La}^{pos}$,
this implies that they are both zero, so $\check{\la} = 0$ as well.

Now we recall\footnote{Let us recall the proof for the sake
of completeness. Note that $\Gr_N \cap \Gr_{N^-}$ parametrizes
maps from the formal disc $\widehat{\cD} \to N^-\backslash G/N$
with a trivialization on the formal punctured disc $\o{\cD}$
(where we have algebraized the former to obtain a map from the
adic disc, i.e., $\Spec$ instead of $\Spf$). Then $N^-N \subset G$
is \emph{closed}, since $N^- \cdot 1 \in G/N$ is a unipotent 
orbit on a quasi-affine variety, and these are always closed.
Therefore, since $\o{\cD}$ maps to $1 \in N^- \backslash G/N$,
$\widehat{\cD}$ must map through $N^-\backslash N^-N/N = \Spec(k)$.}
that $\Gr_N \cap \Gr_{N^-} = \{1\}$
to deduce the claim in this case.

\step 

Next, we construct a useful action of $\bG_m$ on the indschemes
$\fg((t))$. For this, for technical
reasons, we need to fix an auxiliary integer
$h$ that is ``large enough": $h \geq (\check{\rho},\alpha)$
for all positive roots $\alpha$ suffices. (I.e., we can
take $h$ to be the Coxeter number of $G$ minus 1.)

For $\lambda \in \bG_m$, we let $\xi \mapsto \presup{\lambda}{\xi}$
denote the action by loop rotation (i.e., if $\xi = \xi(t)$, 
$\presup{\lambda}{\,\xi}(t) = \xi(\lambda t)$).

Then our action on $\fg((t))$ is given by the formula:

\[
\lambda \star \xi \coloneqq 
\lambda \cdot \Ad_{\check{\rho}(\lambda)} \presup{\lambda^h}{\xi}.
\]

\noindent (Here the $\lambda$ in front is the usual homothety
action on the vector space $\fg((t))$.) Note that: 

\[
\lambda \star f = \lambda \cdot \Ad_{\check{\rho}(\lambda)} \presup{\lambda^h}{f} =
\lambda \cdot \Ad_{\check{\rho}(\lambda)} (f) = 
\lambda \cdot \lambda^{-1} f = f.
\]

Moreover, for $\xi \in \fb[[t]]$, we clearly
have $\lim_{\lambda \to 0} \lambda \star \xi = 0$.
We claim that the same holds for $\xi \in t\fn^{-}[[t]]$.
Indeed, since this action preserves the decomposition into
weight spaces, and also the decomposition via $t\partial_t$-eigenspaces,
we can assume $\xi = t^m f_{\alpha}$ ($m>0$)
for some $\alpha > 0$, $f_{\alpha} \in \fn^-$ a $-\alpha$ root vector.

\[
\lambda \star t^m  f_{\alpha} = 
\lambda^{1-(\check{\rho},\alpha)+m\cdot h} \cdot t^m f_{\alpha}.
\]

\noindent Since $1-(\check{\rho},\alpha)+m\cdot h \geq
1-(\check{\rho},\alpha)+h \geq 1$, we obtain the claim by assumption
on $h$.

\step 

We need actions of $\bG_m$ on $G(K)$ and $\Gr_G$.

Again, for $g \in \Gr_G$ or $G(K)$ and $\lambda \in \bG_m$,
we let $\presup{\lambda}{g}$ denote the 
element obtained by loop rotation. 
Then our action on 
$\Gr_G$
is given by:

\[
\lambda \star g \coloneqq 
\check{\rho}(\lambda) \cdot \presup{\lambda^h}{g}, \hspace{.5cm}
g \in \Gr_G, \lambda \in \bG_m,
\]

\noindent and our action on $G(K)$ is given by:

\[
\lambda \star g \coloneqq 
\Ad_{\check{\rho}(\lambda)} \presup{\lambda^h}{g}, \hspace{.5cm}
g \in G(K), \lambda \in \bG_m.
\]

\noindent Because $\check{\rho}(\lambda) \in T \subset G(O)$,
the projection map $G(K) \to \Gr_G$ is $\bG_m$-equivariant.
(This is also the reason $\check{\rho}(\lambda) \cdot g$ makes sense
in the first place for $g \in \Gr_G$, even if 
$\check{\rho} \not\in \check{\La}$.)

Note that this action is compatible with
our Lie algebra action in the sense that for 
$g \in G(K)$ and $\xi \in \fg((t))$:

\[
\Ad_{\lambda \star g}(\lambda \star \xi) = 
\lambda \star \Ad_g(\xi)
\]

\noindent Indeed, the $\bG_m$-actions of loop rotation
and conjugation by $\check{\rho}(-)$ each have this
property, so the $\star$-action does as well.

In particular, since the $\star$-action
preserves $\fg[[t]]$ (since it mixes homotheties and conjugation
by elements of $T \subset G(O)$), 
we deduce that if $g \in \Spr^{\xi} \subset \Gr_G$,
then $\lambda \star g \in \Spr^{\lambda \star \xi}$. 

\step 

Now suppose that $\xi \in f+\Lie(I)$ and $g \in \Spr^{\xi} \subset \Gr_G$.
Then $\lim_{\lambda \to 0}(\lambda \star g)$ must exist,
since $\Gr_G$ is ind-proper. Moreover, since
$\lim_{\lambda \to 0} \lambda \star \xi = f$ and
$\lambda \star g \in \Spr^{\lambda \star \xi}$,
$\lim_{\lambda \to 0}(\lambda \star g) \in \Spr^f$.\footnote{Here
is the actual argument for this fact:

Let $\bA_{\lambda}^1$ denote $\bA^1$ with coordinate $\lambda$.
Then we have a
closed sub-indscheme 
$\Spr^{\xi,dynam} \subset \bA_{\lambda}^1 \times \Gr_G$ parametrizing 
$(\lambda \in \bA^1, g \in \Spr^{\lambda \star \xi})$. Indeed, it
is obtained by base-change along
$\bA_{\lambda}^1 \xar{\lambda \mapsto \lambda \star \xi} \fg((t))$ 
from the relevant
\emph{universal affine Springer fiber}, which itself is closed in $\fg((t)) \times \Gr_G$ since it is the fiber product
$\big(\fg((t))\times \Gr_G \big)\times_{\fg((t))/G(O)} \fg[[t]]/G(O)$, where we are using the map 
$\fg((t)) \times \Gr_G \xar{(\vph,g)\mapsto \Ad_{g^{-1}}(\vph)} \fg((t))/G(O)$.

Therefore, the projection $\Spr^{\xi,dynam} \to \bA_{\lambda}^1$
is ind-proper, since $\Gr_G$ is ind-proper. Then for
$g \in \Spr^{\xi}$, the map $\lambda \to \lambda \star g$
defines a section of this projection over $\bA_{\lambda}^1 \setminus 0$,
and by the valuative criterion, this section must extend
over $\bA_{\lambda}^1$. Then the value at $0 \in \bA_{\lambda}^1$
obviously equals $\lim_{\lambda \to 0}(\lambda \star g)$ and
lies in $\Spr^{\xi,dynam} \times_{\bA_{\lambda}^1} 0 =
\Spr^{\lim_{\lambda \to 0} \lambda \star \xi} = \Spr^{f}$.}

Now, observe that the $\star$-action preserves $\Gr_N$,
and so $\ol{\Gr}_N$ as well. Therefore, 
if $g \in \Spr^{\xi} \cap \overline{\Gr}_N$,
we deduce that:

\[
\lim_{\lambda \to 0} \lambda \star g \in \Spr^f \cap \ol{\Gr}_N 
\overset{\text{Step \ref{st:f-spr}}}{=} \{1\},
\]

\noindent i.e., $\lim_{\lambda \to 0} \lambda \star g = 1$.
Therefore, it suffices to show that if $g \in \overline{\Gr}_N$
satisfies $\lim_{\lambda \to 0} \lambda \star g = 1$,
then $g = 1$.

Note that this condition obviously implies $g \in \Gr_N$.
Then we have an isomorphism
$\exp:\fn((t))/\fn[[t]] \isom \Gr_N$ that is $\bG_m$-equivariant
for the $\star$-action on the target and the
action $\lambda^{-1} \cdot \lambda \star (-)$ on the source
(i.e., we need to cancel out the homothety occurring in the
definition of the $\star$-action on $\fg((t))$). 

Since this action on $\fn((t))/\fn[[t]]$ preserves the decomposition
into weight spaces and $t\partial_t$-eigenspaces,
$\lim_{\lambda \to 0}(\xi) = 0$ for $\xi \in \fn((t))/\fn[[t]]$ if and only
if this is true for each graded component of $\xi$.
Therefore, it suffices to show that either the limit as
$\lambda \to 0$ of $\lambda^{-1}(\lambda \star t^{-m} e_{\alpha})$ ($m>0$)
does not exist, or if it does exist, is not $0$. We then
compute:

\[
\lambda^{-1}(\lambda \star t^{-m} e_{\alpha}) = 
\lambda^{-mh + (\check{\rho},\alpha)} t^{-m} e_{\alpha}
\]

\noindent and since 
$-mh + (\check{\rho},\alpha) \leq -m + (\check{\rho},\alpha) \leq 0$,
we obtain the claim.

\end{proof}

\begin{rem}\label{r:inf-spr}

This remark will not be needed in what follows, but
as promised earlier, let us indicate the proof that the
intersection above is transverse. 

The tangent space
of the intersection is 
$\mathfrak{spr}^{\xi} \cap \fn((t))/\fn[[t]] \subset \fg((t))/\fg[[t]]$,
where $\mathfrak{spr}^{\xi} \coloneqq 
\{\vph \in \fg((t)) \mid [\xi,\vph] \in \fg[[t]]\}/\fg[[t]]$.\footnote{The
fact that $[\xi,\fg[[t]]] \subset \fg[[t]]$ is of course because
$\xi \in \fg[[t]]$, which corresponds to the fact that
$1\in \Spr^{\xi}$ in the first place.}

For $0 \neq \vph \in \mathfrak{spr}^{\xi}$, the 1-parameter family
of lines $k \cdot (\lambda \star \vph) \in \bP(\fg((t))/\fg[[t]])$ must
admit a limit as $\lambda \to 0$, and this limit must lie
in $\bP(\mathfrak{spr}^f)$.\footnote{To justify the 
limiting process here, we claim that the family of subspaces
$\mathfrak{spr}^{\xi} \subset \fg((t))/\fg[[t]]$ is well-behaved as we vary
$\xi \in \fg[[t]] \times_{\fg} \fg^{reg}$. To see this,
show that $\mathfrak{spr}^{\xi}$ as $\fz(\xi) + \fg[[t]]/\fg[[t]]$
for $\fz(\xi)$ the centralizer of $\xi$, 
and then use the smoothness of the group scheme of regular
centralizers.}

By representation
theory of the principal $\sl_2$ $(e,2\Lie(\check{\rho})(1),f)$), 
$\mathfrak{spr}^f \cap \fn((t))/\fn[[t] = \{0\} \subset \fg((t))/\fg[[t]]$,
i.e., $\bP(\mathfrak{spr}^f) \cap \bP(\fn((t))/\fn[[t]]) = \emptyset$.
Since $\bP(\fn((t))/\fn[[t]]) \subset \bP(\fg((t))/\fg[[t]])$ is closed,
preserved under the $\star$-action,
$\vph$ could not have been in $\fn((t))/\fn[[t]]$.

\end{rem}

\subsection{Regular centralizers} 

We now deduce the following result.

\begin{cor}\label{c:reg-cent}

For $\xi \in f+\Lie(I)$, let $Z(\xi) \subset G(K)$ denote its
centralizer. Then $Z(\xi) \cap N(K)G(O) \subset G(O)$.

\end{cor}

\begin{proof}

Clearly the $Z(\xi)$-orbit through $1 \in \Gr_G$ is contained
in $\Spr^{\xi}$. Therefore, Theorem \ref{t:spr-fib} implies
$Z(\xi)G(O) \cap N(K) G(O) = G(O)$, giving the claim.

\end{proof}

\begin{rem}\label{r:spr-fdim-2}

This may be compared to the finite-dimensional result
that for $\xi \in f+\fb$, $Z(\xi) \cap B = Z(G)$.
Let us deduce this claim from the result from 
Remark \ref{r:spr-fdim-1}.

Note that $Z(\xi)$ acts trivially on the (usual)
Springer fiber for $\xi$. Indeed, by regularity of $\xi$,
$Z(\xi)$ is generated by
the center of $G$ and its connected component $Z(\xi)^0$.
The former obviously acts trivially, and the latter acts trivially
because this Springer fiber is finite (by regularity again).

Therefore, $Z(\xi)$ is contained in any Borel $B^{\prime}$ whose
Lie algebra contains $\xi$. But since the Springer
fiber is contained in $Bw_0B/B$, $B^{\prime}$ is opposed to
$B$, so $Z(\xi) \cap B \subset T$. But clearly for
$\xi$ of this form, $Z(\xi) \cap T = Z(G)$.

\end{rem}

\subsection{}

Next, we obtain:

\begin{cor}\label{c:intersection-1}

$g \in N(K)$ satisfies:

\[
( f+\Lie(I)) \cap \Ad_g(f+\Lie(I)) \neq \emptyset
\]

\noindent if and only if $g \in N(O)$.

\end{cor}

First, we need:

\begin{lem}\label{l:reg-conj}

Any $\xi_1,\xi_2 \in f+\Lie(I)$ are conjugate by an element
of $G(K)$ if and only
if they are conjugate by an element of 
$\o{I} \coloneqq G(O) \underset{G}{\times} N$. 

\end{lem}

\begin{proof}

Recall that the morphism $(f+\fb)/N \to \fg/G$
is smooth. Writing $f+\Lie(I)/\o{I}$ as 
$\fg/G(O) \times_{\fg/G} f+\fb/N$, formal smoothness implies that
any point of $f+\Lie(I)/\o{I}$ lifts to $f+\fb[[t]]/N(O)$.

This lift is completely determined by the characteristic polynomial,
since by Kostant's theory \cite{kostant-slice},
the composition of $(f+\fb)/N \to \fg/G$ with the projection
$\chi:\fg/G \to \fg//G$ is an isomorphism (here $\fg//G$ is the GIT quotient 
$\coloneqq\Spec(\Sym(\fg^{\vee})^G)$), so
$f+\fb[[t]]/N(O) \isom (\fg//G)(O)$. 

Now the claim follows, since the characteristic
polynomial is determined by the underlying $G(K)$-conjugacy class.

\end{proof}

\begin{proof}[Proof of Corollary \ref{c:intersection-1}]

Suppose $\xi \in f+\Lie(I)$ with $\Ad_g(\xi) \in f+\Lie(I)$ as well.
By Lemma \ref{l:reg-conj}, we can find $\gamma \in G(O)$
with $\Ad_{\gamma}(\xi) = \Ad_g(\xi)$. 
Then $g^{-1} \gamma \in Z(\xi)$, but also in $N(K)G(O)$,
so lies in $G(O)$ by Corollary \ref{c:reg-cent}. Since $\gamma$
was already in $G(O)$,
this gives $g \in G(O) \cap N(K) = N(O)$, as desired.

(The converse direction is clear: for $g \in N(O)$, 
$\Ad_g$ preserves $f+\Lie(I)$.)

\end{proof}

\subsection{Return to adolescence}\label{ss:nostalgia}

We now essentially rewrite Corollary \ref{c:intersection-1},
but using the subgroups $\o{I}_n$ instead.

In what follows, for simplicity, we use $dt$ (and $\kappa$) to identify 
$\fg((t))^{\vee} = \fg((t))dt $ with $\fg((t))$. For a subspace
$V \subset \fg((t))$, we use $V^{\perp}$ to indicate the
subspace $(\fg((t))/V)^{\vee} \subset \fg((t))^{\vee} = \fg((t))$.

\begin{cor}\label{c:intersection-2}

For $n>0$, $g \in N(K)$ satisfies:

\[
( f+\Lie\o{I}_n^{\perp}) \cap \Ad_g(f+\Lie\o{I}_n^{\perp}) \neq \emptyset
\]

\noindent if and only if 
$g \in N(K) \cap \o{I}_n = \Ad_{-n\check{\rho}(t)} N(O)$.

\end{cor}

\begin{proof}

We compute the image of both terms of the intersection
under the automorphism $t^n\cdot\Ad_{n\check{\rho}(t)}$ of 
$\fg((t))$ (where $t^n$ is acting by homotheties on this $K$-vector space).

Since the identification $\fg((t)) \simeq \fg((t))^{\vee}$ is
$G(K)$-equivariant, 
$\Ad_{n\check{\rho}(t)}(f+\Lie\o{I}_n^{\perp})$ =
$t^{-n}f + (t^n\fg[[t]]+\fn[[t]])^{\perp}$
(recall that by definition, 
$t^n\fg[[t]]+\fn[[t]] = \Ad_{n\check{\rho}(t)}(\Lie\o{I}_n)$).
This dual lattice is easy to compute:

\[
(t^n\fg[[t]]+\fn[[t]])^{\perp} = \fg[[t]] + t^{-n}\fb[[t]].
\]

\noindent So finally, we obtain:

\[
t^n \Ad_{n\check{\rho}(t)}(f+\Lie\o{I}_n^{\perp}) = 
f+t^n\fg[[t]] + \fb[[t]] \subset f+\Lie(I)
\]

\noindent (since $n>0$).

Moreover, we have:

\[
\begin{gathered}
t^n \Ad_{n\check{\rho}(t)}\Ad_g(f+\Lie\o{I}_n^{\perp}) = 
\Ad_{(n\check{\rho})(t) \cdot g \cdot (-n\check{\rho})(t)}
t^n\Ad_{n\check{\rho}(t)}(f+\Lie\o{I}_n^{\perp}) = \\
\Ad_{(n\check{\rho})(t) \cdot g \cdot (-n\check{\rho})(t)}(f+t^n\fg[[t]] + \fb[[t]] ) \subset
\Ad_{(n\check{\rho})(t) \cdot g \cdot (-n\check{\rho})(t)}(f+\Lie(I)).
\end{gathered}
\]

So if the original intersection is non-trivial, then:

\[
f+\Lie(I) \cap \Ad_{(n\check{\rho})(t)(g)(-n\check{\rho})(t)}(f+\Lie(I)) \neq \emptyset
\]

\noindent so that Corollary \ref{c:intersection-1} gives:

\[
(n\check{\rho})(t)(g)(-n\check{\rho})(t) \in N(O) \Leftrightarrow
g \in \Ad_{-n\check{\rho}(t)} N(O).
\]

\end{proof}

\subsection{}\label{ss:rodier-pf}

We now proceed to our original objective.

\begin{proof}[Proof of Lemma \ref{l:rodier}]

Let $g \in N(K)$.
Note that:

\[
\psi_{\o{I}_n}|_{\o{I}_n \cap \Ad_g \o{I}_n} =
\psi_{\o{I}_n} \circ \Ad_{g^{-1}}|_{\o{I}_n \cap \Ad_g \o{I}_n}
\]

\noindent if and only if their derivatives coincide on the 
corresponding Lie algebras, i.e.:

\[
\Lie(\psi_{\o{I}_n})|_{\Lie\o{I}_n \cap \Ad_g \Lie\o{I}_n} =
\Lie(\psi_{\o{I}_n}) \circ \Ad_{g^{-1}}|_{\Lie\o{I}_n \cap \Ad_g \Lie\o{I}_n}.
\]

\noindent This occurs if and only if there exists 
$\xi \in \fg((t))^{\vee}$ with:

\[
\begin{gathered}
\xi|_{\Lie\o{I}_n} =
\Lie(\psi_{\o{I}_n})|_{\Lie\o{I}_n} \\
\xi|_{\Ad_g\Lie\o{I}_n} =
\Lie(\psi_{\o{I}_n}) \circ \Ad_{g^{-1}}|_{\Ad_g\Lie\o{I}_n}.
\end{gathered}
\]

Now observe that 
$f+\Lie\o{I}_n^{\perp} \subset \fg((t)) = \fg((t))^{\vee}$
is the subspace of functionals on $\fg((t))$ whose
restriction to $\o{I}_n$ is $\Lie(\psi_{\o{I}_n})$.
Similarly, $\Ad_g(f+\Lie\o{I}_n^{\perp})$
is the space of functionals whose restriction to
$\Ad_g\Lie\o{I}_n$ is $\Lie(\psi_{\o{I}_n}) \circ \Ad_{g^{-1}}$.
Therefore:

\[
\xi \in ( f+\Lie\o{I}_n^{\perp}) \cap \Ad_g(f+\Lie\o{I}_n^{\perp})
\]

\noindent meaning that $g \in N(K) \cap \o{I}_n$ by
Corollary \ref{c:intersection-2}.

\end{proof}

\subsection{Application: Semi-classical Theorem \ref{t:!-avg}}

We now give one more application of the above results, which 
will be used in \S \ref{s:skryabin}.

\begin{thm}\label{t:!-avg-cl}

For every $0\leq n \leq m \leq \infty$, the morphism:

\[
f+\Lie\o{I}_n^{\perp} \cap \Lie\o{I}_m^{\perp}/\o{I}_n\cap \o{I}_m \to
f+\Lie\o{I}_m^{\perp}/\o{I}_m
\]

\noindent is a finitely presented closed 
embedding.\footnote{For $m = \infty$ and from the DAG perspective, we
should really say ind-finitely presented, as for $0 \into \colim_n \bA^n$.}

\end{thm}

\begin{proof}

\step\label{st:perp-conj}

Suppose $n>0$ for the moment.
Then we claim that every element of 
$f+\Lie\o{I}_n^{\perp}$ can be conjugated under 
$\o{I}_n$ into $f+t^{-n}\Ad_{-n\check{\rho}(t)}\fb^e[[t]]$.
(Of course, $e$ here fits into our favorite principal $\sl_2$
as before.)

Using $\Ad_{-n\check{\rho}(t)}N(O)$, we can conjugate
any element of $f+t^{-n}\Ad_{-n\check{\rho}(t)}\fb[[t]]$
into $f+t^{-n}\Ad_{-n\check{\rho}(t)}\fb[[t]]$. Therefore, 
it suffices to show that every element of 
$f+\Lie\o{I}_n^{\perp}$ can be conjugated into
$f+t^{-n}\Ad_{-n\check{\rho}(t)}\fb[[t]]$ under $\o{I}_n$;
in fact, it suffices to use $\Ad_{-n\check{\rho}(t)} \cK_n$
for $\cK_n\subset G(O)$ the $n$th congruence subgroup.

Indeed, by a Hensel argument, this reduces to the corresponding
infinitesimal statement. Applying $t^n\Ad_{n\check{\rho}(t)}$,
this statement is:

\[
[\, t^n\fg[[t]],f+\fb[[t]] \, ] + \fb[[t]] = f+t^n\fg[[t]]+\fb[[t]] =
t^n\Ad_{n\check{\rho}(t)}(f+\Lie\o{I}_n^{\perp}).
\]

\noindent (That is a Lie bracket in the first term.)
But this is standard representation of our principal $\sl_2$.

\step

We will need the following construction.

Let $Z^{abs} \subset G(O) \times f+\fb^e[[t]]$ denote the group
scheme of (integral) regular centralizers. Recall that this group scheme
is obtained by taking jets of a smooth group scheme over $f+\fb^e$.

We obtain a pro-smooth group scheme over 
$f+t^{-n}\Ad_{-n\check{\rho}(t)}\fb^e[[t]]$ defined as:

\[
\Ad_{-n\check{\rho}(t)} Z^{abs} \subset 
\Ad_{-n\check{\rho}(t)} G(O) \times f+t^{-n}\Ad_{-n\check{\rho}(t)} \fb^e[[t]].
\]

\noindent Finally, we set 
$Z_n^{abs}$ to be the intersection of
$\Ad_{-n\check{\rho}(t)} Z^{abs} $ with $\o{I}_n$.\footnote{Here
intersecting with $\o{I}_n$ means with the
corresponding constant group scheme. We abuse terminology
in this way throughout.}
For example, $Z_0^{abs} = Z^{abs}$.

\step\label{st:!-avg-cl-zn-kn}

We will now show that $Z_n^{abs}$ can also be obtained
by intersecting $\Ad_{-n\check{\rho}(t)} Z^{abs}$ with
$\Ad_{-n\check{\rho}(t)} \cK_n$.
Assuming this, note that $Z_n^{abs}$ is pro-smooth:
it is conjugated from the $n$th congruence subgroup of 
$Z^{abs}$, which is jets on a smooth group scheme.

By Remark \ref{r:spr-fdim-2}, the group scheme
of regular centralizers over $f+\fb^e$ 
has trivial intersection with $N$. Moreover, it is
straightforward to see that this intersection is transverse;
this is a strictly simpler Lie algebra calculation.

Applying finite order jets, 
we find that no non-trivial element of $N(O/t^n)$
can stabilize an element of $f+\fb[[t]]/t^n\fb[[t]] \subset \fg[[t]]/t^n\fg[[t]]$.

Let us prove the claim from above now. Applying 
$t^n\Ad_{n\check{\rho}(t)}$, it is the same as to show that:

\[
Z^{abs} \cap G(O) \times_{G(O/t^n)} N(O/t^n) = Z^{abs} \cap \cK_n.
\]

\noindent But this is exactly what was shown above:
an element of the left hand side is a pair
$\xi \in f+\fb^e[[t]]$ and $\gamma \in G(O) \times_{G(O/t^n)} N(O/t^n)$
stabilizing it; reducing this equation modulo $t^n$ and applying
our earlier analysis, we find
that $\gamma$ must be the identity modulo $t^n$. 

\step 

Now suppose again that $n>0$.
Note that we have a canonical map:

\[
f+t^{-n}\Ad_{-n\check{\rho}(t)} \fb^e[[t]] \to 
f+\Lie\o{I}_n^{\perp}/\o{I}_n.
\]

\noindent We claim that this map realizes the target
as the classifying stack of $Z_n^{abs}$,
with the corresponding section being given by the Kostant section.

Indeed, this follows immediately from the definition
of $Z_n^{abs}$ and Step \ref{st:perp-conj}.

\step 

Now suppose $n<m$, so $m \neq 0$.
Then similar analysis applies to the quotient stack:

\[
f+\Lie\o{I}_n^{\perp} \cap \Lie\o{I}_m^{\perp}/\o{I}_n\cap \o{I}_m.
\]

\noindent Namely, this is again a classifying stack, this
time over: 

\[
\big(f+\Lie\o{I}_n^{\perp} \cap \Lie\o{I}_m^{\perp}\big) 
\cap \big(f+\fb^e((t))\big) =
f+t^{-n}\Ad_{-n\check{\rho}(t)} \fb^e[[t]].
\]

\noindent The relevant 
group scheme is now the group scheme of centralizers
that lie in $\o{I}_n \cap \o{I}_m$. 

\step

We now prove the theorem. We may assume $n<m$, so $m \neq 0$.

We claim that there is a Cartesian diagram:

\[
\xymatrix{
f+\Lie\o{I}_n^{\perp} \cap \Lie\o{I}_m^{\perp}/\o{I}_n\cap \o{I}_m 
\ar[rr] \ar[d] & &
f+\Lie\o{I}_m^{\perp}/\o{I}_m \ar[d] \\
f+t^{-n}\Ad_{-n\check{\rho}(t)}\fb^e[[t]] \ar@{^(->}[rr] & &
f+t^{-m}\Ad_{-m\check{\rho}(t)}\fb^e[[t]]
}.
\]

\noindent Here the vertical arrows are induced by the Kostant section.
Obviously this claim would imply the theorem.

By our above analysis, this Cartesian product 
is the classifying stack of $Z_m^{abs}$ restricted to
$f+t^{-n}\Ad_{-n\check{\rho}(t)}\fb^e[[t]]$. 

On the other hand, the top left term in this diagram
is the classifying stack of centralizers lying in 
$\o{I}_n \cap \o{I}_m$. 

So we need to show that for any
$\xi \in f+t^{-n}\Ad_{-n\check{\rho}(t)}\fb^e[[t]]$ and any
element $\gamma \in \o{I}_m$ centralizing it, this element lies in
$\o{I}_n$. 

For such $\gamma$, we have a triangular decomposition
$\gamma = \gamma^+ \cdot \gamma^-$ with
$\gamma^+ \in N(K)$ and $\gamma^- \in B^-(K)$
(because $m>0$). 
Because $m$ is larger than $n$, $\gamma^-$ lies in $\o{I}_n$.
Therefore, we need to show that $\gamma^+ \in N(O)$.

For this, observe that:

\[
\xi = \Ad_{\gamma}(\xi) = \Ad_{\gamma^+} \big(\Ad_{\gamma^-}(\xi)\big).
\] 

\noindent Because $\gamma^- \in \o{I}_n$ (as noted above),
$\Ad_{\gamma^-}(\xi) \in f+\Lie\o{I}_n^{\perp}$. But $\xi$ satisfied
the same property, Corollary \ref{c:intersection-2} implies that
$\gamma^+ \in \o{I}_n$ as desired.

\end{proof}

\begin{rem}\label{r:!-avg-cl}

We consider this as a semi-classical version of 
Theorem \ref{t:!-avg}; let us explain why. The reader 
may safely skip this. It freely uses some ideas from 
\S \ref{s:ds} and Appendix \ref{a:hc}.

First, note that Theorem \ref{t:!-avg} \eqref{i:!-avg-exists}
can be reformulated as saying that $\iota_{n,m,*}$
preserves compacts.

The category $\widehat{\fg}_{\kappa}\mod^{\o{I}_n,\psi}$
has the Kazhdan-Kostant filtration with \emph{semi-classical}
category $\QCoh^{ren}(f+\Lie\o{I}_n^{\perp}/\o{I}_n)$
(see \emph{loc. cit}. for the notation). The $*$-averaging
functor $\iota_{n,m,*}:\widehat{\fg}_{\kappa}\mod^{\o{I}_n,\psi} \to
\widehat{\fg}_{\kappa}\mod^{\o{I}_m,\psi}$ has 
associated semi-classical functor given
\textemdash{} up to a mild correction \textemdash{} by
pull-push along the correspondence:

\[
\xymatrix{
& f+\Lie\o{I}_n^{\perp} \cap \Lie\o{I}_m^{\perp}/\o{I}_n \cap \o{I}_m \ar[dl] \ar[dr] & \\
f+\Lie\o{I}_n^{\perp}/\o{I}_n &&
f+\Lie\o{I}_m^{\perp}/\o{I}_m.
}
\] 

\noindent The pullback along the left arrow obviously preserves
compacts, whereas pushforward along the right arrow does because
it is a (finitely-presented) regular embedding.

A word on the \emph{mild correction}: running the calculation
properly, one finds that the pullback should actually 
be a $*$-pullback followed by a $!$-pullback for the morphisms:

\[
f+\Lie\o{I}_n^{\perp} \cap \Lie\o{I}_m^{\perp}/\o{I}_n \cap \o{I}_m \to
f+\Lie\o{I}_n^{\perp}/\o{I}_n \cap \o{I}_m \to 
f+\Lie\o{I}_n^{\perp}/\o{I}_n. 
\]

\noindent But this does not affect the discussion above.
(The reason this appears is that $\iota_{n,m,*}$
is the composition 
$\widehat{\fg}_{\kappa}\mod^{\o{I}_n,\psi} \xar{\Oblv}
\widehat{\fg}_{\kappa}\mod^{\o{I}_n \cap \o{I}_m,\psi} \xar{\Av_*^{\psi}}
\widehat{\fg}_{\kappa}\mod^{\o{I}_n,\psi}$;
the semi-classical version of a forgetful functor involves a 
$*$-pullback along a correspondence, while the semi-classical
version of averaging involves a $!$-pullback along a regular embedding.)

\end{rem}

\section{Drinfeld-Sokolov realization of the generalized vacuum representations}\label{s:ds}

\subsection{A question}\label{ss:ds-intro}

We begin this section with a basic question about $\sW$-algebras:
who are the \emph{generalized vacuum} modules?
By this, at first pass, we mean that we expect a
projective system of modules 
$\sW_{\kappa}^n \in \sW_{\kappa}\mod^{\heart}$ ($n \geq 0$)
playing a similar role to the modules 
$\ind_{t^n\fg[[t]]}^{\widehat{\fg}_{\kappa}}(k)$ for the Kac-Moody algebra.

In more detail, we want that:

\begin{itemize}

\item $\sW_{\kappa}^0$ is the vacuum representation, 
i.e., $\sW_{\kappa}^0 = \sW_{\kappa} \in \sW_{\kappa}\mod^{\heart}$. 

\item The inverse limit of the $\sW_{\kappa}^n$ is the
topological chiral\footnote{This funny name is taken
from \cite{beilinson-top-alg}, who asks that it be used.
It means an associative algebra with respect to the
$\overset{\rightarrow}{\otimes}$-monoidal product
on $\Pro(\Vect^{\heart})$ from \emph{loc. cit}. Note that
in \cite{chiral}, such a thing is called a 
\emph{topological associative algebra}.}
 algebra $\sW_{\kappa}^{as}$ 
associated with the vertex algebra $\sW_{\kappa}$
(c.f. \cite{chiral} \S 3.6.2).

\item For $\fg = \sl_2$, so that $\sW_{\kappa}$ is a Virasoro
algebra, the modules $\sW_{\kappa}^n$ should be induced
from the subalgebra $t^{2n}\on{Der}(\cD)$.\footnote{The $2$ here
is needed for compatibility with the next expectation. At
the critical level, this might
be compared to the fact that the Sugawara element
corresponding to the derivation $t^n\partial_t$ ($n\geq 0$)
acts by zero on the module 
$\ind_{t^{\lfloor{\frac{n}{2}}\rfloor}\fg[[t]]}^{\widehat{\fg}_{crit}}(k)$,
c.f. \cite{hitchin} Theorem 3.7.9.}\footnote{The other expectations
in this list will be verified in this section, but not this one. 
For this comparison,
see Example \ref{e:virasoro}.}

\item The module $\sW_{\kappa}^n$ should have a canonical
filtration\footnote{We maintain the conventions of Appendix \ref{a:hc},
so filtrations are assumed exhaustive. Moreover, since we are speaking
about objects of an abelian category, we are tacitly assuming here
that each $F_i \sW_{\kappa}^n \to \sW_{\kappa}^n$ is injective.}
 $F_{\dot}\sW_{\kappa}^n$ 
compatible with the filtration on $\sW_{\kappa}$ and with
$F_{-1}\sW_{\kappa}^n = 0$.
Recalling that the associated graded of $\sW_{\kappa}^{as}$
is\footnote{This is not quite canonical:
we are using our choice of $dt$ and some choice
of non-degenerate $\Ad$-invariant bilinear form on $\fg$ to make 
$\fg((t))$ self-dual. More canonical would be to take
$\mu^{-1}(\psi)/N(K)$ for $\mu:\fg((t))^{\vee} \to \fn((t))^{\vee}$ the
canonical map.} the algebra of functions on the indscheme 
$f+\fb((t))/N(K)$, the associated graded of $\sW_{\kappa}^n$
should be identified with the structure sheaf
of the closed subscheme:\footnote{Recall that
$f+\fb((t))/N(K) = f+\fb^e((t))$ (for $e$ fitting into
the $\sl_2$-triple $(e,2\Lie(\check{\rho})(1),f)$)
with $f+\fb[[t]]/N(O) = f+\fb^e[[t]]$. So for $n = 0$, we really
do get a closed subscheme, and for general $n$ this follows
because $t^{-n} \Ad_{-n\check{\rho}(t)}$ is an automorphism
of our indscheme.}

\[
t^{-n} \Ad_{-n\check{\rho}(t)}\big(f+\fb[[t]]/N(O)\big) \subset 
f+\fb((t))/N(K).
\]

\noindent The morphism $\sW_{\kappa}^{n+1} \to \sW_{\kappa}^n$
should be strictly compatible with filtrations and should
induce the restriction of functions map when we identify
the associated graded as above.

\item The $\sW_{\kappa}^n$ should form a flat family of
modules as we vary $\kappa$, and moreover, this family
should extend to allow $\kappa \to \infty$ for $\kappa$
non-degenerate. In this
case, recall that $\sW_{\infty}^{as}$ is the algebra of functions
on the indscheme $\Op_G(\o{\cD})$ of opers, i.e., $(f+\fb((t)))dt/N(K)$ where
$N(K)$ acts by the \emph{gauge} action (not the adjoint action).
Then $\sW_{\infty}^n$ should be the structure sheaf
of the subscheme $\Op_G^{\leq n}$ of \emph{opers with singularities of order 
$\leq n$} as defined in \cite{hitchin} \S 3.8
(see also \cite{fg2} \S 2). We remind that this subscheme
is:\footnote{One uses $\check{\rho}(t)$ to define the
structure maps as we vary $n$.}

\[
\Op_G^{\leq n} \coloneqq 
\big(f+\Ad_{-n\check{\rho}(t)}(\fb[[t]])\big)dt/\Ad_{-n\check{\rho}(t)} N(O)
\]

\noindent with $\Ad_{-n\check{\rho}(t)} N(O)$ acting through the gauge 
action.\footnote{So from an \emph{opers-centric} worldview,
the modules $\sW_{\kappa}^n$ are quantizations
of opers with singularity $\leq n$. The existence of such
quantizations implies that the subschemes 
$\Op_G^{\leq n} \subset \Op_G(\o{\cD})$ are
coisotropic with respect to the canonical Poisson
structure on $\Op_G(\o{\cD})$. 
This is straightforward to see from the given description of
$\Op^{\leq n}$,
c.f. \cite{fg2} Lemma 4.4.1
and various points in the discussion of \cite{hitchin} \S 3.6-3.8.}

\end{itemize}

The purpose of this section and the next 
is to construct such modules, which
seem not to exist elsewhere in the literature.

\subsection{}

In fact, we will give two constructions of these modules.

We will present the two constructions somewhat out of order: first, we give
a construction via Drinfeld-Sokolov reduction, and then in 
\S \ref{s:free-field}
give a (perhaps) more elementary construction via the
\emph{free-field realization} of the $\sW$-algebra. 
The reason is that the former is the one relevant for proving
the affine Skryabin theorem. The latter is included in this paper 
for the sake of completeness, and because it 
plays a technical role in deducing 
the categorical Feigin-Frenkel theorem from the
affine Skryabin theorem. 

\subsection{Adolescent Whittaker construction}

The main idea of this section is to take:

\[
\sW_{\kappa}^n = \Psi(\ind_{\Lie \o{I}_n}^{\widehat{\fg}_{\kappa}} \psi_{\o{I}_n})
\]

\noindent up to a cohomological shift, and show that with this
construction satisfies our expectations.

\begin{rem}

To make some of our treatment more elementary, this result
is split across Theorem \ref{t:ds-no-w-str}
and Theorem \ref{t:ds-wstr}; the former does not explicitly mention
$\sW$-algebras, while the latter does.

\end{rem}

\begin{rem}

This construction is motivated by the following considerations.

Note that: 

\[
\ind_{\Lie \o{I}_n}^{\widehat{\fg}_{\kappa}}(\psi) \in 
\widehat{\fg}_{\kappa}\mod^{\o{I}_n,\psi} =
\Whit^{\leq n}(\widehat{\fg}_{\kappa}\mod)
\]

\noindent is compact, and
generates this category whenever $n>0$. Therefore, as
we vary $n$, the objects 
$\iota_{n,!}(\ind_{\Lie \o{I}_n}^{\widehat{\fg}_{\kappa}}(\psi)) \in 
\Whit(\widehat{\fg}_{\kappa}\mod)$ give compact generators.

Anticipating the affine Skryabin theorem
$\Whit(\widehat{\fg}_{\kappa}\mod) \simeq \sW_{\kappa}\mod$
and its compatibility with the functor $\Psi$,
we are computing the images of our compact generators
as modules over the $\sW$-algebra.

\end{rem}

\begin{rem}

For $n = 0$, this theorem says that $\Psi(\bV_{\kappa}) = \sW_{\kappa}$,
which is the definition of the right hand side. The argument
given below recovers
the fundamental results on $\sW_{\kappa}$: 
that it is in cohomological degree
$0$, and that it is a filtered vertex algebra\footnote{
One should use factorization techniques for this, 
c.f. \cite{chiral} \S 3.8.}
with commutative
associated graded the algebra of functions on $f+\fb[[t]]/N(O)$.

The argument is somewhat more conceptual than appears elsewhere
(c.f. \cite{dbt} and \cite{fbz} Chapter 15). It substantially
overlaps with the presentation given in the recent 
survey \cite{arakawa-survey},
but the issue of the convergence of the spectral sequence
is dealt with by a different argument. The key advantage of the
method below is that it avoids the \emph{tensor product decomposition}
appearing in other treatments (even in \cite{fg-weyl}), which is 
always justified by explicit formulae I have not been able to
understand conceptually.\footnote{Though it may well be that a conceptual
explanation of this method exists and I just could not find it. Or perhaps
the argument we give here for the convergence, 
which has some remarkable similarities with the tensor product argument,
is that conceptual explanation.}

\end{rem}

\subsection{Drinfeld-Sokolov reduction}

Recall that there is a continuous functor:\footnote{We remind at this point that e.g.
$\Vect$ denotes the \emph{DG category} of chain complexes
of $k$-vector spaces, and that \emph{canonical} means
defined up to canonical quasi-isomorphism 
(in the $\infty$-categorical sense).
}

\[
\Psi:\widehat{\fg}_{\kappa}\mod \to \Vect
\]

\noindent defined as the composition:

\[
\widehat{\fg}_{\kappa}\mod \overset{\Oblv}{\to}
\fn((t))\mod \xar{C^{\sinf}(\fn((t)),\fn[[t]],(-) \otimes -\psi)}
\Vect
\]

\noindent where $-\psi$ is the 1-dimensional 
$\fn((t))$-module corresponding to the character
$-\psi$, and
$C^{\sinf}$ is the semi-infinite cohomology
functor as defined in \S \ref{ss:sinf}. 
We remind that $\fn[[t]]$ appears in the notation, but
plays a very mild role. 

\subsection{}

The first main result of this section is the following.

\begin{thm}\label{t:ds-no-w-str}

For every $n \geq 0$, the complex
$\Psi(\ind_{\Lie \o{I}_n}^{\widehat{\fg}_{\kappa}} \psi_{\o{I}_n})$
is concentrated in cohomological degree
$-n\Delta$. (Here $\Delta$ is used as in
\S \ref{ss:!-avg}, i.e., $\Delta = 2(\rho,\check{\rho})$.)

\end{thm}

The argument will be given in \S \ref{ss:ds-ss} below.

\begin{rem}

The main issue in what follows is the convergence
of a certain spectral sequence. The approach given
below seems to be the most versatile one for
$n = 0$. However, for $n>0$, the method of
\S \ref{ss:amp-rep-thry-method} can also be adapted
to this purpose and is much more flexible.

\end{rem}

\subsection{}

We will prove the above by an argument about passage to the associated
graded, using the methods Appendix \ref{a:hc}.

First, note that $G(K)$ acts on $\widehat{\fg}_{\kappa}$, so we obtain
an action of $\bG_m$ on the Kac-Moody algebra
via:

\[
-\check{\rho}:\bG_m \to G \subset G(O) \subset G(K).
\]

\noindent This action preserves $\o{I}_n$ for all $n \geq 0$.
Moreover, the character $\psi:\Lie\o{I}_n \to k$ is $\bG_m$-equivariant
if we give the target $k$ the degree $-1$ grading.

Therefore, Appendix \ref{a:hc} produces
\emph{PBW} and \emph{Kazhdan-Kostant} (KK) filtrations on
the categories
$\widehat{\fg}_{\kappa}\mod^{\o{I}_n,\psi}$
and $\fn((t))\mod^{\Ad_{-n\check{\rho}(t)} N(O),\psi}$. 
(We remind that $\Ad_{-n\check{\rho}(t)} N(O) = N(K) \cap \o{I}_n$.)

\subsection{}\label{ss:fil-summary}

We recall the basic facts about these filtrations that we will need.
We will need some of the language and notation from
Appendix \ref{a:hc}: see especially the material about
filtrations in \S \ref{ss:fil-intro} and 
renormalization in \S \ref{ss:ren-start}. 

Throughout, we fix an $\Ad$-invariant identification $\fg \simeq \fg^{\vee}$
and a 1-form $dt$ to obtain $\fg((t))^{\vee} \simeq \fg((t))$;
note that this induces:

\[
\fn((t))^{\vee} \simeq \fg((t))/\fb((t)). 
\]

As in \S \ref{ss:nostalgia}, for $V \subset \fg((t))$, we let
$V^{\perp}$ denote $(\fg((t))/V)^{\vee} \subset \fg((t))^{\vee} = \fg((t))$;
we recall from \emph{loc. cit}. that:

\[
\Lie\o{I}_n^{\perp} = 
\Ad_{-n\check{\rho}(t)}(\fg[[t]]+t^{-n}\fb[[t]]).
\]

\begin{itemize}

\item 

We have \emph{semi-classical} categories:

\[
\begin{gathered}
\widehat{\fg}_{\kappa}\mod^{\o{I}_n,\psi,PBW\mathendash{cl}} = 
\QCoh^{ren}(\Lie\o{I}_n^{\perp}/\o{I}_n) \\
\widehat{\fg}_{\kappa}\mod^{\o{I}_n,\psi,KK\mathendash{cl}} = 
\QCoh^{ren}(f+\Lie\o{I}_n^{\perp}/\o{I}_n)
\end{gathered}
\]

\noindent for $f \in \fn^-((t))$ our principal nilpotent corresponding to $\psi$.

For the $\fn((t))$-categories, observe that:

\[
\Lie\o{I}_n^{\perp}/\big(\Lie\o{I}_n^{\perp} \cap \fb((t))\big) =
\Ad_{-n\check{\rho}(t)}(\fg[[t]]+t^{-n}\fb[[t]]).
\]

and:\footnote{At various points in our discussion, the abundance of symbols
$\Ad_{-n\check{\rho}(t)}$ means something is canonically isomorphic
to the same expression with all such symbols removed. We 
still retain the notation since it is an important bookkeeping device,
especially as we vary $n$.}

\[
\begin{gathered}
\fn((t))\mod^{N(K) \cap \o{I}_n,\psi,PBW\mathendash{cl}} = 
\QCoh^{ren}\Big(
\Ad_{-n\check{\rho}(t)} (\fg[[t]]/\fb[[t]])/
\Ad_{-n\check{\rho}(t)} N(O)
\Big)
\\ 
\fn((t))\mod^{N(K) \cap \o{I}_n,\psi,KK\mathendash{cl}} = 
\QCoh^{ren}\Big(
f+\Ad_{-n\check{\rho}(t)} (\fg[[t]]/\fb[[t]])/
\Ad_{-n\check{\rho}(t)} N(O)
\Big)
\end{gathered}
\]

\noindent where these arise from the identification:

\[
\Lie\o{I}_n^{\perp}/\big(\Lie\o{I}_n^{\perp} \cap \fb((t))\big) = 
\Ad_{-n\check{\rho}(t)}\big((\fg[[t]]+t^{-n}\fb[[t]])/t^{-n}\fb[[t]]\big) =
\Ad_{-n\check{\rho}(t)} (\fg[[t]]/\fb[[t]])
\]

\item 

The forgetful functor:

\[
\widehat{\fg}_{\kappa}\mod^{\o{I}_n,\psi} \to
\fn((t))\mod^{\Ad_{-n\check{\rho}(t)} N(O),\psi}
\]

\noindent is filtered for either filtration. Its underlying
PBW semi-classical functor:

\[
\QCoh^{ren}(\Lie\o{I}_n^{\perp}/\o{I}_n)
\to 
\QCoh^{ren}\Big(
\Ad_{-n\check{\rho}(t)} (\fg[[t]]/\fb[[t]])/
\Ad_{-n\check{\rho}(t)} N(O)
\Big)
\]

\noindent is given by pullback/pushforward along
the obvious structure maps from:

\[
\Lie\o{I}_n^{\perp}/\Ad_{-n\check{\rho}(t)} N(O).
\]

\item 

The functor:

\begin{equation}\label{eq:psi-n-sinf}
C^{\sinf}(\fn((t)),\Ad_{-n\check{\rho}(t)} N(O),
(-) \otimes -\psi): 
\fn((t))\mod^{\Ad_{-n\check{\rho}(t)} N(O),\psi} \to \Vect
\end{equation}

\noindent is filtered for each of these filtrations. 
Here we remind the reader of Notation \ref{notation:h_0=k},
which says:

\[
\begin{gathered}
C^{\sinf}(\fn((t)),\Ad_{-n\check{\rho}(t)} N(O),-) = 
C^{\sinf}(\fn((t)),\Ad_{-n\check{\rho}(t)} \fn[[t]],
\Ad_{-n\check{\rho}(t)} N(O),-) = \\
C^{\sinf}(\fn((t)),\Ad_{-n\check{\rho}(t)} \fn[[t]],-)
\end{gathered}
\]

\noindent where the first equality is a definition
and the second equality is only an equality of functors,
not of filtered functors (the fact that it is an equality of functors
is a consequence of the prounipotence of our group scheme
$\Ad_{-n\check{\rho}(t)} N(O)$).

Its underlying 
PBW (resp. KK) semi-classical functors
are given by $*$-restriction to:

\[
0/\Ad_{-n\check{\rho}(t)} N(O) \subset
\Ad_{-n\check{\rho}(t)} (\fg[[t]]/\fb[[t]])/
\Ad_{-n\check{\rho}(t)} N(O)
\]

(resp. $f/\Ad_{-n\check{\rho}(t)} N(O)
 \subset f+\Ad_{-n\check{\rho}(t)} (\fg[[t]]/\fb[[t]])/
\Ad_{-n\check{\rho}(t)} N(O)$) followed by
global sections, i.e., group cohomology with respect
to $\Ad_{-n\check{\rho}(t)} N(O)$.

\item These two filtrations fit into a \emph{bifiltration} (c.f. \S \ref{ss:kk-bifilt}).
 
In the language of \S \ref{ss:kk-bifilt},
the KK filtration on the PBW-semi-classical categories 
is induced via Example \ref{e:fil-gm} from the 
action:\footnote{Note that in Appendix \ref{a:hc},
we use the sometimes confusing
convention that the action of $\bG_m$ on a scheme
is \emph{expanding} if functions are non-negatively
graded. The reason is that if any group acts on a
scheme, there is an inverse sign for the induced
action on functions. But e.g., the grading
on $\Sym(\fg) = \gr_{\dot}U(\fg)$ corresponds to 
the action of $\bG_m$ by inverse homotheties on
$\fg^{\vee}$.} 

\[
\lambda \cdot x = 
\lambda^{-1}\cdot\Ad_{\check{\rho}(\lambda^{-1})}(x) 
\]

\noindent of $\bG_m$ on 
$\Lie\o{I}_n^{\perp}/\o{I}_n$ and 
$\Ad_{-n\check{\rho}(t)} (\fg[[t]]/\fb[[t]])/
\Ad_{-n\check{\rho}(t)} N(O)$ respectively (c.f. Remark \ref{r:psi-filts-on-cl}).

\end{itemize}

\subsection{}\label{ss:ds-ss}

With these preliminaries aside, we are prepared to prove
the above result.

First, we introduce the following notation, which significantly
reduces the burden in what follows.

\begin{notation}\label{n:det-lines}

For all $n \geq 0$, let $\ell^n$ denote the line
$\det(\Ad_{-n\check{\rho}(t)}\fn[[t]]/\fn[[t]])$. For $n \leq m$,
we let:

\[
\ell^{n,m} \coloneqq 
\det(\Ad_{-m\check{\rho}(t)}\fn[[t]]/\Ad_{-n\check{\rho}(t)}\fn[[t]]) =
\ell^m \otimes \ell^{n,\vee}.
\]

\end{notation}

\begin{proof}[Proof of Theorem \ref{t:ds-no-w-str}]

Here is the strategy: we will construct a Kazhdan-Kostant
filtration on $\ind_{\Lie\o{I}_n}^{\widehat{\fg}_{\kappa}}(\psi)$
to obtain one on $\Psi(\ind_{\Lie\o{I}_n}^{\widehat{\fg}_{\kappa}}(\psi))$,
and it will be easy to show that the associated graded for the latter
is in one cohomological degree.

We would be done, but that the Kazhdan-Kostant filtration on
$\ind_{\Lie\o{I}_n}^{\widehat{\fg}_{\kappa}}(\psi)$ 
\emph{is not bounded below} (and not even complete). 
We will nevertheless show
that the induced filtration \emph{on its Drinfeld-Sokolov} reduction is 
bounded below. 
This will be achieved by a comparison with its PBW filtration,
which is bounded below (although its associated graded is
not in a single cohomological degree).

\step\label{st:ds-filts}

Observe that: 

\[
\ind_{\Lie\o{I}_n}^{\widehat{\fg}_{\kappa}}(\psi)\in 
\widehat{\fg}_{\kappa}\mod^{\o{I}_n,\psi}
\]

\noindent has a canonical KK filtration, denoted
$F_{\dot}^{KK} \ind_{\Lie\o{I}_n}^{\widehat{\fg}_{\kappa}}(\psi)$.
Indeed, this follows by functoriality from Example \ref{e:pbw-kk-k=h}.
This filtration fits into a bifiltration with the PBW filtration
$F_{\dot}^{PBW} \ind_{\Lie\o{I}_n}^{\widehat{\fg}_{\kappa}}(\psi)$.

The induced KK filtration on the vector space:

\[
\gr_{\dot}^{PBW} \ind_{\Lie\o{I}_n}^{\widehat{\fg}_{\kappa}}(\psi) =
\Sym^{\dot} \fg((t))/\Lie\o{I}_n
\]

\noindent is easy to understand concretely
\textendash{} the KK filtration is:

\[
F_i^{KK} \gr_{\dot}^{PBW} \ind_{\Lie\o{I}_n}^{\widehat{\fg}_{\kappa}}(\psi) =
\underset{i^{\prime} \leq i}{\oplus} \,
\underset{j\in \bZ}{\oplus} \, \big(\Sym^{i^{\prime}-j} \fg((t))/\Lie\o{I}_n\big)^j
\]

\noindent where the superscript $j$ indicates the
$-\check{\rho}$-grading (note the sign!). 
(For example, for a positive
root $\alpha$, $t^m e_{\alpha}$ has grading 
$-(\check{\rho},\alpha)$, so the corresponding
element lies in $F_{1-(\check{\rho},\alpha)}^{KK}$
but, if non-zero, this element will not lie in 
$F_{-(\check{\rho},\alpha)}^{KK}$.)

\step 

Note that we can compute 
$\Psi(\ind_{\Lie\o{I}_n}^{\widehat{\fg}_{\kappa}}(\psi))$
as the \emph{Harish-Chandra} version of semi-infinite cohomology
(c.f. \S \ref{ss:sinf-hc}):

\[
\Psi(\ind_{\Lie\o{I}_n}^{\widehat{\fg}_{\kappa}}(\psi)) = 
C^{\sinf}(\fn((t)),\fn[[t]],\Ad_{-n\check{\rho}(t)} N(O),
\ind_{\Lie\o{I}_n}^{\widehat{\fg}_{\kappa}}(\psi) \otimes -\psi)
\]

\noindent by prounipotence of $\Ad_{-n\check{\rho}(t)} N(O)$.

Since the functors:

\[
\widehat{\fg}_{\kappa}\mod^{\o{I}_n,\psi} \xar{\Oblv}
\fn((t))\mod^{\Ad_{-n\check{\rho}(t)} N(O),\psi} 
\xar{C^{\sinf}(\fn((t)),\Ad_{-n\check{\rho}(t)} N(O),
- \otimes -\psi)}
\Vect
\]

\noindent are bifiltered,
we obtain a bifiltration on 
$\Psi(\ind_{\Lie\o{I}_n}^{\widehat{\fg}_{\kappa}}(\psi)) \in \Vect$.
In particular, we obtain the filtrations
$F_{\dot}^{PBW}\Psi(\ind_{\Lie\o{I}_n}^{\widehat{\fg}_{\kappa}}(\psi))$
and $F_{\dot}^{KK}\Psi(\ind_{\Lie\o{I}_n}^{\widehat{\fg}_{\kappa}}(\psi))$.
We will now compute the associated graded complexes.

First, it is convenient to slightly modify the functor
$\Psi$. We define: 

\begin{equation}\label{eq:psi-n-defin}
\Psi_n(-) = 
C^{\sinf}(\fn((t)),\Ad_{-n\check{\rho}(t)}\fn[[t]],-) =
\Psi(-) \otimes 
\ell^{n,\vee} [-n\Delta]
\end{equation}

\noindent (c.f. \eqref{eq:psi-n-sinf}). 
Obviously it differs from $\Psi$ only by a cohomological shift and 
tensoring with a 1-dimensional vector space. (At the level
of filtered functors, there is also a shift in the indexing by
$n\Delta$.)

Then we claim that:

\begin{equation}\label{eq:pbw-gr}
\gr_{\dot}^{PBW} \Psi_n(\ind_{\Lie\o{I}_n}^{\widehat{\fg}_{\kappa}}(\psi)) =
C^{\dot}\big(\Ad_{-n\check{\rho}(t)} N(O),
\Fun(\Ad_{-n\check{\rho}(t)} t^{-n}\fb[[t]])\big)
\end{equation}

\noindent i.e., (derived) global sections of the structure
sheaf of the stack  
$\Ad_{-n\check{\rho}(t)} t^{-n}\fb[[t]]/\Ad_{-n\check{\rho}(t)} N(O)$.

This is a straightforward verification using \S \ref{ss:fil-summary}.
Writing $\Psi_n|_{\widehat{\fg}_{\kappa}\mod^{\o{I}_n,\psi}}$ 
as the appropriate composition:

\[
\widehat{\fg}_{\kappa}\mod^{\o{I}_n,\psi} \xar{\Oblv}
\fn((t))\mod^{\Ad_{-n\check{\rho}(t)} N(O),\psi} 
\to
\Vect
\]

\noindent we find that its semi-classical functor:

\[
\QCoh^{ren}(\Lie\o{I}_n^{\perp}/\o{I}_n) \to \Vect 
\]

\noindent is given by
successive pullbacks and (renormalized) pushforwards along the diagram:

\[
\xymatrix @C=-2pc {
& \Lie\o{I}_n^{\perp}/\Ad_{-n\check{\rho}(t)} N(O) \ar[dl] \ar[dr] & &
0/\Ad_{-n\check{\rho}(t)} N(O) \ar[dl] \ar[dr] \\
\Lie\o{I}_n^{\perp}/\o{I}_n & &
\Ad_{-n\check{\rho}(t)} (\fg[[t]]/\fb[[t]])/\Ad_{-n\check{\rho}(t)} N(O) & &
\Spec(k). 
}
\]

\noindent We recall that 
$\Lie\o{I}_n^{\perp} = \Ad_{-n\check{\rho}(t)}(\fg[[t]] + t^{-n}\fb[[t]])$.
Therefore, since
$\gr_{\dot}^{PBW} \ind_{\Lie\o{I}_n}^{\widehat{\fg}_{\kappa}}(\psi)$
is the structure sheaf of $\Lie\o{I}_n^{\perp}/\Ad_{-n\check{\rho}(t)}$,
we obtain the claim.

The same analysis applies in the KK setting, and we
obtain: 

\[
\gr_{\dot}^{KK} 
\Psi_n(\ind_{\Lie\o{I}_n}^{\widehat{\fg}_{\kappa}}(\psi)) =
C^{\dot}\big(\Ad_{-n\check{\rho}(t)} N(O),
\Fun(f+\Ad_{-n\check{\rho}(t)} t^{-n}\fb[[t]])\big) 
\]

\noindent The major difference is that  
$\Ad_{-n\check{\rho}(t)} N(O)$ acts \emph{freely} on this
locus and the quotient is an affine scheme. Indeed, by
Kostant theory we have:

\[
\begin{gathered}
f+\Ad_{-n\check{\rho}(t)} t^{-n}\fb[[t]]/\Ad_{-n\check{\rho}(t)} N(O) =
t^{-n}\Ad_{-n\check{\rho}(t)} (f+\fb[[t]]) /\Ad_{-n\check{\rho}(t)} N(O) = \\
t^{-n}\Ad_{-n\check{\rho}(t)} (f+\fb^e[[t]])
\end{gathered}
\]

\noindent for $e$ fitting into our principal $\sl_2$ as usual.

In particular, we find that $\gr_{\dot}^{KK} 
\Psi_n(\ind_{\Lie\o{I}_n}^{\widehat{\fg}_{\kappa}}(\psi)) \in \Vect^{\heart}$.
Since $\Psi_n$ differs from $\Psi$ by the cohomological shift
by $n\Delta$, we find that
$\gr_{\dot}^{KK} \Psi(\ind_{\Lie\o{I}_n}^{\widehat{\fg}_{\kappa}}(\psi))$
is concentrated in cohomological degree $-n\Delta$, as expected.

\step As indicated in the preamble, it remains to show that
$F_i^{KK} \Psi_n(\ind_{\Lie\o{I}_n}^{\widehat{\fg}_{\kappa}}(\psi)) = 0$
for all $i<0$.\footnote{Note that due to the filtering conventions,
this means that $F_i^{KK}$ of $\Psi$ of this module vanishes
for $i<n\Delta$.} 

\step 

First, we observe that
$F_i^{PBW} \Psi_n(\ind_{\Lie\o{I}_n}^{\widehat{\fg}_{\kappa}}(\psi)) = 0$ 
for $i<0$. Indeed, as in Remark \ref{r:sinf-pbw-fmla}, 
we have:

\[
\begin{gathered}
F_i^{PBW} \Psi_n(\ind_{\Lie\o{I}_n}^{\widehat{\fg}_{\kappa}}(\psi)) = \\
\underset{m \geq n}{\colim} \, 
F_{i-(m-n)\Delta}^{PBW} \,
C^{\dot}(\Ad_{-m\check{\rho}(t)} \fn[[t]],\Ad_{-n\check{\rho}(t)}N(O),
\ind_{\Lie\o{I}_n}^{\widehat{\fg}_{\kappa}}(\psi)) \otimes \ell^{n,m}
[(m-n)\Delta].
\end{gathered}
\]

Then since the PBW filtration on 
$\ind_{\Lie\o{I}_n}^{\widehat{\fg}_{\kappa}}(\psi)$ vanishes
in negative degrees, the PBW filtration
on: 

\[
C^{\dot}(\Ad_{-m\check{\rho}(t)} \fn[[t]],\Ad_{-n\check{\rho}(t)}N(O),
\ind_{\Lie\o{I}_n}^{\widehat{\fg}_{\kappa}}(\psi))
\]

\noindent vanishes in degrees $< -(m-n)\Delta$
(c.f. Remark \ref{r:hc-cplx}, i.e., this vanishing follows from using the
\emph{standard} filtration on Harish-Chandra cohomology).
This obviously gives the claim because of the shift of
filtration that occurs in the colimit.

\step Now recall that 
$\Psi_n(\ind_{\Lie\o{I}_n}^{\widehat{\fg}_{\kappa}}(\psi))$ 
is \emph{bi}filtered. In particular:

\[
\gr_{\dot}^{PBW} \Psi_n(\ind_{\Lie\o{I}_n}^{\widehat{\fg}_{\kappa}}(\psi))
\overset{\eqref{eq:pbw-gr}}{=}
\Gamma(\Ad_{-n\check{\rho}(t)} t^{-n}\fb[[t]]/\Ad_{-n\check{\rho}(t)} N(O),
\sO_{\Ad_{-n\check{\rho}(t)} t^{-n}\fb[[t]]/\Ad_{-n\check{\rho}(t)} N(O)})
\]

\noindent inherits a KK filtration. We claim that:

\[
F_i^{KK} \gr_{\dot}^{PBW} 
\Psi_n(\ind_{\Lie\o{I}_n}^{\widehat{\fg}_{\kappa}}(\psi)) 
\]

\noindent vanishes in negative degrees.

We will verify this by giving an explicit description of the KK filtration
in this case.

Note that $\bG_m \times \bG_m$ acts on:

\[
\Ad_{-n\check{\rho}(t)} t^{-n}\fb[[t]]/\Ad_{-n\check{\rho}(t)} N(O)
\]

\noindent where the action of
one factor is induced by the inverse homothety action
on the vector space $\Ad_{-n\check{\rho}(t)} t^{-n}\fb[[t]]$,
and the action of the other factor is given by
the adjoint action induced by the cocharacter $\check{\rho}$ (of
the adjoint group). In what follows, we consider this stack
as equipped with the induced \emph{diagonal} action of
$\bG_m$; in particular, the global sections of its structure sheaf
inherit a grading.

Then by Step \ref{st:ds-filts}, the KK filtration is induced
from the grading as:

\begin{equation}\label{eq:kk-pbw-gr}
\begin{gathered}
F_i^{KK} \gr_{\dot}^{PBW} 
\Psi_n(\ind_{\Lie\o{I}_n}^{\widehat{\fg}_{\kappa}}(\psi)) = \\
\oplus_{i^{\prime} \leq i}
\Gamma(\Ad_{-n\check{\rho}(t)} t^{-n}\fb[[t]]/\Ad_{-n\check{\rho}(t)} N(O),
\sO_{\Ad_{-n\check{\rho}(t)} t^{-n}\fb[[t]]/\Ad_{-n\check{\rho}(t)} N(O)})^{i^{\prime}}
\end{gathered}
\end{equation}

\noindent where the outer superscript in the global sections indicates
the grading defined above.

Now we claim that our $\bG_m$-action is \emph{expanding}.
Obviously the inverse homothety action is expanding. Moreover,
the $-\check{\rho}$ adjoint action on both
$\fb$ (so on $\fb((t))$) and on $N(K)$ are, so the same
is true for our quotient stack above. Therefore, the diagonal
action is expanding as well. 

Therefore, our grading is in non-negative degrees only.
By \eqref{eq:kk-pbw-gr}, we obtain the desired vanishing.

\step 

We now make the (simpler) observation that the PBW
filtration on 
$\gr_{\dot}^{KK} 
\Psi_n(\ind_{\Lie\o{I}_n}^{\widehat{\fg}_{\kappa}}(\psi)) $
vanishes in negative degrees.

Indeed, this PBW filtration on functions on 
the scheme
$f+\Ad_{-n\check{\rho}(t)} t^{-n}\fb[[t]]/\Ad_{-n\check{\rho}(t)} N(O)$
comes from degenerating $f$. Precisely, we have
a prestack over $\bA_{\hbar}^1/\bG_m$ defined by
the family:

\[
\hbar f+\Ad_{-n\check{\rho}(t)} t^{-n}\fb[[t]]/\Ad_{-n\check{\rho}(t)} N(O)
\]

\noindent and the\footnote{Note that we should work with 
$\QCoh^{ren}$, so the pushforward is continuous. But this is a place
where one can only feel anxiety, but cannot make a mistake:
since our complexes are bounded from below, renormalized pushforward
coincides with any other notion.}
pushforward of the structure sheaf defines the PBW filtration on our
complex.

The monoid $\bA^1$ acts on the total space of this fibration
through the homothety action of this monoid on $\fg((t))$: this implies the
claim.

\step 

From here, the claim is formal:

Suppose $V \in \BiFil \Vect$ 
is any bifiltered vector space. We denote its underlying filtrations
as $F_{\dot}^{PBW}$ and $F_{\dot}^{KK}$.
We suppose that
$F_i^{PBW} V$, $F_i^{KK} \gr_{\dot}^{PBW} V$, and
$F_i^{PBW} \gr_{\dot}^{KK}$ vanish
for $i<0$. Then we claim that $F_i^{KK} V$ vanishes for $i<0$ as well.

We denote: 

\[
F_{i,j} V = F_i^{KK} F_j^{PBW} V = F_j^{PBW} F_i^{KK} V.
\]

\noindent Then observe that:

\[
F_i^{KK} V = \underset{j}{\colim} \, F_{i,j} V.
\]

\noindent So it suffices to show that if $i<0$, then $F_{i,j} V = 0$.

Since $\Coker(F_{i,j-1} V \to F_{i,j} V) = F_i^{KK} \gr_j^{PBW} V = 0$,
$F_{i,j} V$ is independent of $j$. Therefore, it suffices to show
the above claim when $i,j<0$.

But the same argument shows that $F_{i,j} V$ is independent of
$i$ when $j$ is negative. Therefore, for $i,j<0$, we have:

\[
F_{i,j} V \isom \underset{i^{\prime}}{\colim} \, F_{i^{\prime},j} V =
F_j^{PBW} V = 0
\]

\noindent as desired.

\end{proof}

\begin{rem}

The method used here bears a striking resemblance to the
method of \emph{tensor product decomposition} used
traditionally (e.g. in \cite{dbt} and \cite{fbz}) in the $n = 0$ case above,
i.e., to compute $\Psi(\bV_{\kappa})$.

There, one finds a quasi-isomorphic subcomplex\footnote{In 
\cite{fbz}, it is introduced in \S 15.2 and denoted by
$C_k^{\dot}(\fg)_0$.}
of the usual complex of Drinfeld-Sokolov semi-infinite chains
with ``size"
$\Sym(\fb((t))/\fb[[t]]) \otimes \Lambda^{\dot} \fn[[t]]^{\vee}$;
this complex is closed under the KK filtration and bounded from below
with respect to it, which solves the boundedness problem.

The method used above settled the convergence by a comparison
with the PBW filtration, whose associated graded 
also has this ``size," since it is Lie algebra cohomology for
$\fn[[t]]$ with coefficients in $\Fun(\fb[[t]]) = \Sym(\fb((t))/\fb[[t]])$.

\end{rem}

\subsection{$\sW_{\kappa}$-module structures}

Define: 

\[
\sW_{\kappa}^n \coloneqq 
H^0 \Psi_n(\ind_{\Lie \o{I}_n}^{\widehat{\fg}_{\kappa}} \psi_{\o{I}_n}) =
H^{-n\Delta} \Psi(\ind_{\Lie \o{I}_n}^{\widehat{\fg}_{\kappa}} \psi_{\o{I}_n}) 
\otimes \ell^{n,\vee}
\]

\noindent 
By Theorem \ref{t:ds-no-w-str}, this is the only
non-vanishing cohomology group. We recall that
$\Psi_n$ was defined in \eqref{eq:psi-n-defin}.
Note that tensoring with this line is extremely mild.

We now make some observations about the 
\emph{$\sW_{\kappa}$-module} structure on $\sW_{\kappa}^n$.
(For further observations, see \S \ref{s:free-field}.)

\subsection{Recollections on $\sW$-algebras}\label{ss:w-recollection}

Before proceeding, we 
summarize what facts we will need about $\sW_{\kappa}$.

\subsection{}

Recall that $\sW_{\kappa} \coloneqq \Psi(\bV_{\kappa}) = \sW_{\kappa}^0$
is a \emph{vertex algebra}, which for our purposes means the 
vacuum representation of a chiral (or factorization) algebra in 
the sense of \cite{chiral}.

Here are two conceptual explanations of this fact. 

One may use \cite{chiral} \S 3.8 to give a model for semi-infinite cochains
that factorizes.

Alternatively, one can work as in \cite{km-indcoh} and show that
the categories $\widehat{\fg}_{\kappa}\mod$ is (the fiber of) a unital chiral
category in the sense of \cite{chiralcats} 
with unit the factorization $\bV_{\kappa}$ (or rather,
the Kac-Moody factorization algebra). Moreover, one can show
that our definition of $\Psi:\widehat{\fg}_{\kappa}\mod \to \Vect$ factorizes.
This formally implies that the image of the unit is a vertex algebra in 
the above sense.

(There is also a traditional vertex algebra description: see \cite{fbz} \S 15.)

\subsection{}

In either of the above pictures, we find that $\Psi$ upgrades to a functor:

\[
\Psi: \widehat{\fg}_{\kappa}\mod \to \sW_{\kappa}\mod^{\on{fact}}
\]

\noindent where the right hand side denotes the DG category of
\emph{factorization modules} for (the factorization algebra with
underlying vacuum representation) $\sW_{\kappa}$; this notion
is defined in \cite{chiralcats}.

Some remarks about this notion are in order.

$\sW_{\kappa}\mod^{\on{fact}}$ has a canonical $t$-structure,
and by \cite{chiral} \S 3.6, the heart is the abelian category
of discrete $\sW_{\kappa}^{as}$-modules, which coincides
with vertex modules in the usual sense. We denote this
category by $\sW_{\kappa}\mod^{\heart}$.

It is possible to compute $\sW_{\kappa}\mod^{\on{fact}}$ explicitly:
it turns out to be the left completion of the derived category
of this abelian category. However, such arguments have not
been given in the literature previously (e.g., the corresponding
fact for Kac-Moody algebras is not published), and are somewhat
involved.

Moreover, \emph{we will not need to compute this DG category}
so explicitly: we will only need to know the heart of its $t$-structure.

\subsection{}\label{ss:w-acts-on-psi}

By the above, for $M \in \widehat{\fg}_{\kappa}\mod$,
$H^k \Psi(M) \in \sW_{\kappa}\mod^{\heart}$, i.e., this
object of $\Vect^{\heart}$ has a canonical action of 
the $\sW_{\kappa}$-algebra. 

Moreover, we find that the functors $(H^k \Psi, k \in \bZ)$ are
\emph{cohomological}, i.e., for a morphism 
$f:M \to N \in \widehat{\fg}_{\kappa}\mod$, the boundary morphism:

\[
H^k(\Coker(f)) \to H^{k+1} (M)
\]

\noindent is a morphism of $\sW_{\kappa}$-modules.

\subsection{}

We also need some compatibilities with filtrations.

First, the KK filtration on $\bV_{\kappa}$ makes sense
factorizably, so defines on $\bV_{\kappa}$ a structure
of \emph{filtered} vertex algebra (c.f. \cite{chiral} \S 3.3.11
or \cite{fbz} \S 15). 

Then one can show by mixing \cite{km-indcoh} and
Appendix \ref{a:hc} that
the KK filtration on $\widehat{\fg}_{\kappa}\mod$ defines
a filtration on this category \emph{as a unital 
factorization category}
and that $\Psi$ is a factorizable functor.

Since the induced KK filtration 
on $\sW_{\kappa} = \Psi(\bV_{\kappa}) = 
H^0 \Psi(\bV_{\kappa})$ is 
a filtration in the abelian category of vector spaces,
we deduce that  $\sW_{\kappa}$ is a filtered vertex algebra. Concretely, this means
that $\sW_{\kappa}^{as}$ is filtered as a 
topological chiral algebra. 
We remind that $\gr_{\dot}^{KK} \sW_{\kappa}$ is the
commutative vertex algebra of functions on 
$f+\fb[[t]]/N(O) = f+\fb^e[[t]]$, and that the associated graded
of $\sW_{\kappa}^{as}$ is the corresponding
commutative topological algebra of functions
on the affine Kostant slice $f+\fb((t))/N(K) = f+\fb^e((t))$.

\subsection{}

Suppose $M \in \Fil^{KK} \widehat{\fg}_{\kappa}\mod$ is a 
KK filtered Kac-Moody representation. 
Recall from Appendix \ref{a:hc} that $\Psi(M)$ inherits
a KK filtration.

From the above discussion, for any integer $k$,
$\gr_{\dot}^{KK} \Psi(M) \in \gr_{\dot}^{KK} \sW_{\kappa}\mod^{\on{fact}}$,
which as before is a DG category with a $t$-structure
whose heart is $\IndCoh(f+\fb((t))/N(K))^{\heart}$.\footnote{We recall
that the heart of the natural $t$-structure here is tautologically
the same as the abelian category of discrete
modules over the algebra of functions on this indscheme.}
Therefore, for any integer $k$, 
$H^k(\gr_{\dot}^{KK} \Psi(M)) \in \IndCoh(f+\fb((t))/N(K))^{\heart}$.

\subsection{Description of the modules $\sW_{\kappa}^n$}

Now observe that the modules $\sW_{\kappa}^n$ 
are naturally filtered. Indeed, as above, each module
$\ind_{\Lie\o{I}_n}^{\widehat{\fg}_{\kappa}}(\psi)$ carries
a canonical KK filtration, so it Drinfeld-Sokolov reduction does
as well. 
However, we will renumber the filtration by shifting the
indices by $n\Delta$: this amounts to considering the natural
filtration on $H^0\Psi_n(\ind_{\Lie\o{I}_n}^{\widehat{\fg}_{\kappa}}(\psi))$
instead of on $H^{n\Delta}\Psi$ of this module.
We denote this filtration $F_{\dot}^{KK} \sW_{\kappa}^n$.

We now have the following outcome of the 
proof of Theorem \ref{t:ds-no-w-str}:

\begin{thm}\label{t:ds-wstr}

The filtration on $\sW_{\kappa}^n$ satisfies:

\begin{itemize}

\item $F_i^{KK} \sW_{\kappa}^n = 0$ for $i<0$.

\item $H^k(\gr_{\dot}^{KK} \sW_{\kappa}^n)$ vanishes for 
$k \neq 0$.

\item 
$H^0(\gr_{\dot}^{KK} \sW_{\kappa}^n) \in \IndCoh(f+\fb((t))/N(K))^{\heart}$
is the structure sheaf of the closed subscheme
$t^{-n}\Ad_{-n\check{\rho}(t)}(f+\fb[[t]]/N(O))$ with its natural grading
as a $\bG_m$-invariant subscheme.

\end{itemize}

\end{thm}

\begin{rem}[Cyclicity of $\sW_{\kappa}^n$]\label{r:cyclic}

Note that the algebra of functions
on $f+\fb/N = f+\fb^e$ is a polynomial algebra, and with respect
to the KK grading, is generated by elements of degree $ \geq 1$
(and $\geq 2$ if $\fg$ is semisimple). The same holds for the
affine version, or for the structure sheaf of a $\bG_m$-invariant
subscheme of $f+\fb((t))/N(K)$.

Because $F_{-1}^{KK} \sW_{\kappa}^n = 0$, we deduce that
$F_0^{KK} \sW_{\kappa}^n = \Gr_0^{KK} \sW_{\kappa}^n = k$,
where this copy of $k$ corresponds to the constant functions
on $t^{-n}\Ad_{-n\check{\rho}(t)}(f+\fb[[t]]/N(O))$. In particular,
we obtain a canonical \emph{vacuum vector} $1 \in \sW_{\kappa}^n$.

Moreover, the map $\sW_{\kappa}^{as} \to \sW_{\kappa}^n$
given by acting on this vector is surjective; indeed, this follows
because it is true at the associated graded level. Therefore,
we obtain that $\sW_{\kappa}^n$ is a \emph{cyclic} module
for the $\sW_{\kappa}$-algebra.

\end{rem}

\subsection{Varying $n$}

We conclude this section with the following result.

\begin{thm}\label{t:ds-n/m}

\begin{enumerate}

\item For $m \geq n$, there is a unique map 
$\sW_{\kappa}^m \to \sW_{\kappa}^n$ preserving the
vacuum vectors introduced in Remark \ref{r:cyclic}. 

\item This map is surjective and strictly compatible with filtrations.
On associated graded, it gives the restriction map for functions along
the canonical embedding:

\[
t^{-n}\Ad_{-n\check{\rho}(t)}(f+\fb[[t]]/N(O)) \into
t^{-m}\Ad_{-m\check{\rho}(t)}(f+\fb[[t]]/N(O)).
\]

\item\label{i:ds-lim} The compatible system of maps
$\sW_{\kappa}^{as} \to \sW_{\kappa}^n$ defined by the vacuum
vectors defines an isomorphism:

\[
\sW_{\kappa}^{as} \to \underset{n}{\lim} \, \sW_{\kappa}^n \in \Pro(\Vect^{\heart})
\]

\noindent of pro-vector spaces. 

\end{enumerate}

\end{thm}

\begin{proof}

\step First, we observe that these maps are unique if they 
exist. Indeed, this follows immediately from the cyclicity of the
modules $\sW_{\kappa}^m$.

\step\label{st:wm->wn} We now construct a map 
$\alpha:\sW_{\kappa}^m \to \sW_{\kappa}^n$.

First, note that we have:

\[
\ind_{\Lie (\o{I}_n \cap \o{I}_m)}^{\widehat{\fg}_{\kappa}}(\psi) \onto 
\ind_{\Lie \o{I}_n}^{\widehat{\fg}_{\kappa}}(\psi).
\]

Then we claim that there is a canonical map:

\[
\ind_{\Lie \o{I}_m}^{\widehat{\fg}_{\kappa}}(\psi) \to 
\ind_{\Lie (\o{I}_n \cap \o{I}_m)}^{\widehat{\fg}_{\kappa}}(\psi) 
\otimes
\ell^{n,m}
[(m-n)\Delta] \in \widehat{\fg}_{\kappa}\mod.
\]

\noindent It suffices to construct:

\[
\psi \to 
\ind_{\Lie (\o{I}_n \cap \o{I}_m)}^{\Lie\o{I}_m}(\psi) 
\otimes
\ell^{n,m}
[(m-n)\Delta] \in \Lie\o{I}_m\mod.
\]

\noindent 
More generally, suppose $\fh_1 \subset \fh_2$ is
an open Lie subalgebra in a profinite dimensional Lie algebra.
Then for any module $M \in \fh_2\mod$, we claim that there is
a canonical map:

\[
M \to \ind_{\fh_1}^{\fh_2}\big(\Oblv(M) \otimes \det(\fh_2/\fh_1)[\dim\fh_2/\fh_1]\big).
\]

\noindent Indeed, by the projection formula it suffices
to construct this for $M = k$ the trivial module, and then
it is given by Lemma \ref{l:coh-ind-cpt-hc}.

By composition, we obtain:

\begin{equation}\label{eq:wm->wn}
\ind_{\Lie \o{I}_m}^{\widehat{\fg}_{\kappa}}(\psi) \to 
\ind_{\Lie \o{I}_n}^{\widehat{\fg}_{\kappa}}(\psi) \otimes
\ell^{n,m}
[(m-n)\Delta].
\end{equation}

\noindent Applying $\Psi$ now gives the desired map.

Note that these morphisms compose well
as we vary $n$ and $m$.

\step By our generalities on filtrations from Appendix \ref{a:hc}, 
$\alpha:\sW_{\kappa}^m \to \sW_{\kappa}^n$ is filtered,
and on associated graded gives the
restriction map along our closed embeddings.

This observation actually implies the rest of the results.
The map preserves vacuum vectors by their construction. 
It is surjective because it is surjective at the associated graded
level. For a surjective morphism of
filtered abelian groups with filtrations bounded from below, 
surjectivity at the associated graded level
is equivalent to strictness.
Finally, the projective system $\{\sW_{\kappa}^n\}$ 
gives $\sW_{\kappa}^{as}$ because this is true at the associated
graded level.

\end{proof}

\section {Affine Skryabin theorem}\label{s:skryabin}

\subsection{}

Let $D^+(\sW_{\kappa}\mod^{\heart})$ denote the bounded
below derived category considered as a DG category.
Then define:

\[
\sW_{\kappa}\mod^c \subset D^+(\sW_{\kappa}\mod^{\heart})
\]

\noindent as the full subcategory 
generated by the objects
$\sW_{\kappa}^n \in \sW_{\kappa}\mod^{\heart}$ 
under cones and shifts; recall that these objects where
defined in \S \ref{s:ds}.
Finally, following \cite{dmod-aff-flag} \S 22-3, we define:

\[
\sW_{\kappa}\mod \coloneqq \Ind(\sW_{\kappa}\mod^c).
\]

The purpose of this section is to prove:

\begin{thm}[Affine Skryabin theorem]\label{t:aff-skry}

There is a canonical equivalence:

\[
\Whit(\widehat{\fg}_{\kappa}\mod) \simeq
\sW_{\kappa}\mod.
\]

\noindent This equivalence has the property that the composition:

\[
\widehat{\fg}_{\kappa}\mod \to 
\Whit(\widehat{\fg}_{\kappa}\mod) \simeq \sW_{\kappa}\mod 
\xar{\Oblv} \Vect
\]

\noindent is the Drinfeld-Sokolov functor $\Psi$; here the first
morphism is the canonical morphism coming from identifying
$\Whit$ with coinvariants.

\end{thm}

\subsection{Construction of the $t$-structure}

The core of the proof is a construction of a canonical $t$-structure
on $\Whit(\widehat{\fg}_{\kappa}\mod)$ and an analysis
of some nice properties that it has.

\subsection{}

Recall from \S \ref{s:adolescence} that we have the
categories $\Whit^{\leq n}(\widehat{\fg}_{\kappa}\mod)$,
and adjoint functors $(\iota_{n,m,!},\iota_{n,m}^!)$
for $0 \leq n\leq m \leq \infty$.

For $n <\infty$, note that
$\Whit^{\leq n}(\widehat{\fg}_{\kappa}\mod) = 
\widehat{\fg}_{\kappa}\mod^{\o{I}_n,\psi}$ has a canonical
$t$-structure (the forgetful functor to $\widehat{\fg}_{\kappa}\mod$
is $t$-exact).

The key observation is:

\begin{lem}\label{l:amp}

For all $n\leq m<\infty$, the functor: 

\[
\iota_{n,m,!}[-(m-n)\Delta] 
\overset{Thm. \ref{t:!-avg} \eqref{i:!-avg=*-avg}}{=}
\iota_{n,m,*}[(m-n)\Delta]
\]

\noindent is $t$-exact. 
(Again, $\Delta \coloneqq 2(\rho,\check{\rho})$ as in
\S \ref{ss:!-avg}.)

\end{lem}

\begin{rem}

We prove Lemma \ref{l:amp} in Appendix \ref{a:av}
after giving a sketch below.

\end{rem}

\begin{proof}[Proof sketch for Lemma \ref{l:amp}]

The idea is that for general reasons
$*$-averaging  $\o{I}_n\cap \o{I}_m$ to $\o{I}_m$ has 
cohomological amplitude:

\[
[0,\dim \o{I}_m/\o{I}_n \cap \o{I}_n] = [0,(m-n)\Delta]
\]

\noindent while $!$-averaging has amplitude:

\[
[-\dim \o{I}_m/\o{I}_n \cap \o{I}_n,0] = [-(m-n)\Delta,0].
\]

\noindent Indeed, this follows since they are 
essentially given by $*$ and $!$-versions of de Rham 
cohomology along the \emph{affine} scheme
$\o{I}_m/\o{I}_n \cap \o{I}_n$.

Then recall that $\iota_{n,m,*}$ and $\iota_{n,m,!}$ are a
composition of a forgetful functor, which is $t$-exact,
with such an averaging functor.
So we find that
$\iota_{n,m,*}[(m-n)\Delta]$ is right $t$-exact,
while $\iota_{n,m,!}[-(m-n)\Delta]$ is left $t$-exact.
These functors coincide by Theorem \ref{t:!-avg} \eqref{i:!-avg=*-avg},
so they are $t$-exact.

\end{proof}

\subsection{}

We now have the following result:

\begin{prop}\label{p:t-str-whit}

\begin{enumerate}

\item 

There exists a unique $t$-structure on 
$\Whit(\widehat{\fg}_{\kappa}\mod)$ compatible with
filtered colimits such that
the functors:

\[
\iota_{n,!}[n\Delta]
\]

\noindent are $t$-exact.

\item\label{i:whit-derived}

With respect to this $t$-structure,
$\Whit(\widehat{\fg}_{\kappa}\mod)^+$ is the (DG) bounded below
derived category $D^+(\Whit(\widehat{\fg}_{\kappa}\mod^{\heart}))$
of its heart.

\end{enumerate}

\end{prop}

\begin{convention}

Here we make the decision to always equip 
$\Whit^{\leq n}(\widehat{\fg}_{\kappa}\mod)$ with the
canonical $t$-structure on Harish-Chandra modules,
i.e., realizing the category as $\widehat{\fg}_{\kappa}\mod^{\o{I}_n,\psi}$.
This can be confusing, since we need the shifts above for
the functors to be $t$-exact; but shifting all our $t$-structures
would probably be more confusing. 

\end{convention}

Using the automorphisms 
$[n\Delta]:\Whit^{\leq n}(\widehat{\fg}_{\kappa}\mod) \to 
\Whit^{\leq n}(\widehat{\fg}_{\kappa}\mod)$,
we have:

\[
\Whit(\widehat{\fg}_{\kappa}\mod) = 
\underset{n,\iota_{n,m,!}[-(m-n)\Delta]}{\colim}
\Whit^{\leq n}(\widehat{\fg}_{\kappa}\mod).
\]

\noindent Therefore, the above result is an immediate consequence of the following
plus Lemma \ref{l:bernstein-lunts-tate} (the Bernstein-Lunts theorem).

\begin{lem}\label{l:t-str}

Suppose $i \mapsto \sC_i \in \DGCat_{cont}$ is a filtered
diagram of cocomplete DG categories, each equipped with
a $t$-structure compatible with filtered colimits. 
Suppose every structure
functor $\psi_{i,j}:\sC_i \to \sC_j$ is $t$-exact and
admits a continuous right adjoint $\vph_{i,j}$.

\begin{enumerate}

\item 

Then $\sC \coloneqq \colim_i \sC_i$ admits a unique
$t$-structure such that each $\psi_i:\sC_i \to \sC$
is $t$-exact. 

\item 

If $\sC_i^+ = D^+(\sC_i^{\heart})$ is the
bounded below derived category of its heart for each $i$,
and if the filtered category indexing our colimit is countable, then
$\sC^+ = D^+(\sC^{\heart})$. 

\end{enumerate}

\end{lem}

\begin{proof}

\step 

Define a $t$-structure on $\sC$ by declaring
$\sC^{\leq 0}$ to be generated under colimits
by the subcategories $\psi_i(\sC_i^{\leq 0})$. 
It is equivalent to say that $\sF \in \sC^{\geq 0}$ if and only
if $\vph_i(\sF) \in \sC_i^{\geq 0}$ for all $i$; here
$\vph_i:\sC \to \sC_i$ is the (continuous) right adjoint to
$\psi_i$.

We want to show that the functors $\psi_i$ are 
$t$-exact. Clearly they are right $t$-exact, so it remains
to show left $t$-exactness. 
Suppose $\sF \in \sC_i^{\geq 0}$; we need to
show $\psi_i(\sF) \in \sC^{\geq 0}$. It suffices
to show $\vph_j\psi_i(\sF) \in \sC_j^{\geq 0}$ for all indices $j$.
By standard generalities about filtered co/limits (see \cite{dgcat}),
we have:

\[
\vph_j\psi_i(\sF) = 
\underset{i \to k \leftarrow j}{\colim} \, \vph_{j,k}\psi_{i,k}(\sF_i) 
\in \sC_j.
\]

\noindent But $\psi_{i,k}(\sF_i) \in \sC_k^{\geq 0}$ by $t$-exactness,
and $\vph_{j,k}(\sC_k^{\geq 0}) \subset \sC_j^{\geq 0}$ because
it is right adjoint to a $t$-exact functor, so each term in our
colimit is in degrees $\geq 0$.
Since our $t$-structures are compatible with filtered colimits, we obtain
the claim.

\step 

Now suppose that $\sC_i^+ = D^+(\sC_i^{\heart})$ is the bounded
below derived category of its heart. 
We want to see the same
property for $\sC$. 

Some technical comments first: note that the 
$t$-structure on $\sC_i$ is right separated 
by assumption. By the compatibility of the $t$-structures with filtered
colimits, it is therefore also right complete. It follows formally that
the $t$-structure on $\sC$ is right separated, and again, that it is
right complete.

Therefore, it suffices to show that an injective
object $I$ in the abelian category $\sC^{\heart}$ is 
``actually injective" in $\sC$, i.e., that the following
equivalent conditions are satisfied:

\begin{itemize}

\item $\ul{\Hom}_{\sC}(-,I): \sC \to \Vect^{op}$ is $t$-exact.

\item For $\sF \in \sC^{\geq 0}$, 
$\ul{\Hom}_{\sC}(\sF,I) \in \Vect^{\leq 0}$.

\item For $\sF \in \sC^{\heart}$, $\ul{\Hom}_{\sC}(\sF,I) \in \Vect^{\heart}$.

\end{itemize}

We will do this in what follows.

\step 

We need some general properties about homological algebra
for abelian categories.

Suppose $\sA$ and $\sB$ are Grothendieck abelian categories and
$F:D^+(\sA) \rightleftarrows D^+(\sB):G$ are adjoint with
$F$ $t$-exact. Then recall that $G$ is the derived functor
of $H^0(G): \sB \to \sA$, that is, for any injective $I \in \sB$,
$G(I) = H^0(G(I))$.

Indeed, for $\sF \in \sA$, we have:

\[
\Ext_{\sA}^i(\sF,G(I)) = \Ext_{\sB}^i(F(\sF),I) = 0 \text{ for }i>0.
\]

\noindent A standard argument shows as well that $H^0 G(I)$ is
injective. Therefore, by the long exact sequence,
$\tau^{>0} G(I) = \Coker(H^0(G(I)) \to G(I)) \in D^+(\sA)$ satisfies the above vanishing as well. But since $\tau^{>0} G(I)$ is in degrees $>0$,
the vanishing of positive $\Ext$s is enough to guarantee that it is zero.

Therefore, in our setting, 
the functors $H^0 \vph_{i,j}: \sC_j^{\heart} \to \sC_i^{\heart}$
preserve injective objects, and $\vph_{i,j} = H^0\vph_{i,j}$ when
evaluated on such an object.

\step 

Now we show that $\vph_i(I) = H^0(\vph_i(I))$ for every $i$.
(The argument is unfortunately a little indirect.)

Suppose $i \to j$ is given. We form the cone:

\[
\Coker(\vph_{i,j} H^0(\vph_j(I)) \to \vph_i(I)) = 
\Coker\big(\vph_{i,j} H^0(\vph_j(I)) \to \vph_{i,j}\vph_j(I)\big) = 
\vph_{i,j} \tau^{>0} \vph_j(I).
\]

\noindent Obviously the last term is in cohomological 
degrees $>0$. Moreover,
note that $H^0(\vph_j(I))$ is injective in $\sC_j^{\heart}$,
so the first term $\vph_{i,j} H^0(\vph_j(I)$ 
in our distinguished triangle is in cohomological
degree $0$ by the last step.
Therefore, this is a truncation sequence, and in particular 
we have:

\[
H^0(\vph_{i,j} H^0(\vph_j(I))) \isom H^0(\vph_i(I)). 
\]

We remark again that 
$H^0(\vph_{i,j} H^0(\vph_j(I))) = \vph_{i,j} H^0(\vph_j(I))$. Therefore,
the objects $j \mapsto H^0(\vph_j(I))$ defines an object
of $\lim_j \, \sC_j = \sC$ (the structure maps in this limit being the $\vph_j$).

Recall that by filteredness, any object $\sF$ of $\sC$ can be written
as:

\[
\underset{i}{\colim} \, \psi_i \vph_i(\sF).
\]

\noindent Therefore, the object of $\sC$ constructed above
is:

\[
\underset{i}{\colim} \, \psi_i H^0(\vph_j(I)) =
H^0(\underset{i}{\colim} \, \psi_i \vph_j(I)) = 
H^0(I) = I.
\]

\noindent This proves the claim.

\step 

We now conclude the argument. 
For $\sF \in \sC$, we have 
$\sF = \colim_i \psi_i \vph_i(\sF) \in \sC$,
so:

\[
\ul{\Hom}_{\sC}(\sF,I) = 
\underset{i}{\lim} \, \ul{\Hom}_{\sC_i}(\psi_i\vph_i(\sF),I) =
\underset{i}{\lim} \, \ul{\Hom}_{\sC_i}(\vph_i(\sF),\vph_i(I)).
\]

\noindent Now suppose $\sF \in \sC^{\heart}$. 
Then $\vph_i(\sF_i) \in \sC_i^{\geq 0}$ and $\vph_i(I)$ is
injective by the above, so each term in this limit is in
$\Vect^{\leq 0}$. Because this is a countable limit by
assumption, the limit is in $\Vect^{\leq 1}$. 
But now recall that:

\[
H^1(\Hom_{\sC}(\sF,I)) = 
\Ext_{\sC}^1(\sF,I) =
\Ext_{\sC^{\heart}}^1(\sF,I) = 0
\]

\noindent by injectivity of $I$, completing the argument.

\end{proof}

\begin{notation}

There is risk for notational confusion in using this $t$-structure:
e.g., $\Whit^{\leq n}(\widehat{\fg}_{\kappa}\mod)$ continues
to denote the adolescent Whittaker category, while
$\Whit(\widehat{\fg}_{\kappa}\mod)^{\leq n}$ denotes the
subcategory of $\Whit(\widehat{\fg}_{\kappa}\mod)$ consisting
of objects cohomologically bounded from above by $n$.

\end{notation}

\subsection{Fiber functor}

Observe that we have:

\[
\xymatrix{
\widehat{\fg}_{\kappa}\mod \ar[dr]^{\Psi} \ar[d] & \\
\Whit(\widehat{\fg}_{\kappa}\mod) \ar@{..>}[r]^(.6){\Psi^{\Whit}} &
\Vect 
}
\]

\noindent where we regard $\Whit$ as coinvariants. 
Indeed, this follows because each functor
$C^{\dot}(\Ad_{-n\check{\rho}(t)} \fn[[t]],(-) \otimes -\psi)$
obviously\footnote{Indeed, it suffices to show that
the natural transformation $\Oblv \Av_*^{\psi} \to \id$
induces an isomorphism on this Lie algebra cohomology;
(here everything is with respect to $\Ad_{-n\check{\rho}(t)} N(O)$
and its Lie algebra). It obviously suffices to prove this statement
for $\Ad_{-n\check{\rho}(t)} \fn[[t]] \mod$ 
in place of $\widehat{\fg}_{\kappa}\mod$. Then the 
Lie algebra cohomology functor is corepresented by the trivial
representation, so with the twist by the character it is corepresented
by the 1-dimensional module defined by $\psi$. Since this object
lies in 
$\Ad_{-n\check{\rho}(t)} \fn[[t]] \mod^{\Ad_{-n\check{\rho}(t)} N(O),\psi}$,
and since the forgetful functor from this category
to $\Ad_{-n\check{\rho}(t)} \fn[[t]]\mod$ is fully-faithful, we obtain the
result.}
factors as:

\[
\xymatrix{
\widehat{\fg}_{\kappa}\mod \ar[d]^{\Av_*^{\psi}} \ar[dr] & \\
\widehat{\fg}_{\kappa}\mod^{\Ad_{-n\check{\rho}(t)} N(O),\psi} \ar@{..>}[r] 
& \Vect.
}
\]

We will prove Theorem \ref{t:aff-skry} by an analysis of this
functor.

\begin{warning}\label{w:shifts}

The cohomological shifts can cause a great deal of confusion here.
By the above, for $m \geq n$ the following diagram commutes:

\[
\xymatrix{
\Whit^{\leq n}(\widehat{\fg}_{\kappa}\mod) =
\widehat{\fg}_{\kappa}\mod^{\o{I}_n,\psi} 
\ar[rr]^(.7){\Oblv} \ar[d]^{\iota_{n,m,*}}
& &
\widehat{\fg}_{\kappa}\mod \ar[r]^{\Psi} & 
\Vect \ar@{=}[d] \\
\Whit^{\leq m}(\widehat{\fg}_{\kappa}\mod) =
\widehat{\fg}_{\kappa}\mod^{\o{I}_m,\psi} 
\ar[rr]^(.7){\Oblv} 
& &
\widehat{\fg}_{\kappa}\mod \ar[r]^{\Psi} & \Vect.
}
\]

\noindent Then recall that in the proof of Theorem \ref{t:inv-coinv},
we used cohomological shifts by $2n\Delta$ on
$\Whit^{\leq n}$ to identify 
$\Whit(\sC) \coloneqq \colim_{n,\iota_{n,m,*}} \Whit^{\leq n}(\sC)$
with the colimit under the functors $\iota_{n,m,!}$. 

Therefore, we obtain the commutativity of the following diagram:

\[
\xymatrix{
\Whit^{\leq n}(\widehat{\fg}_{\kappa}\mod) 
\ar[rr]^{\Oblv} \ar[d]^{\iota_{n,!}} & &
\widehat{\fg}_{\kappa}\mod \ar[rr]^{\Psi[-2n\Delta]} & &
\Vect \ar@{=}[d] \\
\Whit(\widehat{\fg}_{\kappa}\mod) 
\ar[rrrr]^{\Psi^{\Whit}} & & & &
\Vect.
}
\]

\noindent This is the reason we use the notation $\Psi^{\Whit}$
(rather than $\Psi$):
we continue to use the notation $\Psi$ for the functor
$\Whit^{\leq n}(\widehat{\fg}_{\kappa}\mod) \to \Vect$ obtained
by thinking of this category as Harish-Chandra modules, and
then $\Psi^{\Whit} \circ \iota_{n,!} = \Psi[-2n\Delta]$.

(Note that the confusion arises here because we are not 
adhering to the principle espoused in Warning \ref{w:adolescent-forget}.)

\end{warning}

\subsection{}

The main observation is that 
$\Psi^{\Whit}:\Whit(\widehat{\fg}_{\kappa}\mod) \to \Vect$ is canonically 
corepresented in a nice way. 
Heuristically, if we think of the source category as the category
of modules over the $\sW_{\kappa}$-algebra, we want to show that
it is corepresented by the projective system $\{\sW_{\kappa}^n\}_{n \geq 0}$
from \S \ref{s:ds}.

\subsection{}

To make this idea precise, define 
$\preprime\sW_{\kappa}^n \in \Whit^{\leq n}(\widehat{\fg}_{\kappa}\mod)$
as:

\[
\ind_{\Lie \o{I}_n}^{\widehat{\fg}_{\kappa}}(\psi) 
\otimes \ell^{n,\vee}
[n\Delta] \in \widehat{\fg}_{\kappa}\mod^{\o{I}_n,\psi} = 
\Whit^{\leq n}(\widehat{\fg}_{\kappa}\mod).
\]

\noindent Here (and throughout this section) we are freely using
Notation \ref{n:det-lines}, so 
$\ell$ is shorthand for a determinant line.
We use the same notation for the induced
object $\iota_{n,!}(\preprime\sW_{\kappa}^n)$
of $\Whit(\widehat{\fg}_{\kappa}\mod)$.

We have the following basic properties.

\begin{itemize}

\item 
$\preprime\sW_{\kappa}^n \in \Whit(\widehat{\fg}_{\kappa}\mod)^{\heart}$.
Indeed, recall that $\iota_{n,!}[n\Delta]$ is $t$-exact, giving the claim.

\item The notation is compatible with \S \ref{s:ds} in the sense that:

\[
\Psi^{\Whit}(\preprime\sW_{\kappa}^n) = \sW_{\kappa}^n 
\in \Vect^{\heart} \subset \Vect.
\]

\noindent Indeed, this was the definition of the right hand side
(we remark that Warning \ref{w:shifts} is important here).

\end{itemize}

\subsection{}

Now observe that in the proof of
Theorem \ref{t:ds-n/m}, we in effect constructed
a canonical map 
$\alpha_{n,m}:\preprime\sW_{\kappa}^m \to \preprime\sW_{\kappa}^n$
for $m \geq n$.

Indeed, in \eqref{eq:wm->wn}, we produced a map:

\[
\ind_{\Lie \o{I}_m}^{\widehat{\fg}_{\kappa}}(\psi) \to 
\ind_{\Lie \o{I}_n}^{\widehat{\fg}_{\kappa}}(\psi) \otimes
\ell^{n,m}
[(m-n)\Delta] \in 
\widehat{\fg}_{\kappa}\mod^{\o{I}_n\cap \o{I}_m,\psi}
\]

\noindent This induces a morphism:

\begin{equation}\label{eq:almost-alpha}
\alpha_{n,m}^{\prime}:
\ind_{\Lie \o{I}_m}^{\widehat{\fg}_{\kappa}}(\psi) \to
\Av_*^{\o{I}_m,\psi} 
(\ind_{\Lie \o{I}_n}^{\widehat{\fg}_{\kappa}}(\psi) \otimes
\ell^{n,m}
[(m-n)\Delta]) \in 
\widehat{\fg}_{\kappa}\mod^{\o{I}_m,\psi}.
\end{equation}

\noindent Note that $\Av_*^{\o{I}_m,\psi}$  
is the same as $\iota_{n,m,*}$ here. Therefore, incorporating
the determinant twists and cohomological shifts, and switching
to adolescent Whittaker notation, we obtain:\footnote{
Regarding the arithmetic of cohomological shifts: 
in the last equality, up to the factor of $\Delta$
we have a contribution $(2m-n)$ from the previous
line; switching from $\iota_{n,m,*}$ to $\iota_{n,m,!}$ means
we need to add a shift by $-2(m-n)$, producing the $-n$ that is displayed.}

\[
\begin{gathered}
\alpha_{n,m}:\preprime\sW_{\kappa}^m = 
\ind_{\Lie \o{I}_m}^{\widehat{\fg}_{\kappa}}(\psi) 
\otimes \ell^{m,\vee}
[m\Delta] \to
\iota_{n,m,*} 
\big(\ind_{\Lie \o{I}_n}^{\widehat{\fg}_{\kappa}}(\psi)\big) \otimes
\ell^{n,m}
[(m-n)\Delta] \otimes \ell^{m,\vee}
[m\Delta] = \\
\iota_{n,m,!} 
\big(\ind_{\Lie \o{I}_n}^{\widehat{\fg}_{\kappa}}(\psi)
\otimes \ell^{n,\vee}
[-n\Delta]
\big)
 = \iota_{n,m,!}(\preprime\sW_{\kappa}^n) 
\in \Whit^{\leq m}(\widehat{\fg}_{\kappa}\mod).
\end{gathered}
\]

We have the following important fact.

\begin{lem}\label{l:alpha-epi}

The morphism 
$\alpha_{n,m}:\preprime\sW_{\kappa}^m \to \preprime\sW_{\kappa}^n \in 
\Whit(\widehat{\fg}_{\kappa}\mod)^{\heart}$ is an epimorphism
in this abelian category.

\end{lem}

\begin{proof}

\step 

By $t$-exactness of $\iota_{m,!}[m\Delta]$, it suffices to show
the corresponding fact in 
$\Whit^{\leq m}(\widehat{\fg}_{\kappa}\mod)$.
Tautologically, this reduces to a fact about
the morphism $\alpha_{n,m}^{\prime}$ from \eqref{eq:almost-alpha}.

Observe that both the source and the target have
canonical KK filtrations in $\widehat{\fg}_{\kappa}\mod^{\o{I}_m,\psi}$. 
We will first verify that the associated graded morphism
is an epimorphism. Afterwards, we will explain why this suffices
to deduce the result (the issue being that the KK filtration is
not bounded from below).

\step 

We have:

\[
\gr_{\dot}^{KK} \ind_{\Lie \o{I}_m}^{\widehat{\fg}_{\kappa}}(\psi) =
\sO_{f+\Lie\o{I}_m^{\perp}/\o{I}_m} \in 
\QCoh^{ren}(f+\Lie\o{I}_m^{\perp}/\o{I}_m).
\]

\noindent As in Remark \ref{r:!-avg-cl}, we have:

\[
\gr_{\dot}^{KK} \Av_*^{\o{I}_m,\psi} 
(\ind_{\Lie \o{I}_n}^{\widehat{\fg}_{\kappa}}(\psi) \otimes
\ell^{n,m} [(m-n)\Delta]) = 
\pi_*(\sO_{f+\Lie\o{I}_n^{\perp} \cap \Lie\o{I}_m^{\perp}/\o{I}_n \cap \o{I}_m})
\]

\noindent for $\pi:f+\Lie\o{I}_n^{\perp} \cap \Lie\o{I}_m^{\perp}/\o{I}_n \cap \o{I}_m \to f+\Lie\o{I}_n^{\perp}/\o{I}_n$ the projection;
note that the determinant twist and cohomological shift
are absorbed due to the ``mild correction" from \emph{loc. cit}.

By construction, our map is the canonical adjunction morphism.
Because $\pi$ is a closed embedding by Theorem \ref{t:!-avg-cl},
this map is an epimorphism as desired. 

\step 

Now we explain why the associated graded map being an isomorphism
suffices. 

Let $L_0 = t\partial_t$. 
We use the term \emph{$(L_0+\check{\rho})$-grading}
to refer to the corresponding grading on $\fg((t))$ induced by
taking the diagonal action of $\bG_m$ with respect to loop rotation and 
$\Ad_{\check{\rho}(-)}$. Note that 
$\psi:\o{I}_n \to k$ is graded for $k$ being given
degree $0$: the point is that $\frac{e_i}{t}$ has 
$(L_0+\check{\rho})$-degree $0$.

Therefore, $\ind_{\Lie \o{I}_n}^{\widehat{\fg}_{\kappa}}$
carries a canonical $(L_0+\check{\rho})$-grading. 
The same formally holds for:

\[
\Av_*^{\o{I}_m,\psi} 
(\ind_{\Lie \o{I}_n}^{\widehat{\fg}_{\kappa}}(\psi) \otimes
\ell^{n,m}
[(m-n)\Delta])
\]

\noindent and our map $\alpha_{n,m}^{\prime}$ is compatible with
these gradings.

Obviously it suffices to show $\alpha_{n,m}^{\prime}$ is an epimorphism
in the case $m>0$, since otherwise $m = n = 0$. 
We claim that in this case, for every integer $i$, the KK filtration on
the $i$th $(L_0+\check{\rho})$-graded component of
our modules is bounded from below.
This combined with the corresponding semi-classical statement 
obviously suffices to show the surjectivity of the morphism
$\alpha_{n,m}^{\prime}$, since it implies it on each graded component.

We show this below.

\step 

To verify the claim about the $(L_0+\check{\rho})$-grading
on $\ind_{\o{I}_n}^{\widehat{\fg}_{\kappa}}(\psi)$,
note that because the KK filtration is
compatible with the grading and separated (non-derivedly),
it suffices to show that the $(L_0+\check{\rho})$-eigenvalues
on $\gr_i^{KK} \ind_{\Lie \o{I}_m}^{\widehat{\fg}_{\kappa}}(\psi)$
are bounded above by some function of $i$ going to $-\infty$ as $i$ does. 

Recall that: 

\[
\gr_i^{KK} \ind_{\Lie \o{I}_m}^{\widehat{\fg}_{\kappa}}(\psi) = 
\underset{j}{\oplus} \Sym^j(\fg((t))/\Lie\o{I}_n)^{i-j}
\]

\noindent where the superscript indicates the $(i-j)$th graded
degree with respect to the $-\check{\rho}$-grading. 

Below, for $\alpha$ in the root lattice of $G$, we use the notation
$|\alpha|$ for $(\check{\rho},\alpha)$.

The $j$th summand above is spanned by elements of the form:

\[
\frac{e_{\alpha_1}}{t^{r_1}} \frac{e_{\alpha_2}}{t^{r_2}} \ldots
\frac{e_{\alpha_k}}{t^{r_k}} 
\frac{f_{\beta_{k+1}}}{t^{r_{k+1}}} \ldots 
\frac{f_{\beta_j}}{t^{r_j}}
\]

\noindent where $e_{\alpha_\ell} \in \fn$ is a non-zero 
vector of weight $\alpha_\ell$, $f_{\beta_\ell} \in \fb^-$ is a non-zero
vector of weight $-\beta_{\ell}$. (Note that $\beta_{\ell}$ 
can be a positive root or zero, and in the latter case $f_0$ can
be any non-zero vector in $\ft$.) That this vector has degree $i-j$
means that:

\begin{equation}\label{eq:i-j}
-\sum_{\ell = 1}^k |\alpha_{\ell}| + \sum_{\ell = k+1}^j |\beta_{\ell}|
= i-j.
\end{equation}

\noindent Finally, note that:

\[
\begin{cases} 
r_{\ell} \geq 1+m|\alpha_{\ell}| & \text{ if } 1\leq \ell \leq k \\
r_{\ell} \geq 1-m(|\beta_{\ell}|+1)
& \text{ if } k<\ell \leq j.
\end{cases}
\] 

\noindent by definition of $\o{I}_m$, and the fact that
$m>0$.

Then the $(L_0+\check{\rho})$-degree
of an element as above is:

\[
\begin{gathered}
\sum_{\ell = 1}^k -r_{\ell}+|\alpha_{\ell}| +
\sum_{\ell = k+1}^j -r_{\ell}-|\beta_{\ell}| \leq \\
\sum_{\ell = 1}^k -1-(m-1)|\alpha_{\ell}| +
\sum_{\ell = k+1}^j -1+(m-1)|\beta_{\ell}| + m 
\overset{\eqref{eq:i-j}}{=}
\\
-j+(m-1)(i-j) + m(j-k) = \\
(m-1) i - mk.
\end{gathered}
\]

\noindent If $m-1>0$, then clearly this goes to $-\infty$ as $i$ does.

To treat the general case (so additionally allowing $m = 1$), 
we need to bound $k$ in terms of $i$. 

For this, we let $\alpha_{max}$ denote the longest root of
$G$. We then apply \eqref{eq:i-j} and the fact that $j \geq 0$
(by its definition) to obtain:

\[
i \geq i-j = 
-\sum_{\ell = 1}^k |\alpha_{\ell}| + \sum_{\ell = k+1}^j |\beta_{\ell}| \geq
-k|\alpha_{\max}|.
\]

\noindent Therefore:

\[
-k \leq \frac{i}{|\alpha_{max}|}
\]

\noindent (safely assuming $G$ is not a torus, so this fraction 
makes sense).

Applying this above, we find that the $(L_0+\check{\rho})$ degrees
are at most:

\[
(m-1) i - mk \leq 
(m-1 + \frac{m}{|\alpha_{max}|}) i
\]

\noindent which does indeed go to $-\infty$ as $i$ does.\footnote{Of course,
this only makes sense if $G$ is not a torus. In that case, $\gr_i^{KK} = 0$
for $i<0$, so the game is over before it even started.}

A similar calculation treats $\Av_*^{\o{I}_m,\psi} 
(\ind_{\Lie \o{I}_n}^{\widehat{\fg}_{\kappa}}(\psi) \otimes
\ell^{n,m}
[(m-n)\Delta])$. We note that this averaging can be explicitly
described as 
$\Gamma(G(K),\delta_{\o{I}_m\o{I}_n}^{\psi})^{\o{I}_n}$
where $\delta_{\o{I}_m\o{I}_n}^{\psi}$ was defined in \S \ref{ss:!=*-pf},
and where the notation indicates that we take global sections on the
loop group with coefficients in this $D$-module, and take right 
$\o{I}_n$-invariant sections.\footnote{In other words,
the $\kappa$-twisted $D$-module
$\delta_{\o{I}_m\o{I}_n}^{\psi})$ does not descend
to $G(K)/\o{I}_n$ because of the presence of the character $\psi$.
But its underlying quasi-coherent sheaf does descend, and we are
taking its global sections (which will still be acted on by $\widehat{\fg}_{\kappa}$);
this is because the exponential $D$-module has
trivial underlying multiplicative quasi-coherent sheaf.}

\end{proof}

\subsection{}

We have the following key observation.

\begin{prop}\label{p:ds-corep}

The pro-object 
$\{\preprime\sW_{\kappa}^n\}_{n \geq 0} \in 
\Pro(\Whit(\widehat{\fg}_{\kappa}\mod))$ canonically corepresents
the functor $\Psi^{\Whit}$.

\end{prop}

\begin{proof}

By definition, 
$\preprime\sW_{\kappa}^n \in \Whit(\widehat{\fg}_{\kappa}\mod)$ 
corepresents
the functor:

\[
\sF \mapsto 
C^{\dot}(\Lie\o{I}_n,\iota_n^!(\sF) \otimes -\psi)\otimes \ell^n [-n\Delta].
\]

\noindent Note that this complex maps by restriction of invariants to:

\[
C^{\dot}(\Ad_{-n\check{\rho}(t)} \fn[[t]],\iota_n^!(\sF) \otimes -\psi)\otimes \ell^n [-n\Delta].
\]

\noindent If we had a shift by \emph{positive} $n\Delta$ instead,
this in turn would canonically map to 
$\Psi(\iota_n^!(\sF))$ by definition of $\Psi$.
As it is, it maps instead to:

\[
\Psi(\iota_n^!(\sF))[-2n\Delta] = \Psi^{\Whit}(\iota_n^!(\sF)).
\]

\noindent Passing to the colimit in $n$, we get a canonical morphism:

\[
\underset{n}{\colim} \, 
C^{\dot}(\Lie\o{I}_n,\iota_n^!(\sF) \otimes -\psi)\otimes \ell^n [-n\Delta] \to
\underset{n}{\colim} \, \Psi^{\Whit}(\iota_{n,!}\iota_n^!(\sF)) =
\Psi^{\Whit}(\sF).
\]

\noindent The left hand side is the functor corepresented by our
pro-object, so we need to see that this morphism is an isomorphism.

This is a straightforward verification. By definition of $\Psi^{\Whit}$, it
suffices to show that for $M \in \widehat{\fg}_{\kappa}\mod$,
the morphism:

\[
\underset{n}{\colim} \,
C^{\dot}(\Lie\o{I}_n,M \otimes -\psi)\otimes \ell^n [n\Delta] \to
\Psi(M) = 
\underset{n}{\colim} \,
C^{\dot}(\Ad_{-n\check{\rho}(t)} \fn[[t]],M \otimes -\psi)\otimes \ell^n [n\Delta]
\]

\noindent is an isomorphism. 

In the notation from the proof of Lemma \ref{l:n->infty} and
using the same method, this follows
from the identity

\[
\underset{m}{\colim} \, C^{\dot}(\Lie(\o{I}_{n,m},M) \isom 
C^{\dot}(\Ad_{-n\check{\rho}(t)} \fn[[t]],M)
\]

\noindent for $n > 0$, which is a straightforward verification
using the fact that $\fh\mod$ is a co/limit in the
standard way for $\fh$ a profinite-dimensional
Lie algebra. (We have omitted $\psi$ because this holds for
any $M \in \Lie\o{I}_n\mod$.) 

\end{proof}

\subsection{}

We obtain the following important consequence of 
the above results.

\begin{thm}\label{t:psi-exact/cons}

The functor $\Psi^{\Whit}$ is $t$-exact.
Its restriction to $\Whit(\widehat{\fg}_{\kappa}\mod)^+$
is conservative.

\end{thm}

\begin{proof}

\step 

By Proposition \ref{p:ds-corep}, $\Psi^{\Whit}$ is corepresented
by the pro-object $\{\preprime\sW_{\kappa}^n\}_{n \geq 0}$. 
Because each of these objects lies in the heart of the
$t$-structure, we obtain that $\Psi^{\Whit}$ is left $t$-exact.

Now note that 
$\Whit(\widehat{\fg}_{\kappa}\mod)$ is compactly generated
by the objects $\preprime\sW_{\kappa}^n$. Indeed, this
follows from the co/limit formalism and the observation
that for $n>0$, $\ind_{\Lie \o{I}_n}^{\widehat{\fg}_{\kappa}}(\psi)$
compactly generates $\widehat{\fg}_{\kappa}\mod^{\o{I}_n,\psi}$
(by prounipotence of $\o{I}_n$). 

Since these compact generators lie
in $\Whit(\widehat{\fg}_{\kappa}\mod)^{\leq 0}$, to verify the
fact that $\Psi$ is right $t$-exact, it suffices to show that
$\Psi^{\Whit}(\preprime\sW_{\kappa}^n) \in \Vect^{\leq 0}$.
But as we noted before, this object is $\sW_{\kappa}^n$, which lies
in $\Vect^{\heart}$.

\step 

Suppose $\sF \in \Whit(\widehat{\fg}_{\kappa}\mod)^{\geq 0}$
is given $\Psi^{\Whit}(\sF) \in \Vect^{>0}$. By right completeness
of the $t$-structure on $\Whit(\widehat{\fg}_{\kappa}\mod)$,
the conservativeness will follow if we can show  
$\sF \in \Whit(\widehat{\fg}_{\kappa}\mod)^{>0}$.

Because the objects 
$\preprime\sW_{\kappa}^n \in  \Whit(\widehat{\fg}_{\kappa}\mod)^{\heart}$
generate $\Whit(\widehat{\fg}_{\kappa}\mod)^{\leq 0}$ under colimits,
it suffices so show that:

\[
\ul{\Hom}_{\Whit(\widehat{\fg}_{\kappa}\mod)}
(\preprime\sW_{\kappa}^n,\sF) \in \Vect^{>0}
\]

\noindent for all $n$. Clearly this complex is in $\Vect^{\geq 0}$,
so we need to show that its $H^0$ vanishes.

Observe that:

\[
H^0 \ul{\Hom}_{\Whit(\widehat{\fg}_{\kappa}\mod)}
(\preprime\sW_{\kappa}^n,\sF) = 
\Hom_{\Whit(\widehat{\fg}_{\kappa}\mod)^{\heart}}
(\preprime\sW_{\kappa}^n,H^0(\sF)).
\]

\noindent Therefore, by Lemma \ref{l:alpha-epi}, the map:

\[
H^0 \ul{\Hom}_{\Whit(\widehat{\fg}_{\kappa}\mod)}
(\preprime\sW_{\kappa}^n,\sF) \to 
H^0 \ul{\Hom}_{\Whit(\widehat{\fg}_{\kappa}\mod)}
(\preprime\sW_{\kappa}^{n+1},\sF) 
\]

\noindent of restriction along $\alpha_{n,n+1}$ is \emph{injective}.
Therefore, it suffices to show that the colimit under $n$ vanishes.
But we have:

\[
\underset{n}{\colim} \, 
H^0 \ul{\Hom}_{\Whit(\widehat{\fg}_{\kappa}\mod)}
(\preprime\sW_{\kappa}^n,\sF) = 
H^0 \big(\underset{n}{\colim} \, 
\ul{\Hom}_{\Whit(\widehat{\fg}_{\kappa}\mod)}
(\preprime\sW_{\kappa}^n,\sF) ) =
H^0 \Psi^{\Whit}(\sF) = 0
\]

\noindent by assumption, giving the result.

\end{proof}

\subsection{Affine Skryabin}

We now prove the result with which we began this section.

\begin{proof}[Proof of Theorem \ref{t:aff-skry}]

\step 

The main step is to compute the heart of our
$t$-structure on $\Whit(\widehat{\fg}_{\kappa}\mod)$.

We will do this using the following paradigm.
Suppose $\sA$ is a Grothendieck abelian category
and $F: \sA \to \Vect^{\heart}$ is a conservative exact functor
that commutes with colimits.

Recall that endomorphisms of the functor $F$ can naturally
be considered as a pro-vector space
$\End(F) \in \Pro(\Vect^{\heart})$. To compute it
explicitly, take $\{\sF_i\} \in \sA$ pro-representing the functor
$F$ and then evaluate $F(\sF_i)$ as a pro-vector space. 
It is standard that $\End(F)$ is a topological chiral algebra
and that the canonical functor:

\[
\sA \to \End(F)\mod(\Vect^{\heart})
\]

\noindent is an equivalence. (C.f. \cite{chiral} \S 3.6.)

We apply this with 
$\sA = \Whit(\widehat{\fg}_{\kappa}\mod)^{\heart}$ and
$F = \Psi^{\Whit,\heart}:\Whit(\widehat{\fg}_{\kappa}\mod)^{\heart} \to \Vect^{\heart}$. 
Note that $\Psi^{\Whit,\heart}$
is exact and conservative by Theorem \ref{t:psi-exact/cons}.
We want to show the topological chiral algebra
defined by this data is $\sW_{\kappa}^{as}$, i.e., the 
one associated with the vertex algebra $\sW_{\kappa}$.

We have a canonical morphism of topological chiral algebras

\[
\sW_{\kappa}^{as} \to \End(\Psi^{\Whit,\heart})
\]

\noindent because $\sW_{\kappa}^{as}$ acts on the cohomologies
of $\Psi$ of any object of $\widehat{\fg}_{\kappa}\mod$
(c.f. \S \ref{ss:w-acts-on-psi}).

Therefore, to show that this map is an isomorphism, we just need to show
it at the level of pro-vector spaces. 
Because $\{\preprime\sW_{\kappa}^n\}_{n \geq 0}$ corepresents
$\Psi^{\Whit,\heart}$, $\End(\Psi^{\Whit,\heart})$ is the
pro-vector space given by:

\[
\{\Psi^{\Whit}(\preprime\sW_{\kappa}^n)\}_{n \geq 0} = 
\{\sW_{\kappa}^n\}_{n \geq 0} = \sW_{\kappa}^{as}
\]

\noindent where the last equality is 
Theorem \ref{t:ds-n/m} \eqref{i:ds-lim}. Clearly this identification
is compatible with the map above, so we obtain the claim.

\step

From here, the theorem is straightforward. 
By Proposition \ref{p:t-str-whit} \eqref{i:whit-derived} and
the above, we have:

\[
\Whit(\widehat{\fg}_{\kappa}\mod)^+ \simeq 
D^+(\sW_{\kappa}\mod^{\heart}).
\]

\noindent Because $\Whit(\widehat{\fg}_{\kappa}\mod)$
is compactly generated by the objects:

\[
\preprime\sW_{\kappa}^n \in \Whit(\widehat{\fg}_{\kappa}\mod)^+
\]

\noindent which correspond under this equivalence to
$\sW_{\kappa}^n \in D^+(\sW_{\kappa}\mod^{\heart})$,
we obtain the theorem by definition of $\sW_{\kappa}\mod$.

\end{proof}

\section{Free-field realization of the generalized vacuum representations}\label{s:free-field}

\subsection{}

In this section, we give another construction of the modules
$\sW_{\kappa}^n$. Though the construction is interesting
in its own right, it also plays a technical role in 
the proof of Theorem \ref{t:ff}.

\subsection{}

Let $\kappa^{\prime} = -\kappa + \kappa_{crit}$.\footnote{The 
correction by $\kappa_{crit}$ plays an essentially 
negligible role in what follows;
see \cite{fg2} \S 10 and \cite{km-indcoh} for some explanations
why it is needed.} 
Let $\widehat{\ft}_{\kappa^{\prime}}$ denote the
Heisenberg extension:

\[
0 \to k \to \widehat{\ft}_{\kappa^{\prime}} \to \ft((t)) \to 0.
\]

\noindent (This is another name for the Kac-Moody extension
associated to $\kappa^{\prime}$ considered
as a symmetric bilinear form for $\ft$.) 
Let $\bV_{\ft,\kappa^{\prime}}$ denote the vacuum representation
$\ind_{\ft[[t]]}^{\widehat{\ft}_{\kappa^{\prime}}}(k)$, considered
as a vertex algebra.

Recall that there is an injective \emph{free-field} homomorphism:

\[
\vph: \sW_{\kappa} \to \bV_{\ft,\kappa^{\prime}} 
\]

\noindent that is a map of vertex algebras; its construction
is recalled in \S \ref{ss:free-field-constr}. In particular, this means
that any module over $\widehat{\ft}_{\kappa^{\prime}}$ can
be considered as a module over $\sW_{\kappa}$ by restriction.

\subsection{}

Now let us revisit the problem from \S \ref{ss:ds-intro}.
We have the generalized vacuum representations:

\[
\bV_{\ft,\kappa^{\prime}}^n \coloneqq 
\ind_{t^n \ft[[t]]}^{\widehat{\ft}_{\kappa^{\prime}}}(k)
\]

\noindent of the Heisenberg algebra. 
We can use these to construct cyclic modules over $\sW_{\kappa}$:
$\bV_{\ft,\kappa^{\prime}}^n$ is a $\sW_{\kappa}$-module by
restriction along $\vph$, and we can take the sub-$\sW_{\kappa}$-module
generated by the canonical vacuum vector in $\bV_{\ft,\kappa^{\prime}}^n$.

The main result of this section is:

\begin{thm}\label{t:free-field}

This construction produces $\sW_{\kappa}^n$ equipped with its
canonical vacuum vector.

\end{thm}

\begin{rem}

Note that because we are comparing cyclic modules with preferred generators,
this theorem uniquely characterizes the isomorphism it describes.

\end{rem}

The above result is quite useful in practice for computing the modules
$\sW_{\kappa}^n$ in cases where $\sW_{\kappa}$ has an explicit description.

\begin{example}\label{e:virasoro}

For $\fg = \sl_2$, this theorem and
the explicit formulae\footnote{
For the reader's convenience, if $L_n = -t^{n+1} \partial_t$
in the Virasoro algebra and $h_i \in \widehat{\ft}_{\kappa}^{\prime}$
is defined by the element $t^i \in k((t)) = \ft((t)) \subset \widehat{\ft}_{\kappa}^{\prime}$ (recalling that the Heisenberg algebra has a canonical vector
space splitting, i.e., it is defined by a 2-cocycle), these formulae say:

\[
\vph(L_n) = \sum_{i+j = n} :h_i h_j: -(n+1)\lambda h_n 
\]

\noindent for an appropriate scalar $\lambda$ depending on the level.
Here $:h_i h_j:$ is the normally-ordered product, so $h_i h_j$ if
$j \geq i$ and $h_j h_i$ otherwise.}
for $\vph$ from \cite{fbz} \S 15.4.14 together with the above result
allow to recover the explicit description of the modules 
$\sW_{\kappa}^n$ proposed in this case in \S \ref{ss:ds-intro}.

\end{example}

\begin{example}\label{e:wn-crit}

At the critical level $\kappa = \kappa_{crit}$, this theorem implies that
$\sW_{\kappa}^n$ is the structure sheaf of $\Op_{\ld{G}}^{\leq n}$
under Feigin-Frenkel. See Lemma \ref{l:wn-match}
for more on this.

\end{example}

\subsection{Proof sketch}\label{ss:free-field-sketch}

The proof of the theorem is based on a straightforward
generalization of the map $\vph$. Namely, we will construct
maps:

\[
\vph_n: \sW_{\kappa}^n \to \bV_{\ft,\kappa^{\prime}}^n
\]

\noindent of $\sW_{\kappa}^n$-modules.  
We will show that it is injective and preserves vacuum vectors
by constructing filtrations and computing this map at the associated
graded level (recall that this is how the vacuum vector in $\sW_{\kappa}^n$
was constructed).

Obviously this would suffice to prove the theorem: a 
generator-preserving
injective map between cyclic modules is an isomorphism.

\begin{rem}[Screening operators?]

Dennis Gaitsgory has suggested that $\vph_n$ might be the first
map in a resolution of $\sW_{\kappa}^n$, as in the $n = 0$ case
(at least in the irrational and critical level cases,
see \cite{ff-critical} and \cite{fg-wakimoto}). 
We record his idea here as a sort-of-conjecture. 

\end{rem}

\subsection{The Wakimoto vertex algebra}

The construction of $\vph$ passes through the theory of 
Wakimoto modules. Since we are trying to generalize this construction,
we must review these. We also refer the reader to
\cite{fg2} \S 10-11 and \cite{frenkel-wakimoto} for some other introductions.

We also use the theory of global sections of $D$-modules on the loop
group, but only in a minor way. The reader familiar with
\cite{dmod-loopgroup} (c.f. also \cite{fg2} \S 21 and \cite{km-indcoh})
will have more than enough information at hand for these constructions.

Let $I$ and $I^-$ denote the Iwahori subgroups defined by 
$B$ and $B^-$. We then form:

\[
\Gamma(G(K),j_{*,dR}(\omega_{I \cdot I^-})
\]

\noindent where this pushforward is as a $\kappa$-twisted $D$-module.
Note that because $I \cdot I^- \subset G(O)$ is open
(it is jets on the open cell $B B^-$),
the theory of $D$-modules on the loop group normalizes
this object to lie in cohomological degree $0$. 
There is a left action on this vector space by 
$\widehat{\fg}_{\kappa}$ (with the central element acting by the
identity), and a commuting right action by\footnote{We will
only need the action of $\widehat{\fb}_{-\kappa+2\kappa_{crit}}^-$,
the induced central extension of $\fb^-((t))$. Moreover, the action
of $\fn^-((t))+\fb^-[[t]]$, which is substantially easier to construct,
will play the main role.}
$\widehat{\fg}_{-\kappa+2\kappa_{crit}}$.

We define $\bW_{\kappa}$ as the semi-infinite cohomology: 

\[
C^{\sinf}\Big(\fn^-((t))+\fb^-[[t]],B(O),
\Gamma\big(G(K),j_{*,dR}\big(\omega_{I \cdot I^-})\big)\Big)
\]

\noindent formed with respect to the right action.
Because we have the commuting left action, 
$\bW_{\kappa} \in \widehat{\fg}_{\kappa}\mod$.

\begin{rem}

In a suitable sense, this is global sections of the 
semi-infinite flag variety $\sFl = G(K)/N^-(K)T(O)$ with coefficients
in the $D$-module on it induced by 
$j_{*,dR}(\omega_{I \cdot I^-}) \in D_{\kappa}(G(K))$.
See \cite{cpsii}, where some of these ideas are introduced.
(But I do not mean to suggest that this perspective 
is especially enlightening.)

\end{rem}

There is a canonical morphism:

\begin{equation}\label{eq:vac-waki}
\bV_{\kappa} \to \bW_{\kappa} 
\end{equation}

\noindent induced by the composition:

\begin{equation}\label{eq:v->w}
\begin{gathered}
\bV_{\kappa} =
\Gamma(G(K),\delta_{G(O)})^{G(O)} \to 
\Gamma(G(K),\delta_{G(O)})^{B^-(O)} \to \\
\Gamma(G(K),j_{*,dR}(\omega_{I \cdot I^-}))^{B^-(O)} \to
C^{\sinf}\Big(\fn^-((t))+\fb^-[[t]],B(O),\Gamma\big(G(K),j_{*,dR}(\omega_{I \cdot I^-})\big)\Big).
\end{gathered}
\end{equation}

\noindent Here $\delta_{G(O)}$ is the pushforward of
$\omega_{G(O)}$, considered as a $\kappa$-twisted $D$-module,
and invariants are for the right actions.

\begin{rem}

Factorization shows that $\bW_{\kappa}$ is a vertex algebra,
and $\bV_{\kappa} \to \bW_{\kappa}$ is a morphism of vertex algebras.

\end{rem}

\subsection{}

We now compute $\bW_{\kappa}$ more explicitly. 
Note that:

\[
\Gamma(G(K),j_{*,dR}(\omega_{I \cdot I^-}))\big) =
\Gamma(N(K),\delta_{N(O)}) \otimes \Gamma(B^-(K),\delta_{B^-(O)}).
\]

\noindent This is compatible with the left action of $\fn((t))$ and the
right action of $\widehat{\fb}_{-\kappa+2\kappa_{crit}}$. (Each of these 
global sections is usually called a \emph{CDO} for the respective groups.)

Note that:

\[
C^{\dot}(\fn^-((t))+\fb^-[[t]],\Gamma(B^-(K),\delta_{B^-(O)})) = 
\bV_{\ft,\kappa^{\prime}}.
\]

\noindent Indeed, the invariants $B(O)$ leave us with a vacuum
representation for a central extension of $\fb((t))$, and 
the rest of the semi-infinite cohomology reduces us to a central
extension of $\ft((t))$ (c.f. \cite{dmod-loopgroup} Theorem 5.5).

The calculation of the exact level
has to do with finer points about Tate extensions: we refer
to the sources above.\footnote{In \cite{fg2}, there
is a potentially frustrating typo in the beginning of \S 10.2 that
might thwart the reader who turns there: what is denoted
$\kappa^{\prime}$ there should be 
$-\kappa+2\kappa_{crit}$ (the sign is wrong there
in the second summand).}

So we find $\bW_{\kappa}$ is isomorphic to a tensor
product of the CDO for $N$ and a Heisenberg algebra;
in particular, it lies in cohomological degree zero.

\begin{rem}\label{r:waki-inj}

The morphism $\bV_{\kappa} \to \bW_{\kappa}$ is injective:
see \cite{frenkel-wakimoto} Theorem 5.1. The argument is proved
by constructing filtrations and passing to the associated graded;
we will essentially reconstruct it (and generalize it) in what follows.

\end{rem}

\subsection{Construction of the free-field homomorphism}\label{ss:free-field-constr}

Note that:

\[
\Psi(\Gamma(N(K),\delta_{N(O)})) = k
\]

\noindent by a similar calculation as above.
Therefore, $\Psi(\bW_{\kappa}) = \bV_{\ft,\kappa^{\prime}}$.

By functoriality, we obtain the morphism 
$\vph:\sW_{\kappa} \to \bV_{\ft,\kappa^{\prime}}$ from
\eqref{eq:vac-waki}. Factorization makes clear that
$\vph$ is a morphism of vertex algebras. 
As in Remark \ref{r:waki-inj}, one can show that it is injective
using filtrations; this argument will be generalized in what follows.

\subsection{Generalization to higher $n$}

We now wish to construct the maps $\vph_n$. 
We do this through a straightforward generalization of the above,
but adapted to the modules 
$\ind_{\Lie\o{I}_n}^{\widehat{\fg}_{\kappa}}(\psi)$ in place
of $\bV_{\kappa}$.

For $n>0$, we let $\delta_{\o{I}_n}^{\psi}$ denote the
$D$-module on $G(K)$ given by pushforward from the
character sheaf on $\o{I}_n$ defined by $\psi$; we normalize
it to lie in cohomological degree $0$. For $n = 0$, we let
$\delta_{\o{I}_n}^{\psi}$ be the $D$-module
considered before: the pushforward from $I\cdot I^-$.

Let $\o{I}_n^-$ denote $B^-(O) \cap \o{I}_n^-$.

Define $\bW_{\kappa}^n$ 
as the semi-infinite cohomology:

\[
C^{\sinf}(\fn^-((t))+\fb^-[[t]],\o{I}_n^-,
\Gamma(G(K),\delta_{\o{I}_n}^{\psi}))
\]

\noindent where the semi-infinite cohomology is again taken
with respect to the right action. Note that 
$\bW_{\kappa}^n \in \widehat{\fg}_{\kappa}\mod$ again.

Writing $\ind_{\Lie{\o{I}_n}}^{\widehat{\fg}_{\kappa}}(\psi)$
as $\Gamma(G(K),\delta_{\o{I}_n}^{\psi})^{\o{I}_n}$
(the invariants being for the right action), we obtain a map
$\ind_{\Lie{\o{I}_n}}^{\widehat{\fg}_{\kappa}}(\psi) \to \bW_{\kappa}^n$,
as in \eqref{eq:v->w}.

\subsection{}

Let us compute $\bW_{\kappa}^n$ more explicitly.

Let $\delta_{\o{I}_n^-} \in D_{\kappa}(B^-(K))$ denote the 
pushforward of $\omega_{\o{I}_n^-}$ with a cohomological
shift to put it in the heart of the $t$-structure.

We use the notation $\o{I}_n^+$ for 
$\Ad_{-n\check{\rho}(t)}N(O) = \o{I}_n \cap N(K)$.
Let $\delta_{\o{I}_n^+}^{\psi} \in D_{\kappa}(N(K))$ be
the pushforward of the character sheaf on $\o{I}_n^+$
defined by $\psi$, again normalized to be in cohomological degree $0$.

Then the triangular decomposition $\o{I}_n = \o{I}_n^+ \cdot \o{I}_n^-$
readily implies:

\[
\Gamma(G(K),\delta_{\o{I}_n}^{\psi}) = 
\Gamma(N(K),\delta_{\o{I}_n^+}^{\psi} ) \otimes 
\Gamma(B^-(K),\delta_{\o{I}_n^-})
\]

\noindent compatible with the left $\fn((t))$ and right
$\widehat{\fb}_{-\kappa+2\kappa_{crit}}^-$ actions.
Note that: 

\[
C^{\sinf}\big(\fn^-((t))+\fb^-[[t]],\o{I}_n^-,
\Gamma(B^-(K),\delta_{\o{I}_n^-})\big)
\]

\noindent is $\bV_{\ft,\kappa^{\prime}}^n$: this
follows because $\Lie\o{I}_n \cap \ft[[t]] = t^n\ft[[t]]$.

\subsection{Generalized free-field morphism}

Note that:

\[
\Psi_n(\Gamma(N(K),\delta_{\o{I}_n^+}^{\psi} )) = k
\]

\noindent for $\Psi_n$ the Drinfeld-Sokolov functor
defined relative to the lattice $\Lie(\o{I}_n^+) \subset \fn((t))$,
as in \eqref{eq:psi-n-defin}.

Combining this with the above, 
we obtain $\Psi_n(\bW_{\kappa}^n) = \bV_{\ft,\kappa^{\prime}}^n$.
Factorization geometry makes $\Psi_n(\bW_{\kappa}^n)$
a vertex module for $\Psi(\bW_{\kappa}) = \bV_{\ft,\kappa^{\prime}}^n$,
and this isomorphism is compatible.

\subsection{Generalized free-field morphisms}

Now by functoriality, we obtain the desired map:

\[
\vph_n: \sW_{\kappa}^n \coloneqq 
\Psi_n(\ind_{\Lie{\o{I}_n}}^{\widehat{\fg}_{\kappa}}(\psi)) \to 
\Psi_n(\bW_{\kappa}^n) = \bV_{\ft,\kappa^{\prime}}^n.
\]

\noindent This is obviously a morphism of $\sW_{\kappa}$-modules
by functoriality.

As in \S \ref{ss:free-field-sketch}, 
it remains to show the following.

\begin{lem}

The morphism $\vph_n$ is injective and preserves vacuum vectors.

\end{lem}

\begin{proof}

Note that $\bW_{\kappa}^n$ carries a canonical KK (i.e., Kazhdan-Kostant)
filtration such that the morphism from
$\ind_{\Lie{\o{I}_n}}^{\widehat{\fg}_{\kappa}}(\psi)$ 
is filtered. Indeed, both are derived from the KK filtration on
$\Gamma(G(K),\delta_{\o{I}_n}^{\psi})$.
The resulting filtration
on $\Psi(\bW_{\kappa}^n)$ is the PBW filtration on this
Heisenberg module: this follows because
$\ft((t))$ has zero $-\check{\rho}$-grading.

In particular, the KK filtrations on the source and target
are both bounded from below.

At the associated level, we obtain a morphism:

\[
\Fun(\sO_{f+t^{-n}\Ad_{n\check{\rho}(t)}\fb[[t]]}/\Ad_{-n\check{\rho}(t)}N(O)) \to 
\Fun(\sO_{t^{-n}\ft[[t]]}).
\]

\noindent It is routine
(and similar to the methods from \S \ref{s:ds}) to see that this
map is obtained by pullback from the 
\emph{Miura transform}:

\[
t^{-n}\ft[[t]] \xar{\xi \mapsto f+\xi} 
f+t^{-n}\Ad_{n\check{\rho}(t)}\fb[[t]]/\Ad_{-n\check{\rho}(t)}N(O).
\]

So obviously this map preserves vacuum vectors:
they correspond to the constant function with value $1$.
Since the filtrations on $\sW_{\kappa}^n$ and
$\bV_{\ft,\kappa^{\prime}}^n$ begin in degree $0$ 
with 1-dimensional $\gr_0$, this implies that $\vph_n$ preserves
vacuum vectors as well.

Then it remains to show that the Miura transform is dominant.
Applying the isomorphism $t^n\Ad_{n\check{\rho}(t)}$,
we are reduced to the $n = 0$ case, where it is well-known:
over the open in $f+\fb[[t]]/N(O)$ corresponding to regular
semisimple elements, the map is finite \'etale.

\end{proof}

\section{Applications}\label{s:applications}

\subsection{}

In this section, we give some applications of the above results.
First, we discuss how Theorem \ref{t:aff-skry} 
provides a systematic framework for understanding exactness
properties of the Drinfeld-Sokolov functor $\Psi$.
Then we give a categorical form of the Feigin-Frenkel theorem, which is
Theorem \ref{t:ff}. 

\subsection{Exactness results for $\Psi$}\label{ss:ds-exactness}

Our main general result is the following.

\begin{thm}\label{t:exactness}

For every $n \geq 0$, the functor:

\[
\Psi[-n\Delta]: \widehat{\fg}_{\kappa}\mod^{\o{I}_n,\psi} \to \Vect
\]

\noindent is $t$-exact. 
Moreover, for every $M \in \widehat{\fg}_{\kappa}\mod^{\o{I}_n,\psi,\heart}$,
the canonical morphism:

\[
H^0((M \otimes \ell^n)^{\o{I}_n,\psi}) \to H^{-n\Delta} \Psi(M)
\]

\noindent is injective; here $\ell^n$ is a determinant line as before,
the superscript $\o{I}_n,\psi$ indicates invariants, and we included
$H^0$ to emphasize that these are non-derived invariants.

\end{thm}

\begin{proof}

The exactness follows Theorem \ref{t:psi-exact/cons}
because $\Psi^{\Whit} \iota_{n,!} = \Psi[-2n\Delta]$, and 
$\iota_{n,!}[n\Delta]$ is $t$-exact. 
The injectivity follows immediately from Lemma \ref{l:alpha-epi}.

\end{proof}

In the particular case $n = 0$, we obtain:

\begin{cor}

The functor $\Psi:\widehat{\fg}_{\kappa}\mod^{G(O)} \to \Vect$
is $t$-exact. 

\end{cor}

\begin{rem}

At critical level, this is a part of \cite{fg-spherical}
Theorem 3.2.
At non-critical level (say for $\fg$ simple), this 
follows from Arakawa exactness, c.f., below. 
(The relevant deduction is the
proof of Proposition 2 in \cite{fg-weyl}, although this reference
is ostensibly at critical level.)

\end{rem}

\subsection{Arakawa exactness}\label{ss:arakawa} 

We now show how the $n = 1$ case of the above recovers 
\emph{Arakawa exactness}.

Let $I^-$ be the negative Iwahori group $G(O) \times_G B^-$,
and let $\o{I}^-$ denote its prounipotent radical
$G(O) \times_G N^-$.

\begin{cor}[Arakawa exactness]\label{c:arakawa}

The functor:

\[
\Psi[-\Delta + \dim(N)]:
\widehat{\fg}_{\kappa}\mod^{\Ad_{-\check{\rho}(t)} \o{I}^-} \to \Vect
\]

\noindent is $t$-exact.

\end{cor}

\begin{rem}

This result generalizes
\cite{arakawa-rep-thry} Main Theorem 1 (1).
First, \emph{loc. cit}. actually uses $\Ad_{-\check{\rho}(t)} I^-$
instead of its prounipotent radical. Moreover, it assumes
that there is $\bZ$-grading given on our modules 
compatible with the $L_0$-grading on the Kac-Moody algebra
\textendash{} the above result removes this restriction
(which is only substantial at critical level).

\end{rem}

\begin{proof}[Proof of Corollary \ref{c:arakawa}]

Recall the main theorem of \cite{bbm}:
for any $\sC$ acted on by $G$, the functor 
$\Av_!^{N,\psi}$ is defined on $\sC^{N^-}$, and
$\Av_!^{N,\psi} = \Av_*^{N,\psi}[2\dim N]$.
(This is not how the authors formulate the result, but
the proof goes through using the methods from the proof
of Theorem \ref{t:!-avg}, which was modeled on \cite{bbm}.)

We can write $\widehat{\fg}_{\kappa}\mod^{\Ad_{-\check{\rho}(t)} \o{I}^-}$
in two steps, by first taking invariants with respect to the conjugated
first congruence subgroup $\Ad_{-\check{\rho}(t)} \cK_1$, and then invariants with respect to 
$N^- =  \Ad_{-\check{\rho}(t)} \o{I}^-/\Ad_{-\check{\rho}(t)} \cK_1$.

Therefore, as in Lemma \ref{l:amp}
by Lemmas \ref{l:av-*-bd} and \ref{l:!-avg-bd}, 
this implies that the functor:

\[
\Av_*^{N,\psi}[\dim N]: 
\widehat{\fg}_{\kappa}\mod^{\Ad_{-\check{\rho}(t)} \o{I}^-}
\to \widehat{\fg}_{\kappa}\mod^{\o{I}_1,\psi}
\]

\noindent is $t$-exact. 

Clearly $\Psi(M) = \Psi(\Av_*^{N,\psi}(M))$ for such 
$M \in \widehat{\fg}_{\kappa}\mod^{\Ad_{-\check{\rho}(t)} \o{I}^-}$.
Since $\Psi[-\Delta]$ is $t$-exact on
$\widehat{\fg}_{\kappa}\mod^{\o{I}_1,\psi}$, we obtain the claim.

\end{proof}

\subsection{Feigin-Frenkel redux}

We now show a (long\footnote{Certainly it was
written in \cite{quantum-langlands-summary} from 2007, though
it must have been anticipated earlier still.} anticipated) version of
Feigin-Frenkel duality.

\subsection{}

Let $\ld{G}$ denote the Langlands dual group to $G$,
defined by some choice of Borel $B$, maximal torus $T$, and
Chevalley generators $e_i \in \fn$. We obtain Langlands dual
data for $\ld{G}$, which we denote similarly.

Note that $\ld{\ft} = \ft^{\vee}$, since $T$ and $\ld{T}$ are
dual tori; this identification is compatible with the natural
actions of the Weyl group $W$. 
Recall that a level $\kappa$ is the same
as a $W$-invariant bilinear form on $\ft$. Therefore, if
$\kappa$ is \emph{non-degenerate}, then $\frac{1}{\kappa}$ makes sense as a level for $\ld{\fg}$: $\kappa$ is a
$W$-invariant isomorphism $\ft \simeq \ft^{\vee}$,
and $\frac{1}{\kappa}$ is its inverse.

Recall that for $\kappa$ a fixed non-degenerate
bilinear form, we can take limits of ``many" constructions
as $\kappa \to \infty$, i.e.,
the definitions of the topological enveloping algebra,
and its category of representations, and the
$\sW_{\kappa}$-algebra, etc., extend naturally
over $\bP_{\kappa}^1$.

\subsection{}\label{ss:ff-duality-intro}

For $\kappa$ as above, we let $\ld{\kappa}$ denote the level
of $\ld{\fg}$ given by:

\[
\ld{\kappa} \coloneqq \frac{1}{\kappa-\kappa_{crit}}+\kappa_{crit}
\]

\noindent where
in the denominator we are using the critical level for $\fg$ and in the
second term it is the critical level for $\ld{\fg}$. Note
that the map $\kappa \mapsto \ld{\kappa}$ is involutive
and sends $\kappa_{crit}$ to $\infty$.

The Feigin-Frenkel duality theorem from 
\cite{feigin-frenkel-duality}
says:

\[
\sW_{\fg,\kappa} \simeq \sW_{\ld{\fg},\ld{\kappa}}
\]

\noindent as vertex algebras.
Here e.g. $\sW_{\fg,\kappa}$ is what we
were denoting $\sW_{\kappa}$ before, and
the right hand side is the $\sW$-algebra on the Langlands
dual side with respect to the Kac-Moody extension
defined by the dual level. This duality theorem
makes sense and is defined in the limit $\kappa \to \infty$.

\begin{warning}\label{w:ff-ratl}

In truth, I don't know where in the literature to find this statement:
Feigin and Frenkel only prove it for irrational levels and critical/level $\infty$.
It seems to be folklore that it is true in this full generality: see e.g.
\cite{arakawa-survey} Remark 5.24. We assume duality holds at all 
levels in what follows.

\end{warning}

\subsection{}

We now have the following form of Feigin-Frenkel.

\begin{thm}[Categorical Feigin-Frenkel duality]\label{t:ff}

There is a canonical equivalence of categories:

\[
\Whit(\widehat{\fg}_{\kappa}\mod) \simeq 
\Whit(\widehat{\ld{\fg}}_{\ld{\kappa}}\mod).
\]

\end{thm}

Here $\Whit$ on one side is with respect to $N(K)$, and on the
other side is with respect to $\ld{N}(K)$.

To prove this, observe that we have an equivalence of
abelian categories $\sW_{\fg,\kappa}\mod^{\heart} \simeq 
\sW_{\ld{\fg},\ld{\kappa}}\mod^{\heart}$ coming
from usual Feigin-Frenkel. This fact combined with affine Skryabin
(Theorem \ref{t:aff-skry}) gives Theorem \ref{t:ff} on bounded
below derived categories. To deduce all of Theorem \ref{t:ff},
we need to match the compact generators under this equivalence.
This follows from:

\begin{lem}\label{l:wn-match}

Under Feigin-Frenkel duality, there is a unique isomorphism:

\[
\sW_{\fg,\kappa}^n \simeq \sW_{\ld{\fg}, \ld{\kappa}}^n
\]

\noindent preserving vacuum vectors and compatible with
Feigin-Frenkel duality.

\end{lem}

\begin{proof}

The basic property satisfied by Feigin-Frenkel is that the diagram:

\[
\xymatrix{
\sW_{\fg,\kappa} \ar[d]^{\vph} \ar[r]^{\simeq} &
\sW_{\ld{\fg}, \ld{\kappa}} \ar[d]^{\vph} \\
\bV_{\ft,\kappa-\kappa_{crit}} \ar[r]^{\simeq} & 
\bV_{\ld{\ft}, \frac{1}{\kappa-\kappa_{crit}}}
}
\]

\noindent commutes, where the bottom arrow is the
obvious isomorphism.

Therefore, this compatibility follows from Theorem \ref{t:free-field}.

\end{proof}

\begin{rem}

We are constantly neglecting twists involving forms on the
disc by having chosen our $dt$. Of course, it is better
to incorporate these twists systematically as in 
\cite{cpsii} \S 2. They come out in the wash: Whittaker
is properly defined with twists incorporated, as is
this category of Kac-Moody representations, as is
Feigin-Frenkel duality.

\end{rem}

\subsection{Critical level}

For $\kappa = \kappa_{crit}$, the above gives:

\begin{cor}\label{c:ff-crit}

There is an equivalence:\footnote{Since $\Op_{\ld{G}}(\o{\cD})$
is an ind-pro affine space, $\IndCoh$ defined in any sense
coincides with $\QCoh$.}

\[
\Whit(\widehat{\fg}_{crit}\mod) \simeq 
\QCoh(\Op_{\ld{G}}(\o{\cD}))
\]

\noindent for $\Op_{\ld{G}}(\o{\cD}) = (f+\fb((t)))dt/N(K)$ the
indscheme of opers on the punctured disc.

For $n>0$, the subcategory
$\Whit^{\leq n}(\widehat{\fg}_{crit}\mod) \subset \Whit(\widehat{\fg}_{crit}\mod)$
is the subcategory of quasi-coherent sheaves set-theoretically supported
on $\Op_{\ld{G}}^{\leq n}$, i.e., the subscheme of opers with
singularities of order $\leq n$.

\end{cor}

The only thing still to explain is the calculation of $\Whit^{\leq n}$.
This follows because we know 
$\Whit^{\leq n}(\widehat{\fg}_{crit}\mod)$ is compactly generated
by $\iota_{n,!}(\ind_{\Lie(\o{I}_n)}^{\widehat{\fg}_{crit}}(\psi))$,
which goes under this equivalence to
$\Psi^{\Whit}(\ind_{\Lie(\o{I}_n)}^{\widehat{\fg}_{crit}}(\psi))$.
Up to shift and determinant twist, this is
the structure sheaf of $\Op_{\ld{G}}^{\leq n}$ by 
Lemma \ref{l:wn-match}.

\appendix

\section{Filtrations, Harish-Chandra modules, and
semi-infinite cohomology}\label{a:hc}

\subsection{}

In this appendix, we give a slightly non-standard construction
of the (quantum) Drinfeld-Sokolov reduction
$\Psi:\widehat{\fg}_{\kappa}\mod \to \Vect$, and
discuss its compatibility with various filtrations. This material
supports the calculations of \S \ref{s:ds}.

\subsection{}\label{ss:sinf-intro}

The treatment we give here is quite lengthy, but this does not reflect
the seriousness of the contents. There is a small number
of ideas, and we summarize them here as themes to keep
an eye towards:

\begin{itemize}

\item Recall that the Drinfeld-Sokolov functor is defined
as $M \mapsto C^{\sinf}(\fn((t)),\fn[[t]],M \otimes (-\psi))$, where
$C^{\sinf}$ is the \emph{semi-infinite cohomology} functor. 
This functor should be thought of as ``cohomology along
$\fn[[t]]$ and homology along $\fn((t))/\fn[[t]]$," although this
does not quite make sense. Here we 
recall that cohomology is well-behaved\footnote{E.g., continuous,
at least for an appropriate definition of the source category.}
for pro-finite dimensional Lie algebras, while homology
is well-behaved for ``discrete" Lie algebras (i.e., non-topologized ones). 

We give a slightly non-standard treatment of this 
semi-infinite cohomology functor, avoiding irrelevant
Clifford algebras. Rather, we define:

\[
C^{\sinf}(\fn((t)),\fn[[t]],-) \coloneqq 
\underset{n \geq 0}{\colim} \, 
C^{\dot}(\Ad_{-n\check{\rho}(t)} \fn[[t]],-)
[\dim (\Ad_{-n\check{\rho}(t)} \fn[[t]]/\fn[[t]])].
\]

\noindent Here we recall that for finite-dimensional Lie algebras,
Lie algebra cohomology and homology differ by a 
determinant twist and cohomological shift, so satisfy 
both covariant and contravariant functoriality (up to these twists and 
shifts) with respect to the Lie algebra. 
There is a relic of this for Lie algebra cohomology for profinite-dimensional
Lie algebras, giving the structure morphisms in the above colimit.
(So in fact, we should have included a twist by the line
$\det(\Ad_{-n\check{\rho}(t)} \fn[[t]]/\fn[[t]])$ to make the above
structure maps canonical.)

So much attention is paid to the construction of such morphisms.

\item The other major theme is \emph{filtrations}. 

The first question is given a (PBW) filtered module over
a Lie algebra, how is its co/homology filtered? What is its associated
graded? 

But there is a more subtle point in working with (finite or affine)
$\sW$-algebras: one would like the associated graded
of a 1-dimensional module corresponding to a character
$\psi$ of $\fn$ (or similarly for $\fn((t))$) to be the skyscraper sheaf
at $\psi \in \fn^{\vee}$ in 
$\QCoh(\fn^{\vee}) = \Sym(\fn)\mod = \gr_{\dot} U(\fn) \mod$.
However, this is impossible: it is a graded module, so must be
at the origin in $\fn^{\vee}$. 

The (standard) solution to this problem is
the \emph{Kazhdan-Kostant} method, which 
\emph{twists} the filtration using $\check{\rho}:\bG_m \to T$.
This is the \emph{only} canonical way of obtaining (non-derived) 
filtrations on $\sW$-algebras.

Finally, we emphasize the relationship between the PBW and
Kazhdan-Kostant filtrations using \emph{bifiltrations}; this
material is used in \S \ref{s:ds} to settle a subtle homological
algebra point regarding some Kazhdan-Kostant filtrations.

\end{itemize}

Each of these amount to completely elementary constructions
and statements about Lie algebra co/homology, and its relation
to Harish-Chandra conditions. 

The reason this section is so long is rather out of a commitment
to develop the theory with an emphasis on the categorical
aspects. A primary reason for this is because, following
\cite{dmod-aff-flag}, the derived category of 
$\fg[[t]]$ or $\fg((t))$-modules is subtle, and is better to
understand through categories than (topological) algebras.
One thing this requires, however, is some general formalism
for working with filtrations on \emph{categories} rather than
on algebras and their modules.

This section also renders the theory in
the $\IndCoh$ formalism of \cite{grbook}. It has
the advantage that it provides a robust
framework for Lie algebra cohomology that does not rely
on explicit formulae. But this accounts for some portion of the
length: we have explained some elementary points in detail,
in the hopes that this is instructive for understanding the formalism.
We also hope that an $\IndCoh$ treatment gives the feeling why
things are the way they are, and that they
could never have been another way. 

Finally, we advise the reader to look to \S \ref{ss:fil-summary}, where
we give another summary of what is actually needed from
this section, which should also help identify what is most
important here.

\subsection{Filtrations}\label{ss:fil-intro}

We begin with some abstract language about filtrations.

\begin{defin}

Let $\sC \in \DGCat_{cont}$ be given. A \emph{filtration}
on $\sC$ is a datum of 
$\widetilde{\sC} \in \QCoh(\bA_{\hbar}^1/\bG_m)\mod$ plus
an isomorphism: 

\[
\sC \simeq
\Fil \sC 
\underset{\QCoh(\bA_{\hbar}^1/\bG_m)}{\otimes} 
\QCoh((\bA_{\hbar}^1\setminus 0)/\bG_m) = 
\underset{\QCoh(\bA_{\hbar}^1/\bG_m)}{\otimes} 
\Vect.
\]

\noindent Here $\bA_{\hbar}^1$ is $\bA^1$ with coordinate
$\hbar$, and $\bG_m$ acts by inverse\footnote{This
is so that 
$\hbar \in \Gamma(\bA_{\hbar}^1,\sO_{\bA_{\hbar}^1})$
has weight $1$ with respect to the $\bG_m$-action: we remind
that when a group acts on a scheme, there is an inverse sign 
in the formula for the induced action on the algebra
of functions.}
homotheties.

An object of $\Fil \sC$ is called a \emph{filtered}
object of $\sC$. Note that there is a canonical restriction
functor $\Fil \sC \to \sC$.
For $\sF \in \sC$, we refer to an
extension $\widetilde{\sF}$ of $\sF$ to 
$\Fil \sC $ as a \emph{filtration} on $\sF$.

We say a morphism $F:\sC \to \sD$ between filtered categories
is \emph{filtered} if we are given the data of
$\Fil F: \Fil \sC \to \Fil \sD$ a $\QCoh(\bA_{\hbar}^1/\bG_m)$-linear
functor.

\end{defin}

In the language of \cite{shvcat}, we would say
$\Fil \sC$ is a \emph{sheaf of categories}
over $\bA_{\hbar}^1/\bG_m$ with fiber $\sC$ at the
open point of this stack.

\begin{notation}

For a filtered category\footnote{We use this terminology, even
though our categories are always of a special type: DG and cocomplete.} 
 $\sC$ as above, we let $\sC^{cl}$ 
denote the fiber of $\Fil \sC$ at $0$; it is called the associated
\emph{semi-classical category}. 
Note that $\bG_m$ acts weakly on $\sC^{cl}$.
For $\sF \in \sC$ filtered, we let $\gr_{\dot} \sF$ denote
the induced object of $\sC^{cl}$, obtained by
taking the fiber at $\hbar = 0$. Note that
$\gr_{\dot}(\sF)$ comes from an object of
$\sC^{cl,\bG_m,w}$, which we also denote by
$\gr_{\dot}(\sF)$. For a filtered functor $F: \sC \to \sD$,
we obtain a functor $F^{cl}:\sC^{cl} \to \sD^{cl}$, which we refer
to as the corresponding \emph{semi-classical} functor.

\end{notation}

\begin{example}\label{e:filt-vect}

$\Vect$ is canonically filtered, with 
$\Fil \Vect \coloneqq \QCoh(\bA_{\hbar}^1/\bG_m)$.\footnote{So
we can rewrite the original definition more evocatively by 
saying that a filtered category is a $\Fil \Vect$-module category
(in $\DGCat_{cont}$.}
In this case, a filtered object is a $\bG_m$-representation with
a degree $1$ endomorphism. A $\bG_m$-representation
is the same as a $\bZ$-graded vector space, and we suggestively
denote the $n$th term of this graded vector space
as $F_n(V)$. Then our degree $1$ endomorphism
is a sequence of maps $F_n(V) \to F_{n+1}(V)$.

The induced object of $\Vect$ is obtained by inverting
this degree $1$ endomorphism $\hbar$ and taking the
degree $0$ component: this is computed as the
colimit $V \coloneqq \colim_n F_n(V)$. (So in this
formalism, filtrations are by definition exhaustive.) 
We compute $\gr V$ by taking the cokernel of
$\hbar$ acting on $\oplus_n F_nV$, which is
$\oplus_n \Coker(F_{n-1} V \to F_n V)$. In this example,
we let $\gr_n V$ denote the $n$th summand.

\end{example}

\begin{example}\label{e:cnst-fil}

The above example generalizes to general $\sC$ in place
of $\Vect$: it has a canonical filtration defined by the
category $\sC \otimes \QCoh(\bA_{\hbar}^1/\bG_m)$,
and the above calculations render as is. We refer to this
as the \emph{constant} filtration on $\sC$.

\end{example}

\begin{example}\label{e:filt-ass}

For $A$ a filtered associative DG algebra, the Rees construction
provides an algebra object $A_{\hbar} \in \QCoh(\bA_{\hbar}^1/\bG_m)$
with generic fiber $A$. Then the category 
$\Fil A\mod \coloneqq A_{\hbar}\mod(\QCoh(\bA_{\hbar}^1/\bG_m))$ 
provides a filtration on $A\mod$. Note that $A\mod^{cl}$ is the DG category
of $\gr_{\dot} A$-modules.
The induced functor $A\mod \to \Vect$ is canonically filtered.

\end{example}

\begin{example}\label{e:fil-gm}

Suppose $\sC$ carries a weak action of $\bG_m$. 
Then we define:

\[
\Fil \sC \coloneqq 
\sC^{\bG_m,w} \underset{\Rep(\bG_m)}{\otimes} 
\QCoh(\bA_{\hbar}^1/\bG_m).
\]

\noindent Here $\Rep(\bG_m)$ acts on each of these categories
as on any weak $\bG_m$-invariants. This tensor
product is considered as acted on by
$\QCoh(\bA_{\hbar}^1/\bG_m)$ in the obvious way.
Note that this really does define a filtration on $\sC$, and
that there is a canonical functor $\sC^{\bG_m,w} \to \Fil \sC$,
given by exterior product with the structure sheaf
of $\bA_{\hbar}^1/\bG_m$.

Concretely, suppose that $\sC = A\mod$ for 
$A = \oplus_{i\in \bZ} A^i$ a $\bZ$-graded
algebra. Note that $A$ inherits a filtration from its grading:
set $F_i A = \oplus_{j \leq i} A^j$.\footnote{Geometrically,
this construction amounts to pullback along the map
$\bA_{\hbar}^1/\bG_m \to \bB \bG_m$.}
Then $\Fil A\mod$ is the DG category of filtered
modules over this filtered algebra. In this case, functor
$A\mod^{\bG_m,w} \to \Fil A\mod$ takes a graded module
and creates a filtered one using the same construction as above.

\end{example}

\subsection{}

We will use the following terminology in what follows:
a filtered vector space $F_{\dot} V$ is \emph{bounded from below}
if for all $i \ll 0$, $F_i V = 0$. Note that the
functor $\gr_{\dot}:\Fil\Vect \to \Vect$ is conservative when
restricted to the subcategory of filtered vector spaces
with bounded below filtrations.\footnote{Although boundedness
from below does not make sense for an arbitrary filtered category,
\emph{(derived) completeness} of a filtration does. However,
we will not use this notion in any significant level of generality,
remarking only that it is not so straightforward to verify in many of 
our examples.}

\subsection{Finite-dimensional setting}

Our formalism for semi-infinite cohomology
will be built from the finite-dimensional
setting, so we spend a while discussing this case. 

We fix $\fh \in \Vect^{\heart}$ a finite-dimensional Lie algebra.

\subsection{}\label{ss:pbw}

The\footnote{Personally, I find the discussion that follows to be not 
so interesting, but its Harish-Chandra generalization (which we will
discuss next) to at least be somewhat clarifying.}
PBW filtration on $U(\fh)$ defines a filtration on 
$\fh\mod$ with $\fh\mod^{cl} = \QCoh(\fh^{\vee})$.

Here is a more geometric perspective.
Let $\exp(\fh)$ denote the formal group associated with $\fh$.
Recall that 
$\fh\mod = \IndCoh(\bB \exp(\fh))$ so that the forgetful functor
corresponds to $!$-pullback to
a point.

We have an induced family $\fh_{\hbar}$ 
of Lie algebras over $\bA_{\hbar}^1$ given by
$\fh \otimes \sO_{\bA_{\hbar}^1}$ with the bracket given by
$\hbar$ times the bracket coming from $\fh$. This family
is obviously $\bG_m$-equivariant, so defines
$\Fil \fh$ a Lie algebra over $\bA_{\hbar}^1/\bG_m$. 
The generic fiber of $\Fil\fh$ is $\fh$, 
and special fiber is the vector space $\fh$ \emph{equipped with the
abelian Lie algebra structure}. We remark that the
latter abelian Lie algebra has underlying formal group 
$\fh_0^{\wedge}$, that is, the formal completion at $0$
\emph{of the vector space $\fh$}; we use this notation at various
points to emphasize the abelian nature.

Therefore, we obtain
$\Fil \fh\mod \coloneqq \IndCoh(\bB \exp(\Fil\fh))$, which is a filtration
on $\fh\mod$. Note that the ``special fiber" 
$\fh\mod^{cl}$ is 
$\IndCoh(\bB \fh_0^{\wedge}) = \Sym(\fh)\mod = \QCoh(\fh^{\vee})$.

\begin{warning}

For obvious reasons, we prefer to think of $\IndCoh(\bB \fh_0^{\wedge})$
as $\QCoh(\fh^{\vee})$, remembering that this category
actually arises from the classifying space of a commutative formal group.
Part of remembering this means internalizing that 
a duality occurs in this transition, so
various pullbacks and pushforwards get swapped.
For example, the fiber functor $\Sym(\fh)\mod \to \Vect$
(which corresponds to $!$-pullback $\Spec(k) \to \bB \fh_0^{\wedge}$)
goes to global sections $\QCoh(\fh^{\vee})\to \Vect$.

\end{warning}

\begin{example}

$U(\fh)$ is filtered, and its associated graded is
the structure sheaf of $\fh^{\vee}$.

\end{example}

\begin{warning}\label{w:fil-abelian}

The construction $\fh \mapsto \Fil \fh$ has some subtleties.
We can rewrite its output: $\Fil \fh \in \LieAlg(\QCoh(\bA_{\hbar}^1/\bG_m))$
is the same as a filtration on $\fh$ plus a Lie bracket on that 
filtered vector space extending the one on $\fh$. In our case,
the filtration is $F_i \fh = 0$ for $i \leq 0$
and $F_i \fh = \fh$ for $i>0$, so that the bracket
$[-,-]:F_i \fh \otimes F_j \fh \to F_{i+j} \fh$ happens to factor
through $F_{i+j-1}\fh$ (so that $\gr_{\dot} \fh$ is abelian,
as we expect). In particular, $\fh = \gr_1 \fh$.

This is slightly confusing for $\fh$ being abelian: 
since the generic and special fibers are the same,
it is easy to wrongly confuse this filtration with the constant one.

\end{warning}

\subsection{}

Now recall that the functor:

\[
C_{\dot}(\fh,-):\fh\mod \to \Vect
\]

\noindent corresponds to the $\IndCoh$-pushforward functor
for $\bB \exp(\fh)$. Therefore, we see that this functor is naturally
filtered, i.e., it extends to give $\Fil\fh\mod \to \Fil\Vect$.
Note that the corresponding semi-classical functor
$\QCoh(\fh^{\vee}) \to \Vect$ is $*$-restriction to $0$.

\begin{rem}\label{r:filt-chev-hom}

Say $F_{\dot} M \in \Fil\fh\mod$, and for simplicity 
let us assume $M$ is in the heart
of the $t$-structure and that $F_iM \into F_{i+1}M$. 
Recall that $C_{\dot}(\fh,M)$
is computed by the complex:

\[
\ldots \to \Lambda^2 \fh \otimes M \to \fh \otimes M \to M \to 0 \to 0 \to \ldots
\]

\noindent with appropriate differentials. 
Then the resulting filtration on $C_{\dot}(\fh,M)$ has $F_i$-term:

\[
\ldots \to \Lambda^2 \fh \otimes F_{i-2} M \to \fh \otimes F_{i-1} M \to F_i M \to 0 \to 0 \to \ldots
\]

\end{rem}

\subsection{}

Moreover, recall that $\omega_{\bB \exp(\fh)}$ is \emph{compact},
as is its extension $\omega_{\bB \exp(\Fil\fh)}$.\footnote{This uses
our assumptions on $\fh$ in a serious way. We recall that
the standard filtration on $\omega$ of any formal prestack
(``with deformation theory" in the \cite{grbook} sense) gives
a bounded free resolution in this case. This corresponds to the
standard resolution of the trivial module for $U(\fh)$ (or its
Rees algebra).}

Note that $\omega_{\bB \exp(\fh)}$ corresponds to the trivial
representation in $\fh\mod$. Therefore, the functor:

\[
C^{\dot}(\fh,-):\fh\mod \to \Vect
\]

\noindent of Lie algebra cohomology has a canonical filtered
structure with graded $\QCoh(\fh^{\vee}) \to \Vect$ given
by the $(\QCoh,!)$-pullback to $0 \in \fh^{\vee}$, i.e., the (continuous) 
right adjoint to the pushforward functor from $0$.

\begin{rem}\label{r:filt-chev-coh}

We retain the notation of Remark \ref{r:filt-chev-hom}.
Recall that $C^{\dot}(\fh,M)$
is computed by the complex:

\[
\ldots \to 0 \to M \to \fh^{\vee} \otimes M \to \Lambda^2 \fh^{\vee} \otimes M 
\to \ldots
\]

\noindent with appropriate differentials. 
Then the resulting filtration on $C^{\dot}(\fh,M)$ has $F_i$-term:

\[
\ldots \to 0 \to F_i M \to \fh^{\vee} \otimes F_{i+1} M \to 
\Lambda^2 \fh^{\vee} \otimes F_{i+2} M 
\to \ldots
\]

Note that (unlike the case of Lie algebra homology) even if
the filtration on $M$ has $F_{-1}M = 0$, we may not
have $F_{-1}C^{\dot}(\fh,M) = 0$ (though
$F_{-\dim \fh-1}C^{\dot}(\fh,M)$ will be zero).

\end{rem}

\subsection{}

We now recall the following fact.

\begin{lem}\label{l:lie-shift}

There is a canonical isomorphism of functors:

\[
C^{\dot}(\fh,(-)\otimes \det(\fh)[\dim\fh]) \simeq C_{\dot}(\fh,-).
\]

\noindent 

Moreover, suppose that we regard $\det(\fh)[\dim\fh]$ as a filtered
$\fh$-module with $F_{\dim\fh-i-1} \det(\fh)[\dim\fh] = 0$
and $F_{\dim\fh+i}\det(\fh)[\dim\fh] = \det(\fh)[\dim\fh]$ for all
$i \geq 0$. Then this isomorphism extends to an isomorphism
of filtered functors, g inducing
the usual isomorphism-up-to-twist-and-shift
between $*$ and $!$-restriction to the point $0 \in \fh^{\vee}$.

\end{lem}

\begin{proof}

Recall that for any $M \in \fh\mod$ (possibly non-filtered),
$C^{\dot}(\fh,M)$ has a canonical filtration
$F_{\dot}^{Chev} C^{\dot}(\fh,M)$ (indexed by non-positive integers
but bounded below) with 
$\gr_{-i}^{Chev} C^{\dot}(\fh,M) = M \otimes \Lambda^i\fh^{\vee}[-i]$.
Similarly, the functor of $!$-restriction along $0 \into \fh^{\vee}$
has a filtration with the same associated graded. One immediately
sees that these ``glue": $C^{\dot}(\fh,-)$ has a canonical filtration
\emph{considered as a filtered functor}.

In particular, we obtain a natural transformation of filtered functors
$\fh\mod \to \Vect$:

\[
(-) \otimes \det(\fh)^{\vee}[-\dim \fh] \to C^{\dot}(\fh,-)
\]

\noindent where on the left hand side 
$\det(\fh)^{\vee}[-\dim \fh]$ is equipped with the filtration
with one jump in degree $-\dim \fh$.
Evaluating this on $U(\fh)$, we claim that the composition:

\[
\det(\fh)^{\vee}[-\dim \fh] \to 
U(\fh) \otimes \det(\fh)^{\vee}[-\dim \fh] \to 
C^{\dot}(\fh,U(\fh))
\]

\noindent is an isomorphism of filtered complexes. Indeed,
it suffices to check this at the associated graded level, where
the claim is standard.

Then observe that using the (filtered) bimodule structure
on $U(\fh)$, $C^{\dot}(\fh,U(\fh))$ can actually be considered
as a filtered $\fh$-module. We claim that the above computation
is true with this extra structure, considering $\det(\fh)^{\vee}$ as an
$\fh$-module in the obvious way.

Indeed, considering
$U(\fh)$ as a filtered $\fh$-module via the \emph{right} action,
the morphism:

\[
U(\fh) \otimes \det(\fh)^{\vee}[-\dim \fh] \to C^{\dot}(\fh,U(\fh))
\] 

\noindent is a tautologically a 
morphism of filtered $\fh$-modules. Moreover, by
a standard argument, this is also true if we consider
$U(\fh)$ as a filtered $\fh$-module via the \emph{adjoint} action.
This shows the claim.

Now observe that by usual Morita theory, 
the datum of $C^{\dot}(\fh,U(\fh))$ as a filtered $\fh$-module completely
is equivalent to the datum of the filtered functor $C^{\dot}$.
So the result follows from the observation that:

\[
C^{\dot}(\fh,U(\fh) \otimes \det(\fh)[\dim \fh]) = k =
C_{\dot}(\fh,U(\fh))
\] 

\noindent as filtered $\fh$-modules.

\end{proof}

\begin{rem}

From the perspective of usual complexes:
for $M$ a complex of $\fh$-modules, Lie algebra homology
is $M \otimes \Lambda^{\dot} \fh$ with the appropriate 
grading and differential, while cohomology is
$M \otimes \Lambda^{\dot} \fh^{\vee}$. Noting that
$\Lambda^i \fh = \det(\fh) \otimes \Lambda^{\dim \fh - i} \fh^{\vee}$
and matching up the differentials and gradings then gives the claim.

\end{rem}

\begin{rem}

This isomorphism amounts to the calculation of
$C^{\dot}(U(\fh))$ as a filtered complex acted on by $\fh$
(using the bimodule structure on $U(\fh)$). Since the claim 
is that the result is 1-dimensional in some cohomological
degree with filtration jumping in one degree, the isomorphism above
is easy to pin down uniquely.

\end{rem}

\subsection{}

We also observe that the above generalizes in the
natural way to the setting
of a morphism $\fh_1 \to \fh_2$ between (finite-dimensional,
non-derived) Lie algebras.
Here generalizing means that we obtain the previous discussion
by considering the structure map $\fh \to 0$.

We draw attention to the following consequence of
Lemma \ref{l:lie-shift} in this setting, which will play an important
role in what follows. Let 
$\ind_{\fh_1}^{\fh_2}:\fh_1\mod \to \fh_2\mod$ denote the
left adjoint to the forgetful functor. Note that by realizing
$\ind_{\fh_1}^{\fh_2}$ as a $\IndCoh$-pushforward, we
find that it is naturally filtered with semi-classical functor
the quasi-coherent pullback along $\fh_2^{\vee} \to \fh_1^{\vee}$. 

\begin{cor}\label{c:coh-ind}

There is a canonical isomorphism of filtered functors:

\[
C^{\dot}\Big(\fh_2,\ind_{\fh_1}^{\fh_2}\big((-) \otimes 
\det(\fh_2/\fh_1)[\dim \fh_2/\fh_1]\big)\Big) = 
C^{\dot}(\fh_1,-).
\]

\noindent Here $\det(\fh_2/\fh_1)[\dim \fh_2/\fh_1]$ is considered as
trivially filtered with a single jump in degree $\fh_2/\fh_1$.

\end{cor}

\begin{rem}

If $\fh_1 \to \fh_2$ is not injective, the quotient above should
be replaced by the cone.

\end{rem}

\subsection{Harish-Chandra setting}\label{ss:hc-fd-start}

Now suppose that $(\fh,K)$ is a finite-dimensional Harish-Chandra pair,
so $\fh$ is as above, $K$ is an affine algebraic group acting on $\fh$,
and we are given a $K$-equivariant morphism 
$\Lie(K) \eqqcolon \fk \to \fh$.

Recall that a group always acts trivially on its classifying
stack. Therefore, the induced action of $K$ on $\bB \exp(\fh)$
factors through an action of $K_{dR}$ on $\bB \exp(\fh)$.
Using somewhat strange notation, we let\footnote{We do not
include $\exp$ in the notation because we believe that 
with the comma, there is no risk for confusing 
the vector space underlying $\fh$ and its associated formal
group, the way there could be in the notation $\bB \fh$.}
$\bB(\fh,K)$
denote the quotient $\bB \exp(\fh)/K_{dR}$. We let
$\fh\mod^K$ denote $\IndCoh(\bB \exp(\fh)/K_{dR})$,
noting that this category is tautologically the $K$-equivariant
category for the induced strong action of $K$ on $\fh\mod$.

\begin{example}

For the Harish-Chandra pair $(\fk,K)$, we obtain
$\bB(\fk,K) = \bB K$.

\end{example}

\begin{example}\label{e:fh-hc-integrates}

Note that the projection map $\Spec(k) \to \bB (\fh,K)$ defines a group
$(\fh,K)$ whose classifying stack is as notated. If $\fh = \Lie(H)$
for $K \subset H$, then $(\fh,K)$ is the formal completion of
$K$ in $H$.

\end{example}

\subsection{}

Now we recall that for any ind-affine nil-isomorphism $f:X \to Y$ of
prestacks, \cite{grbook} \S IV.5.2
associates a \emph{deformation} of this map,
i.e., a prestack $Y_{\hbar}$ over $\bA_{\hbar}^1/\bG_m$
and an $\bA_{\hbar}^1/\bG_m$-morphism 
$X \times \bA_{\hbar}^1/\bG_m \to Y_{\hbar}$
giving $f$ over the generic point.

For example, for $\Spec(k) \to \bB \exp(\fh)$, we obtain
the deformation $\bB \Fil \fh$ corresponding to the PBW filtration.

For a Harish-Chandra datum as above, we obtain a deformation
$\bB \Fil(\fh,K)$
associated with the map $\bB K \to \bB (\fh,K)$.
The special fiber of this deformation is
$(\bB (\fh/\fk)_0^{\wedge} )/K$, where we note that
$(\fh/\fk)_0^{\wedge}$ is\footnote{We will not need
this in practice, but if $\fk \to \fh$ is not injective, then $\fh/\fk$ should
be considered as a complex and understood as the cone;
the corresponding scheme should be understood in the usual sense.}
$\exp(\fh/\fk)$ with $\fh/\fk$ \emph{considered as an abelian Lie algebra}. 

Therefore, we obtain a filtration on 
$\fh\mod^K$ with $\fh\mod^{K,cl} = \QCoh((\fh/\fk)^{\vee}/K)$.

\begin{example}

When $\fk \isom \fh$, this is the constant filtration
on $\Rep(K)$ (in the sense of Example \ref{e:cnst-fil}).

\end{example}

\subsection{}

From now on, we assume $\fk \to \fh$ is injective.

\subsection{}

$\IndCoh$ pushforward and pullback along
the ind-affine nil-isomorphism $\bB K \to \bB(\fh,K)$ induces a pair
of adjoint functors:

\[
\ind = \ind_{\fk}^{\fh}:\Rep(K) \rightleftarrows \fh\mod^K:\Oblv.
\]

These functors are evidently compatible with filtrations.
At the semi-classical level, they induce the adjunction:

\[
\Rep(K) \rightleftarrows \QCoh((\fh/\fk)^{\vee}/K)
\]

\noindent given by pullback and pushforward along the
structure map $(\fh/\fk)^{\vee}/K \to \bB K$.

In particular, we find that $\fh\mod^K$ admits a canonical $t$-structure
for which $\Oblv:\fh\mod^K \to \Rep(K)$ is $t$-exact,
and that the functor $\ind$ is $t$-exact for this $t$-structure as well.

\subsection{}

We focus on the case of \emph{Harish-Chandra cohomology}.

Using the standard resolution of 
$\omega_{\bB (\fh,K)}$ induced by the ind-affine nil-isomorphism
$\bB K \to \bB(\fh,K)$, we find that 
$\omega_{\bB (\fh,K)}$ is compact. Moreover, this compactness
remains true (for the same reason) for 
$\omega_{\bB \Fil(\fh,K)}$.

Therefore, mapping out of it defines a filtered functor:

\[
C^{\dot}(\fh,K,-):\fh\mod^K \to \Vect.
\]

On associated, graded, this functor
is given by taking $!$-restriction
along $0/K \into (\fh/\fk)^{\vee}/K$, and then global sections
(i.e., group cohomology for $K$).

\begin{example}

If $\fk \isom \fh$, then $C^{\dot}(\fk,K,-)$ computes group
cohomology $\fk\mod^K = \Rep(K) \to \Vect$. 
The filtration on this functor
is the constant one.

\end{example}

\begin{example}

If $K$ is unipotent, then $C^{\dot}(\fh,K,-)$
coincides with $C^{\dot}(\fh,-)$
(composed with the forgetful functor $\fh\mod^K \to \fh\mod$).
However, the filtered structures are substantially affected
by the presence of $K$, as can be seen e.g. by looking at the
semi-classical level.

\end{example}

\begin{rem}\label{r:hc-cplx}

We explain the above constructions using usual complexes.
Suppose for simplicity that $K$ is unipotent, so 
$C^{\dot}(\fh,K,-) = C^{\dot}(\fh,-)$ as a non-filtered functor.
In the notation
of Remark \ref{r:filt-chev-hom}, the filtration $F_{\dot} M$
is a filtration in $\fh\mod^K$ if the $F_i M$ are $K$-submodules
of $M$, i.e., if it induces a filtration on the image of 
$M \in \fk\mod^K = \Rep(K)$.

In this case, we obtain a filtration on the cohomological Chevalley complex with $i$th term:

\[
\ldots \to 0 \to F_i M \to 
(\fh/\fk)^{\vee} \otimes F_{i+1} M + \fh^{\vee} \otimes F_i M
\to 
\Lambda^2 (\fh/\fk)^{\vee} \otimes F_{i+2} M +
(\fh/\fk)^{\vee} \wedge \fh^{\vee} \otimes F_{i+1} M + 
\Lambda^2 \fh^{\vee} \otimes F_i M
\to \ldots
\]

\end{rem}

\subsection{}

We now give versions of Corollary \ref{c:coh-ind} in this setting.

\begin{lem}\label{l:coh-ind-hc-from-gp}

Let $\ind$ denote the induction functor $\Rep(K) \to \fh\mod^K$.
There is a canonical isomorphism:

\[
C^{\dot}\Big(\fh,K,\ind\big((-) \otimes 
\det(\fh/K)[\dim \fh/\fk]\big)\Big) = 
C^{\dot}(K,-):\Rep(K) \to \Vect
\]

\noindent of filtered functors. As before, 
$\det(\fh/\fk)[\dim \fh/\fk])$ is filtered with a single jump in
degree $\dim \fh/\fk$.

\end{lem}

\begin{proof}

The proof is similar to Lemma \ref{l:lie-shift}:
we need to evaluate both sides on the regular
representation\footnote{We are lazily not distinguishing
between $\sO_K$ considered as a sheaf and its global sections.} 
$\sO_K$ and consider the
result as a filtered $K$-module. For the right hand side, we obtain
the trivial representation (in cohomological degree $0$, with the
filtration jumping only in degree $0$).

Then one calculates $C^{\dot}(\fh,K,\ind \sO_K )$ as 
$\det(\fh/\fk)^{\vee}[-\dim \fh/\fk]$ using standard 
filtrations
as in Lemma \ref{l:lie-shift}, which gives the claim.
Here we note that for $C^{\dot}(\fh,K,-)$ (considered
merely as a filtered functor) the standard filtration has 
associated graded: 

\[
\gr_{-i}C^{\dot}(\fh,K,-) = 
C^{\dot}(K, (-) \otimes \Lambda^i (\fh/\fk)^{\vee}[-i]).
\]

\end{proof}

Suppose now that $(\fh_1,K) \to (\fh_2,K)$ is a morphism of
Harish-Chandra pairs. Note that $\IndCoh$ pushforward
defines a functor $\ind_{\fh_1}^{\fh_2}:\fh_1\mod^K \to \fh_2\mod^K$
compatible with forgetful functors and satisfying similar properties
as before.

We have the following generalization of Lemma \ref{l:coh-ind-hc-from-gp}.

\begin{lem}\label{l:coh-ind-hc-fd}

There is a canonical isomorphism of filtered functors:

\[
C^{\dot}\Big(\fh_2,K,\ind_{\fh_1}^{\fh_2}\big((-) \otimes 
\det(\fh_2/\fh_1)[\dim \fh_2/\fh_1]\big)\Big) = 
C^{\dot}(\fh_1,K,-): \fh_1\mod^K \to \Vect.
\]

\noindent As before, $\det(\fh_2/\fh_1)[\dim \fh_2/\fh_1]$ is filtered
with a single jump in degree $\dim \fh_2/\fh_1$.

\end{lem}

The argument is similar to those that preceded it, so we omit it.

\subsection{Group actions on filtered categories}\label{ss:q-gp-action}

We now put the above into a more conceptual
framework whose perspective will be convenient at some points.
We essentially rewrite the above using a version of the theory
of group actions on categories.

Associated with the ind-affine nil-isomorphism, $K \to K_{dR}$ one has the
deformation $\Fil K_{dR}$, which is
a \emph{relative group\footnote{One way to see that the deformation formalism
from \cite{grbook} forces a (relative) group structure is to instead work
with the classifying prestacks.} 
prestack} over $\bA_{\hbar}^1/\bG_m$ with
generic fiber $K_{dR}$ and special fiber
$K \ltimes \bB \fk_0^{\wedge}$. Note that
because $\fk_0^{\wedge}$ is a commutative formal group,
$\bB \fk_0^{\wedge}$ actually is a commutative group prestack
(acted on by $K$).

Observe that we have a fiber sequence of relative groups:

\begin{equation}\label{eq:k-dr-seq}
1 \to \exp(\Fil \fk) \to \Fil K \to 
\Fil K_{dR} \to 1
\end{equation}

\noindent where $\Fil K$ just means $K \times \bA_{\hbar}^1/\bG_m$.

Note that the usual convolution monoidal structure makes
$\IndCoh(\Fil K_{dR}) \in \Alg(\QCoh(\bA_{\hbar}^1/\bG_m)\mod)$.

\begin{defin}

Let $\sC$ be a filtered category.
A \emph{(strong) action of $K$ on $\Fil \sC$}
is a $\IndCoh(\Fil K_{dR})$-module structure on 
$\Fil\sC \in \QCoh(\bA_{\hbar}^1/\bG_m)\mod$.\footnote{One may
work equivalently with co-actions by duality. This has the advantage
that one may equate $\IndCoh$ with $\QCoh$ using formal smoothness,
and then at least try to forget $\IndCoh$.}

\end{defin}

\begin{warning}

There is a redundancy in the terminology: $K$ acts on $\Fil \sC$
could wrongly be understood to mean that 
$\Fil \sC \in \DGCat_{cont}$ is a $D(K)$-module category.
However, we are confident that there is no risk for confusion
in our use of the above terminology: the meaning is
completely determined by whether or not there is a $\Fil$ in front
of what is acted upon. 

We also note that there is no such risk for \emph{weak} actions,
i.e., if $K$ acts on $\Fil \sC$, then $K$ acts weakly on $\Fil \sC$
in the sense that $\Fil \sC$ is a $\QCoh(K)$-module category. 

\end{warning}

Note that if $K$ acts on $\Fil \sC$,
then $K$ in particular acts on $\sC$ (in the usual sense).

Moreover, $\sC^{cl}$ receives an action of
$\IndCoh(K \ltimes \bB \fk_0^{\wedge})$. 
This is equivalent to saying that
$\QCoh(\fk^{\vee})$ (with the usual symmetric monoidal
structure) acts on $\sC^{cl}$, and $K$ acts weakly on $\sC^{cl}$,
and these two actions are compatible under the adjoint
action of $K$ on $\fk^{\vee}$. In particular, $\sC^{cl,K,w}$ may
be thought of as a sheaf of categories on $\fk^{\vee}/K$.

Using the ``trivial" action of $K$ on $\Fil \Vect$ as a filtered category,
we obtain an invariants and coinvariants formalism.
By the same arguments as in \cite{dario-*/!} \S 2, the
two are canonically identified.

In particular, we obtain a filtration on $\sC^K$ with
$\sC^{K,cl} = 
(\sC^{cl} \underset{\QCoh(\fk^{\vee})}{\otimes} \Vect)^{K,w}$,
where $\QCoh(\fk^{\vee})$ acts on $\Vect$ through
the restriction to $0$ functor. Geometrically, this means
we take our sheaf of categories on $\fk^{\vee}/K$,
restrict to $0/K$, and then take global sections.

The adjoint functors $\Oblv: \sC^K \rightleftarrows \sC: \Av_*$
carry natural filtrations. semi-classically, 
$\Oblv^{cl}$ is the functor:

\[
(\sC^{cl} \underset{\QCoh(\fk^{\vee})}{\otimes} \Vect)^{K,w} \to
\sC^{cl} \underset{\QCoh(\fk^{\vee})}{\otimes} \Vect \to
\sC^{cl}
\]

\noindent where the first functor is forgetting and the second
is $*$-pushforward along $0 \into \fk^{\vee}$. Of course,
the semi-cl{cl}assical $\Av_*^{cl}$ is the right adjoint to the
functor we just described: $!$-pullback
$0 \into \fk^{\vee}$ and then $*$-pushforward to $\bB K$.

\begin{example}

We can reformulate some of our earlier constructions
in saying that $K$ acts on $\Fil \fh\mod$ for a
Harish-Chandra pair $(\fh,K)$. The associated filtration
on $\fh\mod^K$ is our earlier one.

\end{example}

\subsection{Derived categories}

The following result will play an important role for us, but may be
skipped at first pass.

\begin{lem}[Bernstein-Lunts, \cite{bernstein-lunts}]\label{l:bernstein-lunts-fd}

$\fh\mod^K$ is the
derived category of its heart. (Here we continue to assume
$\fk \to \fh$ is injective.)

\end{lem}

\begin{proof}

If $\fk \to \fh$ is an isomorphism, then 
$\fh\mod^K = \Rep(K)$, and this is a general property
about algebraic stacks (see e.g. \cite{qcoh} Proposition 5.4.3).

In general, we have a $t$-exact forgetful functor
$\Oblv:\fh\mod^K \to \Rep(K)$. Moreover,
this functor admits a left adjoint $\ind$, geometrically given
by $\IndCoh$-pushforward. 

Now observe that our hypothesis implies that the 
tangent complex of the morphism 
$\bB K \to \bB(\fh,K)$ is 
concentrated in cohomological degree $0$: it
is $\fh/\fk$ considered as a $K$-representation.
Therefore, the monad $\Oblv\circ \ind$ has a standard filtration
with associated graded given by tensoring with
$\Sym(\fh/\fk)$, so in particular, this monad is $t$-exact.
Since $\Oblv$ is $t$-exact, we find that $\ind$ is as well.

Therefore, we obtain a similar pair of adjoint functors
between $D(\Rep(K)^{\heart}) = \Rep(K)$ and 
$D(\fh\mod^{K,\heart})$. Both forgetful functors
$\fh\mod^K \to \Rep(K)$ and $D(\fh\mod^{K,\heart}) \to \Rep(K)$
are conservative and commute with colimits, so 
are monadic. Then we observe that the monads on
$\Rep(K)$ are naturally identified, giving the claim.

\end{proof}

\subsection{Kazhdan-Kostant twists}\label{ss:kk-start}

We now discuss how to render the standard solution to a standard
problem in our framework.

We begin by describing the issue. Suppose $\fg$ and $\fn$ are as usual,
and $\psi:\fn \to k$ is a non-degenerate character.
We choose $G$-equivariant symmetric $\fg \simeq \fg^{\vee}$
and take $f \in \fg = \fg^{\vee}$ the principal nilpotent
mapping to $\psi \in \fn^{\vee}$.

Recall that $\sW^{fin}\mod \simeq \fg\mod^{N,\psi}$, 
and that $\sW^{fin}$ is filtered with associated graded
being the algebra of functions on $f+\fb$.
In particular, $C^{\dot}(\fn,(-\psi) \otimes \ind_{\fn}^{\fg}(\psi))$ 
is filtered with associated graded being this algebra of functions.

However, this filtration is \emph{not} induced by the obvious PBW filtration
on $\ind_{\fn}^{\fg}(\psi)$. Indeed, suppose more generally that
$M$ is any (PBW) filtered $\fn$-module. 
Then the induced filtration on $M \otimes -\psi$ \emph{has the
same associated graded} as $M$. (So in the case above, we will
see the DG algebra $\Gamma(\fb/N,\sO_{\fb/N})$ instead.)

The issue is with the
filtration on the 1-dimensional representation $\psi$: jumping
in a single degree, its associated graded is going to
be the augmentation module of $\Sym(\fn)$. In other words,
its associated graded will be a $\bG_m$-equivariant quasi-coherent
sheaf on $\fn^{\vee}$, so it must be the skyscraper at the origin,
though we would rather see the skyscraper at $\psi \in \fn^{\vee}$.

The solution\footnote{Essentially introduced, I believe, in \cite{kostant-whittaker} \S 1, where it is attributed to Kazhdan.}
to this problem is to use $\check{\rho}$ to modify
the filtration on $U(\fn)$, so that its associated graded is
again $\Sym(\fn)$ (but with a modified grading!),
and the module $\psi$ is filtered with associated graded
being the skyscraper at $\psi \in \fn^{\vee}$ instead.
Namely, we set:

\[
\begin{gathered}
F_i^{KK} U(\fn) = \underset{j}{\oplus} \,  
F_{i-j}^{PBW} U(\fn) \cap U(\fn)^j \\
U(\fn)^j \coloneqq 
\{x \in U(\fn) \mid \Ad_{-\check{\rho}(\lambda)}(x) = \lambda^j \cdot x\}.
\end{gathered}
\]

\noindent We emphasize that the grading used here is induced
by $-\check{\rho}: \bG_m \to G^{ad}$, not $\check{\rho}$ itself
(so $\fn$ is negatively graded).
For example, for $i \in \cI_G$, 
$e_i \in F_0^{KK} U(\fn)$, and for $\alpha$ a general 
root, $e_{\alpha} \in F_{1-(\check{\rho},\alpha)}^{KK} U(\fn)$.

One has a similar filtration on $U(\fg)$. Moreover, these
filtrations naturally induce the ``correct" filtration on
$\sW^{fin}$ (e.g., it is a filtration in the abelian category, not
just the derived category).

Our present goal is to render the above ideas in the categorical framework.
We will begin by discussing some generalities, and then apply
these to the example of Harish-Chandra modules.

\begin{warning}\label{w:kk-bdd}

We immediately see that Kazhdan-Kostant filtrations are typically
unbounded from below and are not complete filtrations. 
This causes some technical problems
(e.g., in \S \ref{s:ds}), and requires care.

\end{warning}

\subsection{}\label{ss:kk-vect}

We begin by discussing how to twist a filtration by a grading to
obtain a new filtration. Motivated by our particular
concerns, we use the notation
\emph{PBW} to indicate an ``old" filtration and \emph{KK} to 
indicate a ``new" filtration.

So suppose that $F_{\dot}^{PBW} V \in \Fil \Rep(\bG_m)$, i.e.,
$V$ is equipped with a grading $V = \oplus_j V^j$ and
a filtration $F_{\dot}^{PBW} V$ \emph{as a graded vector space},
so we have a grading $F_i^{PBW} V = \oplus_j F_i^{PBW} V^j$ compatible
with varying $i$ in the natural sense.

Then we can \emph{twist} the filtration $F_{\dot}^{PBW}$ by the grading to
obtain the filtration:

\[
F_i^{KK} V = \underset{j}{\oplus} \, F_{i-j}^{PBW} V^j
\]

\noindent as before.

Here is a geometric interpretation. Recall that a filtered vector
space is the same as a quasi-coherent sheaf on $\bA_{\hbar}^1/\bG_m$.
The data of a compatible filtration and grading as above
is equivalent to a quasi-coherent sheaf on 
$\bA_{\hbar}^1/\bG_m \times \bB \bG_m$, where we recover
the underlying filtered vector space by pulling back along the
first projection.

Formation of the KK filtration corresponds to 
pulling back along the graph of the projection
$\bA_{\hbar}^1/\bG_m \to \bB \bG_m$ instead.

In either perspective, we immediately find that:

\[
\gr_{i}^{KK} V = \underset{j}{\oplus} \,  \gr_{i-j}^{PBW} V^j.
\]

\noindent That is, $\gr_{\dot}^{KK} V = \gr_{\dot}^{PBW} V$
as vector spaces, although the gradings are different.

\subsection{}\label{ss:kk-cat}

We now give a categorical version of the above.

Suppose $\sC$ is a filtered category. We use the notation
$\Fil\!^{PBW} \sC$ for the underlying 
$\QCoh(\bA_{\hbar}^1/\bG_m)$-module category, since we wish
to construct another filtration on $\sC$.

The extra data we need is an action of 
$\QCoh(\bG_m)$ (with its convolution monoidal structure) 
on $\Fil\!^{PBW} \sC$ commuting with the
$\QCoh(\bA_{\hbar}^1/\bG_m)$. In this case, we can again
\emph{twist} our filtration by this action in forming:

\[
\Fil\!^{KK} \sC \coloneqq 
\Fil\!^{PBW} \sC^{\bG_m,w} 
\underset{\QCoh(\bA_{\hbar}^1/\bG_m) \otimes \Rep(\bG_m)}{\otimes} 
\QCoh(\bA_{\hbar}^1/\bG_m)
\]

\noindent where the action on the right term is induced
by the symmetric monoidal functor: 

\[
\QCoh(\bA_{\hbar}^1/\bG_m) \otimes \Rep(\bG_m) \to 
\QCoh(\bA_{\hbar}^1/\bG_m)
\]

\noindent  of pullback along the
graph of the structure map $\bA_{\hbar}^1/\bG_m \to \bB \bG_m$.

More geometrically: we are given the datum of a sheaf of categories
on $\bA_{\hbar}^1/\bG_m \times \bB \bG_m$, and we observe 
that we can form \emph{two} filtered categories from it,
via pullback along the maps:

\[
\id \times p,\Gamma: \bA_{\hbar}^1/\bG_m \to
\bA_{\hbar}^1/\bG_m \times \bB \bG_m
\]

\noindent where $p:\Spec(k) \to \bB \bG_m$ is the tautological
projection, and $\Gamma$ is the graph of the structure map as above.
These two maps coincide over the open point
$\bA_{\hbar}^1\setminus 0/\bG_m$, so are filtrations on the
same category $\sC$. By definition, the pullback along
$\id \times p$ defines the ``PBW" filtration, and pullback along
$\Gamma$ defines the KK filtration.

\begin{rem}\label{r:gr-fil-kk-twist}

There is a tautological functor 
$\Fil\!^{PBW} \sC^{\bG_m,w} \to \Fil\!^{KK} \sC$.

\end{rem}

\begin{rem}

Note that $\sC^{cl}$ is also the fiber at $0$ of $\Fil\!^{KK} \sC$,
but the weak $\bG_m$-action is different: it is the diagonal action
mixing the standard action of $\bG_m$ on $\sC^{cl}$ with the
action coming from the weak $\bG_m$-action on $\Fil\!^{PBW} \sC$.

\end{rem}

\begin{example}

In \S \ref{ss:kk-vect}, it is tautological that the KK twisting
construction is symmetric monoidal. So if $A$ is an 
algebra with compatible filtration $F_{\dot}^{PBW}$ and
grading, we obtain a filtration $F_{\dot}^{KK}$ on $A$ (as an algebra).
Therefore, we obtain two filtrations $\Fil\!^{PBW},\Fil\!^{KK}$ on 
$A\mod$.
Of course, the grading on $A$ induces a weak
$\bG_m$-action on 
$\Fil\!^{PBW} A$, and the general categorical 
construction above produces $\Fil\!^{KK} A$.
The functor $\Fil\!^{PBW} A\mod^{\bG_m,w} \to \Fil\!^{KK} A\mod$
corresponds to taking a graded and PBW filtered
$A$-module and then applying the corresponding 
KK twist to obtain a KK filtered $A$-module.

\end{example}

\subsection{}\label{ss:kk-bifilt}

The reader may skip this material for the time being, and return
to it as necessary. Its purpose is closely tied to 
Warning \ref{w:kk-bdd}: Kazhdan-Kostant filtrations are
typically incomplete and unbounded from below, even when a
corresponding PBW filtration is bounded from below.
To deal with this issue, we wish to yoke the two filtrations on our category.

\begin{defin}

A \emph{bifiltration} on $\sC \in \DGCat_{cont}$ is a 
$\QCoh(\bA_{\hbar_1}^1/\bG_m \times \bA_{\hbar_2}^1/\bG_m)$-module
category $\BiFil \sC$ plus an isomorphism:

\[
\sC \simeq
\BiFil \sC 
\underset{\QCoh(\bA_{\hbar_1}^1/\bG_m \times \bA_{\hbar_2}^1/\bG_m)}
{\otimes} 
\QCoh((\bA_{\hbar_1}^1\setminus 0)/\bG_m \times
(\bA_{\hbar_2}^1\setminus 0)/\bG_m) =
\BiFil \sC
\underset{\QCoh(\bA_{\hbar_1}^1/\bG_m \times \bA_{\hbar_2}^1/\bG_m)}
{\otimes} 
\Vect.
\]

\noindent A \emph{bifiltration} on $\sF$ in $\sC$ is an object
of $\BiFil \sC$ restricting to $\sF$.

\end{defin}

Much of our earlier discussion generalies.
E.g., we have an obvious notion of bifiltered vector spaces, and
so on.

Note that a bifiltration on $\sC$ indeed gives rise to
two filtrations on $\sC$, denoted $\Fil\!^{PBW} \sC$ and $\Fil\!^{KK} \sC$.
These are respectively obtained by restricting $\BiFil \sC$ to the
loci:

\[
\begin{gathered} 
\bA_{\hbar_1}^1/\bG_m \times (\bA_{\hbar_2}^1\setminus 0)/\bG_m =
\bA_{\hbar}^1/\bG_m \\
(\bA_{\hbar_1}^1\setminus 0)/\bG_m \times \bA_{\hbar_2}^1/\bG_m =
\bA_{\hbar}^1/\bG_m.
\end{gathered}
\]

\noindent Note that a bifiltration on $\sF \in \sC$ gives rise to
PBW and KK filtrations on $\sF$.

\begin{rem}\label{r:bifil->fil-on-cl}

Let $\sC^{PBW\mathendash{cl}}$ and $\sC^{KK\mathendash{cl}}$ denote the semi-classical
categories associated with each of these filtrations. We claim
that e.g. $\sC^{PBW\mathendash{cl}}$ carries a natural KK filtration
$\Fil\!^{KK} \sC^{PBW\mathendash{cl}}$; moreover, the weak $\bG_m$-action
on $\sC^{PBW\mathendash{cl}}$ extends to one on $\Fil\!^{KK} \sC^{PBW\mathendash{cl}}$,
and this action commutes with the action of $\QCoh(\bA_{\hbar}^1/\bG_m)$.

Indeed, note that $\sC^{PBW\mathendash{cl},\bG_m,w}$
is the restriction to:

\[
\bB \bG_m \times (\bA_{\hbar_2}^1\setminus 0)/\bG_m \subset
\bA_{\hbar_1}^1/\bG_m \times \bA_{\hbar_2}^1/\bG_m
\]

\noindent so taking $\Fil\!^{KK} \sC^{PBW\mathendash{cl},\bG_m,w}$ as
the restriction to:

\[
\bB \bG_m \times \bA_{\hbar_2}^1/\bG_m 
\]

\noindent  (and applying
de-equivariantization) gives the desired
construction.

Of course, this works symmetrically in PBW and KK.

\end{rem}

\begin{example}\label{e:kk-bifil}

Suppose that we are in the setting of 
\S \ref{ss:kk-cat}, so $\sC$ carries a single (PBW) filtration
and a compatible weak $\bG_m$-action. We claim that
this data induces a bifiltration on $\sC$ inducing the 
PBW and KK filtrations in the sense of \S \ref{ss:kk-cat}.
For this, we note that we have the morphism:\footnote{The notation 
$s\in \sL$ for points of $\bA_{\hbar}^1/\bG_m$
is used because this stack is the moduli of a line bundle plus a section.}

\[
\begin{gathered}
\bA_{\hbar_1}^1/\bG_m \times \bA_{\hbar_2}^1/\bG_m \to
\bA_{\hbar}^1/\bG_m \times \bB\bG_m \\
(s_1 \in \sL_1, s_2\in \sL_2) \mapsto 
(s_1 \otimes s_2 \in \sL_1 \otimes \sL_2, \sL_2)
\end{gathered}
\]

\noindent whose restriction to
$\bA_{\hbar_1}^1/\bG_m \times (\bA_{\hbar_2}^1\setminus 0)/\bG_m$
is the $\id \times p$ and whose restriction to
$(\bA_{\hbar_1}^1\setminus 0)/\bG_m \times \bA_{\hbar_2}^1/\bG_m =
\bA_{\hbar}^1/\bG_m$ is $\Gamma$; so the operation 
of pullback (in the sheaf of categories language) along this
morphism gives the desired structure.

Recall that $\sC^{PBW\mathendash{cl}}$ is canonically isomorphic to 
$\sC^{KK\mathendash{cl}}$ in this case, so we denote them each by
$\sC^{cl}$.

Recall that $\bG_m \times \bG_m$ weakly acts on $\sC^{cl}$:
one factor acts because this is always true for the semi-classical
category, and the other factor acts because of the weak $\bG_m$-action
on $\Fil\!^{PBW} \sC$. 
Then it is straightforward to verify that the KK filtration on 
$\sC^{PBW\mathendash{cl}} = \sC^{cl}$ is induced by taking the diagonal
weak action of $\bG_m$ on $\sC^{cl}$ and applying
Example \ref{e:fil-gm}.

The situation with $\sC^{KK\mathendash{cl}} = \sC^{cl}$ is similar: 
its filtration is induced
by the ``canonical" weak $\bG_m$-action on $\sC^{cl}$, i.e., the
one from the first $\bG_m$-factor above (so is unrelated to
the weak $\bG_m$-action on $\Fil\!^{PBW} \sC$).

\end{example}

\subsection{}\label{ss:psi-pbw-filt}

We now begin to apply the above in the setting of Lie algebras and
Harish-Chandra modules.

Before discussing Kazhdan-Kostant directly, let us discuss
what we can obtain without the additional ``grading" (i.e., weak 
$\bG_m$-action). We use the language of \S \ref{ss:q-gp-action}.

Suppose $\bG_a$ acts on $\Fil \sC$. Let $\psi$ denote the
exponential (alias: Artin-Schreier) 
character sheaf on $\bG_a$. Our problem is to construct
a filtration on $\sC^{\bG_a,\psi}$.

First, note that (forgetting the filtrations) 
we can write $\sC \mapsto \sC^{\bG_a,\psi}$ in two steps:
for $\widehat{\bG}_a$ the formal completion of 
$\bG_a$ at the origin, the group prestack 
$\bB \widehat{\bG}_a = \bG_{dR}/\bG_a$ acts on
$\sC^{\bG_a,w}$. Note that
$\QCoh(\bB \widehat{\bG}_a) = \QCoh(\bA_{\Lie}^1)$ with the
convolution structure on the LHS corresponding to the
tensor product structure on the RHS; here the subscript $\Lie$
is used so we later remember the Lie-theoretic origins of 
this copy of the affine line.
Therefore, $\sC^{\bG_a,w}$ fibers over $\bA_{\Lie}^1$, and we
can take its fiber at $1 \in \bA_{\Lie}^1$, i.e., we can form:

\[
\sC^{\bG_a,w} \underset{\QCoh(\bA_{\Lie}^1)}{\otimes} \Vect
\]

\noindent using the restriction functor along $1 \into \bA_{\Lie}^1$.
It is immediate to see that this tensor product is $\sC^{\bG_a,\psi}$.

In the filtered setting, let
$\Fil \widehat{\bG}_a$ denote the (commutative) relative formal group over
$\bA_{\hbar}^1/\bG_m$ defined by $\Fil \Lie(\bG_a)$.
As above, $\bB \Fil \widehat{\bG}_a$ acts on $\Fil \sC^{\bG_a,w}$.
Note that: 

\[
\QCoh(\bB \Fil \widehat{\bG}_a) \simeq 
\QCoh((\bA_{\Lie}^1 \times \bA_{\hbar}^1)/\bG_m) \in 
\Alg(\QCoh(\bA_{\hbar}^1/\bG_m)\mod)
\]

\noindent where the left hand side is equipped with the convolution
monoidal structure, and the right hand side is equipped with
the tensor product monoidal structure. Moreover, 
since $\bA_{\Lie}^1$ occurs here as the \emph{co}adjoint space 
$\Lie(\bG_a)^{\vee}$, it is naturally 
equipped with the action of 
$\bG_m$ by inverse homotheties; so both $\bA^1$-factors are
acted on in this way, and our graded algebra of functions 
is a polynomial algebra
on the two degree $1$ generators $\hbar$ and $x\hbar$ for 
$x\in \Lie(\bG_a)$ the generator.

\begin{warning}

The reader confused why we see 
$(\bA_{\Lie}^1 \times \bA_{\hbar}^1)/\bG_m$ and
not $\bA_{\Lie}^1 \times (\bA_{\hbar}^1/\bG_m)$
should return to Warning \ref{w:fil-abelian}.

\end{warning}

The upshot is that we obtain a filtration on 
$\sC^{\bG_a,\psi}$ by taking $\Fil \sC^{\bG_a,\psi}$ as
$\Fil \sC^{\bG_a,w}$ and restricting along the diagonal map:

\[
\bA_{\hbar}^1/\bG_m \xar{x \mapsto (x,x)} (\bA_{\Lie}^1 \times \bA_{\hbar}^1)/\bG_m.
\]

\noindent When there is risk for confusion, we refer to this
as the \emph{PBW} filtration on $\sC^{\bG_a,\psi}$ and denote
it by $\Fil\!^{PBW} \sC^{\bG_a,\psi}$.

\begin{rem}

But in our hearts, we know that we would rather restrict
along the \emph{non-existing map}:

\[
\bA_{\hbar}^1/\bG_m \xar{x \mapsto (1,x)} (\bA_{\Lie}^1 \times \bA_{\hbar}^1)/\bG_m.
\]

\end{rem}

\begin{rem}\label{r:psi-pbw-cl}

Note that $\sC^{\bG_a,\psi,cl} = \sC^{\bG_a,cl}$ under this construction.

\end{rem}

\subsection{}\label{ss:psi-kk-filt}

Let $\bG_m$ act on the group $\bG_a$ through inverse\footnote{The 
reader is reminded that we used $-\check{\rho}$ in
\S \ref{ss:kk-start}, so $\fn$ had negative degree with respect to the
$\bG_m$-action. So the sign here is the expected one.}
homotheties.
Since the morphism $\bG_a \to \bG_{a,dR}$ is $\bG_m$-equivariant,
$\bG_m$ acts on $\Fil \bG_{a,dR}$. For $\sC$ ``PBW" filtered, it makes
sense to ask that a weak $\bG_m$-action and a 
(strong) $\bG_a$-action on $\Fil\!^{PBW} \sC$ be 
compatible (with the action
of $\bG_m$ on $\Fil \bG_a$). Note that this in particular gives
a weak $\bG_m$-action on $\Fil\!^{PBW} \sC^{\bG_a,w}$,
so we may form $\Fil\!^{KK} \sC^{\bG_a,w}$.

Observe that if we regard $\Lie \bG_a$ as a
filtered Lie algebra through the PBW method of \S \ref{ss:pbw}
(remembering Warning \ref{w:fil-abelian}) and equip it with the 
above (degree $-1$) grading, then Kazhdan-Kostant twisting
gives $\Lie \bG_a$ equipped with the \emph{constant} filtration
(jumping only in degree $0$). 

It follows formally that $\bB \widehat{\bG}_a$ (with no $\Fil\!$!) 
acts on $\Fil\!^{KK} \sC^{\bG_a,w}$. 
Therefore, $\Fil\!^{KK} \sC^{\bG_a,w}$ has an action of
$\QCoh(\bA_{\Lie}^1)$, and we may take its fiber at $1 \in \bA_{\Lie}^1$
(by approriately tensoring with $\Vect$).
By definition, this is $\Fil\!^{KK} \sC^{\bG_a,\psi}$, the
\emph{Kazhdan-Kostant} filtration on $\sC^{\bG_a,\psi}$.

Recall that $\sC^{cl}$ has commuting (because $\bG_a$ is commutative)
actions of $\QCoh(\bG_a)$ (under convolution) and
$\QCoh(\Lie(\bG_a)^{\vee}) = \QCoh(\bA_{\Lie}^1)$. One immediately
finds that $\sC^{\bG_a,\psi,KK\mathendash{cl}}$, the ``semi-classical" category
for the special fiber, is $\sC^{cl,\bG_a,w}|_{1\in \bA_{\Lie}^1}$,
where the restriction notation means we form the appropriate 
tensor product.

\begin{warning}

This Kazhdan-Kostant filtration on $\sC^{\bG_a,\psi}$ is \emph{not} obtained
by applying the method of \S \ref{ss:kk-cat} to the PBW filtration.
Indeed, the semi-classical categories are different.

\end{warning}

\subsection{}

We now repeat the above to produce a 
\emph{bi}filtration on $\sC^{\bG_a,\psi}$
inducing the PBW and KK filtrations.

Note that $\sC^{\bG_a,w}$ carries a canonical
bifiltration from \S \ref{ss:kk-bifilt}: it is induced by the weak $\bG_m$-action
on $\Fil\!^{PBW} \sC^{\bG_a,w}$. 
Moreover, because $\Lie \bG_a$ is bifiltered by its
PBW and KK filtrations by \S \ref{ss:kk-bifilt}, it follows that the
corresponding bifiltered formal group acts on $\BiFil \sC^{\bG_a,w}$.
Combining our analysis in the PBW and KK cases, we find that
the action of 
$\QCoh(\bA_{\hbar_1}^1/\bG_m \times \bA_{\hbar_2}^1/\bG_m)$ extends to 
an action of:

\[
\QCoh\Big(
(\bA_{\Lie}^1 \times \bA_{\hbar_1}^1)/\bG_m \times 
\bA_{\hbar_2}^1/
\bG_m
\Big)
\]

\noindent where the action on the first two factors is diagonal.
So we obtain our desired bifiltration by setting
$\BiFil \sC^{\bG_a,\psi}$ to be the restriction
of $\BiFil \sC^{\bG_a,w}$ along the map:

\[
\bA_{\hbar_1}^1/\bG_m \times 
\bA_{\hbar_2}^1/
\bG_m \xar{\Delta/\bG_m \times \id}
(\bA_{\Lie}^1 \times \bA_{\hbar_1}^1)/\bG_m \times 
\bA_{\hbar_2}^1/
\bG_m
\]

\noindent where $\Delta$ is the diagonal map
$\bA_{\hbar_1}^1 = \bA^1 \to \bA^1 \times \bA^1 = 
\bA_{\Lie}^1 \times \bA_{\hbar_1}^1$. 
By construction, 
it induces the PBW and KK filtrations on $\sC^{\bG_a,\psi}$,
with terminological conventions consistent with \S \ref{ss:kk-bifilt}.

\begin{rem}\label{r:psi-filts-on-cl}

Recall from Remark \ref{r:bifil->fil-on-cl} that our bifiltration
induces filtrations on $\sC^{\bG_a,\psi,PBW\mathendash{cl}}$ and
$\sC^{\bG_a,\psi,KK\mathendash{cl}}$ (in the notation of \emph{loc. cit}.).
It is straightforward to verify that the (KK) filtration on
$\sC^{\bG_a,\psi,PBW\mathendash{cl}} = \sC^{cl,\bG_a,w}|_{0 \in \bA_{\Lie}^1}$
is as in Example \ref{e:kk-bifil}, i.e., induced by the diagonal
$\bG_m$-action via Example \ref{e:fil-gm}.
The (PBW) filtration on 
$\sC^{\bG_a,\psi,KK\mathendash{cl}} = \sC^{cl,\bG_a,w}|_{1 \in \bA_{\Lie}^1}$
is obtained from degenerating the character, i.e., it is the $\bG_m$-equivariant
sheaf of categories over $\bA^1$ with fiber 
$\sC^{cl,\bG_a,w}|_{\lambda \in \bA_{\Lie}^1}$ at $\lambda \in \bA^1$;
of course, the $\bG_m$-equivariance here comes 
from the $\bG_m$-action on $\Fil\!^{PBW} \sC$. 

\end{rem}

\subsection{}\label{ss:kk-hc-fd}

We now apply the above in the Harish-Chandra setting.

Suppose as before that $(\fh,K)$ is a Harish-Chandra pair with
$\fh$ finite-dimensional and $K$ an affine algebraic group. We suppose
$\Lie(K) \into \fh$ for simplicity.
Suppose moreover that we are given a non-trivial character
$\psi:K \to \bG_a$; let $K^{\prime}$ denote the kernel. We also 
let $\psi$ denote the induced character $\fk \to k$, or the
corresponding $1$-dimensional $\fk$-module; similarly for $-\psi$.

Then since $K$ acts on $\Fil \fh\mod$, $\bG_a$ acts on $\Fil \fh\mod^{K'}$,
so by the above, we obtain a \emph{PBW} filtration 
$\Fil\!^{PBW} \fh\mod^{K,\psi}$ from \S \ref{ss:psi-pbw-filt}.
We have $\fh\mod^{K,\psi,PBW\mathendash{cl}} = \QCoh((\fh/\fk)^{\vee}/K)$.

Suppose now that $\bG_m$ acts on $K$; that the
$K$-action on $\fh$ has been extended to
$\bG_m \ltimes K$; and that the character
$\psi:K \to \bG_a$ is $\bG_m$-equivariant for the inverse homothety
action of $\bG_m$ on $\bG_a$. 
Then $\Fil \fh\mod^{K^{\prime}}$ carries an action of 
$\bG_a = K/K^{\prime}$ and a weak action of $\bG_m$, giving
a datum as in \S \ref{ss:psi-kk-filt}. Therefore, we obtain
a KK filtration $\Fil\!^{KK} \fh\mod^{K,\psi}$. 
We have: 

\[
\fh\mod^{K,\psi,KK\mathendash{cl}} = \QCoh(\psi + (\fh/\fk)^{\vee}/K)
\]

\noindent where $\psi + (\fh/\fk)^{\vee} \subset \fh^{\vee}$ is the 
inverse image of $\psi$ 
under the map $\fh^{\vee} \to \fk^{\vee}$; this locus is
closed under the $K$-action because $\psi$ is a character.

These two filtrations fit into a bifiltration by the general formalism.

\begin{example}\label{e:pbw-kk-k=h}

For $\fk = \fh$, the PBW filtration on 
$\Rep(K) \overset{-\otimes \psi}{\simeq} \fk\mod^{K,\psi}$
is the constant one (as we discussed before),
and the KK filtration is induced from Example \ref{e:fil-gm}
via the weak $\bG_m$-action on $\Rep(K)$. In other words,
we regard $\sO_K$ as a graded coalgebra, so by
\emph{loc. cit}. it inherits a natural coalgebra filtration; then 
the KK filtration is obtained by
considering filtered comodules. Note that the group 
cohomology functor
$\Rep(K) \to \Vect$ is canonically bifiltered, e.g. because
the trivial representation has a canonical\footnote{It 
is the constant bifiltration on the
underlying vector space of the representation.} bifiltration.

\end{example}

It follows formally that the induction functor
$\ind_{\fk}^{\fh}:\Rep(K) = \fk\mod^{K,\psi} \to \fh\mod^{K,\psi}$
is bifiltered. Applying this to the trivial representation with its
canonical bifiltration, we see that 
the functor 
$C^{\dot}(\fk,K,(-)\otimes -\psi): \fh\mod^{K,\psi} \to \Vect$ is also
naturally bifiltered. Its underlying semi-classical functors
for the PBW and KK filtrations are the appropriate
global sections functors:

\[
\begin{gathered}
\Gamma((\fh/\fk)^{\vee}/K,-):\QCoh((\fh/\fk)^{\vee}/K) \to \Vect \\
\Gamma(\psi+ (\fh/\fk)^{\vee}/K,-):
\QCoh(\psi+ (\fh/\fk)^{\vee}/K) \to \Vect.
\end{gathered}
\]

Somewhat more generally,\footnote{This setup appears strange if one
has the finite $W$-algebra example $(\fh,K) = (\fg,N)$, but its infinite-dimensional
version appears in the affine $W$-algebra setup.}
suppose that the character $\psi$ is extended to $\fh$ and
continues to satisfy the
appropriate $\bG_m$-equivariance.

Then for $M \in \fh\mod^{K,\psi}$,
$M \otimes -\psi$ can be considered as an object of $\fh\mod^K$,
so it makes sense to take the Harish-Chandra cohomology:

\[
C^{\dot}(\fh,K,M \otimes -\psi).
\]

\noindent This functor is bifiltered, with PBW and KK 
semi-classical versions given by:

\[
\begin{gathered}
\QCoh((\fh/\fk)^{\vee}/K) \to \Vect \\
\QCoh(\psi+(\fh/\fk)^{\vee}/K) \to \Vect
\end{gathered}
\]

\noindent given by $!$-restriction to $0/K$ or $\psi/K$ 
followed by global sections on this stack (i.e., group cohomology
for $K$).

\begin{rem}\label{r:kk-hc-explicit}

Let us describe what a KK filtration on an object of $\fh\mod^{K,\psi}$
looks like concretely. Suppose $M \in \fh\mod^{K,\psi,\heart}$
observe that $M \otimes -\psi \in \fk\mod^K = \Rep(K)$,
i.e., the natural $\fk$-action integrates to the group.
A sequence:

\[
\ldots \subset F_i^{KK} M \subset F_{i+1}^{KK} M \subset \ldots
\]

\noindent is
a KK filtration if: 

\begin{itemize}

\item It is a filtration of $M$ considered as a module
over the KK-filtered algebra $U(\fh)$. In other words,
if $\fh^j$ indicates the $j$th graded component of $\fh$,
$\fh^j$ maps $F_i^{KK} M$ to $F_{i+j+1}^{KK} M$.

\item Consider $\sO_K$ as a filtered coalgebra
using the $\bG_m$-action on it and 
Example \ref{e:fil-gm}. 
Then the coaction map
$(M \otimes -\psi) \to (M\otimes -\psi) \otimes \sO_K$ should
be filtered. 

If $K$ is connected, this
is equivalent to asking that 
$\fk^j$ acting on $M \otimes -\psi$ takes
$F_i^{KK} M \otimes -\psi$ to $F_{i+j}^{KK} M \otimes -\psi$.

\end{itemize}

Suppose now that $\psi$ is $\bG_m$-equivariantly
extended to $\fh$.
We also suppose that $K$ is unipotent, so $C^{\dot}(\fh,K,-)$
coincides with $C^{\dot}(\fh,-)$ as a non-filtered functor. 
Then the KK filtration on $C^{\dot}(\fh,K,M \otimes -\psi)$
is similar to the filtration from Remark \ref{r:hc-cplx}; its $i$th term is:

\[
0 \to F_i^{KK} M \otimes -\psi \to 
\sum_j 
(\fh^j/\fk^j)^{\vee} \otimes 
F_{i+j+1}^{KK} M \otimes -\psi +
(\fh^j)^{\vee} \otimes F_{i+j}^{KK} M \otimes -\psi \to \ldotsplus
\]

\end{rem}

\subsection{Compact Lie algebras}

We now begin to move to an infinite dimensional setting.
Let $\fh$ be a \emph{profinite-dimensional Lie algebra},
so $\fh = \lim_i \fh/\fh_i$ is a filtered limit of
finite-dimensional Lie algebras $\fh/\fh_i \in \LieAlg(\Vect^{\heart})$, 
and with all structure maps being surjective. Of course,
$\fh_i \subset \fh$ indicates the corresponding
normal open Lie subalgebra, which is of the requisite type.

Following \cite{dmod-aff-flag} \S 22-23, we define:

\[
\fh\mod \coloneqq \underset{i}{\colim} \, 
\fh/\fh_i\mod \in \DGCat_{cont}
\]

\noindent where our structure functors are forgetful
functors. By our assumptions, for
each structure map $\fh/\fh_i \onto \fh/\fh_j$,
the induced functor:

\[
\Oblv: \fh/\fh_j\mod \to \fh/\fh_i\mod
\]

\noindent has a continuous
right adjoint: it is Lie algebra cohomology with respect to
the $\fh_j/\fh_i$. Therefore, the above colimit
is also the limit under these right adjoint functors.

By Lemma \ref{l:t-str}, $\fh\mod$ has a canonical $t$-structure
compatible with filtered colimits with heart the abelian 
category of discrete\footnote{Recall that
these are $\fh$-modules $V$ such that the stabilizer
of any vector in $V$ is open in $\fh$.} 
$\fh$-modules. 
Moreover, if our indexing category is countable,
$\fh\mod^+$ is the
(bounded below) derived category of this abelian category. 

We see that $\fh\mod$ is compactly generated, and that
it has a canonical trivial representation $k \in \fh\mod$ that
is compact. Therefore, we have a continuous 
functor $C^{\dot}(\fh,-): \fh\mod \to \Vect$, which is defined
as the complex of maps from the trivial representation.

\begin{rem}

Note that the $t$-structure on $\fh\mod$ is not 
necessarily left complete.
Indeed, suppose $\fh$ is abelian and infinite-dimensional. 
Then\footnote{In what follows, $\fh^{\vee}$ should always be understood
as the \emph{continuous} dual to $\fh$.}
 $\Ext_{\fh\mod}^{\dot}(k,k) = \Lambda^{\dot} \fh^{\vee}$,
so there are non-zero maps $k \to k[n]$ for each $n \geq 0$.
If the $t$-structure were left complete, we would
have $\oplus_{n\geq 0} k[n] \isom \prod_{n \geq 0} k[n]$
(proof: consider the Postnikov tower for the LHS).
But this is impossible: we would have
constructed a map $k \to \oplus_{n \geq 0} k[n]$ that would
not factor through any finite direct sum, contradicting the
compactness of $k$.

In fact, the $t$-structure is not even left separated.
Here is one explicit way to see this. 
Then $\fh\mod$ is canonically \emph{self-dual}
in the sense of \cite{dgcat}: indeed,
each $\fh/\fh_i\mod$ has a canonical Serre self-duality,
and we tautologically have:\footnote{
Dualizability is no issue because we are in
a co/limit situation.}

\[
\fh\mod^{\vee} = \underset{\Oblv^{\vee}}{\lim} \fh/\fh_i\mod 
\]

\noindent where the notation indicates the
limit under the functors \emph{dual} to the forgetful functors;
these are given by \emph{co}invariants with respect
to the kernels, which differ from the invariants
by a shift and tensoring with a 1-dimensional
representation, according to Lemma \ref{l:lie-shift}. 
This readily implies the claim: we should replace
Serre self-duality on each $\fh/\fh_i\mod$ by its
composition with the functor of shifting by $\dim \fh/\fh_i$
and tensoring with the determinant of the adjoint
representation.

It follows that we have an equivalence
$\bD:\Pro(\fh\mod^{c,op})^{op} \simeq \fh\mod$,
where $\fh\mod^c$ indicates the subcategory of
compact objects. The objects $U(\fh/\fh_i)$ are compact
in $\fh\mod$ (and even generate), and form a filtered projective
system in the obvious way; we denote this
pro-object by $U(\fh)$. The object
$\bD U(\fh) \in \fh\mod$ is then obviously non-zero,
but lives in $\cap_n \fh\mod^{\leq -n}$ because
$\bD U(\fh) = \colim_i \bD U(\fh/\fh_i)$, and
$\bD U(\fh/\fh_i)$ lies in cohomological degree
$-\dim \fh/\fh_i$.

\end{rem}

\begin{rem}

Note that $\fh^{\vee}$ is a Lie coalgebra, so there is a general
formalism of taking comodules over it. It is straightforward to show that
$\fh^{\vee}\comod$ is the left completion of $\fh\mod$.

\end{rem}

\begin{warning}

For $M \in \fh\mod^+$, $C^{\dot}(\fh,M)$ is computed by a standard
complex, but this is \emph{not} true for general $M \in \fh\mod$.
To make this precise, note that for any $M \in \fh\mod$, one can form a 
semi-cosimplicial diagram
$M \overset{\coact}{\underset{0}{\rightrightarrows}} 
M \otimes \fh^{\vee} \rightrightrightarrows
\ldots \in \Vect$ and the canonical morphism:

\[
C^{\dot}(\fh,M) \to \underset{\bDelta_{inj}} {\lim} \, 
\big(
M \rightrightarrows M \otimes \fh^{\vee} 
\rightrightrightarrows
\ldotsplus
\big)
\]

\noindent This map is an equivalence for $M \in \fh\mod^+$,
but not for general $M$. Indeed, considering the right hand side 
as a functor in the variable $M$, it is easy to see that it will not
commute with colimits.

\end{warning}

\subsection{}\label{ss:fil-cpt-lie}

We have a canonical filtration
on $\fh\mod$. Indeed, this follows immediately from the
fact that the functors $\Oblv: \fh/\fh_j \mod \to \fh/\fh_i\mod$
are filtered. 

We have:

\[
\fh\mod^{cl} \simeq \underset{i}{\colim} \, \QCoh((\fh/\fh_i)^{\vee})
\]

\noindent where the colimit is under $*$-pushforward functors
along the closed embeddings $\alpha_{i,j}:(\fh/\fh_j)^{\vee} \into (\fh/\fh_i)^{\vee}$.

We claim that $\fh\mod^{cl}$ is canonically isomorphic
to\footnote{Recall that tensoring with the dualizing sheaf
induces an equivalence $\QCoh(\fh^{\vee}) \isom \IndCoh(\fh^{\vee})$.
We prefer to write the category of $\IndCoh$ rather than $\QCoh$
though because the notation is somewhat simpler.

(Note that the place where this equivalence is 
shown, \cite{indschemes} Theorem 10.0.7, has a
countability hypothesis. This assumption is verified for us 
in our applications,
so the reader may safely assume it in this section. 
But in fact, one readily verifies that this assumption is only
used in finding nice presentations for a fairly general class of
indschemes;
this is no problem for our indscheme $\fh^{\vee}$, 
so one finds that the countability is not needed
in applying their method in the present case.
}
$\IndCoh(\fh^{\vee})$, where $\fh^{\vee}$ is considered
as an indscheme. Indeed, in the standard
$\IndCoh$ notation from \cite{indcoh}, we have commutative diagrams:

\[
\xymatrix{
\IndCoh((\fh/\fh_j)^{\vee}) \ar[d]^{\Psi} \ar[r]^{\alpha_{i,j,*}^{\IndCoh}} &
\IndCoh((\fh/\fh_i)^{\vee}) \ar[d]^{\Psi} \\
\QCoh((\fh/\fh_j)^{\vee}) \ar[r]^{\alpha_{i,j,*}} &
\QCoh((\fh/\fh_i)^{\vee})
}
\]

\noindent with vertical arrows equivalences; of course, these commutative
diagrams have the requisite compatibilities of higher category theory.
Therefore, $\fh\mod^{cl}$ is equivalent to this colimit; since 
the functors $\alpha_{i,j,*}^{\IndCoh}$ admit the continuous right
adjoints $\alpha_{i,j}^!$, we obtain the claim.

As before, the functor $C^{\dot}(\fh,-):\fh\mod \to \Vect$ is
filtered, and with semi-classical functor 
$\IndCoh(\fh^{\vee}) \to \Vect$ given as $!$-restriction to $0 \in \fh^{\vee}$.

\subsection{Actions of group schemes on filtered categories}

It is straightforward to generalize the above definitions to the
Harish-Chandra setting and to compute the outputs.
But it is quite clarifying in this setting 
to generalize the language of \S \ref{ss:q-gp-action},
so we do so.

\subsection{}\label{ss:ren-start}

Suppose $K$ is an affine group scheme; we write $K$
as a filtered limit $\lim_i K/K_i$ for
$K_i \subset K$ a normal subgroup scheme 
with $K/K_i$ an affine algebraic group.
Recall that an action of an algebraic group on a filtered category
induced a weak action the semi-classical category.
As a warm-up, we begin our discussion there.

\begin{defin}

A \emph{weak action} of $K$ on a category $\sC \in \DGCat_{cont}$ 
a $\QCoh(K)$-module structure on $\sC$, where $\QCoh(K)$ is
given the convolution monoidal structure.

A \emph{renormalized (weak)\footnote{If 
$K$ is an extension
of an affine algebraic group by a prounipotent one, 
as is always the case for us, this renormalization has no effect
in the setting of \emph{strong} group actions on categories:
this a consequence of the coincidence of invariants and
coinvariants for such categories (see \cite{dario-*/!}).
So we set the convention that the word \emph{renormalization} 
indicates that we are working with weak group actions.
}
action of $K$}
is an object of $\lim_i \QCoh(K/K_i)\mod$, where the structure functors
$\QCoh(K/K_i)\mod \to \QCoh(K/K_j)\mod$ are given by 
weak invariants with respect to $K_j/K_i$.

\end{defin}

\begin{notation}

We say a renormalized action of $K$ is \emph{on $\sC \in \DGCat_{cont}$}
if the compatible system of $\QCoh(K/K_i)$-module categories
is denoted by $\sC^{K_i\ren}$ and $\sC = \colim \, \sC^{K_i\ren}$.
Note that this is necessarily a co/limit situation, i.e., all structure
functors admit right adjoints.

\end{notation}

\begin{example}

Define $\Rep(K)$ as $\colim \, \Rep(K/K_i) \in \DGCat_{cont}$. 
(This category
should not be confused with $\QCoh(\bB K)$: they have 
$t$-structures and coincide on bounded below derived categories,
but $\Rep(K)$ is always compactly generated while $\QCoh(\bB K)$
may not be; rather, the latter category is the left completion of the
former.) 

Clearly this definition makes sense
for any affine group scheme. Moreover, $\Rep(K_i)$ is weakly acted
upon by $K/K_i$. So setting 
$\Vect^{K_i\ren} \coloneqq \Rep(K_i)$, we obtain a renormalized
action of $K$ on $\Vect$.

\end{example}

\begin{rem}

Note that the trivial representation in $\Rep(K)$ is compact,
so we have a \emph{continuous} group cohomology
functor $C^{\dot}(K,-):\Rep(K) \to \Vect$.

\end{rem}

\begin{rem}

Note that for any $\sC$ with a renormalized action of $K$,
$\Rep(K)$ acts on $\sC^{K\ren}$: indeed, writing
$\sC^{K\ren} = (\sC^{K_i\ren})^{K/K_i,w}$, we find
$\Rep(K_i)$ acts, so taking the colimit over $i$ gives the claim.

In fact, since the colimit defining $\Rep(K)$ is under (symmetric)
monoidal functors, and since:

\[
\QCoh(K/K_i)\mod \xar{(-)^{K/K_i,w}} \Rep(K/K_i)\mod
\]

\noindent is an equivalence (by 1-affineness of $\bB K/K_i$),
we find that the functor $\sC \mapsto \sC^{K\ren}$ 
is actually an equivalence between $\Rep(K)\mod$ and the
(2-)category of categories with a renormalized $K$-action. 

\end{rem}

\begin{rem}

Suppose that $X$ is a quasi-compact quasi-separated 
\emph{classical}\footnote{We have used the word \emph{scheme} 
throughout to mean classical scheme, but are emphasizing it
here because although it may seem unnecessary, it is important
for the Noetherian approximation we are applying.} 
scheme with a $K$-action. By Noetherian approximation,
we can write $X = \lim_i X^i$ under affine morphisms
and $K = \lim_i K/K_i$ as above such that
$X^i$ is finite type and $K/K_i$ acts on $X^i$, with these actions
being compatible in the natural sense as we vary $i$.
Finally, we assume that all structure morphisms among the
$X^i$ are \emph{flat}.

In this case, we obtain a renormalized action of $K$ on
$\QCoh(X)$ by setting:

\[
\QCoh(X)^{K\ren} \coloneqq \underset{i}{\colim} \, \QCoh(X^i/(K/K_i)).
\]

\noindent To emphasize a stacky perspective, we sometimes
use the notation:

\[
\QCoh^{ren}(X/K) = \QCoh(X)^{K\ren}.
\]

All structure maps here are affine, so this is a
co/limit situation. We similarly have $\IndCoh(X)^{K\ren}$; the
flatness of our structure maps implies that this is also a co/limit. 
(All of this is invariant
under choices of presentations as limits.)

Note that for $X = \Spec(k)$, we recover
$\Rep(K)$ by this construction.

More generally, let $X = \colim_j X_j$ be an indscheme, with
the $X_j$ schemes satisfying the above hypotheses.
Moreover, we assume the closed embeddings among the $X_j$
are finitely presented and \emph{eventually
coconnective} (e.g. regular). Note that $\QCoh(X) \coloneqq 
\lim \, \QCoh(X_j)$ is a co/limit, and similarly for 
$\IndCoh(X) \coloneqq \IndCoh(X_j)$.\footnote{Note that
taking $K = \{1\}$, our earlier discussion gave a makeshift definition of 
$\IndCoh(X_i)$.} Clearly each of these categories has a canonical
renormalized action of $K$.

\end{rem}

\subsection{}\label{ss:gp-sch-act}

We now discuss the filtered setting.

\begin{defin}

A \emph{(strong) action of $K$ on a filtered category} is an object
of:

\[
\underset{i}{\lim} \, \IndCoh(\Fil (K/K_i)_{dR})\mod(\QCoh(\bA_{\hbar}^1/\bG_m)\mod)
\]

\noindent i.e., a compatible system of filtered categories acted
on by the $K/K_i$.

\end{defin}

As before, we say this action is on $\Fil \sC$ if our compatible 
system is denoted 
$\Fil \sC^{K_i}$ and $\Fil \sC = \colim_i \Fil \sC^{K_i}$.
Again, this is a co/limit. 

If $K$ has a prounipotent tail, then 
$\sC$ inherits a (strong)
action of $K$ with all notation compatible, i.e.,
the generic fiber $\sC^{K_i}$ of $\Fil \sC^{K_i}$ is
the $K_i$-invariants for this action.

\begin{rem}

At the semi-classical level, we obtain an object of:

\[
\underset{i}{\lim} \, 
\QCoh((\fk/\fk_i)^{\vee}/(K/K_i))\mod \eqqcolon
\ShvCat^{ren}(\fk^{\vee}/K).
\]

\noindent Here we use the sheaf of categories notation
because, by \cite{ainfty}, 

\[
\ShvCat(\fk^{\vee}) \coloneqq 
\underset{i}{\lim} \, \QCoh((\fk/\fk_i)^{\vee})\mod \neq
\QCoh(\fk^{\vee})\mod.
\]

\noindent (Although the RHS is a full subcategory of the LHS.)
We use the superscript \emph{ren} because of the relationship
to our notion of a renormalized action of $K$ on a category.

In the above setup, we use the notation 
$\sC^{cl}|_{(\fk/\fk_i)^{\vee}}^{K_i\ren}$ to indicate
the corresponding object of $\QCoh((\fk/\fk_i)^{\vee}/(K/K_i))\mod$.
With this notation, we are encouraging the reader to imagine 
$\sC^{cl}$ as sitting over $\fk^{\vee}$
and equipped with a compatible weak (or better: renormalized) 
action of $K$. 

Then observe that the filtration on $\sC^K$ has 
$\sC^{K,cl} = \sC^{cl}|_0^{K\ren}$, and more generally,
$\sC^{K_i,cl} = \sC^{cl}|_{(\fk/\fk_i)^{\vee}}^{K_i\ren}$.

\end{rem}

\begin{rem}

Note that in the above setting, we may reformulate our
semi-classical data in saying that we have a compatible
system of categories 
$\sC^{cl}|_{(\fk/\fk_i)^{\vee}}$ equipped with renormalized
$K$-actions and $\QCoh((\fk/\fk_i)^{\vee})$-module
category structures and satisfying the natural compatibilities.

We let $\sC^{cl}$ denote the limit of this diagram. Note that this is actually
a co/limit situation. Then $\sC^{cl}$ can be thought of as the
global sections of the sheaf of categories on $\fk^{\vee}$ that
our datum induced. Note that $\sC$ itself actually is a filtered
category, with $\sC^{cl}$ as its semi-classical version.

Note that the place to be careful about making
mistakes in distinguishing sheaves of categories from
module categories is that we may have:

\[
\sC^{cl}|_{(\fk/\fk_i)^{\vee}} \neq 
\sC^{cl} \underset{\QCoh(\fk^{\vee})}{\otimes} \QCoh((\fk/\fk_i)^{\vee}).
\]

\end{rem}

\subsection{}\label{ss:hc-cpt}

Now suppose that we are in a Harish-Chandra setting:
we assume we are given a projective system
of Harish-Chandra data $(\fh/\fh_i,K/K_i)$
with $\fk/\fk_i \to \fh/\fh_i$ injective. Our two projective 
systems $\fh/\fh_i$ and $K/K_i$ are assumed to satisfy our
earlier (e.g., finiteness) hypotheses.

We then set:

\[
\Fil \fh\mod^K \coloneqq \underset{i}{\colim} \, 
\Fil \fh/\fh_i\mod^{K/K_i} \in \DGCat_{cont}.
\]

\noindent Note that this is a co/limit situation;
for the right adjoint to the forgetful functor:

\[
\Fil \fh/\fh_j\mod^{K/K_j} \to
\Fil \fh/\fh_i\mod^{K/K_i} 
\]

\noindent is given by 
(the filtered version of) 
Harish-Chandra cohomology with respect to $(\fh_j/\fh_i,K_j/K_i)$.

Note that this construction makes sense for each $K_i$ in place of
$K$. Moreover, $K/K_i$ acts on $\Fil \fh\mod^{K_i}$, and
we have natural compatibilities as we take invariants.
Therefore, the above data defines an action of $K$ on $\Fil \fh\mod$.

Note that $\fh\mod^{K,cl} = \IndCoh^{ren}((\fh/\fk)^{\vee}/K)$:
the calculation is the same as the one we gave for
$\fh\mod^{cl}$. More precisely, recall that our semi-classical data
is an object of $\ShvCat^{ren}(\fk^{\vee}/K)$: it is defined
by the compatible system:

\[
i \mapsto \IndCoh((\fk/\fk_i)^{\vee}/(K/K_i))
\] 

\noindent where for $K_i \subset K_j \subset K$, 
we note that $!$-pullback
along the morphism:

\[
(\fk/\fk_i)^{\vee}/K_i \to (\fk/\fk_j)/K_j
\]

induces an equivalence:

\[
\IndCoh((\fk/\fk_i)^{\vee}/(K/K_i)) 
\underset{\QCoh((\fk/\fk_i)^{\vee}/(K/K_i)) }{\otimes}
\QCoh(\fk/\fk_j)^{\vee}/(K/K_j)) \isom 
\IndCoh((\fk/\fk_i)^{\vee}/(K/K_i)) 
\]

\noindent because we can identify $\IndCoh$ with $\QCoh$ by
smoothness.

We have a natural filtered Harish-Chandra cohomology functor
$\fh\mod^K \to \Vect$ with expected semi-classical version given
by $!$-restriction to $0/K$ followed by group cohomology.

\subsection{}\label{ss:kk-cpt}

Now suppose in the above setting that we have
compatible $\bG_m$-actions on each $K/K_i$ and
$\fh/\fh_i$. Suppose moreover that we are given a 
$\bG_m$-equivariant character
$\psi:K \to \bG_a$ for the action of $\bG_m$ on $\bG_a$ by
inverse homotheties.

The construction of \S \ref{ss:kk-hc-fd} applies as is, giving
a Kazhdan-Kostant filtration on $\fh\mod^{K,\psi}$
fitting into a bifiltration with the
PBW filtration. 
It has similar properties to the finite-dimensional version,
up to the various differences we saw above between the 
finite and profinite-dimensional settings.

\subsection{}

Now suppose that we are given $\fh^0 \subset \fh$ an open
subalgebra, so $\fh/\fh^0$ is finite-dimensional. 
We assume the pair
$(\fh^0,K)$ satisfies the profinite-dimensional Harish-Chandra conditions
as above, so $\fk \subset \fh^0 \subset \fh$ and $\fh^0$ is
a $K$-submodule of $\fh$.

We have the following version of 
Corollary \ref{c:coh-ind} and Lemma \ref{l:coh-ind-hc-fd}.

\begin{lem}\label{l:coh-ind-cpt-hc}

\begin{enumerate}

\item\label{i:coh-ind-cpt-lie}
The forgetful functor $\fh\mod \to \fh^0\mod$ admits
a left adjoint $\ind_{\fh^0}^{\fh}$ \emph{as a filtered functor}.
The induced semi-classical functor $(\ind_{\fh^0}^{\fh})^{cl}$
is the $(\IndCoh,*)$ pullback functor:

\[
\IndCoh(\fh^{0,\vee}) \to \IndCoh(\fh^{\vee})
\]

\noindent i.e., the left adjoint to the $\IndCoh$-pushforward along
the projection $\fh^{\vee} \to \fh^{0,\vee}$.

There is a canonical isomorphism of filtered functors
$\fh\mod \to \Vect$:

\[
C^{\dot}\Big(\fh,\ind_{\fh^0}^{\fh}\big((-) \otimes 
\det(\fh/\fh^0)[\dim \fh/\fh^0]\big)\Big) = 
C^{\dot}(\fh^0,-)
\]

\noindent where $\det(\fh/\fh^0)[\dim \fh/\fh^0])$ is filtered
with a single jump in degree $\dim \fh/\fh^0$.

\item\label{i:coh-ind-cpt-hc} The functor 
$\ind_{\fh^0}^{\fh}$ is a morphism of filtered categories acted
on by $K$. The induced functor
$\ind_{\fh^0}^{\fh}: \fh^0\mod^K \to \fh\mod^K$ 
has semi-classical version:

\[
\IndCoh^{ren}((\fh^0/\fk)^{\vee}/K) \to \IndCoh^{ren}((\fh/\fk)^{\vee}/K)
\]

\noindent again given by $(\IndCoh,*)$-pullback.

There is a canonical isomorphism of
filtered functors $\fh^0\mod^K \to \Vect$:

\[
C^{\dot}\Big(\fh,K,\ind_{\fh^0}^{\fh}\big((-) \otimes 
\det(\fh/\fh^0)[\dim \fh/\fh^0]\big)\Big) = 
C^{\dot}(\fh^0,-).
\]

\item\label{i:coh-ind-cpt-kk} Suppose now that we are given the extra 
data of \S \ref{ss:kk-cpt}, and suppose that
we are given a character $\psi:\fh \to k$ extending
the same-named character on $\fk$.
Then there is a canonical
isomorphism of \emph{bi}filtered functors
$\fh^0\mod^{K,\psi} \to \Vect$:

\[
C^{\dot}\Big(\fh,K,\ind_{\fh^0}^{\fh}\big((-) \otimes (-\psi) \otimes 
\det(\fh/\fh^0)[\dim \fh/\fh^0]\big)\Big) = 
C^{\dot}(\fh^0,K, - \otimes (-\psi)).
\]

\noindent Each of these functors has semi-classical version:

\[
\IndCoh^{ren}(\psi+(\fh^0/\fk)^{\vee}/K) \to \Vect
\]

\noindent given by $!$-restriction to $\psi/K$ followed by 
$\Gamma^{\IndCoh}$ (i.e., group cohomology with respect to $K$).

\end{enumerate}

\end{lem}

\begin{proof}

These results follow immediately from their finite-dimensional
counterparts by passing to the limit.

\end{proof}

\subsection{Tate setting}\label{ss:tate-lie-start}

Our treatment here follows \cite{km-indcoh} at some points.

Suppose $\fh \in \Pro(\Vect^{\heart})$
is a Tate Lie algebra, i.e., $\fh$ is a limit under surjective
maps of (possibly infinite dimensional) vector spaces,
has a continuous Lie bracket, and an open 
profinite dimensional Lie subalgebra.\footnote{As
in \cite{beilinson-top-alg} \S 1.4, 
a topological Lie algebra structure on 
a Tate vector space automatically has a
basis by open Lie subalgebras.}

Suppose moreover that we are given a Harish-Chandra
datum $(\fh,K)$ with $K$ a group scheme and $\fk \into \fh$ an \emph{open}\footnote{This is a serious condition: for example,
$K$ can not be trivial if the topology on $\fh$ is non-trivial.} 
subalgebra. We are 
going to define a filtered category $\fh\mod$ acted on by $K$.

First, observe that the \emph{group prestack} $(\fh,K)$ 
from \S \ref{ss:hc-fd-start}
still makes sense,
receiving a canonical ind-affine nil-isomorphism $\bB K \to \bB (\fh,K)$.
The definition from \emph{loc. cit}. does not make sense as is:
de Rham spaces and formal completions are best avoided in 
infinite type. Instead, we assume $\fh = \Lie(H)$ for a
group indscheme $H$ with $K \subset H$ a compact open 
subgroup;\footnote{In particular, we assume $H/K$ is ind-finite type:
this forces $H$ to be reasonable in the sense of \cite{hitchin}.}
then $(\fh,K)$ is the formal completion of $K$ in $H$.
In general, one can appeal e.g. to \cite{hitchin} 7.11.2 (v) for the 
construction.

We form the simplicial diagram:

\[
\ldots 
K \backslash (\fh,K) \overset{K}{\times} (\fh,K) /K
\rightrightrightarrows
K\backslash (\fh,K) /K
\rightrightarrows
\bB K
\]

\noindent given by applying the Cech construction
to the morphism $\bB K \to \bB (\fh,K)$; note that the geometric
realization of this diagram is $\bB (\fh,K)$.
Moreover, note that each term in the simplicial diagram is of the
form ``an ind-finite type indscheme modulo an action of $K$."
Therefore, $\IndCoh^{ren}$ makes sense for each term
of this diagram. We define $\fh\mod^K$ as the totalization:

\[
\fh\mod^K \coloneqq 
\Tot \Big(
\Rep(K) = \IndCoh^{ren}(\bB K) \rightrightarrows
\IndCoh^{ren}(K\backslash (\fh,K) /K)
\rightrightrightarrows
\ldots 
\Big)
\]

\noindent Since all the morphisms in our simplicial diagram
are ind-affine nil-isomorphisms, this is a co/limit.
The Beck-Chevalley formalism easily implies that
the forgetful functor $\Oblv:\fh\mod^K \to \Rep(K)$ is
monadic, and in particular, admits
a left adjoint $\ind_{\fk}^{\fh}$.

It is easy to see that $\ind_{\fk}^{\fh} \circ \Oblv$
is conservative and $t$-exact, so $\Oblv$ is monadic
and $\fh\mod^K$ has a canonical $t$-structure with
$\Oblv$ and $\ind_{\fk}^{\fh}$ both $t$-exact.
We have:

\begin{lem}[Bernstein-Lunts, \cite{bernstein-lunts}]\label{l:bernstein-lunts-tate}

If $K = \lim_i K/K_i$ is a countable inverse limit,
$\fh\mod^{K,+}$ is the
bounded below derived category of the heart of its $t$-structure.

\end{lem}

Indeed, $\Rep(K)^+ = D^+(\Rep(K)^{\heart})$ by
Lemma \ref{l:t-str} (which is where the countability hypothesis
enters),
and then the same argument as in the finite-dimensional 
Lemma \ref{l:bernstein-lunts-fd} applies.

\subsection{}

Now note that the deformations defined in 
\cite{grbook} \S IV.5.2 makes sense in the infinite type setup and
are well-behaved in our setup.
Applying this to $\bB K \to \bB(\fh,K)$, we obtain a
prestack $\Fil \bB (\fh,K)$ over $\bA_{\hbar}^1/\bG_m$ with
special fiber $(\bB (\fh/\fk)_0^{\wedge})/K$.

We can form the above Cech construction along this deformation
and imitate the above construction to obtain a filtration
on $\fh\mod^K$. We claim that $\fh\mod^{K,cl}$ is canonically
isomorphic to $\QCoh^{ren}((\fh/\fk)^{\vee}/K)$, with
$(\fh/\fk)^{\vee}$ the continuous dual considered as an affine scheme.
Indeed, we need to compute:

\[
\Tot \Big(
\IndCoh^{ren}(\bB K) \rightrightarrows
\IndCoh^{ren}( (\fh/\fk)_0^{\wedge}/K)
\rightrightrightarrows
\ldots 
\Big).
\]

\noindent Here it makes sense to replace
$\fh/\fk$ by any $K$-representation $V$. Since any $K$-representation
is the union of its finite-dimensional representations,
we find that the above is $\IndCoh^{ren}(\bB (\fh/\fk)_0/K)$
(where the renormalization makes sense because
$\bB (\fh/\fk)_0$ is the appropriate colimit of the classifying
stacks acted on by $K$ corresponding to finite-dimensional 
finite-dimensional subrepresentations of $\fh/\fk$).
Clearly $\IndCoh^{ren}(\bB (\fh/\fk)_0^{\wedge}/K) \simeq 
\QCoh^{ren}((\fh/\fk)^{\vee}/K)$ as desired.

Finally, note that if $K = \lim_i K/K_i$ as before, then 
the above construction makes sense for each of the compact
open normal subgroup schemes $K_i \subset K$. 
Moreover, $K_i/K_j$-invariants for $\Fil \fh\mod^{K_j}$ are 
easily seen to give $\Fil \fh\mod^{K_i}$.

This is exactly the data to define the
filtered category $\fh\mod$ acted on by $K$. Note that:

\[
\fh\mod^{cl} = \IndCoh(\fh^{\vee}) \coloneqq
\underset{i}{\colim} \, \QCoh((\fh/\fk_i)^{\vee}) \in \DGCat_{cont}
\]

\noindent with the colimit being under $*$-pushforward functors;
we label this colimit as $\IndCoh$ 
for the same reason as in \S \ref{ss:fil-cpt-lie}.

More precisely, recall that our semi-classical data is
a renormalized sheaf of categories on $\fk^{\vee}/K$. 
This sheaf of categories is described in the same way as
in \S \ref{ss:hc-cpt}.

\subsection{}

Next, we note that the above makes sense even if
$\fk$ is not an open subalgebra. 

More precisely, and with apologies
for the notation change, choose $H_0$ a group scheme
and a Harish-Chandra datum $(H_0,\fh)$ with $\fh_0 \subset \fh$
open. Then suppose that $K$ is a group subscheme of $H_0$,
with no hypothesis that it be compact open (e.g., $K$ could be trivial).

Then as above, we have an action of $H_0$ on $\Fil\fh\mod$. We claim
that given any action of $H_0$ on $\Fil \sC$, we can restrict to obtain an 
action of $K$ on $\Fil \sC$.

Indeed, if $H_0 = \lim_i \, H_0/H_i$ for $H_i$ a normal subgroup
scheme, note that $H_i K$ is a compact open subgroup scheme of $H$;
we then set:

\[
\Fil \sC^K \coloneqq \underset{i}{\colim} \, \Fil \sC^{H_i K} \in \DGCat_{cont}.
\] 

\noindent We remark that this is a co/limit. Replacing $K$ by
a compact open subgroup (of $K$), we obtain the requisite data.

\subsection{Semi-infinite cohomology}\label{ss:sinf}

We now make a more stringent assumption on $\fh$:
suppose that it is a \emph{union}\footnote{Note that this 
assumption is not satisfied for the Kac-Moody Lie algebra.
There is a semi-infinite cohomology theory for such algebras, but
it is a more subtle and will not be needed in this paper.}
of open pro-finite dimensional
subalgebras $\fh = \colim_i \, \fh_i$. We fix an initial index ``$0$" and let
$\fh_0$ denote the corresponding open subalgebra.

We assume that for every $\fh_i \subset \fh_j$, the action of
$\fh_i$ on $\det(\fh_j/\fh_i)$ is trivial; e.g., this is automatically
the case if $\fh$ is ind-pronilpotent. For later use, we observe that in 
this case:

\begin{equation}\label{eq:pullout-det}
\ind_{\fh_1}^{\fh_2}(M \otimes \det(\fh_j/\fh_i)) =
\ind_{\fh_1}^{\fh_2}(M) \otimes \det(\fh_j/\fh_i)
\end{equation}

\noindent since we are just tensoring by a line. (This line 
is essentially just a placeholder, ensuring the canonicity of
various isomorphisms.) 

In this case, we obtain a \emph{semi-infinite cohomology}
functor:

\[
C^{\sinf}(\fh,\fh_0,-): \fh\mod \to \Vect
\]

\noindent defined as follows. 

Note that any of the compact open subalgebras $\fh_i$,
we have a \emph{forgetful} functor 
$\fh\mod \to \fh_i\mod$, which is conservative and admits
the left adjoint $\ind_{\fh_i}^{\fh}$.
Then we claim that the induced map:

\[
\fh\mod \to \underset{i}{\lim} \, \fh_i\mod
\]

\noindent is an equivalence. Indeed, both sides are clearly
monadic over $\fh_0\mod$.
This is a co/limit, so we also obtain:

\[
\underset{i}{\colim} \, \fh_i\mod \isom \fh\mod \in \DGCat_{cont}
\]

\noindent where we are using the induction functors on the left
hand side.

We then define a functor:

\[
C^{\sinf}(\fh,\fh_0,\ind_{\fh_i}^{\fh}(-)): \fh_i\mod \to \Vect
\]

\noindent as:

\[
C^{\dot}(\fh_i,(-) \otimes \det(\fh_i/\fh_0)[\dim \fh_i/\fh_0]).
\]

\noindent By Lemma \ref{l:coh-ind-cpt-hc} \eqref{i:coh-ind-cpt-lie}
and \eqref{eq:pullout-det},
for $\fh_i \subset \fh_j \subset \fh$,
we have canonical isomorphisms:

\[
C^{\sinf}(\fh,\fh_0,\ind_{\fh_i}^{\fh}(-)) \simeq 
C^{\sinf}\big(\fh,\fh_0,\ind_{\fh_j}^{\fh_i} \ind_{\fh_i}^{\fh_j}(-)\big).
\]

\noindent These are compatible as we vary indices, so we obtain
the desired functor $C^{\sinf}(\fh,\fh_0,-)$.

\begin{notation}

This functor depends only in a
mild way on $\fh_0$, but it is convenient for our purposes
to keep it in the notation.

\end{notation}

\begin{rem}\label{r:sinf-colim}

Here is another perspective. 
Note that by the general co/lim formalism for a filtered diagram,
any $M \in \fh\mod$ can be written as 
$\colim_i \, \ind_{\fh_i}^{\fh} M \in \fh\mod$, 
i.e., we forget down to $\fh_i$ and then induce. So we find:

\[
C^{\sinf}(\fh,\fh_0,M) = 
\underset{i}{\colim} \, C^{\dot}(\fh_i,M \otimes \det(\fh_i/\fh_0)[\dim \fh_i/\fh_0]).
\]

\noindent Applying this formula to $M \in \fh\mod^{\heart}$
(or a bounded below chain complex of such objects) and
computing $C^{\dot}(\fh_i,-)$ by the standard resolution, 
one recovers the usual complex computing semi-infinite cohomology
in this case; i.e., this perspective recovers the classical one.

\end{rem}

The above analysis applies just as well in the filtered setting,
so we obtain a canonical filtration on $C^{\sinf}(\fh,\fh_0,-)$.
The semi-classical functor:

\[
\IndCoh(\fh^{\vee}) \to \Vect
\]

\noindent is given by $!$-restricting\footnote{Precisely, recall that
$\IndCoh(\fh^{\vee})$ was defined as the colimit under pushforwards
of $\QCoh$ of its reasonable subschemes;
then our $!$-restriction here is the right adjoint to the
structural functor $\QCoh((\fh/\fh_0)^{\vee}) \to \IndCoh(\fh^{\vee})$.}
to obtain an object of $\QCoh((\fh/\fh_0)^{\vee})$; and 
noting that this is $\QCoh$ of an affine scheme, we then
take $*$-restriction to $0 \in (\fh/\fh_0)^{\vee}$.
Indeed, by Lemma \ref{l:coh-ind-cpt-hc}, the functor
$C^{\dot}(\fh_i,(-) \otimes \det(\fh_i/\fh_0)[\dim \fh_i/\fh_0])$ semi-classically
gives the functor:

\[
\IndCoh(\fh_i^{\vee}) \to \Vect
\]

\noindent that is the composition of 
$!$-restricting to $(\fh_i/\fh_0)^{\vee}$ and
then $*$-restricting to $0$; passing to the colimit in $i$ gives the claim.

\begin{rem}\label{r:sinf-pbw-fmla}

From the perspective of Remark \ref{r:sinf-colim}, we have:

\[
F_i C^{\sinf}(\fh,\fh_0,M) = 
\underset{i}{\colim} \, 
F_{i-\dim \fh_i/\fh_0}
C^{\dot}(\fh_i,M \otimes \det(\fh_i/\fh_0)[\dim \fh_i/\fh_0]).
\]

\noindent (The shift in indexing reflects the repeatedly emphasized fact that 
that our determinant lines are considered as 
filtered with a jump in degree $\dim \fh_i/\fh_0$.)

\end{rem}

\subsection{}\label{ss:sinf-hc}

Now observe that the above
all makes sense in the Harish-Chandra setting as well.
Indeed, if we have the Harish-Chandra datum $(\fh,K)$ with 
$K$ a group scheme, then note that $\fk \subset \fh_i$
for $i\gg 0$, so the formula:

\[
C^{\sinf}(\fh,\fh_0,K,-) \coloneqq 
\underset{i}{\colim} \, C^{\dot}(\fh_i,K,M \otimes \det(\fh_i/\fh_0)[\dim \fh_i/\fh_0])
\] 

\noindent makes sense by cofinality and defines a filtered functor:

\[
C^{\sinf}(\fh,\fh_0,K,-):\fh\mod^K \to \Vect.
\]

\noindent If $\fh_0 = \fk$, then the corresponding semi-classical
functor:

\[
\IndCoh^{ren}((\fh/\fk)^{\vee}/K) \to \Vect
\]

\noindent is given by $!$-restriction to $0/K$ followed by group
cohomology with respect to $K$. 
If $\fh_0 \subset \fk$ with $\det(\fk/\fh_0)$ a trivial $\fh_0$-representation,
then we can twist to reduce to this case,
i.e., the functor:

\[
C^{\sinf}(\fh,\fh_0,K, (-) \otimes \det(\fk/\fh_0)^{\vee}[-\dim \fk/\fh_0]):
\fh\mod^K \to \Vect.
\]

\noindent will have the semi-classical functor described above.

\begin{notation}\label{notation:h_0=k}

When $\fh_0 = \fk$, we use the notation
$C^{\sinf}(\fh,K,-)$ in place of $C^{\sinf}(\fh,\fk,K,-)$.

\end{notation}

\subsection{KK version}

Suppose now that $(\fh,K)$ is equipped with a $\bG_m$-action
as before.
Suppose moreover that $K$ is equipped with a character
$\psi:K \to \bG_a$ that is $\bG_m$-equivariant for the
inverse homothety action on the target.

Then $\fh\mod^{K,\psi}$ makes sense, and inherits PBW and KK
filtrations fitting into a bifiltration as before; this is immediate 
from our earlier constructions and the general KK formalism.

We have:

\[
\begin{gathered}
\fh\mod^{K,\psi,PBW\mathendash{cl}} = \IndCoh^{ren}((\fh/\fk)^{\vee}/K) \\
\fh\mod^{K,\psi,KK\mathendash{cl}} = \IndCoh^{ren}(\psi+(\fh/\fk)^{\vee}/K)
\end{gathered}
\]

\noindent as before. The KK filtration on the former and the PBW
filtration on the latter are as in Remark \ref{r:psi-filts-on-cl}.

\begin{rem}

Suppose now that $\fh$ satisfies the hypotheses of
\S \ref{ss:sinf}, and let $\fh_0$ be as in \emph{loc. cit}.
Suppose that $\psi:\fk\to k$ is extended to a character
$\psi:\fh\to k$, so the functor:

\[
C^{\sinf}\big(\fh,\fh_0,K,(-) \otimes -\psi\big): \fh\mod^{K,\psi} \to \Vect
\]

\noindent makes sense. Note that if the subalgebras 
$\fh_i \subset \fh$ are all graded subalgebras (with respect to
the grading on $\fh$ induced by the $\bG_m$-action),
and $\psi:\fh \to k$ is graded for the degree $-1$ grading on the target,
then $C^{\sinf}\big(\fh,\fh_0,K,(-) \otimes -\psi\big)$ is naturally bifiltered.
If $\fk = \fh_0$, then the induced semi-classical functors:

\[
\begin{gathered}
\QCoh^{ren}((\fh/\fk)^{\vee}/K) \to \Vect \\
\QCoh^{ren}(\psi+(\fh/\fk)^{\vee}/K) \to \Vect
\end{gathered}
\]

\noindent are given by $*$-restriction to $0/K$ and $\psi/K$ followed
by global sections, i.e., group cohomology with respect to $K$.
Here we note that our hypothesis means $\fk$ is open in
$\fh$, so 
$(\fh/\fk)^{\vee}$ is a \emph{scheme},
not an indscheme; so our definition of $\IndCoh^{ren}$ in this
infinite type setting
means that it tautologically coincides with $\QCoh^{ren}$.

\end{rem}

\subsection{Central extensions}

Finally, we explain the straightforward extension of the
the above to \emph{central extensions} of $\fh$.

\subsection{}

Suppose that in the above notation,
we are given $\widehat{\fh}$ a Tate Lie algebra and a central extension:

\[
0 \to k \to \widehat{\fh} \to \fh \to 0
\]

\noindent of $\fh$ by 
the abelian Lie algebra $k$.

We suppose that $K$ is as in \S \ref{ss:tate-lie-start},
so $\fk \subset \fh$ is an open subalgebra. We suppose moreover
that we are given a Harish-Chandra datum $(\widehat{\fh},K \times \bG_m)$
compatible in the sense that the projections:

\[
\begin{gathered}
\widehat{\fh} \to \fh \\
K \times \bG_m \xar{p_1} K
\end{gathered}
\]

\noindent induce a morphism of Harish-Chandra data. 
Moreover, we assume that the structural morphism 
$k \xar{x \mapsto (0,x)} \fk \times k = \Lie(K \times \bG_m) \to \widehat{\fh}$
is the given embedding of $k$ into $\widehat{\fh}$; and that
the $\bG_m \subset K \times \bG_m$-action on $\widehat{\fh}$
is trivial.

\begin{rem}

Note that the extension $\widehat{\fh}$ is canonically split over $\fk$.

\end{rem}

\begin{rem}

Of course, everything that follows generalizes to the case where
the central $\bG_m$ is replaced by a torus. 
We mention this because it is necessary
for the setting of affine Kac-Moody algebras for non-simple $\fg$,
as in \S \ref{ss:km-conv}.

\end{rem}

\subsection{}

We want to form the DG category $\widehat{\fh}_1\mod$, which
morally is the DG category of $\widehat{\fh}$-modules on
which $1 \in k \subset \widehat{\fh}$ acts by the identity.
For this, we will construct an action of $\QCoh(\bA^1)$
on $\widehat{\fh}\mod$; it then makes sense to take the fiber at
$1 \in \bA^1$ (by tensoring with the restriction
to $1$ functor $\QCoh(\bA^1) \to \Vect$).

Indeed, note that our Harish-Chandra datum induces
an action of $\bG_{m,dR}$ on $\bB (\widehat{\fh},K)$: $\bG_m$ acts
because it acts on $\fh$ commuting with $K$, and
the action of the formal group is trivial because our Harish-Chandra
data was extended to $(\widehat{\fh},K \times \bG_m)$.
Moreover, the underlying $\bG_m$-action is canonically trivial,
because $\bG_m$ acts trivially on $\widehat{\fh}$.
Therefore, we obtain an action of 
$\bB \widehat{\bG}_m = \bG_{m,dR}/\bG_m$ on 
$\bB (\widehat{\fh},K)$.\footnote{Here
$\widehat{\bG}_m$ is the formal group of $\bG_m$,
i.e., the hat notation is being used in a different way
from $\widehat{\fh}$.} 

Then observe that
$\QCoh(\bB \widehat{\bG}_m) \simeq \QCoh(\bA^1)$,
with the convolution monoidal structure on the left hand side 
corresponding to the tensor product on the right hand side, so
we obtain the desired action on $\widehat{\fh}\mod$ by definition of this
category, and therefore the definition of the category $\widehat{\fh}_1\mod$.

\begin{example}

A splitting of the Lie algebra
morphism $\widehat{\fh} \to \fh$ gives an identification
$\widehat{\fh}_1\mod \simeq \fh\mod$ compatible
with all extra structures.

\end{example}

\subsection{}

There is a filtered version of the above, quite similar
to \S \ref{ss:psi-pbw-filt}.
We use the $\bG_m$ action on $\Fil \widehat{\fh}\mod$ as above,
finding that $\widehat{\fh}_1\mod$ is filtered
with semi-classical category
$\IndCoh(\fh^{\vee})$. So the situation is not sensitive to the central 
extension.

The rest of the usual package generalizes as is to this setting.
We have an action of $K$ on $\Fil \widehat{\fh}_1\mod$, so 
obtain a filtration on $\widehat{\fh}_1\mod^K$. 
If we have compatible $\bG_m$-actions on $\widehat{\fh}$ and $K$
with $k \subset \widehat{\fh}$ acted on trivially,  
we obtain a bifiltration on $\widehat{\fh}_1\mod^K$; if $K$ is equipped with 
an appropriately $\bG_m$-equivariant additive character, we obtain a
bifiltration on $\widehat{\fh}_1\mod^{K,\psi}$ as well. The semi-classical
categories are as expected, i.e., the same as if we worked with $\fh$ instead
of its central extension, and the various restriction and induction 
functors satisfy the standard functoriality properties at 
the semi-classical level.

\section{Proof of Lemma \ref{l:amp}}\label{a:av}

\subsection{} 

We give two proofs of this result: a geometric one
based in the theory of $D$-modules, and 
a representation theoretic one. 

The former approach, which was sketched after the
statement of Lemma \ref{l:amp},
is more versatile
and conceptual. But for technical reasons,
we only know how to apply this method 
for $n$ sufficiently large.\footnote{We remark
that this is enough to establish Theorem \ref{t:aff-skry}.
In turn, this is enough to show Corollary \ref{c:ff-crit},
which implies Lemma \ref{l:amp} in general. Also, for
$G = GL_n$, Beraldo's refinement \cite{dario-*/!} of
Theorem \ref{t:!-avg} can be applied to obtain Lemma \ref{l:amp} 
at general level (using $D$-module methods).}

The second one is more ad hoc. The idea is that we can
compute the associated graded of this functor
using (the proof of) Theorem \ref{t:!-avg-cl} and verify
exactness here. However, the problems with
unboundedness of Kazhdan-Kostant filtrations come in here,
and we use some tricks to circumvent this. 

\begin{rem}

There is a homology between the two approaches:
$(\check{\rho},\alpha_{max})$ is involved in the technical issues
on both sides. Perhaps this hints at a more systematic solution.

\end{rem}

\subsection{}

We begin with the $D$-module approach.
Since $\sC = \widehat{\fg}_{\kappa}\mod$ and its Harish-Chandra
variants are fairly general
examples of categories acted on by a group,
we introduce some axiomatics about the relationship between
such group actions and $t$-structures. We then establish
general results about $\Av_*$ and $\Av_!$, and 
deduce Lemma \ref{l:amp} from here.

\subsection{Axiomatics} 

Fix $H$ an affine algebraic 
group and $\sC$ a DG category acted on weakly by $H$.

Suppose $\sC$ is equipped with a $t$-structure compatible with 
filtered colimits. Note that $\QCoh(H) \otimes \sC$ inherits a $t$-structure:
$(\QCoh(H) \otimes \sC)^{\leq 0}$ is generated under colimits
by objects $\sO_H \boxtimes \sF$ for $\sF \in \sC^{\leq 0}$.

\begin{lem}\label{l:tstr-gp}

The following conditions are equivalent:

\begin{enumerate}

\item\label{i:tstr-3} The functor $\Oblv \circ \Av_*^w: \sC \to \sC$ is $t$-exact.

\item\label{i:tstr-1} The functor $\act: \QCoh(H) \otimes \sC \to \sC$ is
$t$-exact.

\item\label{i:tstr-2} The functor $\coact: \sC \to \QCoh(H) \otimes \sC$
is $t$-exact.

\item\label{i:tstr-4} $\sC^{H,w}$ admits a $t$-structure such that 
$\Oblv$ and $\Av_*^w$ are $t$-exact.

\item\label{i:tstr-5} The $\QCoh(H)$-linear equivalence:

\begin{equation}\label{eq:tstr-eq}
\QCoh(H) \otimes \sC \to \QCoh(H) \otimes \sC 
\end{equation}

\noindent induced by $\coact: \sC \to \QCoh(H) \otimes \sC$
is $t$-exact.

\end{enumerate}

\end{lem} 

\begin{proof}

Note that we have a functor
$p_2^*:\sC \xar{\sF\mapsto \sO_H \boxtimes \sF}  \QCoh(H) \otimes \sC$,
which admits the conservative right adjoint $p_{2,*}$.
We claim $p_2^*$ and $p_{2,*}$ are $t$-exact.
Indeed, $p_2^*$ is tautologically right $t$-exact, so $p_{2,*}$ is
left $t$-exact. But from the definition of the $t$-structure, we
see $p_{2,*}$ is right $t$-exact as well, so $t$-exact. Then
since $p_{2,*}p_2^* = \sO_H \otimes -$ is $t$-exact, we obtain
the $t$-exactness of $p_2^*$ as well.

We will deduce the other conditions from \eqref{i:tstr-3}.
Since e.g. \eqref{i:tstr-4} obviously implies it, this suffices.

Recall from the Beck-Chevalley formalism that we have:

\[
p_{2,*} \coact = \act p_2^* = \Oblv\Av_*^w.
\]

\noindent Since $p_{2,*}$ is $t$-exact and conservative, we see that
\eqref{i:tstr-3} implies \eqref{i:tstr-1}.

We now deduce \eqref{i:tstr-2} from \eqref{i:tstr-3} 
and the consequent \eqref{i:tstr-1}.
Note that $t$-exactness of $\coact$ implies that
its right adjoint $\act$ is left $t$-exact. Since
$p_2^*(\sC^{\leq 0})$ generates
$(\QCoh(H) \otimes \sC)^{\leq 0}$, it suffices to show
$\act p_2^*$ is left $t$-exact, but this is clear since
$ \Oblv\Av_*^w$ is $t$-exact by assumption.

For \eqref{i:tstr-4}, observe that 
$\sC^{H,w}$ is the limit of a cosimplicial diagram with
$t$-exact structure maps in the underlying semi-cosimplicial
diagram (by \eqref{i:tstr-1}).
This implies the existence of a $t$-structure
with $\Oblv:\sC \to \sC^{H,w}$ $t$-exact. To see
$\Av_*^w$ is $t$-exact, it suffices to see that
$\Oblv \Av_*^w$ is, but this is given.

Finally, note that the equivalence \eqref{eq:tstr-eq}
intertwines the functors $p_2^*$ and $\coact$. 
Therefore, it suffices to see that 
$\coact(\sC^{\leq 0})$ generates 
$(\QCoh(H) \otimes \sC)^{\leq 0}$ under colimits.
But this follows because $\act$ is $t$-exact and conservative.

\end{proof}

If these equivalent conditions are satisfied, we say the $t$-structure
is \emph{compatible} with the weak action of $H$. 

\subsection{}

Now suppose that $H$ acts strongly on $\sC$.

We say that the action is compatible with the $t$-structure
if it is compatible for the weak action.
It is equivalent to say that:

\[
\coact[-\dim H]:\sC \to D(H) \otimes \sC
\]

\noindent is $t$-exact.
As in the weak setting, $\sC^H$ inherits a $t$-structure
with $\Oblv:\sC^H \to \sC$ being $t$-exact.

\begin{lem}\label{l:av-*-bd}

In the above setting, the functor $\Av_*:\sC \to \sC^H$
has cohomological amplitude $[0,\dim H]$.

More generally, for $K \subset H$ with $H/K$ affine, 
the functor $\Av_*:\sC^K \to \sC^H$ has cohomological amplitude
$[0,\dim H/K]$.
 
\end{lem}

\begin{proof}

$\Av_*$ is left $t$-exact because it is right adjoint
to a $t$-exact functor.

For the upper bound on the amplitude, note that
weak averaging from $\sC^{K,w} \to \sC^{H,w}$ is
$t$-exact because $H/K$ is affine. Observe that
weak averaging is given by convolution with
$D_{H/K}$, $*$-averaging is given by convolution
with the constant $D$-module $k_{H/K}$, and
then use the de Rham resolution of $k_{H/K}$,
which consists of free $D$-modules in degrees
$[0,\dim H/K]$, to complete the
argument.

\end{proof}

\subsection{$!$-averaging}

We now want a version of the above for $!$-averaging.
It is essentially the same, but slightly more subtle because
$!$-averaging may not be defined. 

Moreover,
the proof of Lemma \ref{l:av-*-bd} in the case where $H/K$
was affine used the fact that de Rham cohomology on an affine 
scheme is right $t$-exact. The corresponding fact for 
compactly supported de Rham cohomology is harder to 
show (for non-holonomic $D$-modules), and is the main theorem of \cite{b-fn}.

\subsection{}

Suppose in the above setting that are given $K_1,K_2$ two
subgroups of $H$, and characters
$\psi_i:K_i \to \bG_a$ that coincide on $K_1 \cap K_2$.
Suppose that for every $\sC \in \DGCat_{cont}$ acted
on by $H$, the functor $\Av_!^{\psi_2}:\sC^{K_1,\psi_1} \to \sC^{K_2,\psi_2}$
given by restricting to the intersection $K_1 \cap K_2$ and
then $!$-averaging is defined functorially in $\sC$.

\begin{lem}\label{l:!-avg-bd}

Suppose $\sC$ is acted on by $H$, and equipped with a $t$-structure
compatible with the $t$-structure. Suppose moreover that
the $t$-structure on $\sC$ is \emph{compactly generated},
i.e., $\sC^{\leq 0}$ is compactly generated 
(in the sense of general category theory).

Then under the above hypotheses, if $K_2/K_1 \cap K_2$
is affine,
then $\Av_!^{\psi_2}:\sC^{K_1,\psi_1} \to \sC^{K_2,\psi_2}$ has
cohomological amplitude
${[-\dim K_2/K_1 \cap K_2,0]}$.

\end{lem}

We need the following result, which 
appeared already as \cite{kernels} Lemma 4.1.3. We include
the proof for the reader's convenience.

\begin{lem}\label{l:left-exact}

Let $\sC \in \DGCat_{cont}$ be equipped with a compactly
generated $t$-structure. 
Let $F:\sD_1 \to \sD_2 \in \DGCat_{cont}$ be given,
and suppose that the categories $\sD_i$ are
equipped with $t$-structures,
and that $F$ is left $t$-exact. 
Then:

\[
\id_{\sC} \otimes F: \sC \otimes \sD_1 \to \sC \otimes \sD_2
\]

\noindent is left $t$-exact.

\end{lem}

\begin{proof}

Let $\sF \in \sC^{\leq 0}$ be compact. 
Let $\bD \sF:\sC \to \Vect$ denote the corresponding
continuous functor $\ul{\Hom}_{\sC}(\sF,-)$.
Note that $\bD \sF$ is left $t$-exact because
$\sF \in \sC^{\leq 0}$.

We have induced functors:

\[
\bD \sF \otimes \id_{\sD_i}: \sC \otimes \sD_i \to 
\Vect \otimes \sD_i = \sD_i.
\]

The main observation is that
$\sG \in \sC \otimes \sD_i$ lies in 
$(\sC \otimes \sD_i)^{\geq 0}$ if and only if:

\[
\bD \sF \otimes \id_{\sD_i} (\sG) \in \sD_i^{\geq 0}
\]

\noindent for all $\sF$ as above.
Indeed, for $\sH \in \sD_i$, the (possibly non-continuous)
composite functor:

\[
\sC \otimes \sD_i 
\xar{\bD \sF \otimes \id_{\sD_i}}
\sD_i \xar{\ul{\Hom}_{\sD_i}(\sH,-)} \Vect
\]

\noindent coincides with
$\ul{\Hom}_{\sC \otimes \sD_i}(\sF \boxtimes \sH,-)$,
as follows by observing that it is the right adjoint
to the functor $k \mapsto \sF \boxtimes \sH$.
Taking $\sH \in \sD_i^{\leq 0}$, this 
immediately implies the observation.

Therefore, we need to show
that for $\sG \in (\sC \otimes \sD_1)^{\geq 0}$, we have: 

\[
(\bD \sF \otimes \id_{\sD_2}) \circ 
(\id_{\sC} \otimes F) (\sG) \in \sD_2^{\geq 0}
\]

\noindent for all compact $\sF \in \sC^{\leq 0}$.
By functoriality, we have:
 
\[
(\bD \sF \otimes \id_{\sD_2}) \circ 
(\id_{\sC} \otimes F) (\sG) =
F \circ (\bD \sF \otimes \id_{\sD_1})(\sG).
\]

\noindent Because 
$(\bD \sF \otimes \id_{\sD_1})(\sG) \in \sD_1^{\geq 0}$
by the above, we obtain the claim from 
left $t$-exactness of $F$.

\end{proof}

\begin{proof}[Proof of Lemma \ref{l:!-avg-bd}]

The functor $\Av_!^{\psi_2}$ 
is right $t$-exact because it is a left adjoint to a $t$-exact functor.
So it remains to show the other bound.

First, suppose that $\sC = D(H)$. 
Then $D(H)^{K_1,\psi_1}$ is compactly generated
by coherent $D$-modules.
Therefore, for the $t$-structure on
$D(H)^{K_1,\psi_1}$, compact objects are closed under truncations.
So it suffices to show that every compact object
of $D(H)^{K_1,\psi_1,\geq 0}$ maps to 
$D(H)^{K_2,\psi_2,\geq -\dim K_2/K_1 \cap K_2}$.

This follows immediately from the fact that $!$-pushforward 
is left $t$-exact on coherent objects, which is Theorem 3.3.1
of \cite{b-fn}. 
(The cohomological
shift by $\dim K_2/K_1 \cap K_2$ 
arises because 
$!$-averaging is $!$-convolution with a dualizing $D$-module.)

For general $\sC$, we use the commutative diagram:

\[
\xymatrix{
\sC^{K_1,\psi_1} \ar[rr]^{\coact} \ar[d]^{\Av_!^{\psi_2}} & &
D(H)^{K_1,\psi_1} \otimes \sC \ar[d]^{\Av_!^{\psi_2} \otimes \id_{\sC}} \\
\sC^{K_2,\psi_2} \ar[rr]^{\coact} & &
D(H)^{K_2,\psi_2} \otimes \sC.
}
\]

\noindent The horizontal arrows are obviously
conservative and $t$-exact up to shift
(by assumption on the action on $\sC$),
while the right vertical arrow has the correct amplitude
by the above and Lemma \ref{l:left-exact}.
This immediately implies the same for the left vertical arrow.

\end{proof}

\begin{rem}

In the case $\sC = \fh\mod$, the argument given
amounts to using the Beilinson-Bernstein localization functor 
to pass from the Lie algebra to $D$-modules.

\end{rem}

\begin{rem}\label{r:ignorance}

The above works just as well when $H$ is a group scheme
and the $K_i$ are compact open subgroup schemes:
indeed, there is a normal compact open subgroup scheme
of $H$ contained in the $K_i$, reducing the problem to 
the finite-dimensional version. But it is not clear how to
show the lemma for $H$ being the loop group, since
$\coact$ is no longer $t$-exact up to shift (it maps into 
infinitely connective objects).

\end{rem}

\begin{rem}

The above works just as well in the setting of twisted
$D$-modules.

\end{rem}

\subsection{Geometric proof of Lemma \ref{l:amp} for 
$n$ large enough}\label{ss:amp-coxeter}

We will show Lemma \ref{l:amp} for
$n \geq (\check{\rho},\alpha_{max})$ (alias: the Coxeter
number of $G$ minus 1). 

The right exactness is immediately 
given by Lemma \ref{l:av-*-bd}.
The issue in applying Lemma \ref{l:!-avg-bd} is that
we need $\Av_!$ to be defined and functorial for
subgroups of a group scheme, not a group indscheme
such as $G(K)$.

But in the given range of $n$,
$\o{I}_n$ and $\o{I}_{n+1}$ are 
both contained in $\Ad_{-(n+1)\check{\rho}(t)} G(O)$.
Since the existence and functoriality of $\Av_!$ is really
about convolution identities, this means that for any category
strongly acted on by $\Ad_{-(n+1)\check{\rho}(t)} G(O)$,
we can $!$-average from $(\o{I}_n,\psi)$ to
$(\o{I}_{n+1},\psi)$, and this
$!$-averaging coincides with the $*$-averaging up to
the shift by $2\Delta$ from Theorem \ref{t:!-avg}.
Now Lemma \ref{l:!-avg-bd} applies
and gives the desired left $t$-exactness
for $m = n+1$, which evidently suffices.

\begin{rem}

If for $\sC$ we had $D$-modules on a reasonable indscheme
$X$ acted on by $G(K)$ (or the $\kappa$-twisted version of this notion),
then we could apply \cite{b-fn} directly, 
without needing the general Lemma \ref{l:!-avg-bd}.
That is, we would not need any restrictions on $n$.

\end{rem}

\subsection{Representation theoretic approach}\label{ss:amp-rep-thry-method}

We now indicate a representation theoretic approach
to treat Lemma \ref{l:amp} for all $n$.

\begin{proof}[Proof of Lemma \ref{l:amp}]

\setcounter{steps}{0}

\step\label{st:amp-cl}

Note that by the general formalism from Appendix \ref{a:hc},
$\iota_{n,m,*}:\Whit^{\leq n}(\widehat{\fg}_{\kappa}\mod) \to 
\Whit^{\leq m}(\widehat{\fg}_{\kappa}\mod)$ is filtered
for the KK filtration with associated semi-classical functor: 

\[
\QCoh^{ren}(f+\Lie \o{I}_n^{\perp}/\o{I}_n) \to
\QCoh^{ren}(f+\Lie \o{I}_m^{\perp}/\o{I}_m)
\]

\noindent given by push/pull along the correspondence:

\begin{equation}\label{eq:av-*-cl-corr}
\vcenter{\xymatrix{
& f+\Lie\o{I}_n^{\perp} \cap \Lie\o{I}_m^{\perp}/\o{I}_n \cap \o{I}_m \ar[dl] \ar[dr] & \\
f+\Lie\o{I}_n^{\perp}/\o{I}_n &&
f+\Lie\o{I}_m^{\perp}/\o{I}_m
}}
\end{equation} 

\noindent up to cohomological shift and a determinant
twist. The main observation is that this functor is $t$-exact.
(The ``up to cohomological shift" is compatible with the
shift by $(m-n)\Delta$ in Lemma \ref{l:amp}.)

Indeed, the pushforward in this correspondence is obviously
$t$-exact because the map is affine.
It remains to see that the left leg of the correspondence is flat
(and in fact, smooth).

This follows from the explicit description
of both sides from the proof of Theorem \ref{t:!-avg-cl}. 
Indeed, first say $n>0$ for simplicity. Then 
both sides are classifying stacks over
$f+t^{-n}\Ad_{-n\check{\rho}(t)}\fb^e[[t]]$ by \emph{loc. cit}.
Moreover, the relevant group schemes are congruence subgroups
of jets into the group scheme of regular centralizers. We then obtain
the claim from the smoothness
of that group scheme.

If $n = 0$ and $m>n$, then the relevant map
$f+\fg[[t]] \cap \Lie\o{I}_m^{\perp}/G(O) \cap \o{I}_m \to
\fg[[t]]/G(O)$ factors through
$\fg^{reg}(O)/G(O)$, which is the classifying stack
over $f+\fb^e[[t]]$ of jets into regular centralizers.
So the same analysis applies.

\step 

To show $\iota_{n,m,*}[(m-n)\Delta]$ is $t$-exact, it
suffices to show that it is left $t$-exact, since
Lemma \ref{l:av-*-bd} implies the right $t$-exactness.
For this, it suffices to show that it suffices to show that
for $\sF \in \widehat{\fg}_{\kappa}\mod^{\o{I}_n,\psi,\heart}$,
$\iota_{n,m,*}(\sF)[(m-n)\Delta]$ is also in cohomological
degree $0$. 

For $n>0$, it suffices to take $\sF$ to be a quotient
of $\ind_{\o{I}_n}^{\widehat{\fg}_{\kappa}}(\psi)$. Indeed,
such quotients generate the abelian category under
extensions and filtered colimits. Similarly, for $n = 0$,
it suffices to take $\sF$ to be a quotient of a Weyl module
(i.e., a quotient of 
$\bV_{\kappa}^{\lambda} \coloneqq \ind_{\fg[[t]]}^{\widehat{\fg}_{\kappa}}(V^{\lambda})$
for $\lambda \in \Lambda^+$ a dominant coweight, where
$V^{\lambda}$ is the highest weight representation 
of $G$, and is acted on by $\fg[[t]]$ through the quotient $\fg$).

Here is a wrong conclusion to the argument, which we correct
in what follows. The modules $\ind_{\o{I}_n}^{\widehat{\fg}_{\kappa}}(\psi)$
(resp. $\bV_{\kappa}^{\lambda}$)
have KK filtrations, so the quotient $\sF$ inherits one as well.
Therefore, $\iota_{n,m,*}(\sF)$ has a canonical filtration. 
By Step \ref{st:amp-cl},
$\Gr_{\dot} \iota_{n,m,*}(\sF)[(m-n)\Delta]$ is concentrated in cohomological
degree $0$. 

However, because the KK filtration on 
$\ind_{\o{I}_n}^{\widehat{\fg}_{\kappa}}(\psi)$
is not bounded below, 
\emph{it is not clear that the filtration on $\Psi(\sF)$
is bounded below in this case} (and probably it is not).
That is, the argument from the proof of Theorem \ref{t:ds-no-w-str}
does not adapt well to this setting. So we give a different method below, which essentially uses different bookkeeping to
avoid this issue.

\step Of course, it suffices to treat the case where $G$ is not a torus, so 
we assume this in one follows. We first additionally suppos that $n > 0$.
The main idea is to imitate the trick of using a (weighted)
$L_0$-grading from the proof of Theorem \ref{t:spr-fib}.

Let $h \in \bQ^{>1}$ be a rational number (greater than 1)
to be specified later. This choice defines a grading
on the Kac-Moody algebra with degrees lying in
$h\bZ \subset \bQ$ as follows.
Note that the Kac-Moody algebra has 
canonical $L_0\coloneqq t\partial_t$ and 
$\check{\rho}$-gradings. 
Consider it as equipped with the
grading $-h\check{\rho}-(h-1)L_0$
(so e.g., $t^i e_{\alpha}$ has degree
$-h(\check{\rho},\alpha)-(h-1)i$).

The subalgebras $\Lie \o{I}_n,\Lie\o{I}_m$ 
are obviously graded. Moreover, the
character $\psi:\Lie \o{I}_n \to k$ vanishes
on homogeneous components
apart from degree $-1$, so we can use the KK formalism
from Appendix \ref{a:hc}. Note that
there is no problem in using fractional indices, though
our filtration will be graded similarly. (Clearing
denominators, it is the same as renormalizing the
PBW filtration to have the same associated
graded, but with jumps only at multiples of the denominator
of $h$.) Let us refer to this as the
KK' filtration on $\widehat{\fg}_{\kappa}\mod$, etc. 
Note that if $h = 1$, this is recovering the usual
KK filtration.

A straightforward calculation shows we can 
take $h$ so that the induced KK' filtration on 
$\ind_{\Lie \o{I}_n}^{\widehat{\fg}_{\kappa}}(\psi)$
to be \emph{bounded from below} (it is essential that
$n>0$ here).

Of course, the same boundedness occurs
for the induced KK' filtration on $\sF$,
any quotient of our induced module.
 
It is straightforward to see that the induced KK' filtration on:

\[
C^{\dot}(\Lie\o{I}_m,\o{I}_n \cap \o{I}_m, \sF \otimes -\psi_{\o{I}_m})
\]

\noindent is then bounded from below as well. 
First, observe that (for any $h>1$) there is a compact
open subalgebra of $\Lie\o{I}_n$ on which the
$-h\check{\rho}-(h-1)L_0$-degrees are negative.
It follows that the degrees on
$\Lambda^i \big(\Lie\o{I}_n\big)^{\vee}$ are bounded
from below \emph{independently of $i$}, since a compact
open subalgebra has finite codimension. This shows that
the induced filtration on:

\[
C^{\dot}(\o{I}_n,\sF \otimes -\psi)
\]

\noindent is bounded from below, or similarly for
$\o{I}_n\cap \o{I}_m$.
The Harish-Chandra cohomology appearing above
differs from the latter group cohomology 
by tensoring with the exterior algebra
of $\Ad_{-m\check{\rho}(t)\fn[[t]]}/\Ad_{-n\check{\rho}(t)}\fn[[t]]$,
so the result follows.

This Chevalley complex computes:

\[
\ul{\Hom}_{\widehat{\fg}_{\kappa}\mod^{\o{I}_m}}(
\ind_{\Lie \o{I}_m}^{\widehat{\fg}_{\kappa}} \psi,
\iota_{n,m,*}(\sF))
)
\]

\noindent by definition of $\iota_{n,m,*}$ as $*$-averaging.
To compute the associated graded,
one takes 
$\gr_{\dot}^{KK'}(\sF) \in \QCoh(f+\Lie\o{I}_n^{\perp}/\o{I}_n)^{\heart}$,
applies pull-push along the correspondence
\eqref{eq:av-*-cl-corr}, applies the cohomological shift
by $(m-n)\Delta$ and the determinant twist, and then
applies global sections on the stack 
$f+\Lie\o{I}_m^{\perp}/\o{I}_m$. 

The upshot is that
the resulting object of $\Vect$ is in cohomological
degrees $\geq (m-n)\Delta$ by the exactness
of our pull-push operation and because of the cohomological
shift. This means the same is true for the Chevalley complex
above. Because
$\ind_{\Lie \o{I}_m}^{\widehat{\fg}_{\kappa}} \psi$
generates $\Whit^{\leq m}(\widehat{\fg}_{\kappa}\mod)^{\leq 0}$
under colimits, we finally obtain that
$\iota_{n,m,*}(\sF)$ is in cohomological degrees
$\geq (m-n)\Delta$, hence is in degree $(m-n)\Delta$,
as was desired.

\step 

It remains to define $h$ and check the desired boundedness.
For this, let $\alpha_{max}$ denote the highest root, 
and take:

\[
h \coloneqq
\frac{n(\check{\rho},\alpha_{max})}{1+(n-1)(\check{\rho},\alpha_{max})}.
\]

\noindent (E.g., for $n = 1$, $h$ is one less than the
Coxeter number of $G$.)

We want to see that KK' filtration on
$\ind_{\Lie \o{I}_n}^{\widehat{\fg}_{\kappa}}(\psi)$ is
bounded below: in fact, we will see 
$F_{-1}^{KK'} \ind_{\Lie \o{I}_n}^{\widehat{\fg}_{\kappa}}(\psi) = 0$.
It suffices to show that the non-zero graded degrees on:

\[
\gr_{\dot}^{KK'} \ind_{\Lie \o{I}_n}^{\widehat{\fg}_{\kappa}}(\psi) =
\Sym^{\dot}(\fg((t))/\Lie\o{I}_n)
\]

\noindent are $\geq 0$.
Note that in the notation from the proof
of Lemma \ref{l:alpha-epi},
this associated graded is an algebra
generated by elements $\frac{e_{\alpha}}{t^r}$
($r\geq n(\check{\rho},\alpha)+1$) and 
$\frac{f_{\beta}}{t^r}$ ($r\geq -n(\check{\rho},\beta)-n+1$),
which have gradings:

\[
\begin{gathered}
-h(\check{\rho},\alpha)+(h-1)r+1 \\
h(\check{\rho},\beta)+(h-1)r+1.
\end{gathered}
\]

\noindent We need to show that these numbers are
each $\geq 0$ for $\alpha$ (resp. $\beta$) a 
positive root (resp. or zero) and $r$ in the appropriate range.

Regarding the ``$\alpha$ inequality," note that:

\begin{equation}\label{eq:alpha-h-ineq}
h \geq
\frac{n(\check{\rho},\alpha)}{1+(n-1)(\check{\rho},\alpha)}.
\end{equation}

\noindent Then the bound on $r$ means
the KK' degree of $\frac{e_{\alpha}}{t^r}$ is:

\[
-h(\check{\rho},\alpha)+(h-1)r + 1 \geq 
-h(\check{\rho},\alpha) + (h-1)(n(\check{\rho},\alpha)+1) + 1 = 
h\big(1+(n-1)(\check{\rho},\alpha)\big)-n(\check{\rho},\alpha)
\]

\noindent which is non-negative by \eqref{eq:alpha-h-ineq}.

For the second inequality, first note that:\footnote{
Here the manipulation for the potentially dangerous value
$n = 1$ is obviously justified.
}

\begin{equation}\label{eq:beta-h-ineq}
(n-1) h = \frac{n(n-1)}
{\frac{1}{(\check{\rho},\alpha_{max})}+n-1} <
\frac{n(n-1)}{n-1} = n.
\end{equation}

\noindent Then the bound on
$r$ gives the degree of $\frac{f_{\beta}}{t^r}$ as:

\[
h(\check{\rho},\beta) + (h-1) r + 1 \geq 
h(\check{\rho},\beta) + (h-1) \big(-n(\check{\rho},\beta)-n+1\big) +1 =
(\check{\rho},\beta)\big(n-(n-1)h\big)
\]

\noindent which is non-negative by \eqref{eq:beta-h-ineq}
(recall our normalization that $\beta$ is $0$ or a \emph{positive} 
root).

\step 

Finally, we treat the case $n = 0$. Here are three arguments.

Observe that (e.g. by 
Theorems \ref{t:aff-skry} and \ref{t:psi-exact/cons}),
it suffices to show that 
$\Psi:\widehat{\fg}_{\kappa}\mod^{G(O)} \to \Vect$
is $t$-exact.

First, this result can be found in the literature:
at non-critical level, this is \cite{fg-weyl} Proposition 2
plus the Sugawara construction, and
at critical level this is \cite{fg-spherical} Theorem 3.2.

Second, one can organize the above differently:
\cite{fg-weyl} uses Arakawa exactness in an essential
way, and our generalization Corollary \ref{c:arakawa}
of it, which removes the use of the 
\emph{extended} affine Kac-Moody algebra, allows
one to use the Frenkel-Gaitsgory method directly.

Finally, note that any object of 
$\widehat{\fg}_{\kappa}\mod^{G(O)}$ has a 
$\check{\rho}$-grading, and morphisms
preserve these gradings.
Therefore, 
$\sF$ ($\coloneqq $ a quotient of $\bV_{\kappa}^{\lambda}$)
has canonical $\check{\rho}$-gradings,
and also inherits PBW and KK filtrations from
$\bV_{\kappa}^{\lambda}$. These satisfy the usual
compatibility in the KK formalism. Therefore, we can apply the
method from Theorem \ref{t:ds-no-w-str} to obtain the
desired result.

\end{proof}

\bibliography{bibtex}{}

\begin{thebibliography}{FGKV}

\bibitem[AB]{arkhipov-bezrukavnikov}
Sergey Arkhipov and Roman Bezrukavnikov.
\newblock {Perverse sheaves on affine flags and {L}anglands dual group}.
\newblock {\em Israel J. Math.}, 170:135--183, 2009.
\newblock With an appendix by Bezrukavrikov and Ivan Mirkovi{\'c}.

\bibitem[AG]{dmod-loopgroup}
Sergey Arkhipov and Dennis Gaitsgory.
\newblock Differential operators on the loop group via chiral algebras.
\newblock {\em International Mathematics Research Notices}, 2002(4):165--210,
  2002.

\bibitem[Ara1]{arakawa-report}
Tomoyuki Arakawa.
\newblock {Infinite dimensional Lie algebras, vertex algebras and
  $\sW$-algebras}.
\newblock {\em 数理解析研究所講究録}, 1502:96--106, 2006.

\bibitem[Ara2]{arakawa-rep-thry}
Tomoyuki Arakawa.
\newblock Representation theory of {$\sW$}-algebras.
\newblock {\em Invent. Math.}, 169(2):219--320, 2007.

\bibitem[Ara3]{arakawa-survey}
Tomoyuki Arakawa.
\newblock {Introduction to $\sW$-algebras and their representation theory}.
\newblock {\em arXiv preprint arXiv:1605.00138}, 2016.

\bibitem[BBM]{bbm}
Roman Bezrukavnikov, Alexander Braverman, and Ivan Mirkovic.
\newblock Some results about geometric {W}hittaker model.
\newblock {\em Adv. Math.}, 186(1):143--152, 2004.

\bibitem[BD1]{hitchin}
Sasha Beilinson and Vladimir Drinfeld.
\newblock Quantization of {H}itchin's integrable system and {H}ecke
  eigensheaves.
\newblock Available at
  \url{http://www.math.uchicago.edu/~mitya/langlands/hitchin/BD-hitchin.pdf}.

\bibitem[BD2]{chiral}
Sasha Beilinson and Vladimir Drinfeld.
\newblock {\em Chiral algebras}, volume~51 of {\em American Mathematical
  Society Colloquium Publications}.
\newblock American Mathematical Society, Providence, RI, 2004.

\bibitem[Bei]{beilinson-top-alg}
Sasha Beilinson.
\newblock Remarks on topological algebras.
\newblock {\em Moscow Math. Journ}, 8(1):1--20, 2008.

\bibitem[Ber]{dario-*/!}
Dario Beraldo.
\newblock {Loop group actions on categories and Whittaker invariants}.
\newblock {\em arXiv preprint arXiv:1310.5127}, 2013.

\bibitem[Bez]{perverse-coherent}
Roman Bezrukavnikov.
\newblock {Perverse coherent sheaves (after Deligne)}.
\newblock {\em arXiv preprint math/0005152}, 2000.

\bibitem[BL]{bernstein-lunts}
Joseph Bernstein and Valery Lunts.
\newblock {Localization for derived categories of $(\mathfrak{g}, K)$-modules}.
\newblock {\em Journal of the American Mathematical Society}, 8(4):819--856,
  1995.

\bibitem[BO]{bershadsky-ooguri}
Michael Bershadsky and Hirosi Ooguri.
\newblock Heiddensl (n) symmetry in conformal field theories.
\newblock {\em Communications in Mathematical Physics}, 126(1):49--83, 1989.

\bibitem[BS]{w-symmetry}
Peter Bouwknegt and Kareljan Schoutens.
\newblock {$\sW$-symmetry in conformal field theory}.
\newblock {\em Physics Reports}, 223(4):183--276, 1993.

\bibitem[dBT]{dbt}
Jan de~Boer and Tjark Tjin.
\newblock {The relation between quantum $\sW$-algebras and Lie algebras}.
\newblock {\em Communications in Mathematical Physics}, 160(2):317--332, 1994.

\bibitem[DS]{drinfeld-sokolov}
V.~G. Drinfeld and V.~V. Sokolov.
\newblock Lie algebras and equations of {K}orteweg-de {V}ries type.
\newblock In {\em Current problems in mathematics, {V}ol. 24}, Itogi Nauki i
  Tekhniki, pages 81--180. Akad. Nauk SSSR, Vsesoyuz. Inst. Nauchn. i Tekhn.
  Inform., Moscow, 1984.

\bibitem[DSK]{kac-desole}
Alberto De~Sole and Victor~G Kac.
\newblock {Finite vs affine $\sW$-algebras}.
\newblock {\em Japanese Journal of Mathematics}, 1(1):137--261, 2006.

\bibitem[FBZ]{fbz}
Edward Frenkel and David Ben-Zvi.
\newblock {\em Vertex algebras and algebraic curves}.
\newblock Number~88. American Mathematical Soc., 2004.

\bibitem[FF1]{ff-ds}
Boris Feigin and Edward Frenkel.
\newblock {Quantization of the Drinfeld-Sokolov reduction}.
\newblock {\em Physics Letters B}, 246(1):75--81, 1990.

\bibitem[FF2]{feigin-frenkel-duality}
Boris Feigin and Edward Frenkel.
\newblock {Duality in $\sW$-algebras}.
\newblock {\em International Mathematics Research Notices}, 1991(6):75--82,
  1991.

\bibitem[FF3]{ff-critical}
Boris Feigin and Edward Frenkel.
\newblock {Affine Kac-Moody algebras at the critical level and Gelfand-Dikii
  algebras}.
\newblock {\em International Journal of Modern Physics A}, 7:197--215, 1992.

\bibitem[FG1]{fg2}
Edward Frenkel and Dennis Gaitsgory.
\newblock {Local geometric Langlands correspondence and affine Kac-Moody
  algebras}.
\newblock In {\em Algebraic geometry and number theory}, pages 69--260.
  Springer, 2006.

\bibitem[FG2]{fg-wakimoto}
Edward Frenkel and Dennis Gaitsgory.
\newblock {Geometric realizations of Wakimoto modules at the critical level}.
\newblock {\em Duke Mathematical Journal}, 143(1):117--203, 2008.

\bibitem[FG3]{dmod-aff-flag}
Edward Frenkel and Dennis Gaitsgory.
\newblock {$D$}-modules on the affine flag variety and representations of
  affine {K}ac-{M}oody algebras.
\newblock {\em Represent. Theory}, 13:470--608, 2009.

\bibitem[FG4]{fg-spherical}
Edward Frenkel and Dennis Gaitsgory.
\newblock Local geometric {L}anglands correspondence: the spherical case.
\newblock In {\em Algebraic analysis and around}, volume~54 of {\em Adv. Stud.
  Pure Math.}, pages 167--186. Math. Soc. Japan, Tokyo, 2009.

\bibitem[FG5]{fg-weyl}
Edward Frenkel and Dennis Gaitsgory.
\newblock Weyl modules and opers without monodromy.
\newblock In {\em Arithmetic and Geometry Around Quantization}, pages 101--121.
  Springer, 2010.

\bibitem[FGKV]{fgkv}
Edward Frenkel, Dennis Gaitsgory, David Kazhdan, and Kari Vilonen.
\newblock Geometric realization of {W}hittaker functions and the {L}anglands
  conjecture.
\newblock {\em J. Amer. Math. Soc.}, 11(2):451--484, 1998.

\bibitem[FGV]{fgv}
Edward Frenkel, Dennis Gaitsgory, and Kari Vilonen.
\newblock Whittaker patterns in the geometry of moduli spaces of bundles on
  curves.
\newblock {\em The Annals of Mathematics}, 153(3):699--748, 2001.

\bibitem[Fre1]{frenkel-wakimoto}
Edward Frenkel.
\newblock Wakimoto modules, opers and the center at the critical level.
\newblock {\em Advances in Mathematics}, 195(2):297--404, 2005.

\bibitem[Fre2]{frenkel-book}
Edward Frenkel.
\newblock {\em Langlands correspondence for loop groups}, volume 103.
\newblock Cambridge University Press, 2007.

\bibitem[FZ]{zamolodchikov-fateev}
V.~A. Fateev and A.~B. Zamolodchikov.
\newblock Conformal quantum field theory models in two dimensions having
  {$Z_3$} symmetry.
\newblock {\em Nuclear Phys. B}, 280(4):644--660, 1987.

\bibitem[Gai1]{qcoh}
Dennis Gaitsgory.
\newblock {Quasi-coherent sheaves on stacks}.
\newblock Available at \url{http://math.harvard.edu/~gaitsgde/GL/QCohtext.pdf}.

\bibitem[Gai2]{quantum-langlands-summary}
Dennis Gaitsgory.
\newblock {Quantum Langlands correspondence}.
\newblock {\em arXiv preprint arXiv:1601.05279}, 2007.

\bibitem[Gai3]{lesjours}
Dennis Gaitsgory.
\newblock Whittaker categories.
\newblock Available at
  \url{http://math.harvard.edu/~gaitsgde/GL/LocalWhit.pdf}. Formerly titled
  \emph{Les jours et les travaux}, 2008.

\bibitem[Gai4]{dgcat}
Dennis Gaitsgory.
\newblock Generalities on {DG} categories.
\newblock Available at \url{http://math.harvard.edu/~gaitsgde/GL/textDG.pdf},
  2012.

\bibitem[Gai5]{indcoh}
Dennis Gaitsgory.
\newblock Ind-coherent sheaves.
\newblock {\em Moscow Mathematical Journal}, 13(3):399--528, 2013.

\bibitem[Gai6]{km-indcoh}
Dennis Gaitsgory.
\newblock {Kac-Moody representations (Day IV, Talk 3)}.
\newblock Available at
  \url{https://sites.google.com/site/geometriclanglands2014/notes}. Unpublished
  conference notes, 2014.

\bibitem[Gai7]{shvcat}
Dennis Gaitsgory.
\newblock Sheaves of categories and the notion of 1-affineness.
\newblock In T.~Pantev, C.~Simpson, B.~To{\"e}n, M.~Vaqui{\'e}, and G.~Vezzosi,
  editors, {\em Stacks and Categories in Geometry, Topology, and Algebra:},
  Contemporary Mathematics. American Mathematical Society, 2015.

\bibitem[Gai8]{kernels}
Dennis Gaitsgory.
\newblock Functors given by kernels, adjunctions and duality.
\newblock {\em Journal of Algebraic Geometry}, 25(3):461--548, 2016.

\bibitem[GG]{gan-ginzburg}
Wee~Liang Gan and Victor Ginzburg.
\newblock {Quantization of Slodowy slices}.
\newblock {\em International Mathematics Research Notices}, 2002(5):243--255,
  2002.

\bibitem[GR1]{ainfty}
Dennis Gaitsgory and Sam Raskin.
\newblock Acyclic complexes and 1-affineness.
\newblock Available at \url{math.mit.edu/~sraskin/Ainfty.pdf}, 2015.

\bibitem[GR2]{grbook}
Dennis Gaitsgory and Nick Rozenblyum.
\newblock {\em A study in derived algebraic geometry}.
\newblock Available at \url{http://math.harvard.edu/~gaitsgde}.

\bibitem[GR3]{indschemes}
Dennis Gaitsgory and Nick Rozenblyum.
\newblock {DG indschemes}.
\newblock {\em Contemp. Math}, 610:139--251, 2014.

\bibitem[Kos1]{kostant-slice}
Bertram Kostant.
\newblock Lie group representations on polynomial rings.
\newblock {\em American Journal of Mathematics}, 85(3):327--404, 1963.

\bibitem[Kos2]{kostant-whittaker}
Bertram Kostant.
\newblock {On Whittaker vectors and representation theory}.
\newblock {\em Inventiones mathematicae}, 48(2):101--184, 1978.

\bibitem[Los]{losev-survey}
Ivan Losev.
\newblock {Finite $\sW$-algebras}.
\newblock In {\em Proceedings of the International Congress of Mathematicians},
  volume 901. World Scientific, 2010.

\bibitem[Mir]{mirkovic-notes}
Ivan Mirkovic.
\newblock {Zastava spaces}.
\newblock Available at
  \url{http://people.math.umass.edu/~mirkovic/A.Notes/xx.LoopGrassmannians/BC.ZastavaSpaces.pdf}.

\bibitem[Pre]{premet}
Alexander Premet.
\newblock Special transverse slices and their enveloping algebras.
\newblock {\em Advances in Mathematics}, 170(1):1--55, 2002.

\bibitem[Ras1]{chiralcats}
Sam Raskin.
\newblock {Chiral categories}.
\newblock Available at \url{math.mit.edu/~sraskin/chiralcats.pdf}, 2015.

\bibitem[Ras2]{cpsi}
Sam Raskin.
\newblock {Chiral principal series categories I: finite dimensional
  calculations}.
\newblock Available at \url{http://math.mit.edu/~sraskin/cpsi.pdf}, 2015.

\bibitem[Ras3]{dmod}
Sam Raskin.
\newblock {$D$-modules on infinite dimensional varieties}.
\newblock Available at \url{http://math.mit.edu/~sraskin/dmod.pdf}, 2015.

\bibitem[Ras4]{locsys}
Sam Raskin.
\newblock {On the notion of spectral decomposition in local geometric
  Langlands}.
\newblock arXiv preprint arXiv:1511.01378, 2015.

\bibitem[Ras5]{b-fn}
Sam Raskin.
\newblock {A generalization of the $b$-function lemma}.
\newblock Available at \url{http://math.mit.edu/~sraskin/b-function.pdf}, 2016.

\bibitem[Ras6]{cpsii}
Sam Raskin.
\newblock {Chiral principal series categories II: the factorizable Whittaker
  category}.
\newblock Available at \url{math.mit.edu/~sraskin/cpsii.pdf}, 2016.

\bibitem[Rod]{rodier}
Fran{\c{c}}ois Rodier.
\newblock {Mod{\`e}le de Whittaker et caract{\`e}res de repr{\'e}sentations}.
\newblock In {\em Non-commutative harmonic analysis}, pages 151--171. Springer,
  1975.

\bibitem[Wan]{wang}
Weiqiang Wang.
\newblock {Nilpotent orbits and finite $\sW$-algebras}.
\newblock {\em arXiv preprint arXiv:0912.0689}, 2009.

\bibitem[Zam]{zamolodchikov}
A.~B. Zamolodchikov.
\newblock Infinite extra symmetries in two-dimensional conformal quantum field
  theory.
\newblock {\em Teoret. Mat. Fiz.}, 65(3):347--359, 1985.

\end{thebibliography}
\bibliographystyle{alphanum}

\end{document}